\documentclass[11pt,letterpaper]{amsart}
\usepackage{graphicx}
\usepackage{algorithm}
\usepackage{algpseudocode}
\usepackage{amsmath}
\usepackage{amsthm}
\usepackage{amsfonts}
\usepackage{amssymb}
\usepackage{bbm}
\usepackage{bm}
\usepackage{mathrsfs}
\usepackage{mathtools}
\usepackage{thmtools}

\usepackage{mathrsfs}

\usepackage{tabulary}
\usepackage{booktabs}

\usepackage{fancyref}
\usepackage{hyperref}
\usepackage{listings}
\usepackage{nameref}
\usepackage[section]{placeins}
\usepackage{soul}
\usepackage[normalem]{ulem}
\hypersetup{colorlinks=true, linkcolor=black}
\usepackage{xcolor}

\usepackage{lipsum}

\newtheorem{thm}{Theorem}[section]
\newtheorem{lem}[thm]{Lemma}
\newtheorem{cor}[thm]{Corollary}
\newtheorem{defi}[thm]{Definition}

\newtheorem{prop}[thm]{Proposition}

\theoremstyle{remark}
\newtheorem{rmk}[thm]{Remark}

\declaretheoremstyle[headfont=\normalfont]{normalhead}

\newcommand{\rd}{\mathrm{d}}
\newcommand{\N}{\mathbb{N}}
\newcommand{\Z}{\mathbb{Z}}

\newcommand{\R}{\mathbb{R}}

\newcommand{\T}{\mathbb{T}}
\newcommand{\E}{\mathbb{E}}
\newcommand{\Pb}{\mathbb{P}}
\newcommand{\D}{\mathbb{D}}
\newcommand{\I}{\mathbb{I}}
\newcommand{\defeq}{\vcentcolon=}
\newcommand{\eqdef}{=\vcentcolon}

\newcommand{\wsto}{\overset{\ast}{\rightharpoonup}}

\newcommand{\Law}{\mathrm{Law}}

\def\XXint#1#2#3{{\setbox0=\hbox{$#1{#2#3}{\int}$ }
\vcenter{\hbox{$#2#3$ }}\kern-.6\wd0}}

\graphicspath{ {./HW1/} }
\numberwithin{equation}{section}

\title[Coupling of Kinetic and Graph]{Coupling and Tensorization of \\ Kinetic Theory and Graph Theory}
\author{Datong Zhou}
\address{Laboratoire Jacques-Louis Lions and Laboratoire de Probabilités, Statistique et Modélisation, Sorbonne Université, 75005 Paris, France}
\email{datong.zhou@sorbonne-universite.fr}
\thanks{This project has received funding from the European Union’s Horizon 2020 research and innovation programme under the Marie Skłodowska-Curie grant agreement No. 101034255.}

\begin{document}

\maketitle

\begin{abstract}

We study a non-exchangeable multi-agent system and rigorously derive a strong form of the mean-field limit. The convergence of the connection weights and the initial data implies convergence of large-scale dynamics toward a deterministic limit given by the corresponding extended Vlasov PDE, at any later time and any realization of randomness. This is established on what we call a bi-coupling distance defined through a convex optimization problem, which is an interpolation of the optimal transport between measures and the fractional overlay between graphs. The proof relies on a quantitative stability estimate of the so-called observables, which are tensorizations of agent laws and graph homomorphism densities. This reveals a profound relationship between mean-field theory and graph limiting theory, intersecting in the study of non-exchangeable systems.

\end{abstract}

\tableofcontents

\section{Introduction}

Graph theory and kinetic theory are two prominent disciplines of mathematics, each with its unique methodologies and applications. The purpose of this article is to draw the attention of both communities to a promising intersection that has been less explored. Specifically, the concepts and techniques from graph theory have significant potential for investigating non-exchangeable multi-agent systems, where the connection weights $(w_{i,j})_{i,j=1}^N$ can be naturally interpreted as a \emph{weighted, directed graph}.

A canonical example of non-exchangeable systems is the following:

Consider a domain $\D$, which is either $\R^d$ or $\T^d$ (the $d$-dimensional torus). Let $N \geq 1$ denote the number of agents and $X_i(t) \in \D$ denote the state of the $i$-th agent. Let $\mu: \D \to \R^d$ denote the velocity field representing the individual dynamics, and let $\sigma: \D \times \D \to \R^d$ be the interaction kernel describing the binary interactions between agents. Furthermore, let $(w_{i,j})_{i,j=1}^N$ be a matrix that quantifies the intensity of the interaction exerted by the $j$-th agent on the $i$-th agent, scaled by a factor of $1/N$. This setup leads to the following system of ordinary differential equations (ODEs):
\begin{equation*} \label{eqn:multi_agent_ODE}
\begin{aligned}
\frac{\rd X_i}{\rd t} = \mu(X_i) + \frac{1}{N} \sum_{j \in [N]} w_{i,j} \sigma(X_i,X_j), \quad \forall i \in [N].
\end{aligned}
\end{equation*}
Here, for any natural number $N$, we use the notation $[N] = \{1, \dots, N\}$. This notation will be used consistently throughout the article.

The framework can be naturally extended to stochastic multi-agent systems by incorporating additive noise. Specifically, we consider the following system of stochastic differential equations (SDEs):
\begin{equation} \label{eqn:multi_agent_SDE}
\begin{aligned}
\rd X_i = \mu(X_i)\, \rd t + \frac{1}{N} \sum_{j \in [N]} w_{i,j} \sigma(X_i,X_j)\, \rd t + \nu\, \rd B_i, \quad \forall i \in [N].
\end{aligned}
\end{equation}
In this system, $B_i(t)$ for $i \in [N]$ are independent Brownian motions, and the constant $\nu \geq 0$ quantifies the intensity of the noise. When $\nu = 0$, \eqref{eqn:multi_agent_SDE} reduces to the ODE system.

\subsection{A first result: coupling of mean-field limit and graph limit}

We begin by introducing the foundational definitions and presenting our first theorem, which demonstrates the interplay between mean-field theory and graph limit theory concepts in the analysis of multi-agent systems.
\begin{defi} [\textbf{Bi-coupling distances, discrete}] \label{defi:coupling_distance_1}
A $(N_1,N_2)$-coupling is defined as a matrix $\gamma = (\gamma_{i_1,i_2})_{i_1 \in [N_1], i_2 \in [N_2]} \in \R^{N_1 \times N_2}$ such that all entries satisfy $\gamma_{i_1,i_2} \geq 0$ and
\begin{equation*}
\begin{aligned}
\sum_{i_1 \in [N_1]} \gamma_{i_1,i_2} = \frac{1}{N_2}, \quad \forall i_2 \in [N_2],
\\
\sum_{i_2 \in [N_2]} \gamma_{i_1,i_2} = \frac{1}{N_1}, \quad \forall i_1 \in [N_1].
\end{aligned}
\end{equation*}
Let $\Pi(N_1,N_2)$ denote the set of all $(N_1,N_2)$-couplings.

For $n = 1,2$, consider
\begin{equation*}
\begin{aligned}
w^{(n)} = (w_{i,j}^{(n)})_{i,j \in [N_n]} \in \R^{N_n \times N_n} \text { and } X^{(n)} = (X_i^{(n)})_{i \in [N_n]} \in \D^{N_n}.
\end{aligned}
\end{equation*}
The bi-coupling distance from the $L^2 \to L^2$ operator norm and the Wasserstein-1 disntance for the pairs $(w^{(n)}, X^{(n)})$ is defined as
\begin{equation*}
\begin{aligned}
&d_{L^2 \to L^2,W_1} \big[ (w^{(1)}, X^{(1)}), (w^{(2)}, X^{(2)})\big]
\\
& \ \defeq \inf_{\gamma \in \Pi(N_1,N_2)} \bigg[ \sum_{i_1 \in [N_1]} \sum_{i_2 \in [N_2]} |X^{(1)}_{i_1} - X^{(2)}_{i_2}| \, \gamma_{i_1,i_2}
\\
& \quad + \frac{1}{\sqrt{N_1 N_2}}\Big( \|N_2 \cdot w^{(1)} \gamma - N_1\cdot \gamma w^{(2)} \|_{\ell^2 \to \ell^2} + \|N_2 \cdot \gamma^T w^{(1)} - N_1\cdot w^{(2)} \gamma^T \|_{\ell^2 \to \ell^2}\Big) \bigg],
\end{aligned}
\end{equation*}
where $\gamma^T$ denotes the transpose of $\gamma$, and the products such as $(w^{(1)} \gamma) \in \R^{N_1 \times N_2}$ should be understood as the standard matrix multiplication.
\end{defi}

The minimization problem in the definition is remarkably convex.
When $w^{(1)}, w^{(2)}$ are all-one matrices, the second term vanishes and this bi-coupling distance degenerates to the Wasserstein distance in the Kantorovich formulation between empirical measures $\frac{1}{N_1} \sum_{i \in [N_1]} \delta_{X^{(1)}_{i}}$ and $\frac{1}{N_2} \sum_{i \in [N_2]} \delta_{X^{(2)}_{i}} \in \mathcal{P}(\D)$, which has been extensively used for decades in the development of mean-field theory \cite{braun1977vlasov, dobrushin1979vlasov, neunzert2006introduction, sznitman1991topics}. For Sinkhorn's matrix scaling algorithm of the approximate solutions, see \cite{cuturi2013sinkhorn}. On the other hand, when the entries of $X^{(1)}, X^{(2)}$ are all identical, the first term vanishes. The remaining part of this bi-coupling distance is minorly modified from the distance studied in \cite{boker2021graph, dell2018lovasz, grebik2022fractional} that identifies the fractional isomorphism of graphs, whose theoretical foundation traces back to earlier works such as \cite{ramana1994fractional, tinhofer1986graph, tinhofer1991note}.
As noted in \cite{boker2021graph}, approximate solutions can be derived using interior-point methods in convex optimization; for a detailed treatment, see \cite{nesterov1994interior}.
In graph theory, $\gamma \in \Pi(N_1,N_2)$ is also called a \emph{fractional overlay}.

With these definitions in place, we now present our first theorem.
\begin{thm} [\textbf{Main result, discrete only}] \label{thm:main_discrete}
Let $\D = \T$, and let $\mu \in W^{1,\infty}(\T)$, $\sigma \in W^{2,\infty}(\T \times \T)$, $\nu \geq 0$. Let $\{w^{(n)}\}_{n \in \N}$ be a sequence of connection weight matrices $w^{(n)} \in \R^{N_n \times N_n}$, and let $\{X^{(n)}_0\}_{n \in \N}$ be a sequence of initial data $X^{(n)}_0 \in \T^{N_n}$. Assume that $N_n \to \infty$, $\sup_{n} \sup_{i,j} |w^{(n)}_{i,j}| < \infty$, and that the initial data $X^{(n)}_0$ are deterministic. Then, the following holds:
\begin{itemize}
\item (Well-posedness). For each $n \in \N$, there exists a unique solution $X^{(n)}(t)$ of the ODE or SDE system \eqref{eqn:multi_agent_SDE} with connection weights $w^{(n)}$ and initial data $X^{(n)}_0$.
(When $\nu > 0$, $X^{(n)}(t)$ are stochastic processes depending on the randomness of independent Brownian motions $(B^{(n)}_{i})_{i \in [N_n]}$.)

\item (Compactness). There exists a Cauchy subsequence of $(w^{(n)}, X^{(n)}_0)$ (which we still index by $n$) in the sense that 
\begin{equation*}
\begin{aligned}
\lim_{n \to \infty} \sup_{n_1,n_2 \geq n} d_{L^2 \to L^2,W_1} \big[ (w^{(n_1)}, X^{(n_1)}_0), (w^{(n_2)}, X^{(n_2)}_0)\big] = 0.
\end{aligned}
\end{equation*}

\item (Stability). If the sequence $(w^{(n)}, X^{(n)}_0)$ is Cauchy, then at any later time $t > 0$,
\begin{equation*}
\begin{aligned}
\lim_{n \to \infty} \sup_{n_1,n_2 \geq n} \E \Big( d_{L^2 \to L^2,W_1} \big[ (w^{(n_1)}, X^{(n_1)}(t)), (w^{(n_2)}, X^{(n_2)}(t))\big] \Big) = 0.
\end{aligned}
\end{equation*}

\end{itemize}

\end{thm}

The full version of our main result is provided later as Theorem~\ref{thm:main_limit}. Informally, it states that ``If $(w^{(n)}_0, X^{(n)}_0)$ initially converges to a deterministic limit, denoted by $(w, f_0)$, where $w \in L^\infty(\I \times \I)$ and $f_0 \in L^\infty(\I, \mathcal{M}(\D))$ (with $\D = \T$) for some atomless standard probability space $\I$, then there exists a unique distributional solution $f$ to an extended Vlasov equation
\begin{equation} \label{eqn:multi_agent_Vlasov_first}
\begin{aligned}
& \frac{\partial}{\partial t}f(t,x,\xi) + \nabla_x \cdot \left( \left[ \mu(x) + \int_{\I} w(\xi,\zeta) \int_{\D} \sigma(x, y) f(t, y, \zeta) \, \rd y \rd \zeta \right] f(t, x, \xi) \right)
\\
& \hspace{8.5cm} - \frac{\nu^2}{2} \Delta_x f(t,x,\xi) = 0,
\end{aligned}
\end{equation}
with initial data $f_0$; and for any $t \geq 0$, $(w^{(n)}, X^{(n)}(t))$ converges in expectation to $(w, f(t))$.''
Rigorously defining such convergence to the limit requires significantly more definitions and notation. Thus, we pause here and postpone the statement of Theorem~\ref{thm:main_limit} until the end of the introduction.

\subsection{Summary of key innovations}

We take this opportunity to summarize the key innovations of our work, some of which may involve specific background knowledge that will be provided in later sections. Our hope is that readers from different backgrounds can be inspired by at least part of these ideas before proceeding further on to their first reading.

\begin{itemize}

\item \textbf{(Bi-coupling).}
Our bi-coupling distance takes an interpolative form, aligning two non-exchangeable systems $(w^{(1)},X^{(1)})$ and $(w^{(2)},X^{(2)})$ by a coupling measure $\gamma$.
This stands in sharp contrast to the existing literature (see, for example, \cite{bayraktar2023graphon, bet2024weakly, chiba2018mean, jabin2024non, kaliuzhnyi2018mean}) establishing the mean-field limit of the form \eqref{eqn:multi_agent_Vlasov_first}, which essentially relied on a ``bijective'' correspondence between agents \emph{a priori} established according to $w^{(1)}$, $w^{(2)}$ (see Section~\ref{subsec:grpahon}), and studied the laws of $X^{(1)}$, $X^{(2)}$ using the method of characteristics.
Our approach presents several advantages that are related:

\begin{itemize}

\item \textbf{(Convexity).}
The minimization problem associated with the bi-coupling distance is \emph{convex}. This is much more computationally feasible compared to existing results, which typically require performing highly non-convex optimizations to find the best correspondence between $w^{(1)}$, $w^{(2)}$.

\item \textbf{(Weaker requirements).}
The bi-coupling distance provides a much less restrictive condition for identifying similar large-scale behavior.
For example, two multi-agent systems with connection weight graphs, which are merely \emph{fractionally isomorphic} but not isomorphic, are considered very different according to previous results. However, our results suggest that their dynamics can be recognized as identical when having suitable initial data.
For an in-depth discussion of the difference between isomorphism and fractional isomorphism of graphs, we refer the readers to \cite{dell2018lovasz, grebik2022fractional} and references therein.

\item \textbf{(Empirical data).}
The bi-coupling distance characterizes the mean-field limit on the level of \emph{empirical data}. More precisely, it gives conditions under which two systems $(w^{(1)},X^{(1)})$ and $(w^{(2)},X^{(2)})$ actually lead to similar dynamics in each realization of randomness.
This is stronger than the previous results on law convergence and, as discussed in Section~\ref{subsec:posteriori}, seems incompatible with any a priori correspondence between the agents. However, it aligns naturally with the \emph{a posteriori correspondence} viewpoint suggested by the interpolative form of the bi-coupling distance.

Remarkably, the asymptotic independence of agents $(X_i)_{i \in [N]}$, which is critical in the context of a priori correspondence, does not appear explicitly in our discussion.
Instead, a much weaker form of independence seems to be integrated into the main result.

\end{itemize}

\item \textbf{(Connection to graph limit theory).}
The topology induced by the bi-coupling distance is equivalent to the weak-* convergence of an infinite hierarchy of measures, referred to in this article as \emph{observables}. Each observable is a weighted sum of empirical data $(X_i)_{i \in [N]}$, whose weights are carefully chosen according to the connection weights $(w_{i,j})_{i,j \in [N]}$.
Similar observables, defined from the marginal laws of $(X_i)_{i \in [N]}$, have been investigated in \cite{jabin2025mean, jabin2023mean}.
Nevertheless, our bi-coupling distance is novel, and its underlying mathematical structure is noteworthy:

We reformulate the observables as a tensorized extension of \emph{graph homomorphism densities} (see \cite{lovasz2012large}), by considering a graph whose edge ``weights'', depending on both $(w_{i,j})_{i,j \in [N]}$ and $(X_i)_{i \in [N]}$, are defined on a certain Hilbert space.
With this interpretation, by properly extending the Counting Lemma and Inverse Counting Lemma \cite{boker2021graph, borgs2008convergent, grebik2022fractional, lovasz2012large}, we reveal a delicate correspondence between the hierarchy of observables and the low-dimensional data $(w_{i,j})_{i,j \in [N]}$ and $(X_i)_{i \in [N]}$, which consequently leads to our proposal of the bi-coupling distance.

It should also be noted that our results cannot fully recover those of \cite{jabin2025mean, jabin2023mean}, as they consider potentially sparser connection weights that are not discussed in this article.

\end{itemize}

We now provide a far-from-complete overview of many-particle and multi-agent systems, focusing as much as possible on the aspects relevant to the contributions of this article.

\subsection{Many-particle systems}
The classical and relatively well-understood models are those assuming $w_{i,j} \equiv 1$ for all $i \neq j$. Under this assumption, the ODE or SDE system \eqref{eqn:multi_agent_SDE} simplifies to
\begin{equation} \label{eqn:many_particle_SDE}
\begin{aligned}
\rd X_i = \mu(X_i)\, \rd t + \frac{1}{N} \sum_{j \in [N] \setminus \{i\}} \sigma(X_i, X_j)\, \rd t + \nu\, \rd B_i, \quad \forall i \in [N].
\end{aligned}
\end{equation}
In this system, we refer to the agents $(X_i)_{i \in [N]}$ as ``particles'' due to their indistinguishable nature. Specifically, any permutation of the indices of $(X_i)_{i \in [N]}$ results in a reordered sequence $(X_{\pi(i)})_{i \in [N]}$ that also satisfies the ODE or SDE system \eqref{eqn:many_particle_SDE}.

For illustration, consider the deterministic case ($\nu = 0$) and also assume $\sigma(x, x) = 0$ for all $x \in \D$ (allowing us to take $w_{i,i} = 1/N$ without altering the system). With these assumptions, the binary interaction received by $X_i$ in \eqref{eqn:multi_agent_SDE} is a naive average given by $\frac{1}{N}\sum_{j \in [N]} \sigma(X_i, X_j)$.

Suppose that we have two solutions $(X_i(t))_{i \in [N]}$ and $(Y_i(t))_{i \in [N]}$ of \eqref{eqn:multi_agent_SDE}, possibly with different initial conditions. We define the \emph{coupling distance} between these two solutions as
\begin{equation*}
\begin{aligned}
D_{c}(t) \defeq \frac{1}{N} \sum_{i \in [N]} | X_i(t) - Y_i(t) |.
\end{aligned}
\end{equation*}
Assuming that both $\mu: \D \to \R^d$ and $\sigma: \D \times \D \to \R^d$ are bounded and Lipschitz continuous, we can apply Picard's argument for ODEs (or SDEs when
$(X_i(t))_{i \in [N]}$ and $(Y_i(t))_{i \in [N]}$ depend on the same $(B_{i})_{i \in [N]}$). The Grönwall’s inequality shows that the coupling distance satisfies:
\begin{equation*}
\begin{aligned}
D_c(t) \leq D_c(0) \, e^{L t},
\end{aligned}
\end{equation*}
where $L$ depends on the Lipschitz constant of $\mu$ and $\sigma$. This indicates that the difference between the two solutions, in terms of coupling distance, remains controlled over time, provided it starts sufficiently small.

Note that index permutations can be applied to $(X_i(t))_{i \in [N]}$ solving \eqref{eqn:many_particle_SDE}, hence the coupling between two solutions $(X_i(t))_{i \in [N]}$ and $(Y_i(t))_{i \in [N]}$ is not necessarily unique.
Essentially, one can consider the \emph{Wasserstein-1 distance} between the \emph{empirical measures} associated with the particle distributions.
The empirical measure $\rho_N$ is defined as
\begin{equation*}
\begin{aligned}
\rho_N(t,x) \defeq \frac{1}{N} \sum_{i \in [N]} \delta_{X_i(t)}(x),
\end{aligned}
\end{equation*}
where $\delta_{X_i(t)}$ is the Dirac delta measure at $X_i(t)$.
The Wasserstein-1 distance between two probability distributions $f, g \in \mathcal{P}(\D)$ can be defined through the following Kantorovich definition:
\begin{equation*}
\begin{aligned}
W_1(f, g) \defeq \inf_{\gamma \in \Pi(f, g)} \int_{\D \times \D} | x - y | \, \rd \gamma(x, y) = \inf_{\gamma \in \Pi(f, g)} \E_{(X, Y) \sim \gamma} |X - Y|,
\end{aligned}
\end{equation*}
where $\Pi(f, g) \in \mathcal{P}(\D \times \D)$ denotes the set of all couplings (joint probability measures) with marginals $f$ and $g$.
In the deterministic case, it is straightforward to check that the empirical measure $\rho_N$ are distributional solutions of the following Vlasov equation (with $\nu = 0$)
\begin{equation} \label{eqn:many_particle_Vlasov}
\begin{aligned}
& \frac{\partial}{\partial t}f(t,x) + \nabla_x \cdot \left( \left[ \mu(x) + \int_{\D} \sigma(x, y) f(t, y) \, \rd y \right] f(t, x) \right) - \frac{\nu^2}{2} \Delta_x f(t,x) = 0.
\end{aligned}
\end{equation}
One can extend the argument for $D_c$ to any $\E_{(X, Y) \sim \gamma} |X - Y|$. This essentially results the well-posedness and stability of \eqref{eqn:many_particle_Vlasov}, that for any initial data $f_0,g_0 \in \mathcal{P}(D)$, the unique distributional solutions $f, g \in C([0,T]; \mathcal{P}(\D))$ of the Vlasov equation \eqref{eqn:many_particle_Vlasov} has Wasserstein-1 distance bounded by
\begin{equation*}
\begin{aligned}
W_1(f(t), g(t)) \leq W_1(f_0, g_0) \, e^{L t}.
\end{aligned}
\end{equation*}

When stochasticity is incorporated into the system ($\nu > 0$), the empirical measure $\rho_N$ becomes random due to the presence of stochastic terms $\nu\, \rd B_i$ in particle dynamics. Consequently, these empirical measures do not exactly satisfy the deterministic Vlasov equation \eqref{eqn:many_particle_Vlasov}. However, as the number of particles $N \to \infty$, one can still expect some form of convergence of the empirical measures to a deterministic limit governed by the Vlasov equation.
We refer to \cite{sznitman1991topics} for the derivation of such convergence when the coefficients are Lipschitz bounded.

A central concept in this analysis is \emph{propagation of chaos}. It asserts that as $N \to \infty$, the particles become asymptotically \emph{independent and identically distributed} (i.i.d.). In other words, the joint distribution of any finite subset of particles converges to a product of identical single-particle distributions. As a consequence, at any later time $t$, the empirical measure $\rho_N(t)$ converges in expectation with respect to the Wasserstein-1 distance to the deterministic measure that solves the Vlasov equation.

The mean-field limit of many-particle systems with identical connections is an active area of research in its own right, with much still to be understood. When the interaction kernel $\sigma(x, y)$ is singular, rigorously establishing the mean-field limit becomes increasingly difficult. Even for models where the limit has been justified, quantifying the convergence rate as $N \to \infty$ remains a significant problem.
For a more comprehensive review, we refer the reader to survey articles \cite{chaintron2022propagation1, chaintron2022propagation2}.

\subsection{Non-exchangeable systems} \label{subsec:non_exchangeable_intro}

Much less explored are the
non-exchangeable systems, where the connection weights $w_{i,j}$ in ODE or SDE systems such as \eqref{eqn:multi_agent_SDE} can all differ. However, there has been a growing interest in such systems over the past decade due to their critical importance in many applications. Traditional definitions of the mean-field limit are often insufficient in this context, as the non-identical connection weights fundamentally alter the structure of the problem.
Conceptually speaking, agents play distinct roles within the system, resulting in diverse interactions and dynamic behaviors. As a result, any attempt to couple agents or average their behavior must account for their inherent differences.

One way to generalize the concept of the mean-field limit, as discussed in \cite{chiba2018mean, kaliuzhnyi2018mean} (and many more recent works, which we will introduce gradually), is by using the theory of graphons \cite{lovasz2006limits} and its extensions, which serve as natural tools to describe the graph limit of the connection weights. In this approach, it is expected that as the number of agents $N \to \infty$, the large-scale behavior of the ODE or SDE system \eqref{eqn:multi_agent_SDE} can be effectively captured by the extended Vlasov equation \eqref{eqn:multi_agent_Vlasov_first}, restated here:
\begin{equation} \label{eqn:multi_agent_Vlasov}
\begin{aligned}
& \frac{\partial}{\partial t}f(t,x,\xi) + \nabla_x \cdot \left( \left[ \mu(x) + \int_{\I} w(\xi,\zeta) \int_{\D} \sigma(x, y) f(t, y, \zeta) \, \rd y \rd \zeta \right] f(t, x, \xi) \right)
\\
& \hspace{8.5cm} - \frac{\nu^2}{2} \Delta_x f(t,x,\xi) = 0.
\end{aligned}
\end{equation}
This is an equation about $f(t,x,\xi)$, which we call the extended density function. Here, $x \in \D$ represents the state of agents, and $\xi \in \I$ is an additional variable introduced to account for the indices $i = 1,\dots,N$ as the number of agents $N \to \infty$.
Specifically, $\xi$ serves as a continuous identifier that distinguishes between different particles in the continuum limit.
For convenience, we mainly consider $\I = [0,1]$ equipped with the Lebesgue measure, but other atomless standard probability spaces will be used when necessary. The flexibility in selecting $\I$ is due to the fact that the definition relies solely on an integral kernel $w(\xi,\zeta): \I \times \I \to \R$, which corresponds to the discrete $w_{i,j}$-weights in the finite system.

One way to illustrate how the extended Vlasov equation provides the right approximation of the multi-agent system is to again consider the deterministic case, $\nu = 0$. Let $\I = [0,1]$, and define the extended empirical measure $\rho_N$ as
\begin{equation*}
\begin{aligned}
\rho_N(t, x, \xi) \defeq \sum_{i \in [N]} \mathbbm{1}_{\left[\frac{i-1}{N}, \frac{i}{N}\right)}(\xi) \, \delta_{X_i(t)}(x).
\end{aligned}
\end{equation*}
Next, define the piecewise-extended kernel $w_N: [0,1] \times [0,1] \to \R$ by
\begin{equation*}
\begin{aligned}
w_N(\xi, \zeta) \defeq \sum_{i,j \in [N]} \mathbbm{1}_{\left[\frac{i-1}{N}, \frac{i}{N}\right) \times \left[\frac{j-1}{N}, \frac{j}{N}\right)}(\xi, \zeta) \, w_{i,j}.
\end{aligned}
\end{equation*}
It is straightforward to verify that the pair $(w_{i,j})_{i,j \in [N]}, (X_i(t))_{i \in [N]}$ solving the ODE system \eqref{eqn:multi_agent_SDE} with $\nu = 0$ is equivalent, in the distributional sense, to having $f(t,x,\xi) = \rho_N(t,x,\xi)$ and $w(\xi, \zeta) = w_N(\xi, \zeta)$ solve the Vlasov PDE \eqref{eqn:multi_agent_Vlasov} with $\nu = 0$.

For the case $\nu > 0$, one can adapt the classical arguments from \cite{sznitman1991topics} to establish asymptotic independence. This concept is weaker than chaos in exchangeable systems, as the agents' distributions are not necessarily identical when the equations are not fully symmetric.
We outline this approach in the following and refer to \cite{jabin2025mean} for further details.

Consider the following McKean-Vlasov SDE system:
\begin{equation} \label{eqn:McKean_Vlasov_SDE}
\begin{aligned}
\rd \bar{X}_i &= \mu(\bar{X}_i)\, \rd t + \E \left( \frac{1}{N} \sum_{j \in [N]} w_{i,j} \sigma(\bar{X}_i, \bar{X}_j) \right)\, \rd t + \nu\, \rd B_i, \quad \forall i \in [N],
\end{aligned}
\end{equation}
where the binary interaction is defined in terms of the expectation. It is straightforward to verify that if the initial data $(\bar{X}_i(0))_{i \in [N]}$ are independent, then the states $(\bar{X}_i(t))_{i \in [N]}$ remain independent for any later time $t > 0$.

When the coefficients $\mu, \sigma$ are reasonably smooth,
coupling methods can be used to show that the solution $(X_i(t))_{i \in [N]}$ of \eqref{eqn:multi_agent_SDE} is approximated by the solution $(\bar{X}_i(t))_{i \in [N]}$ of \eqref{eqn:McKean_Vlasov_SDE} in the following sense: Starting with the same $(w_{i,j})_{i,j \in [N]}$ and same independent initial data, one can establish an a priori estimate that for any later time $t$, the averaged difference
\begin{equation} \label{eqn:propagation_of_independence}
\E \bigg( \frac{1}{N} \sum_{i \in [N]} |X_i(t) - \bar{X}_i(t)| \bigg) \to 0, \quad \text{ as } N \to \infty. 
\end{equation}
Define the extended law function $p_N(t) : [0,1] \to \mathcal{P}(\D)$ as
\begin{equation*}
p_N(t, x, \xi) \defeq \sum_{i \in [N]} \mathbbm{1}_{\left[\frac{i-1}{N}, \frac{i}{N}\right)}(\xi) \Law_{\bar{X}_i}(x).
\end{equation*}
It is straightforward that the pair $(w_{i,j})_{i,j \in [N]}$ and $(\bar{X}_i(t))_{i \in [N]}$ solving \eqref{eqn:McKean_Vlasov_SDE} with $\nu \geq 0$ is equivalent to having $f(t, x, \xi) = p_N(t, x, \xi)$ and $w(\xi, \zeta) = w_N(\xi, \zeta)$ solve the Vlasov PDE \eqref{eqn:multi_agent_Vlasov} in the distributional sense with $\nu \geq 0$.

However, for non-exchangeable systems, justifying the Vlasov PDE \eqref{eqn:multi_agent_Vlasov} is only halfway towards establishing the mean-field limit, as it merely considers a fixed connection matrix $(w_{i,j})_{i,j \in [N]}$ and compares the difference between $(X_i(t))_{i \in [N]}$ in \eqref{eqn:multi_agent_SDE} and $(\bar{X}_i(t))_{i \in [N]}$ in \eqref{eqn:McKean_Vlasov_SDE} (or $p_N(t,x,\xi)$ solving \eqref{eqn:multi_agent_Vlasov}).
To derive the mean-field limit from a sequence of multi-agent systems satisfying \eqref{eqn:multi_agent_SDE}, one must take the $n \to \infty$ limit of both $(w_{i,j}^{(n)})_{i,j \in [N_n]}$ and $(X_i^{(n)})_{i \in [N_n]}$.

If sufficient a priori knowledge of the connection weights is available, for example, when there exists a predefined kernel $w \in C(\I \times \I)$ on a low-dimensional configuration space $\I$, and the discrete connection weights arise from discretizing this kernel (e.g., each agent $i$ is associated with some $\xi_i \in \I$ and $w_{i,j} = w(\xi_i,\xi_j)$), then the convergence of $(w_{i,j}^{(n)})_{i,j \in [N_n]}$ is natural to expect. In this scenario, one can effectively reduce the problem to the exchangeable setting by enlarging the phase space to $\D \times \I$, where each agent’s state is $(x_i,\xi_i)$ and there is no dynamics in the $\xi$-dimension.

A more challenging question arises when one only has access to $(w_{i,j}^{(n)})_{i,j \in [N_n]}$, with no clear low-dimensional embedding like $w_{i,j} = w(\xi_i,\xi_j)$.
This is particularly relevant to many models where the existence of a low-dimensional structure is far from evident. To characterize the limit of connection weights without imposing overly strong a priori assumptions, graph limit theories have been widely borrowed in the study of non-exchangeable systems.

\subsection{Graphon theory} \label{subsec:grpahon}

A notable example of graph limiting theories is the \emph{graphon} theory, which was initially developed in the seminal work \cite{lovasz2006limits} for $L^\infty([0,1]^2)$ kernels. More recently, in \cite{borgs2019Lp, borgs2018Lp}, it was extended to $L^p([0,1]^2)$ kernels, which applies to limits for connection weights that do not necessarily follow the $O(1/N)$ scaling discussed in this article.
Here, we present a brief introduction to the graphon theory with the most classical $L^\infty$ definition.
\begin{defi} \label{defi:graphon}
For any $w \in L^\infty([0,1]^2)$, the cut norm is defined as
\begin{equation*}
\begin{aligned}
\|w\|_{\square} := \sup_{S, T \subseteq [0,1]} \left| \int_{S \times T} w(\xi, \zeta) \rd \xi \, \rd \zeta \right|.
\end{aligned}
\end{equation*}
Given a bound $W > 0$, the set of graphons is defined by
\begin{equation*}
\begin{aligned}
\mathcal{G}_W := \left\{ w \in L^\infty([0,1]^2) : 0 \leq w(\xi,\zeta) \leq W, \ \forall \xi,\zeta \in [0,1] \text{ and } w \text{ is symmetric} \right\}.
\end{aligned}
\end{equation*}
For any two graphons $w, u \in \mathcal{G}_W$, we define the labeled cut distance as
\begin{equation*}
\begin{aligned}
d_\square(w, u) := \|w - u\|_{\square},
\end{aligned}
\end{equation*}
and the (unlabeled) cut distance as
\begin{equation*}
\begin{aligned}
\delta_\square(w, u) = \inf_{\Phi} \|w - u^\Phi\|_{\square},
\end{aligned}
\end{equation*}
where $\Phi$ ranges over all bijective measure-preserving maps $\Phi: [0,1] \rightarrow [0,1]$, and $u^\Phi \in \mathcal{G}_W$ denotes the rearranged graphon, defined by
\begin{equation*}
\begin{aligned}
u^\Phi(\xi, \zeta) \defeq u(\Phi(\xi), \Phi(\zeta)).
\end{aligned}
\end{equation*}
\end{defi}

The labeled cut distance is equivalent to the $L^\infty \to L^1$ operator norm of the adjacency operators. More precisely, one has
\begin{equation*}
\begin{aligned}
d_\square(w^{(n)}, w) \leq \|T_{w^{(n)}} - T_w\|_{L^\infty \to L^1} \leq 4 d_\square(w^{(n)}, w),
\end{aligned}
\end{equation*}
where
\begin{equation*}
\begin{aligned}
T_w(\psi)(\xi) := \int_0^1 w(\xi, \zeta)\, \psi(\zeta)\, \rd \zeta, \quad \xi \in [0,1].
\end{aligned}
\end{equation*}
From this, it is clear that
\begin{equation*}
\begin{aligned}
d_\square(w^{(n)}, w) \leq \|w^{(n)} - w\|_{L^1},
\end{aligned}
\end{equation*}
and $d_\square$ induces a strictly weaker topology than the $L^1$ norm on $\mathcal{G}_W$ the space of graphons.

The unlabeled cut distance $\delta_\square$ is not a metric in $\mathcal{G}_W$, because a graphon $w$ and its rearrangement $w^{\Phi}$ have $\delta_\square(w,w^{\Phi})$ = 0.
Instead, it is a metric in the quotient space $\mathcal{G}_W / \sim$, which contains equivalent classes of graphons up to measure-preserving rearrangements.
The space $(\mathcal{G}_W / \sim, \delta_\square)$ is remarkably a \emph{compact} metric space, as shown in \cite{lovasz2006limits} (see also \cite{borgs2019Lp, borgs2018Lp} for the $L^p$ case).
In other words, for any sequence of graphons $w^{(n)} \in \mathcal{G}_W$, there exists a graphon limit $w \in \mathcal{G}_W$ such that, after suitable rearrangements and passing to a subsequence, the unlabeled cut distance satisfies $d_\square(w^{(n)}, w) \to 0$.

Let us emphasize that the compactness property does not hold directly on $(\mathcal{G}_W, d_\square)$ when rearrangements are not allowed, nor can it easily be extended by replacing $d_\square(w^{(n)}, w)$ with stronger distances, such as $\|w^{(n)} - w\|_{L^1}$.
This is because we merely assume that the $1/N$-rescaled connection weights are uniformly bounded, i.e.,
\begin{equation*}
\begin{aligned}
\sup_{n}\sup_{i,j}|w_{i,j}^{(n)}| < \infty.
\end{aligned}
\end{equation*}
In simple terms, the rearrangement on $[0,1]$ (or the index permutation on $[N]$) serves to group agents with ``similar connection weight properties'' together, which provides a weak yet crucial regularity of $w^{(n)}$ and ensures compactness in terms of cut distance.

Such measure-preserving rearrangement $\Phi$ in the definition of the unlabeled cut distance $\delta_\square$ is particularly relevant in the study of non-exchangeable systems. It is straightforward to verify that if $(w, f)$ is a solution to the Vlasov PDE \eqref{eqn:multi_agent_Vlasov} and $\Phi: [0,1] \to [0,1]$ is a measure-preserving bijection, then the rearranged pair
\begin{equation*}
\begin{aligned}
w^\Phi(\xi, \zeta) = w(\Phi(\xi), \Phi(\zeta)), \quad f^\Phi(t, x, \xi) = f(t, x, \Phi(\xi))
\end{aligned}
\end{equation*}
also forms a solution. This property corresponds to the permutation of indices in the ODE/SDE system \eqref{eqn:multi_agent_SDE}, which does not fundamentally alter the underlying equation.
Consequently, by allowing measure-preserving rearrangements on $[0,1]$ (and possibly passing to a subsequence), one can assume the a priori convergence of the kernels $d_\square(w^{(n)}, w) \to 0$ in the Vlasov PDE \eqref{eqn:multi_agent_Vlasov}.
If $f^{(n)} \to f$ can be established in the appropriate metric (given the convergence of $w^{(n)}$ and the initial data), then $f$ is justified as the solution to the Vlasov PDE \eqref{eqn:multi_agent_Vlasov} in the limit of the McKean-Vlasov system \eqref{eqn:McKean_Vlasov_SDE}. Furthermore, if the propagation of independence \eqref{eqn:propagation_of_independence} is given, $f$ is also justified as the mean-field limit of the multi-agent system \eqref{eqn:multi_agent_SDE}.

The first introduction of graphons to the mean-field theory of non-exchangeable systems of the form \eqref{eqn:multi_agent_SDE} was possibly made in \cite{chiba2018mean, kaliuzhnyi2018mean}, where the mean-field limit was established through a quantitative stability estimate of $f^{(n)} \to f$ in terms of
\begin{equation} \label{eqn:fiberwise_extension}
\begin{aligned}
& d_F \left( f^{(n)}(t, \cdot, \cdot), f(t, \cdot, \cdot) \right)
\\
& \ \defeq \bigg[ \int_{0}^{1} \Big( d_{\mathcal{P}} \big( f^{(n)}(t, \cdot, \xi), f(t, \cdot, \xi) \big) \rd \xi \Big)^p \bigg]^{1/p},
\end{aligned}
\end{equation}
with $p = 1$ and the metric $d_{\mathcal{P}}$ inside the integral chosen as the bounded Lipschtiz metric (which is a truncated Wasserstein-1) on $\mathcal{P}(\D)$, the space of probability measures.
The original results of \cite{chiba2018mean, kaliuzhnyi2018mean}, however, considered $\|w^{(n)} - w\|_{L^1} \to 0$ instead of $d_\square(w^{(n)}, w) \to 0$, and thus still required an a priori limiting kernel $w$ along with its ad hoc discretization.

The graphon approach has been further developed, for example, by \cite{bayraktar2023graphon, bet2024weakly, jabin2024non}.
With $d_\square(w^{(n)}, w) \to 0$, the stability of $f^{(n)} \to f$ in terms of a metric of the form \eqref{eqn:fiberwise_extension}, with $p = 2$ and $d_{\mathcal{P}}$ chosen as the Wasserstein-2 metric was achieved in both \cite{bayraktar2023graphon, bet2024weakly}.
Remarkable recent work \cite{ayi2024mean} broadened the graphon framework by using hypergraphons (as seen in \cite{elek2012measure, roddenberry2023limits}) and defining mean-field limits where interactions extend beyond pairwise exchanges.

This strategy has also been applied to the graph limits defined in alternative ways. For example, \cite{gkogkas2022graphop} examined graph-ops, as introduced in \cite{backhausz2022action}, from a more operator-theoretic perspective, while \cite{kuehn2022vlasov} developed and analyzed a specialized notion of digraph measures. Both works established the stability of $f^{(n)} \to f$ by choosing $d_{\mathcal{P}}$ in \eqref{eqn:fiberwise_extension} as the bounded Lipschitz metric and formally taking $p = \infty$ (which corresponds to the essential supremum taken over $\xi$).
While these alternative definitions provide valuable limiting objects for sparser connection weights, convergence $w^{(n)} \to w$ can be more intricate, making the compactness property out of reach of current graph limit theories.

Let us also mention that when the connection weights are extremely sparse, for example, induced by graphs with bounded degrees, the averaging effect in binary interactions becomes insufficient and one cannot justify \eqref{eqn:propagation_of_independence}. In fact, since the agents are not asymptotically independent in generic, the dynamics is fundamentally different and should be described using frameworks other than extended Vlasov PDEs like \eqref{eqn:multi_agent_Vlasov}. This topic cannot be adequately discussed in this article, and we refer instead to \cite{lacker2023local, lacker2023marginal, oliveira2020interacting} and the references therein.

\subsection{Empirical data and a posteriori correspondence} \label{subsec:posteriori}

The convergence behavior of empirical data has been much less discussed in the context of non-exchangeable systems.

In the \emph{exchangeable} systems, the stability of the Vlasov equation, combined with the propagation of chaos, suggests that the empirical measure $\rho_N(t, x) = \frac{1}{N} \sum_{i \in [N]} \delta_{X_i(t)}(x)$ converges to a deterministic limit $f(t, x)$ as $N \to \infty$. For instance, for sufficiently smooth dynamics, one can expect the convergence of Wasserstein-1 distance in expectation, i.e.
\begin{equation*}
\begin{aligned}
\lim_{N \to \infty} \E \left[ W_1 \left( \rho_N(t, \cdot), f(t, \cdot) \right) \right] = 0.
\end{aligned}
\end{equation*}

However, in the \emph{non-exchangeable} systems, there is a larger gap between the convergence of laws to the limit and the convergence of empirical data directly to the same limit. To see it, recall our definition of extended law
\begin{equation*}
p_N(t, x, \xi) \defeq \sum_{i \in [N]} \mathbbm{1}_{\left[\frac{i-1}{N}, \frac{i}{N}\right)}(\xi) \Law_{\bar{X}_i}(x),
\end{equation*}
and extended empirical measure
\begin{equation*}
\begin{aligned}
\rho_N(t, x, \xi) \defeq \sum_{i \in [N]} \mathbbm{1}_{\left[ \frac{i-1}{N}, \frac{i}{N} \right)} (\xi) \delta_{X_i(t)}(x).
\end{aligned}
\end{equation*}
For a fixed $t \geq 0$ and $\xi \in [0,1]$, $\rho_N(t, \cdot, \xi)$ is a Dirac delta measure centered at $X_i(t)$ if $\xi$ lies within the interval corresponding to the $i$-th particle.
However, when there is randomness in the initial data or in the dynamics, the limit of $p_N(t, x, \xi)$, that is, $f(t,x,\xi)$, should not concentrate on a Dirac delta for a fixed $t \geq 0$ and $\xi \in [0,1]$.
This implies that when the metric takes the $\xi$-fibered form as in \eqref{eqn:fiberwise_extension}, the convergence of $\rho_N$ to $f$ is typically not allowed.

The issue is fundamental in non-exchangeable systems. Agents play distinct roles due to the connection weights $(w_{i,j}^{(n)})_{i,j \in [N_n]}$. To capture this, the aforementioned approach \emph{a priori} establishes a fine correspondence between agents $i \in [N_n]$ and pieces of density in the continuum limit $\xi \in [0,1]$, independent of time and any specific realization of randomness. However, this would result in insufficient averaging of agents when starting from empirical data.
Hence, even when the laws of the agents are independent, characterizing the convergence of empirical data to the extended Vlasov equation still requires additional effort. One possible approach is to adopt the definition given later in Definition~\ref{defi:dirac_blow_up}, where agents $i \in [N_n]$ are not \emph{a priori} associated with small portions of $[0,1]$ equipped with the Lebesgue measure, but rather more precisely associated with parts of the distribution $f(t, \rd x, \rd \xi)$.

One approach that considers an a priori correspondence and achieves a stability result directly from empirical data in an SDE system is given in \cite{jabin2024non}.
The proof relies on a metric tailored to the limiting $w \in L^\infty([0,1]^2)$. This metric induces the weak-* topology on $\mathcal{M}([0,1] \times \mathbb{R})$ and allows for a form of local averaging among nearby agents, resulting in convergence $\rho_N \wsto f$ and mean trajectories of agents like in the exchangeable case.

In this article, our bi-coupling distance as in Definition~\ref{defi:coupling_distance_1} takes an interpolative form, allowing the coupling measure $\gamma$ to adjust with time and specific realizations of randomness. In this sense, we refer to it as \emph{a posteriori} correspondence.

The benefit of a posteriori correspondence is evident. Consider a toy example where the connection weights are all identically $1/N$, yet we pretend not to know this and consider the systems as non-exchangeable. We further simplify \eqref{eqn:multi_agent_SDE} by setting $\mu = \nu = 0$, so that $\rd X_i = \nu B_i$ on the torus $\T$. Suppose $N_n \to \infty$ and $\rho_n(0, x) = \frac{1}{N_n}\sum_{i \in [N_n]} \delta_{X^{(n)}_i(0)}(x)$ converging to $f_0$ in Wasserstein-2. Then, as time evolves, the dynamics should asymptotically follow $\partial_t f = \frac{\nu^2}{2}\Delta f$. This implies that at any later time $t$, we still have the Wasserstein-2 convergence, and since the topologies are equivalent, we also have the Wasserstein-1 convergence. In other words, we can write
\begin{equation*}
\begin{aligned}
\lim_{n \to \infty} \sup_{n_1,n_2 \geq n} \inf_{\gamma \in \Pi(N_{n_1},N_{n_2})} \left[ \sum_{i_1 \in [N_{n_1}]} \sum_{i_2 \in [N_{n_2}]} |X^{(n_1)}_{i_1}(t) - X^{(n_2)}_{i_2}(t)|^p \gamma_{i_1,i_2} \right] = 0,
\end{aligned}
\end{equation*}
which is a degenerate form of our bi-coupling distance (when $p=1$).
However, if we try to fix a particular $\gamma \in \Pi(N_{n_1},N_{n_2})$ a priori and remove the infimum in the above expression, then at any later time $t$, we cannot achieve convergence for either $p=1$ or $p=2$.

From the above example, it becomes clear that even the classical mean-field limit for exchangeable systems should be understood as relying on a posteriori correspondence. Our discussion simply makes this perspective explicit. 

However, for non-exchangeable systems, adopting a posteriori correspondence becomes much more challenging. With an a priori correspondence, one can first analyze the McKean–Vlasov system \eqref{eqn:McKean_Vlasov_SDE} and then establish a loose form of the mean-field limit based on the propagation of independence \eqref{eqn:propagation_of_independence}. The a posteriori correspondence viewpoint does not seem to allow for such a decomposition easily. Specifically, $\Law_{\bar{X}_i}$ (the law of a particular agent indexed by $i$) becomes very unnatural, and the weights and states of agents become more intricately entangled, resulting in the method of characteristics seemingly inaccessible.

\subsection{Tensorization of agent law and graph homomorphism density}

A recent practice, as seen in \cite{jabin2025mean, jabin2023mean, lacker2024quantitative}, aims to extend the classical mean-field theory tool of the BBGKY hierarchy to the non-exchangeable case.
The article \cite{lacker2024quantitative} provides a detailed quantification of the propagation of independence in relative entropy for non-exchangeable systems. The two related works \cite{jabin2025mean, jabin2023mean} study the large-scale behavior by defining certain weighted $k$-marginals, which do not rely on any correspondence between different systems at the microscopic level.
We describe the arguments from \cite{jabin2025mean, jabin2023mean} in some detail, as our strategy benefits from these works on multiple levels.

Let us begin with the extended Vlasov PDE~\eqref{eqn:multi_agent_Vlasov_first}. By integrating it over the variable $\xi \in \I$ (thereby not explicitly distinguishing individual agents), one has:
\begin{equation} \label{eqn:multi_agent_Vlasov_linearize}
\begin{aligned}
& \frac{\partial}{\partial t}\tau(1,w,f)(t,x_1) + \nabla_x \cdot \Big( \mu(x_1) \tau(1,w,f)(t,x_1) \Big)
\\
& + \nabla_x \cdot \bigg( \int_{\D} \sigma(x_1, x_2) \tau(2,w,f)(t,x_1,x_2) \, \rd x_2 \bigg) - \frac{\nu^2}{2} \tau(1,w,f)(t,x_1) = 0,
\end{aligned}
\end{equation}
where the distribution functions $\tau(1,w,f)(t,x)$ and $\tau(2,w,f)(t,x_1,x_2)$ are defined as
\begin{equation*}
\begin{aligned}
\tau(1,w,f)(t,x) &\defeq \int_{\I} f(t,x,\xi) \rd \xi
\\
\tau(2,w,f)(t,x_1,x_2) &\defeq \int_{\I} \int_{\I} w(\xi_1,\xi_2) f(t,x_1,\xi_1) f(t,x_2,\xi_2) \rd \xi_1 \rd \xi_2.
\end{aligned}
\end{equation*}
This is a linear PDE that describes the time evolution of $\tau(1, w, f)$, which represents the average behavior of all agents. In the exchangeable case, the evolution of $f(t, x)$ can be expressed as a linear equation involving both $f(t, x)$ and its tensorization $f(t, x_1) f(t, x_2)$. In the non-exchangeable setting, this relationship extends to $\tau(1, w, f)$ depending on $\tau(2, w, f)$, defined as a weighted integral of the tensorized form $f(t, x_1, \xi_1) f(t, x_2, \xi_2)$. The connection weights $w$ are incorporated into $\tau(2, w, f)$ in such a way that $w$ no longer explicitly appear in the resulting PDE~\eqref{eqn:multi_agent_Vlasov_linearize}.

One can also derive a linear equation for $\tau(2, w, f)$, which involves weighted integrals of 3-tensorized forms, specifically $f(t, x_1, \xi_1) f(t, x_2, \xi_2) f(t, x_3, \xi_3)$. This recursive process gives rise to an extended family called \emph{observables}, indexed by (directed) trees $T \in \mathcal{T}$ and denoted by $\tau(T, w, f)$, which reads
\begin{equation*}
\begin{aligned}
\tau(T,w,f)\big(t,(x_i)_{i \in \mathsf{v}(T)}\big) &\defeq \int_{\I^{\mathsf{v}(T)}} \prod_{(k,l) \in \mathsf{e}(T)} w(\xi_{k}, \xi_{l}) \prod_{m \in \mathsf{v}(T)} \Big( f(t,x_m,\xi_m) \rd \xi_m \Big).
\end{aligned}
\end{equation*}
These observables, first identified in \cite{jabin2025mean}, establish an infinite hierarchy of dependencies among the distribution functions:
\begin{equation*}
\begin{aligned}
\tau(T,w,f) \text{ depends on } \tau(T,w,f) \text{ and all } \tau(T + i,w,f),
\end{aligned}
\end{equation*}
where $(T+i)$ is obtained from $T$ by adding a new vertex and connecting it to the vertex $i \in \mathsf{v}(T)$ by a new edge.

This hierarchical structure introduces a significant challenge, which appears in many kinetic theory problems. Quantitative analysis of one $\tau(T, w, f)$ requires knowledge of the next-level observables. Although it is possible to establish energy estimates at each level, combining these estimates typically results in a priori bounds that resemble a Taylor expansion with large factorial coefficients. Such coefficients can cause a blow-up in some finite time $t > 0$, a phenomenon commonly referred to as \emph{factorial blow-up}. 
However, in \cite{jabin2025mean}, by establishing a priori estimates for the extended Vlasov PDE~\eqref{eqn:multi_agent_Vlasov_first} and extending these estimates to all observables in the hierarchy, $L^2$-quantitative stability of the observables was achieved for a sufficiently smooth dynamics at all times $t \geq 0$. Without referring to the explicit convergence rate, the critical estimate in \cite{jabin2025mean} can be restated as follows: If
\begin{equation*}
\begin{aligned}
\tau(T,w^{(n)},f^{(n)}_0) \to \tau(T)^{(\infty)}_0 \in L^2(\D^{\mathsf{v}(T)}),
\end{aligned}
\end{equation*}
then at any later $t \geq 0$,
\begin{equation*}
\begin{aligned}
\tau(T,w^{(n)},f^{(n)}(t)) \to \tau(T)^{(\infty)}(t) \in L^2(\D^{\mathsf{v}(T)}).
\end{aligned}
\end{equation*}
Notably, still with sufficiently smooth dynamics, the observables were proved to be $W^{1,\infty}$-regular and to have certain moment bounds (which are necessary when $\D = \R^d$). This additional regularity can ensure compactness in terms of the $L^2$ norm of the observables $\tau(T, w, f)$ and guarantee that the limit exists without additional assumptions on the weights $w^{(n)}$.

Investigating the extended Vlasov PDE~\eqref{eqn:multi_agent_Vlasov_first} through the observables completely avoids any direct argument of the graph limit.
Thus, this approach can even be extended to sparse kernels $w$ lying in the space $L^\infty_\xi \mathcal{M}_\zeta \cap L^\infty_\zeta \mathcal{M}_\xi$, where the interpretation of this notation differs from the conventional definitions of Bochner spaces. This allows connection weights to be sparsely distributed, which we do not discuss in this article. In fact, addressing the compactness and sparsity of the connection weights simultaneously was one of the primary motivations behind the development of the observables in \cite{jabin2025mean}.

Extending this hierarchical convergence to work directly for multi-agent systems like~\eqref{eqn:multi_agent_SDE} was addressed by \cite{jabin2023mean}, although the rigorous result was merely done on $\D = \R$ about a different SDE system (integrate-and-fire neuron model).
This approach corresponds to the classical argument that the BBGKY hierarchy converges to the Vlasov hierarchy in the exchangeable case, where the propagation of chaos is not proved separately but is instead integrated into the derivation of the mean-field limit.
The observable $\tau_N (T,w,X)$ for the multi-agent systems is defined on $\mathcal{M}(\D^{\mathsf{v}(T)})$ for fixed time $t$, as
\begin{equation*}
\begin{aligned}
& \tau_N (T,w,X)(t,\rd z) 
\\
 \ & \defeq \frac{1}{N^{|\mathsf{v}(T)|}} \sum_{\forall j \in \mathsf{v}(T), \ i_j \in [N] \text{ distinct}} \bigg( \prod_{(l,m) \in \mathsf{e}(T)} w_{i_l, i_m} \bigg) \Law_{(X^{i_j}(t))_{j \in \mathsf{v}(T)}} (\rd z),
\end{aligned}
\end{equation*}
which is a weighted sum of marginal laws with the same number of agents.
From the Liouville equation governing the full joint law of $(X^{i}(t))_{i \in [N]}$, a similar hierarchical structure was derived.
Quantitative estimates on the hierarchy were established using a strategy similar to that of \cite{jabin2025mean}, but within a tactically chosen negative Sobolev space, specifically a tensorization of $H^{-1}(\R)$.
Compared to $L^2$, this space is particularly helpful for systems with a finite number of agents, as it includes singular distributions such as the Dirac delta.

\hfill

Readers who are familiar with graph theory may have noticed that if we take the integration over the entire domain $\D^{\mathsf{v}(T)}$, the observables $\tau (T,w,f)$ and $\tau_N (T,w,X)$ simply give the homomorphism densities for (directed) graphons and graphs. In this sense, the observables are joint densities of graph homomorphism densities tensorized with agent distributions that are independently multiplied \cite{jabin2025mean} or as marginals of correlated distributions \cite{jabin2023mean}. This formulation effectively yields a large-scale description of the multi-agent system.

The strategy of constructing observables is highly extensive. For multi-agent systems where there are two types of binary interactions between agents with distinct connection weights $w$ and $u$, observables can be defined by assigning two different colors to the tree edges. 
Another important extension involves indexing observables by tree graphs that incorporate both ``opaque edges'' and ``transparent edges''. In the definition of such observables, opaque edges correspond to factors $w$, while transparent edges correspond to factors $1$. 
The stability argument can be readily applied within this context, as building the hierarchy merely requires the addition of ``opaque edges''.
From a graph homomorphism perspective, this definition may seem redundant at first glance because, if a tree includes a ``transparent edge'', the homomorphism density can be factored into two independent components. For example, two vertices connected by a transparent edge yield
\begin{equation*}
\begin{aligned}
\tau(\dots)(t,x_1,x_2) = \int_{\I \times \I} \mathbbm{1}(\xi_{1}, \xi_{2}) f(t,x_1,\xi_1) f(t,x_2,\xi_2) \rd \xi_1 \rd \xi_2
\end{aligned}
\end{equation*}
which can be integrated separately to obtain
\begin{equation*}
\begin{aligned}
\tau(\dots)(t,x_1,x_2) = \int_{\I} f(t,x_1,\xi_1) \rd \xi_1 \ \int_{\I} f(t,x_2,\xi_2) \rd \xi_2.
\end{aligned}
\end{equation*}
However, for a finite system, this $\tau_N$ corresponds to the $2$-marginal law of the agents. The convergence of observables essentially implies that the $2$-marginal law tensorizes, that is, it becomes the product of the $1$-marginal laws.
By standard tools from kinetic theory, this extends to that all $k$-marginal laws asymptotically factorize as products of the $1$-marginal laws, which is an alternative formulation of chaos. Consequently, the empirical measure converges to the deterministic $1$-marginal law in expectation. This argument can be generalized to any observable by connecting the corresponding tree to its copy via a ``transparent edge''. The convergence essentially implies that the ``empirical observables''
\begin{equation*}
\begin{aligned}
\frac{1}{N^{|\mathsf{v}(T)|}} \sum_{\forall j \in \mathsf{v}(T), \ i_j \in [N] \text{ distinct}} \bigg( \prod_{(l,m) \in \mathsf{e}(T)} w_{i_l, i_m} \bigg) \bigotimes_{j \in \mathsf{v}(T)} \delta_{X^{i_j}(t)}(\rd z),
\end{aligned}
\end{equation*}
which are random distributions derived from individual realizations of the SDE, converge in expectation to $\tau (T,w,f)$.
In this sense, \cite{jabin2023mean} provided another result, apart from \cite{jabin2024non}, on direct convergence from empirical data of the SDE system to a deterministic limit.

\subsection{Counting Lemma and Inverse Counting Lemma}
A natural question is whether it is possible to establish a more straightforward metric on the pair $(w,f)$ or $(w,X)$ that is equivalent to the convergence of a certain family of observables. (As we have illustrated, these interesting families are far from unique.)

The ``projection'' of this problem to graph theory is relatively well studied.
The celebrated Counting Lemma and Inverse Counting Lemma \cite{borgs2008convergent} establish that convergence in the (unlabeled) cut distance $\delta_\square$ is equivalent to the convergence of graph homomorphisms for all simple graphs. More recently, it was shown by \cite{dell2018lovasz, grebik2022fractional} that two graphs (or graphons) having identical graph homomorphisms for all trees is equivalent to the \emph{fractional isomorphism} of the two. In \cite{boker2021graph}, metrics were introduced that relax the strict conditions of fractional isomorphism, showing that the convergence of these metrics implies the convergence of graph homomorphisms for all tree graphs.

Our main result is closely related to this problem.
In particular, we prove in the article that the convergence in the bi-coupling distance, as defined in Definition~\ref{defi:coupling_distance_1} and in the following Definition~\ref{defi:coupling_distance}), is equivalent to the convergence of observables in a larger family than what has been illustrated here.
Our proof of this equivalence adopts techniques developed from all the aforementioned results about graph limits.
The stability in the bi-coupling distance, as stated in Theorem~\ref{thm:main_discrete} and in the following Theorem~\ref{thm:main_limit}, actually reflect the stability of these observables, which is proved following the approach in \cite{jabin2023mean}.

\subsection{Some conventions}
Last, before we state the main result. Let us explain a bit about the conventions adopted in this article.

\begin{itemize}
\item
In our discussion, we frequently work with vectors in Hilbert spaces. To emphasize this technical key point, we denote \textbf{vector-valued quantities in Hilbert spaces using bold fonts}. This convention distinguishes these vectors from other types of variables within our framework.

Consequently, random variables will not be represented in bold fonts. Let us remark here that the only variables that are potentially random in our model are $X$, representing the states, and $B$, representing the Brownian motion. We will remind the readers when these variables are discussed in a random context to avoid any confusion.

\item
We also use standard probability spaces (which are measure spaces with total measure $1$) in our definition of the mean-field limit. However, these probability spaces do not represent actual randomness in the model. To emphasize this distinction, we use the notation $\I$ instead of the more commonly used $\Omega$ for such probability spaces.

\item
Throughout the paper, we perform computations involving kernels on $\I_1 \times \I_2$, which correspond to compositions of adjacency operators. Specifically, consider the kernels $w: \I_1 \times \I_1 \to \mathbb{R}$ and $\gamma: \I_1 \times \I_2 \to \mathbb{R}$. The composition of adjacency operators satisfies $T_{w} T_{\gamma} = T_{w \gamma}$,
where the composed kernel $w \gamma$ is defined by
\begin{equation*}
\begin{aligned}
(w \gamma)(\xi_1, \xi_2) \defeq \int_{\I_1} w(\xi_1, \zeta) \gamma(\zeta, \xi_2) \, \rd\zeta.
\end{aligned}
\end{equation*}
In some instances, $\gamma$ may be just a measure on $\I_1 \times \I_2$, and to rigorously define the composition in the distributional sense, we need to write
\begin{equation*}
\begin{aligned}
(w \gamma)(\xi_1, \xi_2) = \int_{\I_1} w(\xi_1, \zeta) \, \gamma(\rd\zeta, \xi_2).
\end{aligned}
\end{equation*}
However, inserting $\rd\zeta$ into the middle of the expression can be visually cumbersome for readers. To maintain clarity and ease of notation, when there are integrals on measures or equations that should be understood in the distributional sense, we adopt the slight abuse of notation like $\gamma(\zeta, \xi_2) \, \rd\zeta$.

\item
Many of our estimates rely on the assumption of uniform bounds. To maintain concise notation, we adopt the following convention: for any Banach space $\mathcal{B}$, we denote by $\mathcal{B}_{\leq 1}$ the subset of $\mathcal{B}$ consisting of elements with norm bounded by $1$. As will be evident in the proofs, the bound of $1$ is chosen for convenience most of the time and can be replaced by any positive constant.

\item
Finally, our stability estimates rely on multiple metrics that give equivalent notions of convergence, where compactness arguments yield uniform bounds from one metric to the other. We denote these bounds using $C^{\downarrow}_a$, where $a$ represents some additional parameter. For any fixed $a$, $C^{\downarrow}_a$ represents a non-decreasing, continuous function $\mathbb{R}_+ \to \mathbb{R}_+$ satisfying
\begin{equation*}
\begin{aligned}
\lim_{r \downarrow 0} C^{\downarrow}_a(r) = 0.
\end{aligned}
\end{equation*}
We express the equivalence of convergence given by two metrics $d_1,d_2$ as
\begin{equation*}
\begin{aligned}
d_1 \leq C^{\downarrow}_a(d_2), \quad d_2 \leq C^{\downarrow}_a(d_1).
\end{aligned}
\end{equation*}
This notation says that the metric $d_1,d_2$ can be controlled by the other metric through a certain function $C^{\downarrow}_a$.

\end{itemize}

\subsection{Main result}

Let us begin by rephrasing Definition~\ref{defi:coupling_distance_1} with greater generality.
\begin{defi} [\textbf{Coupling}] \label{defi:coupling}
Let $(\I_1, \mathscr{B}_1, P_1)$ and $(\I_2, \mathscr{B}_2, P_2)$ be two standard probability spaces.
A coupling of $P_1$ and $P_2$ is a probability measure $\gamma \in \mathcal{P}(\I_1 \times \I_2)$ on the product $\sigma$-algebra $\mathscr{B}_1 \times \mathscr{B}_2$ such that:
\begin{equation*}
\begin{aligned}
  &\text{for all measurable sets } A_1 \subseteq \I_1, \quad \gamma(A_1 \times \I_2) = P_1(A_1), \\
  &\text{for all measurable sets } A_2 \subseteq \I_2, \quad \gamma(\I_1 \times A_2) = P_2(A_2).
\end{aligned}
\end{equation*}
We also say that it is a coupling of $\I_1$ and $\I_2$ when there is no ambiguity of $P_1, P_2$. We denote by $\Pi(P_1, P_2)$ or $\Pi(\I_1, \I_2)$ the set of all couplings.
\end{defi}
\noindent
Note that when $\I_1 = \{1,\dots,N_1\}$ and $\I_2 = \{1,\dots,N_2\}$ are equipped with a uniform measure of $1/N_1$ and $1/N_2$ at each point, this definition coincides with the $\Pi(N_1,N_2)$ coupling for discrete sets in Definition~\ref{defi:coupling_distance_1}.

\begin{defi} [\textbf{Bi-coupling distances}] \label{defi:coupling_distance}
For $n = 1,2$, let $\I_n$ be standard probability spaces, $w^{(n)} \in L^\infty(\I_n \times \I_n)$ and $X^{(n)}: \I_n \to \D$. Define the 
bi-coupling distance from an operator norm $L^p \to L^q$ and Wasserstein-1 distance, as 
\begin{equation*}
\begin{aligned}
&d_{L^p \to L^q,W_1} \big[ (w^{(1)}, X^{(1)}), (w^{(2)}, X^{(2)})\big]
\\
& \ \defeq \inf_{\gamma \in \Pi(\I_1,\I_2)} \bigg[ \int_{\I_1 \times \I_2} |X^{(1)}(\xi_1) - X^{(2)}(\xi_2)| \, \gamma(\rd \xi_1, \rd \xi_2)
\\
& \hspace{2.5cm} + \| T_{w^{(1)}} \circ T_{\gamma} - T_{\gamma} \circ T_{w^{(2)}} \|_{L^p(\I_2) \to L^q(\I_1)}
\\
& \hspace{2.5cm} + \|T_{\gamma^T} \circ T_{w^{(1)}} - T_{w^{(2)}} \circ T_{\gamma^T} \|_{L^p(\I_1) \to L^q(\I_2)} \bigg],
\end{aligned}
\end{equation*}
where $\gamma^T \in \Pi(\I_2,\I_1)$ is defined as $\gamma^T(\xi,\zeta) = \gamma(\zeta,\xi)$, and $T_{w^{(1)}},T_{w^{(2)}}, T_{\gamma}, T_{\gamma^T}$ are the adjacency operators.

The bi-coupling distance from the cut norm and the Wasserstein-1 distance, denoted by $d_{\square,W_1}$, is defined analogously, with the $L^p \to L^q$ operator norm in the above formulation replaced by the cut norm $\|\cdot\|_{\square}$ of the kernels whose adjacency operators are $(T_{w^{(1)}} \circ T_{\gamma} - T_{\gamma} \circ T_{w^{(2)}})$ and $(T_{\gamma^T} \circ T_{w^{(1)}} - T_{w^{(2)}} \circ T_{\gamma^T})$.

\end{defi}
\noindent
Again, when $p = q = 2$, this definition coincides with the discrete definition of $d_{L^2 \to L^2,W_1}$ in Definition~\ref{defi:coupling_distance_1}.

The following definition addresses the solutions of extended Vlasov PDE~\eqref{eqn:multi_agent_Vlasov_first}.
Intuitively speaking, what we obtained from \eqref{eqn:multi_agent_Vlasov_first} are distributions along the $\xi$-fibers, and we are splitting them back to Dirac deltas.
This definition seems novel and is the conceptual key step that allows us to go beyond a priori correspondence and fiberwise differences in $\xi$, as discussed in Section~\ref{subsec:posteriori}.
\begin{defi} [\textbf{Random variable lifted from law}] \label{defi:dirac_blow_up}
Let $\I$ be a standard probability space, $w \in L^\infty(\I \times \I)$, and $f \in L^\infty(\I; \mathcal{M}(\D))$ such that for a.e. $\xi \in \I$, one has $f(\rd x, \xi) \in \mathcal{P}(\D)$.
Define probability space $\D \rtimes_f \I$ as the product space $\D \times \I$ equipped with $f \in \mathcal{P}(\D \times \I)$ as its probablity measure, i.e.
\begin{equation*}
\begin{aligned}
&\forall \text{ measurable sets } A \subseteq \D, B \subseteq \I, \quad f(A \times B) = \int_{A \times B} f(x,\xi) \rd x \rd \xi.
\end{aligned}
\end{equation*}
Define the $f$-lift of the pair
\begin{equation*}
\begin{aligned}
w \in L^\infty(\I \times \I), \quad f \in L^\infty(\I; \mathcal{M}(\D)),
\end{aligned}
\end{equation*}
as the pair
\begin{equation*}
\begin{aligned}
w_f \in L^\infty((\D \rtimes_f \I) \times (\D \rtimes_f \I)), \quad X_f: \D \rtimes_f \I \to \D,
\end{aligned}
\end{equation*}
where
\begin{equation*}
\begin{aligned}
w_f(\cdot,\cdot) = w(p_{\I}(\cdot),p_{\I}(\cdot)), \quad X_f(\cdot) = p_{\D}(\cdot),
\end{aligned}
\end{equation*}
and the mappings $p_{\I}: \D \times \I \to \I$, $p_{\D}: \D \times \I \to \D$ are the canonical projections for product sets.

To simplify notation, when the domain and function are clear from context, we omit the subscript and refer simply to ``$(w,X)$ being the lift of $(w,f)$''.

\end{defi}

We are now ready to state the full version of Theorem~\ref{thm:main_discrete}.
\begin{thm}[\textbf{Main result}] \label{thm:main_limit}

Let $\D = \T$, and let $\mu \in W^{1,\infty}(\T)$, $\sigma \in W^{2,\infty}(\T \times \T)$, $\nu \geq 0$.
For each $n \in \N$,
let $N_n \in \N \cup \{\infty\}$, $w^{(n)} \in L^\infty(\I_n \times \I_n)$ and $X^{(n)}: \I_n \to \T$ be either of the following:
\begin{itemize}
\item[--] For some $N_n \geq 1$, $\I_n = \{1,\dots,N_n\}$ equipped with uniform atomic measure $1/N_n$ on each point. The pair $w^{(n)}\in \R^{N_n \times N_n}$, $X^{(n)}_0 \in \T^{N_n}$ is connection weight matrices and initial data for the ODE/SDE system \eqref{eqn:multi_agent_SDE}.

\item[--] For formally $N_n = \infty$, $\I_n'$ is an atomless probability space. The pair $w^{(n)} \in L^\infty(\I_n' \times \I_n')$, $f^{(n)}_0 \in L^\infty(\I_n'; \mathcal{M}(\T))$ that for a.e. $\xi \in \I_n'$, 
$f^{(n)}_0(\cdot,\xi) \in \mathcal{P}(\T)$, is the weight kernel and initial data for the extended Vlasov PDE~\eqref{eqn:multi_agent_Vlasov_first}.

Let $\I_n = \I_n' \ltimes_{f^{(n)}_0} \T$, $w^{(n)} \in L^\infty(\I_n \times \I_n)$ and $X^{(n)}_0: \I_n \to \T$ be the lift.
\end{itemize}
Assume that $N_n \to \infty$, $\|w^{(n)}\|_{L^\infty} \leq w_{\max} < \infty$, and the initial data $X^{(n)}_0$ are deterministic.
Then the following holds:
\begin{itemize}
\item (Compactness). There exists a subsequence of $(w^{(n)}, X^{(n)}_0)$ (which we still index by $n$), an atomless standard probability space $\I_\infty$, and a pair $w^{(\infty)} \in L^\infty(\I_\infty' \times \I_\infty')$, $f^{(\infty)}_0 \in L^\infty(\I_\infty'; \mathcal{M}(\T))$ that for a.e. $\xi \in \I_\infty'$, 
$f^{(\infty)}_0(\cdot,\xi)$ concentrate as a Dirac delta. Let $\I_\infty = \I_\infty' \ltimes_{f^{(\infty)}_0} \T$, $w^{(\infty)} \in L^\infty(\I_\infty \times \I_\infty)$ and $X^{(\infty)}_0: \I_\infty \to \T$ be the lift. Then
\begin{equation*}
\begin{aligned}
\lim_{n \to \infty} d_{L^2 \to L^2,W_1} \big[ (w^{(n)}, X^{(n)}_0), (w^{(\infty)}, X^{(\infty)}_0)\big] = 0.
\end{aligned}
\end{equation*}

\item (Well-posedness). For each $n \in \N \cup \{\infty\}$, there exists either
\begin{itemize}
\item
an unique solution $X^{(n)}(t)$ of ODE or SDE system \eqref{eqn:multi_agent_SDE} with connection weights $w^{(n)}$ and initial data $X^{(n)}_0$, for all $t \geq 0$.
(When $\nu > 0$, $X^{(n)}(t)$ are stochastic processes only depend on the randomness of independent Brownian motions $(B^{n}_{i})_{i \in [N_n]}$.)
\item
an unique solution $f^{(n)}(t)$ of the extended Vlasov PDE~\eqref{eqn:multi_agent_Vlasov_first} with weight kernel $w^{(n)}$ and initial data $f^{(n)}_0$, for all $t \geq 0$, 
$f^{(n)}(t) \in L^\infty(\I_n'; \mathcal{M}(\T))$ that for a.e. $\xi \in \I_n'$, 
$f(t,\cdot,\xi) \in \mathcal{P}(\T)$.

Let $\I_n(t) = \I_n' \ltimes_{f^{(n)}(t)} \T$, $w^{(n)} \in L^\infty(\I_n(t) \times \I_n(t))$ and $X^{(n)}(t): \I_n(t) \to \T$ be the lift. (Note that the definition of $\I_n(t)$ now depends on time $t$).

\end{itemize}

\item (Stability).
If the initial data $(w^{(n)}, X^{(n)}_0)$ is converging to $(w^{(\infty)}, X^{(\infty)}_0)$ in $d_{L^2 \to L^2,W_1}$, then at any later time $t > 0$,
\begin{equation*}
\begin{aligned}
\lim_{n \to \infty} \E\bigg[ d_{L^2 \to L^2,W_1} \big[ (w^{(n)}, X^{(n)}(t)), (w^{(\infty)}, X^{(\infty)}(t))\big] \bigg] = 0.
\end{aligned}
\end{equation*}

\end{itemize}

\end{thm}

\subsection{Organization of the article}

The remainder of this article is organized as follows.

In Sections~\ref{subsec:compactness_argument}-\ref{subsec:proving_main}, we introduce all the concepts and definitions necessary for our analysis. Assuming several technical results, we provide a proof of Theorem~\ref{thm:main_limit} at the end of Section~\ref{subsec:proving_main}.  We make a conclusion for this article and discuss further perspectives in Section~\ref{subsec:perspectives}.

Section~\ref{sec:new_technical} contains the proofs of most of these technical results, with exceptions for Lemma~\ref{lem:main_observable}, Lemma~\ref{lem:graphon_compactness_first}, Lemma~\ref{lem:main_counting} and Lemma~\ref{lem:main_tree_counting}.
This is due to their length and because they are essentially straightforward extensions of results about mean-field theory in \cite{jabin2025mean, jabin2023mean}, results about graphon theory in \cite{borgs2008convergent, lovasz2006limits}, and results about graph fractional isomorphism in \cite{boker2021graph, dell2018lovasz, grebik2022fractional}.
For clarity, the proofs of these results are presented respectively in Appendix~\ref{sec:kinetic}, \ref{sec:graph_1} and \ref{sec:graph_2}. We make no claim of novelty for any of the proof techniques presented in these sections, which are written in detail solely to establish the necessary results for this article. Where possible, intermediate results are cited directly rather than derived in detail. We encourage readers to consult the original works and the monograph \cite{lovasz2012large} about graph limits for further motivation and insights.

\section{The whole picture and the proof of Theorem~\ref{thm:main_limit}} \label{sec:picture}

\subsection{Uniform metric bound from compactness argument} \label{subsec:compactness_argument}

As explained in the introduction, the proof relies heavily on compactness arguments. Below, we outline the basic lemmas that are used throughout the article. The proofs are postponed to Section~\ref{subsec:compactness_argument_proof}.
\begin{lem} \label{lem:uniform_bound_1}
Let $E$ be a compact topological space and let $d_1$ and $d_2$ be two metrics on $E$ that induce this topology. Then, there exists a non-decreasing function $C^{\downarrow} : \mathbb{R}_+ \to \mathbb{R}_+ $ with $ \lim_{r \downarrow 0} C^{\downarrow}(r) = 0$ such that for all $x, y \in E$,
\begin{equation*}
\begin{aligned}
d_1(x, y) \leq C^{\downarrow}(d_2(x, y)) \quad \text{and} \quad d_2(x, y) \leq C^{\downarrow}(d_1(x, y)).
\end{aligned}
\end{equation*}
\end{lem}

The following lemma serves as a complement to Lemma~\ref{lem:uniform_bound_1}, providing a bound to the expectation of distance in the case of random variables.

\begin{lem} \label{lem:uniform_bound_2}
Let $D_1$ and $D_2$ be non-negative random variables. Suppose there exists a constant $A > 0$ and a non-decreasing function $C^{\downarrow} : \mathbb{R}_+ \to \mathbb{R}_+$ satisfying $\lim_{r \downarrow 0} C^{\downarrow}(r) = 0$ such that, almost surely,
\begin{equation*}
\begin{aligned}
D_1 \leq C^{\downarrow}(D_2) \quad \text{and} \quad D_1 \leq A.
\end{aligned}
\end{equation*}
Then, there exists a non-decreasing function $C^{\downarrow}_A : \mathbb{R}_+ \to \mathbb{R}_+$ satisfying $\lim_{r \downarrow 0} C^{\downarrow}_A(r) = 0$ such that
\begin{equation*}
\begin{aligned}
\mathbb{E}[D_1] \leq C^{\downarrow}_A\left( \mathbb{E}[D_2] \right).
\end{aligned}
\end{equation*}
\end{lem}

With these lemmas in hand, and given that we work within compact spaces throughout the article, we will not emphasize the distinction between topological equivalence and uniform bounds across metrics in later statements, and we will use topological equivalence to simplify the statements.

\subsection{Functional-valued graphons} \label{subsec:functional_graphons}

Now we introduce some concepts in graph limit theory generalized to higher dimensions.
The proofs are postponed to Section~\ref{subsec:functional_graphons_proof}.

We begin with the definition of norms for kernels with values in Hilbert space $\mathcal{H}$, which are extended from the operator norm and cut norm.
When $\mathcal{H} = \R$, this will serve as our definition of the cut norm $\|\cdot\|_{\square}$ for kenrels on general $\I_1 \times \I_2$.

\begin{defi} [\textbf{Norms for Hilbert-valued kernels}] \label{defi:graphon_norm_Hilbert}
Let $\mathcal{H}$ be a separable Hilbert space, and let $\I_1, \I_2$ be standard probability spaces.
For $\bm{w} \in L^\infty(\I_1 \times \I_2; \mathcal{H})$ and any $1 \leq p,q \leq \infty$, define
\begin{equation*}
\begin{aligned}
\|\bm{w}\|_{p,q;\mathcal{H}} \defeq 
\sup_{\bm{e} \in \mathcal{H}_{\leq 1}, \ f \in L^{q'}(\I_2), \ g \in L^{p}(\I_1)} \bigg\langle \bm{e}, \int_{\I_1 \times \I_2} f(\xi) \bm{w}(\xi,\zeta) g(\zeta) \rd \xi \rd \zeta \bigg\rangle_{\mathcal{H}},
\end{aligned}
\end{equation*}
where $q' = q/(q-1)$ is the Hölder conjugate on $[1,+\infty]$.
In addition, define
\begin{equation*}
\begin{aligned}
\|\bm{w}\|_{\square;\mathcal{H}} \defeq \sup_{\bm{e} \in \mathcal{H}_{\leq 1}, \ S \subset \I_1, \ T \subset \I_2} \bigg\langle \bm{e}, \int_{S \times T} \bm{w}(\xi,\zeta) \rd \xi \rd \zeta \bigg\rangle_{\mathcal{H}},
\end{aligned}
\end{equation*}
where $S \subset \I_1, T \subset \I_2$ are any measurable subsets.

\end{defi}
\noindent
Note that when $\mathcal{H} = \R$, the norm $\|w\|_{p,q;\R}$ corresponds to the $L^p(\I_2) \to L^q(\I_1)$ operator norm of the adjacency operator $T_w$. More generally, our definition is equivalent to
\begin{equation*}
\begin{aligned}
\|\bm{w}\|_{p,q;\mathcal{H}} & =
\sup_{\bm{e} \in \mathcal{H}_{\leq 1}} \| \langle \bm{e}, \bm{w}\rangle_{\mathcal{H}} \|_{\square},
\\
\|\bm{w}\|_{p,q;\mathcal{H}} & = 
\sup_{\bm{e} \in \mathcal{H}_{\leq 1}} \| \langle \bm{e}, \bm{w}\rangle_{\mathcal{H}} \|_{L^p(\I_2) \to L^q(\I_1)}.
\end{aligned}
\end{equation*}
However, when $\mathcal{H}$ is infinite-dimensional, $\|\bm{w}\|_{p,q;\mathcal{H}}$ represents a weaker notion than the $L^p(\I_2) \to L^q(\I_1; \mathcal{H})$ norm of $T_{\bm{w}}$.

The following lemma is a direct consequence of interpolation theory for function spaces.

\begin{lem} [\textbf{Norm interpolations}] \label{lem:Hilbert_graph_interpolation}
For any $\bm{w} \in L^\infty(\I_1 \times \I_2; \mathcal{H})$,
\begin{equation*}
\begin{aligned}
\|\bm{w}\|_{\square;\mathcal{H}} \leq \|\bm{w}\|_{\infty,1;\mathcal{H}} \leq 4\, \|\bm{w}\|_{\square;\mathcal{H}}.
\end{aligned}
\end{equation*}
For any $1< p \leq \infty$ and $1 \leq q < \infty$,
\begin{equation*}
\begin{aligned}
\|\bm{w}\|_{\infty,1;\mathcal{H}} \leq \|\bm{w}\|_{p,q;\mathcal{H}} \leq \|\bm{w}\|_{\infty,1;\mathcal{H}}^{1-\theta} \|\bm{w}\|_{L^\infty(\I_2 \times \I_1; \mathcal{H})}^\theta.
\end{aligned}
\end{equation*}
with $\theta = \max\left\{\frac{1}{p}, 1-\frac{1}{q}\right\}$.
\end{lem}

From the norm introduced here, we define three ways to measure the distance between kernels. Firstly, the ``labeled distances'' are simply to take the norm of the difference, which of course requires the two kernels to lie in the same space $\I \times \I$.
\begin{defi} \label{defi:labeled_distance_Hilbert}
Let $\mathcal{H}$ be a separable Hilbert space, and let $\I$ be standard probability spaces.
For $\bm{w}_1, \bm{w}_2 \in L^\infty(\I \times \I; \mathcal{H})$, define labeled distances
\begin{equation*}
\begin{aligned}
d_{\square;\mathcal{H}}(\bm{w}_1,\bm{w}_2) &\defeq \|\bm{w}_1 - \bm{w}_2\|_{\square;\mathcal{H}},
\\
d_{p,q;\mathcal{H}}(\bm{w}_1,\bm{w}_2) &\defeq \|\bm{w}_1 - \bm{w}_2\|_{p,q;\mathcal{H}},
\end{aligned}
\end{equation*}
for all $1 \leq p,q \leq \infty$.
\end{defi}

Next, we introduce ``unlabeled distances'' in analogy to the unlabeled cut distance in graph theory. They can be understood as pseudometrics on $L^\infty(\I \times \I; \mathcal{H})$.
However, for the sake of convenience in our arguments, we extend the definition to kernels where the underlying spaces, $\I_1 \times \I_1$ and $\I_2 \times \I_2$, differ.

\begin{defi} [\textbf{Unlabeled distances}] \label{defi:unlabeled_distance_Hilbert}
Let $\mathcal{H}$ be a separable Hilbert space, and let $\I_1, \I_2$ be atomless standard probability spaces.
For $\bm{w}_1 \in L^\infty(\I_1 \times \I_1; \mathcal{H})$ and $\bm{w}_2 \in L^\infty(\I_2 \times \I_2; \mathcal{H})$, define unlabeled distances

\begin{equation*}
\begin{aligned}
\delta_{\square;\mathcal{H}}(\bm{w}_1,\bm{w}_2) &\defeq \inf_{\Phi} \|\bm{w}_1 - \bm{w}_2^{\Phi} \|_{\square;\mathcal{H}},
\\
\delta_{p,q;\mathcal{H}}(\bm{w}_1,\bm{w}_2) &\defeq \inf_{\Phi} \|\bm{w}_1 - \bm{w}_2^{\Phi} \|_{p,q;\mathcal{H}},
\end{aligned}
\end{equation*}
for all $1 \leq p,q \leq \infty$.
Here, $\bm{w}_2^{\Phi} = \bm{w}_2(\Phi(\cdot),\Phi(\cdot))$ and $\Phi$ ranges on all bijective measure-preserving maps $\Phi: \I_1 \rightarrow \I_2$.

\end{defi}

The necessity of the atomless assumption in the definition of unlabeled distances finds a direct analogy in the classical optimal transport theory.
The presence of atoms can lead to situations where the mass must be ``split'' to satisfy the push-forward condition, which is not permissible within the Monge definition of the Wasserstein distance.

Lastly, we introduce the concept of ``coupling distances'', which allows any coupling between probability spaces.

\begin{defi} [\textbf{Coupling distances}] \label{defi:coupling_distance_Hilbert}
Let $\mathcal{H}$ be a separable Hilbert space, and let $\I_1$ and $\I_2$ be standard probability spaces.
For $\bm{w}_1 \in L^\infty(\I_1 \times \I_1; \mathcal{H})$ and $\bm{w}_2 \in L^\infty(\I_2 \times \I_2; \mathcal{H})$, define coupling distances
\begin{equation*}
\begin{aligned}
\gamma_{\square;\mathcal{H}}(\bm{w}_1,\bm{w}_2) &\defeq \inf_{\gamma \in \Pi(\I_1,\I_2)} \bigg( \|\bm{w}_1 \gamma -\gamma \bm{w}_2\|_{\square;\mathcal{H}} + \|\gamma^T \bm{w}_1 - \bm{w}_2 \gamma^T \|_{\square;\mathcal{H}} \bigg),
\\
\gamma_{p,q;\mathcal{H}}(\bm{w}_1,\bm{w}_2) &\defeq \inf_{\gamma \in \Pi(\I_1,\I_2)} \bigg( \|\bm{w}_1 \gamma -\gamma \bm{w}_2\|_{p,q;\mathcal{H}} + \|\gamma^T \bm{w}_1 - \bm{w}_2 \gamma^T \|_{p,q;\mathcal{H}} \bigg),
\end{aligned}
\end{equation*}
for all $1 \leq p,q \leq \infty$.
Here, $\gamma^T \in \Pi(\I_2, \I_1)$ is given by $\gamma^T(\xi, \zeta) = \gamma(\zeta, \xi)$ for all $(\xi, \zeta) \in \I_2 \times \I_1$. The term $(\bm{w}_1 \gamma) \in L^\infty(\I_1 \times \I_2; \mathcal{H})$ is defined as
\begin{equation*}
\begin{aligned}
(\bm{w} \gamma)(\xi,\zeta) = \int_{\I_1} \bm{w}(\xi,\iota) \gamma(\iota, \zeta) \rd \iota \in \mathcal{H}, \quad \forall (\xi,\zeta) \in \I_1 \times \I_2.
\end{aligned}
\end{equation*}
Similarly, $(\gamma \bm{w}_2) \in L^\infty(\I_1 \times \I_2; \mathcal{H})$, as well as $(\gamma^T \bm{w}_1)$ and $(\bm{w}_2 \gamma^T) \in L^\infty(\I_2 \times \I_1; \mathcal{H})$, are defined in the natural way.
\end{defi}

The following proposition states that the definitions we have introduced are pseudometrics. However, we do not provide a detailed characterization of the resulting space here, as properties such as completeness and compactness remain unclear to us unless we restrict ourselves to specific compact subsets of $\mathcal{H}$.
Even with such assumptions, the proof requires separate techniques arising from graph limit theory. These properties are added later in Lemma~\ref{lem:graphon_compactness_first}.

\begin{prop} \label{prop:pseudometrics}
The unlabeled distances $\delta_{\square;\mathcal{H}}$, $\delta_{p,q;\mathcal{H}}$ and $\gamma_{\square;\mathcal{H}}$, $\gamma_{p,q;\mathcal{H}}$ are pseudometrics, i.e. satisfying non-negativity, symmetry, and the triangle inequality, with the distance from any element to itself being zero.

\end{prop}

The following proposition implies that for $L^\infty$-bounded kernels, all unlabeled distances induce the same notion of convergence, and similarly, all coupling distances do as well. This result is a straightforward consequence of Lemma~\ref{lem:Hilbert_graph_interpolation}. Additionally, we show that coupling distances are relaxations of unlabeled distances.
\begin{prop} \label{prop:unlabeled_and_coupling}
The following statements hold for unlabeled distances (``$\delta$'') if $\I_1$, $\I_2$ are atomless standard probability spaces and hold for coupling distances (``$\gamma$'') if $\I_1$, $\I_2$ are standard probability spaces:
For any $\bm{w}_1 \in L^\infty(\I_1 \times \I_1; \mathcal{H})$ and $\bm{w}_2 \in L^\infty(\I_2 \times \I_2; \mathcal{H})$,
\begin{equation*}
\begin{aligned}
\delta_{\square;\mathcal{H}}(\bm{w}_1,\bm{w}_2) \leq \delta_{\infty,1;\mathcal{H}}(\bm{w}_1,\bm{w}_2) \leq 4\delta_{\square;\mathcal{H}}(\bm{w}_1,\bm{w}_2),
\\
\gamma_{\square;\mathcal{H}}(\bm{w}_1,\bm{w}_2) \leq \gamma_{\infty,1;\mathcal{H}}(\bm{w}_1,\bm{w}_2) \leq 4\gamma_{\square;\mathcal{H}}(\bm{w}_1,\bm{w}_2).
\end{aligned}
\end{equation*}
For all $1< p \leq \infty$ and $1 \leq q < \infty$,
\begin{equation*}
\begin{aligned}
\delta_{\infty,1;\mathcal{H}}(\bm{w}_1,\bm{w}_2) \leq \delta_{p,q;\mathcal{H}}(\bm{w}_1,\bm{w}_2) \leq \delta_{\infty,1;\mathcal{H}}(\bm{w}_1,\bm{w}_2)^{1-\theta} w_{\max}^\theta,
\\
\gamma_{\infty,1;\mathcal{H}}(\bm{w}_1,\bm{w}_2) \leq \gamma_{p,q;\mathcal{H}}(\bm{w}_1,\bm{w}_2) \leq \gamma_{\infty,1;\mathcal{H}}(\bm{w}_1,\bm{w}_2)^{1-\theta} w_{\max}^\theta,
\end{aligned}
\end{equation*}
with $\theta = \max\left\{\frac{1}{p}, 1-\frac{1}{q}\right\}$ and $w_{\max} = \max\{ \|\bm{w}_1\|_{L^\infty}, \|\bm{w}_2\|_{L^\infty} \}$.

Moreover, if $\I_1$, $\I_2$ are atomless standard probability spaces, then
\begin{equation*}
\begin{aligned}
\gamma_{\square;\mathcal{H}}(\bm{w}_1,\bm{w}_2) &\leq 2\delta_{\square;\mathcal{H}}(\bm{w}_1,\bm{w}_2),
\\
\gamma_{p,q;\mathcal{H}}(\bm{w}_1,\bm{w}_2) &\leq 2\delta_{p,q;\mathcal{H}}(\bm{w}_1,\bm{w}_2),
\end{aligned}
\end{equation*}
for all $1< p \leq \infty$ and $1 \leq q < \infty$.

\end{prop}

\subsection{Bi-coupling distances and functional-valued coupling distances} \label{subsec:bi_coupling_and_functional}
We now present how our bi-coupling distances are associated with the more abstract coupling distances on $L^\infty(\I \times \I; \mathcal{H})$. The proofs are postponed to Section~\ref{subsec:bi_coupling_and_functional_proof}.

To begin with, we clarify our definitions of the Fourier transform and Sobolev spaces to avoid any ambiguity in subsequent calculations.

\begin{defi}[\textbf{Fourier transform and Sobolev spaces on the torus}] \label{defi:negative_Sobolev_torus}
Let $\T$ denote the one-dimensional torus. Specifically, we consider $\T = \R/\Z$, where the points $x$ and $(x + n)$ for any $n \in \Z$ are identified.

For a function $f \in L^1(\T)$, the Fourier transform (or Fourier series coefficients) of $f$ is the sequence $( \hat{f}(m) )_{m \in \Z}$ defined by
\begin{equation*}
\begin{aligned}
\hat{f}(m) \defeq \mathcal{F}(f)(m) \defeq \int_{\mathbb{T}} f(x) e^{-2\pi i m x} \, \rd x, \quad \text{for all } m \in \Z.
\end{aligned}
\end{equation*}
This definition extends to generalized functions such as $f \in \mathcal{M}(\T)$ using the theory of tempered distributions and duality.

The Sobolev norms $H^{s}(\mathbb{T})$, for $s \in \R$, are defined using the Fourier transform:
\begin{equation*}
\begin{aligned}
\| f \|_{H^s(\mathbb{T})} \defeq \left( \sum_{m \in \mathbb{Z}} (1 + 4 \pi^2 m^2 )^s \Big| \hat{f}(m) \Big|^2 \right)^{\frac{1}{2}}.
\end{aligned}
\end{equation*}
\end{defi}

Next, we introduce a family of critical convolutional kernels associated with negative Sobolev spaces, following the notation and definitions in \cite{jabin2023mean}.
\begin{defi}[\textbf{$\Lambda_s$ kernels}]
\label{defi:Lambda_s}
For any $s > 1/4$, define the operator $\Lambda_s$ on $L^2(\T)$ via its Fourier transform. The Fourier coefficients of $\Lambda_s$ are given by
\begin{equation*}
\begin{aligned}
\hat{\Lambda}_s(m) \defeq \frac{1}{(1 + 4\pi^2 m^2)^s}, \quad \text{for all } m \in \mathbb{Z}.
\end{aligned}
\end{equation*}
For simplicity, we denote $\Lambda = \Lambda_1$.
\end{defi}

Next, we present some results regarding the $\Lambda_s$ kernels, which are critical to our analysis.

\begin{lem}
\label{lem:negative_Sobolev_torus}
The kernels $\Lambda_{s}$ satisfy the following properties:
\begin{enumerate}
\item Function spaces: If $s > 1/4$, $\Lambda_{s} \in L^2(\T)$. Also, if $s > 1/2$, $\Lambda_{s} \in L^\infty(\T)$.
  
\item Kernel properties: If $s > 1/4$, the function $\Lambda_{s}$ is positive, symmetric, and satisfies
\begin{equation*}
\begin{aligned}
\int_{\T} \Lambda_s(x) \, dx = 1.
\end{aligned}
\end{equation*}

\item Sobolev norm representations: If $s > 1/2$, the Sobolev norm $H^{-s}(\T)$ can be expressed as
\begin{equation*}
\begin{aligned}
\| f \|_{H^{-s}(\T)} &= \left( \int_{\T \times \T} f(x) \Lambda_s(x - y) f(y) \, \rd x \, \rd y \right)^{\frac{1}{2}}
\\
&= \| \Lambda_{s/2} \star f \|_{L^2(\T)}.
\end{aligned}
\end{equation*}
  
\item Explicit form of $\Lambda(x)$:
The kernel $\Lambda = \Lambda_1$ is explicitly given by
\begin{equation*}
\begin{aligned}
\Lambda(x) \defeq \sum_{l = -\infty}^\infty \frac{1}{2} \exp\left( -|x + l| \right).
\end{aligned}
\end{equation*}
Define
\begin{equation*}
\begin{aligned}
\Lambda_{\max} &\defeq \max_{x \in \T} \Lambda(x) = \sum_{l = -\infty}^\infty \frac{1}{2} \exp(-|l|), \\
\Lambda_{\min} &\defeq \min_{x \in \T} \Lambda(x) = \sum_{l = -\infty}^\infty \frac{1}{2} \exp\left( -\left| l + \tfrac{1}{2} \right| \right).
\end{aligned}
\end{equation*}
  
\item Norm of probability distributions:
The $H^{-1}(\T)$ norm of the Dirac delta at any $x \in \T$ is
\begin{equation*}
\begin{aligned}
\| \delta_x \|_{H^{-1}(\T)} = \sqrt{ \Lambda_{\max} },
\end{aligned}
\end{equation*}
and for any $f \in \mathcal{P}(\T)$ that is not concentrated as a Dirac delta, we have the strict inequality
\begin{equation*}
\begin{aligned}
\| f \|_{H^{-1}(\T)} < \sqrt{ \Lambda_{\max} }.
\end{aligned}
\end{equation*}
  
\item Lipschitz estimate for $\Lambda(x)$:
The following Lipschitz-type estimate holds for all $x \in \T$:
\begin{equation*}
\begin{aligned}
2 (\Lambda_{\max} - \Lambda_{\min}) |x| \leq \Lambda(0) - \Lambda(x) \leq \Lambda_{\max} |x|.
\end{aligned}
\end{equation*}
\end{enumerate}
\end{lem}

Next, we establish the compactness of negative Sobolev spaces.

\begin{lem} \label{lem:Hilbert_measure_weak_star}
For any $s > \frac{1}{2}$, the canonical embedding $\mathcal{M}(\T) \to H^{-s}(\T)$ is a compact operator. Moreover, on the space $\mathcal{M}_{\leq 1}(\T)$, the $H^{-s}(\T)$ norm induces the weak-* topology of the measures.
\end{lem}

We now demonstrate how a pair $(w, X)$ from Theorem~\ref{thm:main_limit} can be reformulated as a kernel with values in the Hilbert space $H^{-1}(\T) \oplus H^{-1}(\T)$. This is one of the most conceptually important definitions in this article.
\begin{defi} [\textbf{The kernel associated to $(w, X)$}] \label{defi:functionalize_graphon}
From the pair $w \in L^\infty(\I \times \I)$ and $X : \I \to \T$, we define the kernel $\bm{w}_{w,X} \in L^\infty(\I \times \I; \mathcal{M}(\T) \oplus \mathcal{M}(\T))$ as follows:
\begin{equation*}
\begin{aligned}
\bm{w}_{w,X}(\xi,\zeta) \defeq
\begin{pmatrix}
\delta_{X(\zeta)}
\\
w(\xi,\zeta) \delta_{X(\zeta)}
\end{pmatrix}
\in \mathcal{M}(\T) \oplus \mathcal{M}(\T), \quad \forall (\xi,\zeta) \in \T \times \T.
\end{aligned}
\end{equation*}
\end{defi}
\noindent
It is straightforward to see that $\bm{w}_{w,X} \in L^\infty(\I \times \I; H^{-1}(\T) \oplus H^{-1}(\T))$ from the inclusion $\mathcal{M}(\T) \subset H^{-1}(\T)$.

The following lemma shows that the bi-coupling distances defined on pairs of $(w,X)$ and the coupling distances defined on kernels $\bm{w}_{w,X}$ are topologically equivalent.
We avoid stating the lemma as convergence at this point because the closedness and compactness of the topology rely on Lemma~\ref{lem:graphon_compactness_first}, which is presented later in this section. Nevertheless, the uniform bound across the metrics, as suggested in Lemma~\ref{lem:uniform_bound_1}, can be directly achieved in the proof.
\begin{lem} \label{lem:coupling_equivalence}
Let $\mathcal{H} = H^{-1}(\T) \oplus H^{-1}(\T)$. For each $n \in \mathbb{N}$, let $\I_n$ be a probability space, let $w^{(n)} \in L^\infty(\I_n \times \I_n)$, and let $X^{(n)}: \I_n \to \T$ be measurable. Assume that there exists a constant $w_{\max} < \infty$ such that $\|w^{(n)}\|_{L^\infty(\I_n \times \I_n)} \leq w_{\max}$ for all $n \in \mathbb{N}$. Then, the following are all equivalent notions of Cauchy sequence:
\begin{itemize}
\item[(1)] under coupling distance from cut norm and $\mathcal{H} = H^{-1}(\T) \oplus H^{-1}(\T)$:
\begin{equation*}
\begin{aligned}
\lim_{n \to \infty} \sup_{n_1,n_2 \geq n} \gamma_{\square;\mathcal{H}}(\bm{w}_{w^{(n_1)},X^{(n_1)}},\bm{w}_{w^{(n_2)},X^{(n_2)}}) = 0,
\end{aligned}
\end{equation*}
\item[(2)] under coupling distance from operator norm ($1< p \leq \infty$ and $1 \leq q < \infty$) and $\mathcal{H} = H^{-1}(\T) \oplus H^{-1}(\T)$:
\begin{equation*}
\begin{aligned}
\lim_{n \to \infty} \sup_{n_1,n_2 \geq n} \gamma_{p,q;\mathcal{H}}(\bm{w}_{w^{(n_1)},X^{(n_1)}},\bm{w}_{w^{(n_2)},X^{(n_2)}}) = 0,
\end{aligned}
\end{equation*}
\item[(3)] under bi-coupling distance, from the cut norm and the Wasserstein-1 distance
\begin{equation*}
\begin{aligned}
\lim_{n \to \infty} \sup_{n_1,n_2 \geq n} d_{\square,W_1} \big[ (w^{(n_1)}, X^{(n_1)}), (w^{(n_2)}, X^{(n_2)})\big] = 0,
\end{aligned}
\end{equation*}
\item[(4)] under bi-coupling distance, from operator norm ($1< p \leq \infty$ and $1 \leq q < \infty$) and Wasserstein-1 distance
\begin{equation*}
\begin{aligned}
\lim_{n \to \infty} \sup_{n_1,n_2 \geq n} d_{L^p \to L^q,W_1} \big[ (w^{(n_1)}, X^{(n_1)}), (w^{(n_2)}, X^{(n_2)})\big] = 0.
\end{aligned}
\end{equation*} 

\end{itemize}

\end{lem}

This equivalence is closely related to Part 5 of Lemma~\ref{lem:negative_Sobolev_torus}, which states that Dirac deltas are the only elements in $\mathcal{M}_{\leq 1}(\T)$ that achieve the maximum in the $H^{-1}(\T)$ norm.
This ensures that when a Dirac delta is coupled with something different, there is always a positive contribution to the coupling distance $\gamma_{\square;\mathcal{H}}$.
While we can define a similar form for any $(w,f)$
\begin{equation*}
\begin{aligned}
\bm{w}_{w,f}(\xi,\zeta) =
\begin{pmatrix}
f(\zeta)
\\
w(\xi,\zeta) f(\zeta)
\end{pmatrix}
\in H^{-1}(\T) \oplus H^{-1}(\T), \quad \forall (\xi,\zeta) \in \T \times \T,
\end{aligned}
\end{equation*}
it is unclear whether the coupling disntace $\gamma_{\square;\mathcal{H}}$ would still satisfy an assertion analogous to Lemma~\ref{lem:coupling_equivalence}.
This provides technical reassurance for the discussion in Section~\ref{subsec:posteriori} that identifying a $\xi$-fiberwise law of agents $f(t,x,\xi)$ is quite unnatural.

\subsection{Functional-valued homomorphism density: Observables} \label{subsec:observables}

Now we explain how the coupling distances on $L^\infty(\I \times \I; \mathcal{H})$ are related to the observables (graph homomorphism densities).
The proofs are postponed to Section~\ref{subsec:observables_proof}.

To begin, we need to clarify several definitions further to avoid ambiguity, starting with the definition on the graph theory side.

\begin{defi} [\textbf{Oriented simple graph}]
An oriented simple graph $F = (\mathsf{v}(F), \mathsf{e}(F))$ consists of:
\begin{itemize}
\item A set of vertices $\mathsf{v}(F)$.
\item A set of directed edges $\mathsf{e}(F) \subset \mathsf{v}(F) \times \mathsf{v}(F)$ such that:
\begin{equation*}
\begin{aligned}
\text{ if } (i,j) \in \mathsf{e}(F), \text{ then } (j,i) \notin \mathsf{e}(F).
\end{aligned}
\end{equation*}
\end{itemize}
Define $\mathcal{G}$ as the set of all oriented simple graphs.

Additionally, for an oriented simple graph $F \in \mathcal{G}$, define $\mathsf{v}'(F)$ as
\begin{equation*}
\begin{aligned}
\mathsf{v}'(F) \defeq \{j \in \mathsf{v}(F): \exists i \in \mathsf{v}(F) \text{ s.t. } (i,j) \in \mathsf{e}(F)\}.
\end{aligned}
\end{equation*}

\end{defi} 
\noindent
In this context, the notation $(i, j) \in \mathsf{e}(F)$ indicates a directed edge from vertex $i$ to vertex $j$. Therefore, the set $\mathsf{v}'(F)$ represents all vertices in $F$ that serve as the starting points (i.e., sources) of at least one directed edge.
\begin{rmk}
The notion of an \emph{oriented simple graph} here is more restrictive than the definitions of a \emph{directed simple graph} commonly found in the literature, where symmetric pairs of directed edges (bi-directed edges) are allowed. 
This broader notion, however, does not play a role in our analysis.

\end{rmk}

We also need to clarify the following definitions from functional analysis.
\begin{defi} [\textbf{Product and direct/tensor products}] \label{defi:functional_analysis_definitions}
Let $S$ be a finite set. We define the following constructs:
\begin{itemize}

\item Given a topological space or measure space $\D$, the product space indexed by $S$ is defined as
\begin{equation*}
\begin{aligned}
\D^{S} \coloneqq \prod_{i \in S} \D.
\end{aligned}
\end{equation*}
For functions defined on the product space $\D^{S}$, the tensorization is defined by
\begin{equation*}
\begin{aligned}
f^{\otimes S}(x) \coloneqq \prod_{i \in S} f(x_i),
\end{aligned}
\end{equation*}
where $x = (x_i)_{i \in S} \in \D^{S}$.

\item Given a Banach space $\mathcal{B}$, we define the direct sum and tensor product spaces indexed by $S$ as follows:
\begin{equation*}
\begin{aligned}
\mathcal{B}^{\oplus S} \coloneqq \bigoplus_{i \in S} \mathcal{B}, \quad \mathcal{B}^{\otimes S} \coloneqq \bigotimes_{i \in S} \mathcal{B}.
\end{aligned}
\end{equation*}
Here, $\mathcal{B}^{\oplus S}$ denotes the direct sum of copies of $\mathcal{B}$ indexed by $S$, and $\mathcal{B}^{\otimes S}$ denotes the tensor product of copies of $\mathcal{B}$ indexed by $S$. (The tensor product should be understood as the projective tensor product, or equivalently, when $\mathcal{B}$ is a Hilbert space, the Hilbert tensor product.)

\item When the index set $S$ is the set $[k] = \{1, \dots, k\}$ for some positive integer $k$, we adopt the more conventional notations:
\begin{equation*}
\begin{aligned}
\D^{k} \coloneqq \D^{[k]}, \quad f^{\otimes k} \coloneqq f^{\otimes [k]}, \quad \mathcal{B}^{\oplus k} \coloneqq \mathcal{B}^{\oplus [k]}, \quad \mathcal{B}^{\otimes k} \coloneqq \mathcal{B}^{\otimes [k]}.
\end{aligned}
\end{equation*}

\end{itemize}

\end{defi}
\noindent Throughout the article, the index sets $S$ we will mainly use are $\mathsf{v}(F)$, $\mathsf{v}'(F)$ and $\mathsf{e}(F)$.

In the following, we summarize a few key results regarding the tensorization of $\mathcal{M}(\T) \subset H^{-s}(\T)$ that are required for our analysis.
\begin{prop}
\label{prop:negative_Sobolev_tensorized}
The following properties are true:
\begin{enumerate}

\item Sobolev norm representations: For any $s > 1/2$, the norm of the tensorized Hilbert space $(H^{-s}(\T))^{\otimes k}$ can be expressed as
\begin{equation*}
\begin{aligned}
\| f \|_{(H^{-s}(\T))^{\otimes k}} & = \bigg[ \sum_{m_1,\dots,m_k \in \Z} \bigg( \prod_{i \in [k]} (1 + 4 \pi^2 m_i^2 )^{-s} \bigg) \Big| \hat{f}(m_1,\dots,m_k) \Big|^2 \bigg]^{\frac{1}{2}}
\\
&= \left( \int_{\T^k \times \T^k} f(x) \Lambda_s^{\otimes k}(x - y) f(y) \, \rd x \, \rd y \right)^{\frac{1}{2}}
\\
&= \| \Lambda_{s/2}^{\otimes k} \star f \|_{L^2(\T^k)}.
\end{aligned}
\end{equation*}

\item Equivalent of topologies: For any $s > \frac{1}{2}$ and any $k \in \N$, the canonical embedding
\begin{equation*}
\begin{aligned}
\mathcal{M}(\T^k) \simeq (\mathcal{M}(\T))^{\otimes k} \to (H^{-s}(\T))^{\otimes k}
\end{aligned}
\end{equation*}
is a compact operator. Moreover, on the space $\mathcal{M}_{\leq 1}(\T^k)$, the $(H^{-s}(\T))^{\otimes k}$ norm induces the weak-* topology of the measures.

\end{enumerate}
\end{prop}

The next definition extends the classical notion of graphon homomorphism density using tensorized Hilbert spaces.

\begin{defi} [\textbf{Tensorized graph homomorphism}] \label{defi:tensorized_homomorphism}
For any $\bm{w} \in L^\infty(\I \times \I; \mathcal{H})$ and $F \in \mathcal{G}$, define the tensorized graph homomorphism density $\bm{t}(F,\bm{w}) \in \mathcal{H}^{\otimes \mathsf{e}(F)}$ as
\begin{equation*}
\begin{aligned}
\bm{t}(F,\bm{w}) \defeq \int_{\I^{\mathsf{v}(F)}} \bigotimes_{(i,j) \in \mathsf{e}(F)} \bm{w}(\xi_i,\xi_j) \prod_{i \in \mathsf{v}(F)} \rd \xi_i.
\end{aligned}
\end{equation*}

\end{defi}
\noindent
We would like to investigate the tensorized graph homomorphism of the kernel $\bm{w}_{w,X}$ associated with $(w,X)$ as defined in Definition~\ref{defi:functionalize_graphon}, restated here:
\begin{defi} [\textbf{The kernel associated to $(w, X)$, restate}] \label{defi:functionalize_graphon_restate}
For any $w \in L^\infty(\I \times \I)$ and any $X : \I \to \T$, define the kernel $\bm{w}_{w,X} \in L^\infty(\I \times \I; \mathcal{M}(\T)^{\oplus \{0,1\}})$ as follows:
\begin{equation*}
\begin{aligned}
\bm{w}_{w,X}(\xi,\zeta) \defeq
\begin{pmatrix}
\delta_{X(\zeta)}
\\
w(\xi,\zeta) \delta_{X(\zeta)}
\end{pmatrix}
\in \mathcal{M}(\T)^{\oplus \{0,1\}}, \quad \forall (\xi,\zeta) \in \T \times \T.
\end{aligned}
\end{equation*}
\end{defi}
\noindent
For simplicity in subsequent calculations, we take the index set of the direct sum to be $\{0,1\}$ in the restated definition.
This kernel $\bm{w}_{w,X}$ remarkably satisfies the following proposition:
\begin{prop} \label{prop:diagonal_concentration}
Let $w \in L^\infty(\I \times \I)$ and $X : \I \to \T$. Take the kernel $\bm{w}_{w,X}$ as defined in Definition~\ref{defi:functionalize_graphon_restate}.
Then by Definition~\ref{defi:tensorized_homomorphism},
\begin{equation*}
\begin{aligned}
& \bm{t}(F,\bm{w}_{w,X})
\\
& \ \in (\mathcal{M}(\T)^{\oplus \{0,1\}})^{\otimes \mathsf{e}(F)} \simeq (\mathcal{M}(\T)^{\otimes \mathsf{e}(F)})^{\oplus (\{0,1\}^{\mathsf{e}(F)})} \simeq \mathcal{M}(\T^{\mathsf{e}(F)})^{\oplus (\{0,1\}^{\mathsf{e}(F)})}.
\end{aligned}
\end{equation*}
For each $s \in \{0,1\}^{\mathsf{e}(F)}$, the component $\big( \bm{t}(F,\bm{w}_{w,X}) \big)_s \in \mathcal{M}(\T^{\mathsf{e}(F)})$ is given by
\begin{equation*}
\begin{aligned}
\big( \bm{t}(F,\bm{w}_{w,X}) \big)_s = \int_{\I^{\mathsf{v}(F)}} \prod_{(i,j) \in \mathsf{e}(F)} \big( w(\xi_i,\xi_j) \big)^{s_{i,j}} \bigotimes_{(i,j) \in \mathsf{e}(F)} \delta_{X(\xi_j)} \prod_{i \in \mathsf{v}(F)} \rd \xi_i,
\end{aligned}
\end{equation*}
and is supported in the subspace
\begin{equation*}
\begin{aligned}
\T^{\mathsf{e}(F)}|_{\textnormal{diag}} \defeq
\big\{
(x_{i,j})_{(i,j) \in \mathsf{e}(F)} \in \T^{\mathsf{e}(F)} : x_{i_1,j} = x_{i_2,j}, \ \forall (i_1,j), (i_2,j) \in \mathsf{e}(F) \big\}.
\end{aligned}
\end{equation*}

\end{prop}
\noindent
The proposition suggests that $\bm{t}(F,\bm{w}_{w,X})$ lies in a specific subspace of $(\mathcal{M}(\T)^{\oplus \{0,1\}})^{\otimes \mathsf{e}(F)}$, namely
\begin{equation*}
\begin{aligned}
\bm{t}(F,\bm{w}_{w,X}) \in \mathcal{M}(\T^{\mathsf{e}(F)}|_{\textnormal{diag}} )^{\oplus (\{0,1\}^{\mathsf{e}(F)})}.
\end{aligned}
\end{equation*}
Consequently, we introduce two types of definitions for observables to reflect this structure.
\begin{defi} \label{defi:plain_observables}
We define the following constructs:
\begin{itemize}
\item For any $w \in L^\infty(\I \times \I)$, any $X : \I \to \T$ and any $F \in \mathcal{G}$, define the normal observable $\bm{t}(F,w,X) \in \mathcal{M}(\T^{\mathsf{e}(F)}|_{\textnormal{diag}} )^{\oplus (\{0,1\}^{\mathsf{e}(F)})}$ as 
\begin{equation*}
\begin{aligned}
\bm{t}(F,w,X) \defeq
\bm{t}(F,\bm{w}_{w,X}).
\end{aligned}
\end{equation*}

\item For any $w \in L^\infty(\I \times \I)$, any $X : \I \to \T$ and any $F \in \mathcal{G}$, define the plain observable $\bm{\tau}(F,w,X) \in (\mathcal{M}(\T)^{\otimes \mathsf{v}'(F)})^{\oplus (\{0,1\}^{\mathsf{e}(F)})} \simeq \mathcal{M}(\T^{\mathsf{v}'(F)})^{\oplus (\{0,1\}^{\mathsf{e}(F)})}$ by the following:
On each $s \in \{0,1\}^{\mathsf{e}(F)}$, the component $(\bm{\tau}(F,w,X))_s \in \mathcal{M}(\T^{\mathsf{v}'(F)})$ is given by
\begin{equation*}
\begin{aligned}
\big( \bm{\tau}(F,w,X) \big)_s \defeq \int_{\I^{\mathsf{v}(F)}} \prod_{(i,j) \in \mathsf{e}(F)} \big( w(\xi_i,\xi_j) \big)^{s_{i,j}} \bigotimes_{j \in \mathsf{v}'(F)} \delta_{X(\xi_j)} \prod_{i \in \mathsf{v}(F)} \rd \xi_i.
\end{aligned}
\end{equation*}

\item Moreover, for any $w\in L^\infty(\I \times \I)$, any $f \in L^\infty(\I; \mathcal{M}(\T))$ that for a.e. $\xi \in \I$, 
$f(\xi,\cdot) \in \mathcal{P}(\T)$ and any $F \in \mathcal{G}$, define the plain observable $\bm{\tau}(F,w,f) \in (\mathcal{M}(\T)^{\otimes \mathsf{v}'(F)})^{\oplus (\{0,1\}^{\mathsf{e}(F)})} \simeq \mathcal{M}(\T^{\mathsf{v}'(F)})^{\oplus (\{0,1\}^{\mathsf{e}(F)})}$ by the following:
On each $s \in \{0,1\}^{\mathsf{e}(F)}$, the component $(\bm{\tau}(F,w,f))_s \in \mathcal{M}(\T^{\mathsf{v}'(F)})$ is given by
\begin{equation*}
\begin{aligned}
\big( \bm{\tau}(F,w,f) \big)_s \defeq \int_{\I^{\mathsf{v}(F)}} \prod_{(i,j) \in \mathsf{e}(F)} \big( w(\xi_i,\xi_j) \big)^{s_{i,j}} \bigotimes_{j \in \mathsf{v}'(F)} f(\xi_j) \prod_{i \in \mathsf{v}(F)} \rd \xi_i.
\end{aligned}
\end{equation*}

\end{itemize}

\end{defi}
\noindent
These definitions are related to the following proposition.
\begin{prop} \label{prop:normal_plain_equivalence}

The following canonical map is bijective
\begin{equation*}
\begin{aligned}
\mathsf{i}_{\mathsf{v}'(F) \to \mathsf{e}(F)}: 
\T^{\mathsf{v}'(F)} & \to
\T^{\mathsf{e}(F)}|_{\textnormal{diag}}
\\
(x_j)_{j \in \mathsf{v}'(F)} & \mapsto (x_{i,j} = x_j)_{(i,j) \in \mathsf{e}(F)}.
\end{aligned}
\end{equation*}
For any $w \in L^\infty(\I \times \I)$ and $X : \I \to \T$,
\begin{equation*}
\begin{aligned}
(\mathsf{i}_{\mathsf{v}'(F) \to \mathsf{e}(F)}^\#)^{\oplus (\{0,1\}^{\mathsf{e}(F)})} (\bm{\tau}(F,w,X)) = \bm{t}(F,w,X),
\end{aligned}
\end{equation*}
where $(\mathsf{i}_{\mathsf{v}'(F) \to \mathsf{e}(F)}^\#)^{\oplus (\{0,1\}^{\mathsf{e}(F)})}$ is the pushforward map of measures defined on
\begin{equation*}
\begin{aligned}
\mathsf{i}_{\mathsf{v}'(F) \to \mathsf{e}(F)}^\# &: \mathcal{M}(\T^{\mathsf{v}'(F)}) \to \mathcal{M}(\T^{\mathsf{e}(F)}|_{\textnormal{diag}})
\\
(\mathsf{i}_{\mathsf{v}'(F) \to \mathsf{e}(F)}^\#)^{\oplus (\{0,1\}^{\mathsf{e}(F)})} &: \mathcal{M}(\T^{\mathsf{v}'(F)})^{\oplus (\{0,1\}^{\mathsf{e}(F)})} \to \mathcal{M}(\T^{\mathsf{e}(F)}|_{\textnormal{diag}})^{\oplus (\{0,1\}^{\mathsf{e}(F)})}
\end{aligned}
\end{equation*}

In addition, for any
$w: L^\infty(\I \times \I)$ and $f \in L^\infty(\I; \mathcal{M}(\T))$ that for a.e. $\xi \in \I$, 
$f(\xi,\cdot) \in \mathcal{P}(\T)$.
Let the pair
\begin{equation*}
\begin{aligned}
w \in L^\infty((\I \ltimes_f \T) \times (\I \ltimes_f \T)), \quad X: \I \ltimes_f \T \to \T
\end{aligned}
\end{equation*}
be the lift of $(w,f)$. Then
\begin{equation*}
\begin{aligned}
\bm{\tau}(F,w,f) = \bm{\tau}(F,w,X).
\end{aligned}
\end{equation*}

\end{prop}

While either definition of observables would suffice to determine the other, introducing both offers distinct advantages. The normal observables $\bm{t}(F,w,X)$ are more aligned with the graph limit theory, as we obtain it directly through the tensorized homomorphism densities and can extend almost all the arguments from \cite{borgs2008convergent, lovasz2006limits} and \cite{boker2021graph, dell2018lovasz, grebik2022fractional}.
In contrast, the plain observables $\bm{\tau}(F,w,X)$ align more closely with the mean-field theory, which simplifies the associated hierarchies of PDEs when the arguments of \cite{jabin2025mean, jabin2023mean} are reproduced. Moreover, the plain observables $\bm{\tau}(F,w,X)$ have the unique advantage of being directly defined from $(w, f)$. Using a tensorized density approach to define normal observables $\bm{t}(F,w,X)$ from $(w, f)$ would lead to correlations within densities on each $\xi$-fiber. This would result in the support of the measure falling outside of the $\T^{\mathsf{e}(F)}|_{\textnormal{diag}}$ subspace, causing the subsequent arguments to break down entirely.

We now summarize all equivalent notions of convergence for these definitions that will be used throughout the article. Unlike in Lemma~\ref{lem:coupling_equivalence}, we state the lemma here in terms of convergence, as the weak-* compactness of Borel measures on compact sets is a well-known classical result.

\begin{prop} \label{prop:normal_plain_equivalence}
Let $n \in \N$. For any sequence of standard probability spaces $\I_l$, kernels $w^{(n)} \in L^\infty(\I_n \times \I_n)$, states $X^{(n)}: \I_n \to \T$, and any subset $\mathcal{G}' \subseteq \mathcal{G}$,
the following notions of convergence are all equivalent:
\begin{itemize}
\item[(1)] under weak-* convergence of all normal observables:
\begin{equation*}
\begin{aligned}
\forall F \in \mathcal{G}', \bm{t}(F,w^{(n)},X^{(n)}) \wsto \bm{t}(F)_\infty \in \mathcal{M}(\T^{\mathsf{e}(F)})^{\oplus (\{0,1\}^{\mathsf{e}(F)})},
\end{aligned}
\end{equation*}
\item[(2)] under weak-* convergence of all components of plain observables:
\begin{equation*}
\begin{aligned}
\forall F \in \mathcal{G}', \forall s \in \{0,1\}^{\mathsf{e}(F)}, 
\big( \bm{\tau}(F,w^{(n)},X^{(n)}) \big)_s \wsto \big( \bm{\tau}(F)_\infty \big)_s \in \mathcal{M}(\T^{\mathsf{v}'(F)}),
\end{aligned}
\end{equation*}
\item[(3)] under tensorized negative Sobolev convergence of all normal observables:
\begin{equation*}
\begin{aligned}
\forall F \in \mathcal{G}', \lim_{n \to \infty} \|\bm{t}(F,w^{(n)},X^{(n)}) - \bm{t}(F)_\infty\|_{(H^{-1}(\T)^{\otimes \mathsf{e}(F)})^{\oplus (\{0,1\}^{\mathsf{e}(F)})}} = 0,
\end{aligned}
\end{equation*}
\item[(4)] under tensorized negative Sobolev convergence of all components of plain observables:
\begin{equation*}
\begin{aligned}
& \forall F \in \mathcal{G}', \forall s \in \{0,1\}^{\mathsf{e}(F)}, 
\\
& \hspace{1cm} \lim_{n \to \infty} \big\| \big( \bm{\tau}(F,w^{(n)},X^{(n)}) \big)_s - \big( \bm{\tau}(F)_\infty \big)_s \big\|_{H^{-1}(\T)^{\otimes \mathsf{v}'(F)}} = 0,
\end{aligned}
\end{equation*}

\item[(5)] under the following metrization of normal observables by weighted supremum:
\begin{equation*}
\begin{aligned}
& \lim_{n \to \infty} \sup_{F \in \mathcal{G}'} \bigg[ (4(1+w_{\max})\sqrt{\Lambda_{\max}})^{-|\mathsf{e}(F)|}
\\
& \hspace{2cm} \|\bm{t}(F,w^{(n)},X^{(n)}) - \bm{t}(F)_\infty\|_{(H^{-1}(\T)^{\otimes \mathsf{e}(F)})^{\oplus (\{0,1\}^{\mathsf{e}(F)})}} \bigg] = 0,
\end{aligned}
\end{equation*}
\item[(6)] under the following metrization of plain observables by weighted supremum:
\begin{equation*}
\begin{aligned}
& \lim_{n \to \infty} \sup_{F \in \mathcal{G}', s \in \{0,1\}^{\mathsf{e}(F)}} \bigg[ (4(1+w_{\max})\sqrt{\Lambda_{\max}})^{-|\mathsf{e}(F)|}
\\
& \hspace{3cm} \big\| \big( \bm{\tau}(F,w^{(n)},X^{(n)}) \big)_s - \big( \bm{\tau}(F)_\infty \big)_s \big\|_{H^{-1}(\T)^{\otimes \mathsf{v}'(F)}} \bigg] = 0.
\end{aligned}
\end{equation*}
\end{itemize}
\end{prop}

An example of a subset $\mathcal{G}' \subseteq \mathcal{G}$ from the above definition, which is of particular interest to us, is defined below.
\begin{defi} [\textbf{Oriented tree}] \label{defi:directed_trees}

An oriented tree $T = (\mathsf{v}(T), \mathsf{e}(T))$ is an oriented simple graph (i.e., $T \in \mathcal{G}$) satisfying the following conditions:

\begin{itemize}
\item There are no undirected cycles in $T$. That is, there does not exist a sequence of distinct vertices $v_1, v_2, \dots, v_k \in \mathsf{v}(T)$ with $k \geq 2$ such that by taking $v_{k+1} = v_1$,
\begin{equation*}
\begin{aligned}
(v_i, v_{i+1}) \in \mathsf{e}(T) \quad \text{or} \quad (v_{i+1}, v_i) \in \mathsf{e}(T), \quad \forall i \in [k]
\end{aligned}
\end{equation*}
  
\item The underlying undirected graph of \( T \) is connected. That is, for any two vertices \( u, v \in \mathsf{v}(T) \), there exists a sequence of vertices \( u = w_0, w_1, \dots, w_k = v \) such that
\begin{equation*}
\begin{aligned}
(w_i, w_{i+1}) \in \mathsf{e}(T) \quad \text{or} \quad (w_{i+1}, w_i) \in \mathsf{e}(T), \quad \forall i \in [k-1].
\end{aligned}
\end{equation*}
\end{itemize}
Define $\mathcal{T} \subset \mathcal{G}$ as the set of all oriented trees.

\end{defi}
\begin{rmk}
Our notion of oriented trees is broader than the definitions of directed trees commonly found in the literature, where edge directions are often strictly specified. Specifically, our definition gives a larger family of trees than that presented in \cite{jabin2025mean, jabin2023mean}.

\end{rmk}

\subsection{Extending existing literature and proving Theorem~\ref{thm:main_limit}} \label{subsec:proving_main}

The following lemmas extend key results from the existing literature. Proofs for these lemmas will be provided in subsequent sections.

The first lemma establishes well-posedness and stability under an alternative metric.
\begin{lem} [\textbf{Main result in observables}] \label{lem:main_observable}
Let $\mu \in W^{1,\infty}(\T)$, $\sigma \in W^{2,\infty}(\T \times \T)$, $\nu \geq 0$.
For each $n \in \N$,
let $N_n \in \N \cup \{\infty\}$, $w^{(n)} \in L^\infty(\I_n \times \I_n)$ and $X^{(n)}: \I_n \to \T$ be either of the following:
\begin{itemize}
\item For some $N_n \geq 1$, $\I_n = \{1,\dots,N_n\}$ equipped with uniform atomic measure $1/N_n$ on each point, and $w^{(n)}\in \R^{N_n \times N_n}$ and $X^{(n)}_0 \in \T^{N_n}$, be connection weight matrices and initial data for the ODE/SDE system \eqref{eqn:multi_agent_SDE}.

\item For formally $N_n = \infty$, $\I_n'$ is an atomless probability space and $w^{(n)} \in L^\infty(\I_n' \times \I_n')$, $f^{(n)}_0 \in L^\infty(\I_n'; \mathcal{M}(\T))$ that for a.e. $\xi \in \I_n'$, 
$f^{(n)}_0(\cdot,\xi) \in \mathcal{P}(\T)$, be the weight kernel and initial data for the extended Vlasov PDE~\eqref{eqn:multi_agent_Vlasov_first}.

Let $\I_n = \I_n' \ltimes_{f^{(n)}_0} \T$, $w^{(n)} \in L^\infty(\I_n \times \I_n)$ and $X^{(n)}_0: \I_n \to \T$ be the lift.
\end{itemize}
Assume that $N_n \to \infty$, $\|w^{(n)}\|_{L^\infty} \leq w_{\max} < \infty$, and the initial data $X^{(n)}_0$ are deterministic.
Then the following holds:
\begin{itemize}

\item (Well-posedness). For each $n \in \N \cup \{\infty\}$, there exists either
\begin{itemize}
\item
an unique solution $X^{(n)}(t)$ of ODE or SDE system \eqref{eqn:multi_agent_SDE} with connection weights $w^{(n)}$ and initial data $X^{(n)}_0$, for all $t \geq 0$.
(When $\nu > 0$, $X^{(n)}(t)$ are stochastic processes only depend on the randomness of independent Brownian motions $(B^{n}_{i})_{i \in [N_n]}$.)
\item
an unique solution $f^{(n)}(t)$ of the extended Vlasov PDE~\eqref{eqn:multi_agent_Vlasov_first} with weight kernel $w^{(n)}$ and initial data $f^{(n)}_0$, for all $t \geq 0$, 
$f^{(n)}(t) \in L^\infty(\I_n'; \mathcal{M}(\T))$ that for a.e. $\xi \in \I_n'$, 
$f(t,\cdot,\xi) \in \mathcal{P}(\T)$.

Let $\I_n(t) = \I_n' \ltimes_{f^{(n)}(t)} \T$, $w^{(n)} \in L^\infty(\I_n(t) \times \I_n(t))$ and $X^{(n)}(t): \I_n(t) \to \T$ be the lift. (Note that the definition of $\I_n(t)$ now depends on time $t$).

\end{itemize}

\item (Stability).
If the initial data $(w^{(n)}, X^{(n)}_0)$ is converging to $(w^{(\infty)}, X^{(\infty)}_0)$ in the sense of
\begin{equation*}
\begin{aligned}
& \lim_{n \to \infty} \sup_{T \in \mathcal{T}, s \in {\{0,1\}}^{\mathsf{e}(T)}} \bigg[ (4(1+w_{\max})\sqrt{\Lambda_{\max}})^{-|\mathsf{e}(T)|}
\\
& \hspace{1.5cm} \big\|\big(\bm{\tau}(T,w^{(n)},X^{(n)}_0)\big)_s - \big(\bm{\tau}(T,w^{(\infty)},X^{(\infty)}_0)\big)_s \big\|_{H^{-1}(\T)^{\otimes \mathsf{v}'(T)}} \bigg] = 0,
\end{aligned}
\end{equation*}
then at any later time $t > 0$,
\begin{equation*}
\begin{aligned}
&\lim_{n \to \infty} \E\Bigg[ \sup_{T \in \mathcal{T}, s \in {\{0,1\}}^{\mathsf{e}(T)}} \bigg[ (4(1+w_{\max})\sqrt{\Lambda_{\max}})^{-|\mathsf{e}(T)|}
\\
& \hspace{0.5cm} \big\|\big(\bm{\tau}(T,w^{(n)},X^{(n)}(t))\big)_s - \big(\bm{\tau}(T,w^{(\infty)},X^{(\infty)}(t)) \big)_s \big\|_{H^{-1}(\T)^{\otimes \mathsf{v}'(T)}} \bigg] \Bigg] = 0.
\end{aligned}
\end{equation*}

\end{itemize}

\end{lem}
\noindent
The full proof of this lemma is provided in Appendix~\ref{sec:kinetic}, where we apply techniques from \cite{jabin2025mean, jabin2023mean}.

The next lemma complements the previous one by addressing the compactness. Although it is stated in terms of the unlabeled distance $\delta_{\square;\mathcal{H}}$, Proposition~\ref{prop:unlabeled_and_coupling} implies that $\gamma_{\square;\mathcal{H}} \leq \delta_{\square;\mathcal{H}}$, thereby ensuring compactness with respect to the weaker coupling distance $\gamma_{\square;\mathcal{H}}$ as well.
\begin{lem} [\textbf{Compactness lemma}] \label{lem:graphon_compactness_first}
Let $\mathcal{B}$ be a Banach space compactly embedded into a separable Hilbert space $\mathcal{H}$.
For any sequence $\{\bm{w}_n\}_{n=1}^\infty \subset L^\infty(\I_n \times \I_n; \mathcal{B})$ that each $\I_n$ is an atomless probability space and satisfying uniform bound that
\begin{equation*}
\begin{aligned}
\sup_n \|\bm{w}_n\|_{L^\infty(\I_n \times \I_n; \mathcal{B})} < w_{\max} < \infty,
\end{aligned}
\end{equation*}
up to an extraction of subsequence (which we still index by $n$), there exists $\bm{w} \in L^\infty(\I \times \I; \mathcal{B})$ that 
\begin{equation*}
\begin{aligned}
\lim_{n \to \infty}
\delta_{\square;\mathcal{H}}(\bm{w}_n, \bm{w}) = 0.
\end{aligned}
\end{equation*}
Moreover, if the sequence $\{\bm{w}_n\}_{n=1}^\infty$ is defined from $w^{(n)} \in L^\infty(\I_n \times \I_n)$ and $X^{(n)} : \I_n \to \T$, and $\bm{w}_n = \bm{w}_{w^{(n)},X^{(n)}} \in L^\infty(\I_n \times \I_n; H^{-1}(\T)^{\oplus \{0,1\}})$ defined as in Definition~\ref{defi:functionalize_graphon_restate},
there exist limiting $w \in L^\infty(\I \times \I)$ and $X: \I \to \T$, such that $\bm{w} = \bm{w}_{w,X}$.

\end{lem}
\noindent
The full proof of this lemma is presented in Appendix~\ref{sec:graph_1}, extending the results from \cite{lovasz2006limits}, but following the proof outlined in the textbook \cite{lovasz2012large}.

The next lemma is an extension of the celebrated Counting Lemma and Inverse Counting Lemma \cite{borgs2008convergent}.
\begin{lem} [\textbf{Counting and Inverse Counting Lemma I}] \label{lem:main_counting}
Let $\mathcal{B}$ be a compact Banach space embedded into a separable Hilbert space $\mathcal{H}$.
If $\{\bm{w}_n\}_{n=1}^\infty \cup \{\bm{w}\} \subset L^\infty([0,1]^2; \mathcal{B})$ satisfy uniform bound that
\begin{equation*}
\begin{aligned}
\sup_n \|\bm{w}_n\|_{L^\infty([0,1]^2; \mathcal{B})} < w_{\max} < \infty
\end{aligned}
\end{equation*}
Then the following notions of convergence are equivalent:

\begin{itemize}
\item
under tensorized negative Sobolev convergence of all graph homomorphism of all oriented simple graphs:
\begin{equation*}
\begin{aligned}
\forall F \in \mathcal{G}, \quad 
\lim_{n \to \infty} \| \bm{t}(F,\bm{w}_n) - \bm{t}(F,\bm{w}) \|_{\mathcal{H}^{\otimes \mathsf{e}(F)}} = 0,
\end{aligned}
\end{equation*}

\item under unlabeled distance from cut norm and $\mathcal{H}$:
\begin{equation*}
\begin{aligned}
\lim_{n \to \infty} \delta_{\square;\mathcal{H}}( \bm{w}_n, \bm{w}) = 0.
\end{aligned}
\end{equation*}

\end{itemize}
\end{lem}
The complete proof of this lemma is presented in Appendix~\ref{sec:graph_1}, following the approach outlined in the textbook \cite{lovasz2012large}. While this lemma is not directly applied in the main proof, it is a critical component of the subsequent lemma, which demonstrates that all equivalent notions of convergence described in Lemma~\ref{lem:coupling_equivalence} and Proposition~\ref{prop:normal_plain_equivalence} are mutually equivalent across both definitions. 
In particular, this implies that the bi-coupling distance used in Theorem~\ref{thm:main_limit} and the metric used in Lemma~\ref{lem:main_observable} are equivalent.

\begin{lem} [\textbf{Counting and inverse counting lemma II}] \label{lem:main_tree_counting}

Let $\mathcal{B}$ be a compact Banach space embedded into a separable Hilbert space $\mathcal{H}$.
If $\{\bm{w}_n\}_{n=1}^\infty \cup \{\bm{w}\} \subset L^\infty([0,1]^2; \mathcal{B})$ satisfy uniform bound that
\begin{equation*}
\begin{aligned}
\sup_n \|\bm{w}_n\|_{L^\infty([0,1]^2; \mathcal{B})} < w_{\max} < \infty
\end{aligned}
\end{equation*}
Then the following notions of convergence are equivalent,

\begin{itemize}
\item
under tensorized negative Sobolev convergence of all graph homomorphism of all oriented trees:
\begin{equation*}
\begin{aligned}
\forall T \in \mathcal{T}, \quad 
\lim_{n \to \infty} \| \bm{t}(T,\bm{w}_n) - \bm{t}(T,\bm{w}) \|_{\mathcal{H}^{\otimes |\mathsf{e}(F)|}} = 0,
\end{aligned}
\end{equation*}

\item under coupling distance from cut norm and $\mathcal{H}$:
\begin{equation*}
\begin{aligned}
\lim_{n \to \infty} \gamma_{\square;\mathcal{H}}( \bm{w}_n, \bm{w}) = 0.
\end{aligned}
\end{equation*}

\end{itemize}

\end{lem}
The complete proof of this lemma is presented in Appendix~\ref{sec:graph_2}, which extends the results from \cite{boker2021graph, dell2018lovasz, grebik2022fractional}. Our proof mainly follows the arguments in \cite{grebik2022fractional}.

\begin{rmk}
The proof of Lemma~\ref{lem:main_tree_counting} is the only part where we do not yet know how to replace the compactness argument with an explicit uniform bound. If such a replacement could be achieved, it would yield a quantitative stability estimate with an explicit rate.
\end{rmk}

We are now ready to prove Theorem~\ref{thm:main_limit}, assuming that all the lemmas hold. In fact, the lemmas have already completed most of the proof. It only remains to assemble them in the appropriate order.
\begin{proof}[Proof of Theorem~\ref{thm:main_limit}]

Take $\mathcal{B} = \mathcal{M}(\T)^{\oplus \{0,1\}}$ and $\mathcal{H} = H^{-1}(\T)^{\oplus \{0,1\}}$.
By extending $w^{(n)} \in L^\infty(\I_n \times \I_n)$ and $X^{(n)}: \I_n \to \T$ to $(\I_n \times [0,1])$ we may assume the underlying spaces are atomless without changing the sequence $\bm{w}_{w^{(n)},X^{(n)}}$ in $\gamma_{\square;\mathcal{H}}$.
The compactness in $\delta_{\square;\mathcal{H}}$ follows directly from Lemma~\ref{lem:graphon_compactness_first}. By Proposition~\ref{prop:unlabeled_and_coupling}, this also implies the compactness in $\gamma_{\square;\mathcal{H}}$ and $\gamma_{p,q;\mathcal{H}}$ for all $1< p \leq \infty$ and $1 \leq q < \infty$.
By Lemma~\ref{lem:coupling_equivalence}, Proposition~\ref{prop:normal_plain_equivalence} (choosing $ \mathcal{G}' = \mathcal{T}$), and Lemma~\ref{lem:main_tree_counting}, we find that the compactness in such $\gamma_{p,q;\mathcal{H}}$ is further equivalent to the compactness under bi-coupling distances $d_{L^p \to L^q,W_1}$ (for all $1< p \leq \infty$ and $1 \leq q < \infty$) and under the maximum over all observables in Lemma~\ref{lem:main_observable}.

By Lemma~\ref{lem:main_observable}, we obtain well-posedness and stability in expectation with respect to the maximum over all observables. 
Applying Lemma~\ref{lem:coupling_equivalence}, Proposition~\ref{prop:normal_plain_equivalence}, and Lemma~\ref{lem:main_tree_counting} again, along with Lemma~\ref{lem:uniform_bound_1} and Lemma~\ref{lem:uniform_bound_2}, we conclude stability in expectation with respect to bi-coupling distances $d_{L^p \to L^q,W_1}$ for all $1< p \leq \infty$ and $1 \leq q < \infty$.

\end{proof}

\subsection{Conclusions and perspectives} \label{subsec:perspectives}

In this article, we rigorously derived the convergence of \emph{empirical data} of a \emph{non-exchangeable} system \eqref{eqn:multi_agent_SDE} on $\T$ toward a deterministic limit solving the corresponding extended Vlasov equation \eqref{eqn:multi_agent_Vlasov_first}. This was established on a distance that is defined through a \emph{convex} optimization problem, which is an interpolation of the optimal transport between measures and the fractional overlay between graphs.

However, proving this result requires a long detour, delving deeply into more concepts of kinetic theory and graph theory. In particular, most of the technical parts of the article were written in terms of \emph{observables}, which can be understood as a tensorization of agent laws and graph homomorphism density.
This work perhaps suggests that the two fields should not only pay attention to each other’s results but also strive for deeper exchanges of ideas.

There are several directions that we wish to explore further, starting from this first result.

\begin{itemize}

\item

Firstly, there is the straightforward problem of explicit convergence rate. Let us remark that the stability in Theorem~\ref{thm:main_limit} can be made quantitative in terms of observables. However, this a priori bound is, roughly speaking, double-exponential in time (see \cite{jabin2025mean, jabin2023mean} for a more detailed discussion).
The absence of convergence rate in the main theorem is due to the significant role compactness arguments played in at least one part of our proofs, namely the proof of Lemma~\ref{lem:main_tree_counting}.
We acknowledge that fully adapting the compactness arguments within our analysis to explicitly quantitative bounds may present challenges.
Moreover, if the goal is to improve the convergence rate, the prospects of the current proof remain uncertain. Even when ``projected onto the two components'', that is, by considering the particle system and graph limit techniques in the proof separately, the best convergence rates available in the existing literature are quite poor. Achieving a practical convergence rate may require entirely different approaches.
It might also be more practical to begin with problems that impose stronger assumptions on connection weights, even at the expense of reduced generality.

\item

A second interesting open question is to what extent our results can be adapted to the various alternative definitions of graph limits found in the existing literature. We expect (though not necessarily correctly) that establishing an analogous result based on $L^p$ graphons with $p > 1$ would be technically highly challenging but could follow essentially the same roadmap as outlined in this article. Also, it could be that the resulting convergence is only valid for the bi-coupling distance with a weaker operator norm (such as $L^\infty \to L^1$), which might result in less favorable properties from an optimization perspective.
A more non-trivial extension is toward the extended kernels discussed in \cite{jabin2025mean, jabin2023mean}, which corresponds to very sparse connection weights. While all our definitions are at least well-defined (in the sense that no integrability issue during the definition process), the current proof relies on the inverse counting lemma of standard graphon, which cannot be directly defined under the conditions in \cite{jabin2025mean, jabin2023mean}. To reproduce the results, the minimum requirement would be a counting lemma for tree homomorphisms and fractional overlay that does not depend on the standard inverse counting lemma. However, additional challenges are difficult to foresee without further exploration.

\item

A third concern, shared with \cite{jabin2023mean}, is how to extend the results to other types of dynamics. Our proof relies heavily on $H^{-1}$ estimates in one dimension. Extending this to higher dimensions is conceptually straightforward, as long as a sufficiently large exponent $s$ is chosen, $\mathcal{M}(\R^d)$ is included by $H^{-s}(\R^d)$ by Sobolev embedding.
However, applying the Leibniz rule in such $H^{-s}(\R^d)$ would require greater differentiability of $\mu$ and $\sigma$ compared to what is assumed in this article. To establish results under weaker differentiability, more advanced techniques from commutator estimates and harmonic analysis may be necessary.

\item

A final interest lies on the application side. Even without convergence rates, our results are compelling from a computational perspective and may warrant numerical experiments. Moreover, both the bi-coupling distance and observables may hold some significance from a data perspective. The optimization problem associated with bi-coupling distance can be interpreted as revealing the correspondence between different parts of two multi-agent systems. Meanwhile, the observables project the system into a high-dimensional space, which could be helpful in studying the intrinsic properties of a single multi-agent system, which are often hidden behind permutation invariance.

\end{itemize}

\section{Proof of the technical results} \label{sec:new_technical}

We present in this section the proofs of the technical results from Section~\ref{sec:picture}.

\subsection{Uniform metric bound from compactness argument} \label{subsec:compactness_argument_proof}

\begin{proof} [Proof of Lemma~\ref{lem:uniform_bound_1}]
We only prove $d_1(x, y) \leq C^{\downarrow}(d_2(x, y))$ because the two inequalities are obviously symmetric in the argument.
Assume, for the sake of contradiction, that there exists $r > 0$ such that for every $n \in \N$, there exist points $x_n, y_n \in E$ with $d_2(x_n, y_n) < \frac{1}{n}$ but $d_1(x_n, y_n) > r$.

Since $E$ is compact under the topology induced by $d_1$ and $d_2$, the sequence $\{x_n\}$ has a convergent subsequence (which we still denote by $\{x_{n}\}$ for convenience) converging to some $x \in E$.
Given that $d_2(x_n, y_n) < \frac{1}{n}$, it follows that $d_2(y_n, x) \to 0$. But since $d_1$ and $d_2$ induce the same topology, $d_1(y_n, x) \to 0$ as well.

Using the triangle inequality for $d_1$:
\begin{equation*}
\begin{aligned}
d_1(x_{n}, y_{n}) \leq d_1(x_{n}, x) + d_1(x, y_{n}).
\end{aligned}
\end{equation*}
Since $x_{n} \to x$ and $y_{n} \to x$, we have $d_1(x_{n}, x) \to 0$ and $d_1(x, y_{n}) \to 0$. Therefore, $d_1(x_{n}, y_{n}) \to 0$.

This contradicts the assumption that $d_1(x_{n}, y_{n}) > r > 0$.
\end{proof}

\begin{proof} [Proof of Lemma~\ref{lem:uniform_bound_2}]

Consider an arbitrary threshold $r > 0$. We can decompose the expectation of $D_1$ based on the event $\{ D_2 \leq r \}$ and its complement $\{ D_2 > r \}$:
\begin{equation*}
\begin{aligned}
\mathbb{E}[D_1] = \mathbb{E}\left[ D_1 \cdot \mathbbm{1}_{\{ D_2 \leq r \}} \right] + \mathbb{E}\left[ D_1 \cdot \mathbbm{1}_{\{ D_2 > r \}} \right].
\end{aligned}
\end{equation*}
Using the given almost sure bounds, we have:
\begin{equation*}
\begin{aligned}
\mathbb{E}[D_1] & \leq \mathbb{E}\left[ C^{\downarrow}(D_2) \cdot \mathbbm{1}_{\{ D_2 \leq r \}} \right] + \mathbb{E}\left[ A \cdot \mathbbm{1}_{\{ D_2 > r \}} \right].
\\
& = C^{\downarrow}(r) \cdot \mathbb{P}(D_2 \leq r) + A \cdot \mathbb{P}(D_2 > r).
\end{aligned}
\end{equation*}
Noting that $\mathbb{P}(D_2 > r) \leq \frac{\mathbb{E}[D_2]}{r}$ by Markov's inequality, we obtain:
\begin{equation*}
\begin{aligned}
\mathbb{E}[D_1] \leq C^{\downarrow}(r) + A \cdot \frac{\mathbb{E}[D_2]}{r}.
\end{aligned}
\end{equation*}
Thus, by optimizing the choice of $r$, we obtain the desired function $C^{\downarrow}_A$.

\end{proof}

\subsection{Functional-valued graphons} \label{subsec:functional_graphons_proof}
\begin{proof}[Proof of Lemma~\ref{lem:Hilbert_graph_interpolation}]

We begin with the first line. 
For any fixed vector $\bm{e}$ with $\|\bm{e}\|_{\mathcal{H}} \leq 1$, standard arguments from graphon theory (see \cite{lovasz2012large}~Chapter~8.2.4) imply that $f \in L^\infty(\I_1), g \in L^\infty(\I_2)$ are saturated only when $f, g = \pm 1$ almost everywhere. This yields, for any $\bm{e} \in \mathcal{H}_{\leq 1}$,
\begin{equation*}
\begin{aligned}
\sup_{\ f \in L^{\infty}(\I_2), \ g \in L^{\infty}(\I_1)} \bigg\langle \bm{e}, \int_{\I_1 \times \I_2} f(\xi) \bm{w}(\xi,\zeta) g(\zeta) \rd \xi \rd \zeta \bigg\rangle_{\mathcal{H}} \leq 4 \| \langle \bm{e}, \bm{w}\rangle_{\mathcal{H}} \|_{\square} \leq 4 \|\bm{w}\|_{\square;\mathcal{H}}
\end{aligned}
\end{equation*}
Taking the supremum over $\bm{e} \in \mathcal{H}_{\leq 1}$, we conclude the first line.

For the second line, note that $\I_1$ and $\I_2$ are probability spaces (that is, each has total measure 1). It is straightforward to apply H\"older arguments to get
\begin{equation*}
\begin{aligned}
\|\bm{w}\|_{\infty,1;\mathcal{H}} \leq \|\bm{w}\|_{p,q;\mathcal{H}} \leq \|\bm{w}\|_{L^\infty(\I_2 \times \I_1; \mathcal{H})}, \quad \forall 1 \leq p,q \leq \infty.
\end{aligned}
\end{equation*}
For the reverse direction, consider again any fixed $\bm{e} \in \mathcal{H}_{\leq 1}$, which gives
\begin{equation*}
\begin{aligned}
\sup_{\ f \in L^{q'}(\I_2), \ g \in L^{p}(\I_1)} \bigg\langle \bm{e}, \int_{\I_1 \times \I_2} f(\xi) \bm{w}(\xi,\zeta) g(\zeta) \rd \xi \rd \zeta \bigg\rangle_{\mathcal{H}} \leq \| \langle \bm{e}, \bm{w}\rangle_{\mathcal{H}} \|_{L^p \to L^q}
\end{aligned}
\end{equation*}
Applying the Riesz-Thorin or Marcinkiewicz interpolation theorem with $\theta = \max\left\{\frac{1}{p}, 1 - \frac{1}{q}\right\}$, we obtain
\begin{equation*}
\begin{aligned}
\| \langle \bm{e}, \bm{w}\rangle_{\mathcal{H}} \|_{L^p \to L^q} & \leq
\| \langle \bm{e}, \bm{w}\rangle_{\mathcal{H}} \|_{L^\infty \to L^1}^{1-\theta} \| \langle \bm{e}, \bm{w}\rangle_{\mathcal{H}} \|_{\text{Second Endpoint}}^{\theta}
\\
& \leq \|\bm{w}\|_{\infty,1;\mathcal{H}}^{1-\theta} \|\bm{w}\|_{L^\infty(\I_2 \times \I_1; \mathcal{H})}^\theta,
\end{aligned}
\end{equation*}
Taking the supremum over $\bm{e} \in \mathcal{H}_{\leq 1}$ completes the second line.

\end{proof}

\begin{proof} [Proof of Proposition~\ref{prop:pseudometrics}]

The non-negativity and the property that the distance from any element to itself is zero are straightforward. It remains to prove the symmetry and the triangle inequality.
Our arguments are written in the $\|\cdot \|_{p,q;\mathcal{H}}$-norms; however, it is straightforward to extend to $\|\cdot \|_{\square;\mathcal{H}}$-norms.

Note that any bijective measure-preserving map preserves the $L^p$ norms, i.e. for such $\Phi: \I_1 \to \I_2$, and for any function $f \in L^p(\I_1)$, the composition satisfies
\begin{equation*}
\begin{aligned}
\|f\|_{L^p(\I_1)} = \|f \circ \Phi^{-1} \|_{L^p(\I_2)},
\end{aligned}
\end{equation*}
by the change of variables formula.
This essentialy results that for any $\bm{w} \in L^\infty(\I_2 \times \I_2; \mathcal{H})$ and any bijective measure-preserving $\Phi: \I_{1} \to \I_2$,
\begin{equation*}
\begin{aligned}
&\|\bm{w}^{\Phi} \|_{p,q;\mathcal{H}}
\\
& \ = \sup_{\bm{e} \in \mathcal{H}_{\leq 1}, \ f \in L^{q'}(\I_1), \ g \in L^{p}(\I_1)} \bigg\langle \bm{e}, \int_{\I_1 \times \I_1} f(\xi) \bm{w}(\Phi(\xi),\Phi(\zeta)) g(\zeta) \rd \xi \rd \zeta \bigg\rangle_{\mathcal{H}},
\\
& \ = \sup_{\bm{e} \in \mathcal{H}_{\leq 1}, \ f \in L^{q'}(\I_1), \ g \in L^{p}(\I_1)} \bigg\langle \bm{e}, \int_{\I_2 \times \I_2} f(\Phi^{-1}(\xi)) \bm{w}(\xi,\zeta) g(\Phi^{-1}(\zeta)) \rd \xi \rd \zeta \bigg\rangle_{\mathcal{H}},
\\
& \ \leq \|\bm{w} \|_{p,q;\mathcal{H}}.
\end{aligned}
\end{equation*}
Apply the argument to the opposite direction with $\Phi^{-1}$, we conclude that
\begin{equation*}
\begin{aligned}
\|\bm{w}^{\Phi} \|_{p,q;\mathcal{H}} = \|\bm{w} \|_{p,q;\mathcal{H}}.
\end{aligned}
\end{equation*}
Hence, for $\bm{w}_1 \in L^\infty(\I_1 \times \I_1; \mathcal{H})$, $\bm{w}_2 \in L^\infty(\I_2 \times \I_2; \mathcal{H})$, and bijective measure-preserving map $\Phi: \I_1 \rightarrow \I_2$,
\begin{equation*}
\begin{aligned}
\|\bm{w}_1 - \bm{w}_2^{\Phi} \|_{p,q;\mathcal{H}}
= \|\bm{w}_1^{\Phi^{-1}} - \bm{w}_2 \|_{p,q;\mathcal{H}},
\end{aligned}
\end{equation*}
which symmetrizes the optimization problem in unlabeled distances.
For the coupling distances, the optimization of
\begin{equation*}
\begin{aligned}
\|\bm{w}_1 \gamma -\gamma \bm{w}_2\|_{p,q;\mathcal{H}} + \|\gamma^T \bm{w}_1 - \bm{w}_2 \gamma^T \|_{p,q;\mathcal{H}}
\end{aligned}
\end{equation*}
can be symmetrized by simple interchanging the roles of $\gamma$ and its transpose $\gamma^T$.
This completes the proof of symmetry.

To establish the triangle inequality, consider arbitrary elements $\bm{w}_n \in L^\infty(\I_n \times \I_n; \mathcal{H})$, $n = 1,2,3$ along with mappings $\Phi_{n+1,n}: \I_{n} \to \I_{n+1}$, $n = 1,2$. We then have the following inequality:
\begin{equation*}
\begin{aligned}
\delta_{p,q;\mathcal{H}}(\bm{w}_1,\bm{w}_3) & \leq \big\|\bm{w}_1 - \bm{w}_3^{\Phi_{3,2} \circ \Phi_{2,1}} \big\|_{p,q;\mathcal{H}}
\\
& \leq \big\|\bm{w}_1 - \bm{w}_2^{\Phi_{2,1}} \big\|_{p,q;\mathcal{H}} + \big\|\bm{w}_2^{\Phi_{2,1}} - \bm{w}_3^{\Phi_{3,2} \circ \Phi_{2,1}} \big\|_{p,q;\mathcal{H}}
\\
& \leq \big\|\bm{w}_1 - \bm{w}_2^{\Phi_{2,1}} \big\|_{p,q;\mathcal{H}} + \big\|\bm{w}_2 - \bm{w}_3^{\Phi_{3,2}} \big\|_{p,q;\mathcal{H}}.
\end{aligned}
\end{equation*}
Taking the infimum over $\Phi_{2,1}$ and $\Phi_{3,2}$, we obtain the triangle inequality
\begin{equation*}
\begin{aligned}
\delta_{p,q;\mathcal{H}}(\bm{w}_1,\bm{w}_3) \leq \delta_{p,q;\mathcal{H}}(\bm{w}_1,\bm{w}_2) + \delta_{p,q;\mathcal{H}}(\bm{w}_2,\bm{w}_3).
\end{aligned}
\end{equation*}
Also, consider arbitrary elements $\bm{w}_n \in L^\infty(\I_n \times \I_n; \mathcal{H})$, $n = 1,2,3$ along with coupling $\gamma_{1,2} \in \Pi(\I_1,\I_2), \gamma_{2,3} \in \Pi(\I_2,\I_3)$. We then have
\begin{equation*}
\begin{aligned}
\gamma_{p,q;\mathcal{H}}(\bm{w}_1,\bm{w}_3) & \leq \|\bm{w}_1 \gamma_{12} \gamma_{23} -\gamma_{12} \gamma_{23} \bm{w}_3\|_{p,q;\mathcal{H}}
\\
& \hspace{0.5cm} + \|(\gamma_{12} \gamma_{23})^T \bm{w}_1 - \bm{w}_3 (\gamma_{12} \gamma_{23})^T \|_{p,q;\mathcal{H}}.
\end{aligned}
\end{equation*}
The first term can, for example, be expanded as 
\begin{equation*}
\begin{aligned}
& \|\bm{w}_1 \gamma_{12} \gamma_{23} -\gamma_{12} \gamma_{23} \bm{w}_3\|_{p,q;\mathcal{H}}
\\
& \ = \sup_{\bm{e} \in \mathcal{H}_{\leq 1}}
\| T_{\langle \bm{e}, \bm{w}_1 \rangle_{\mathcal{H}} } \circ T_{\gamma_{12}} \circ T_{\gamma_{23}} - T_{\gamma_{12}} \circ T_{\gamma_{23}} \circ T_{\langle \bm{e}, \bm{w}_3 \rangle_{\mathcal{H}} } \|_{L^p(\I_3) \to L^q(\I_1)}
\\
& \ \leq \sup_{\bm{e} \in \mathcal{H}_{\leq 1}}
\| T_{\langle \bm{e}, \bm{w}_1 \rangle_{\mathcal{H}} } \circ T_{\gamma_{12}} \circ T_{\gamma_{23}} - T_{\gamma_{12}} \circ T_{\langle \bm{e}, \bm{w}_2 \rangle_{\mathcal{H}} } \circ T_{\gamma_{23}} \|_{L^p(\I_3) \to L^q(\I_1)}
\\
& \quad + \sup_{\bm{e} \in \mathcal{H}_{\leq 1}}
\| T_{\gamma_{12}} \circ T_{\langle \bm{e}, \bm{w}_2 \rangle_{\mathcal{H}} } \circ T_{\gamma_{23}} - T_{\gamma_{12}} \circ T_{\gamma_{23}} \circ T_{\langle \bm{e}, \bm{w}_3 \rangle_{\mathcal{H}} } \|_{L^p(\I_3) \to L^q(\I_1)}
\\
& \ \leq \sup_{\bm{e} \in \mathcal{H}_{\leq 1}}
\| T_{\langle \bm{e}, \bm{w}_1 \rangle_{\mathcal{H}} } \circ T_{\gamma_{12}} - T_{\gamma_{12}} \circ T_{\langle \bm{e}, \bm{w}_2 \rangle_{\mathcal{H}} } \|_{L^p(\I_2) \to L^q(\I_1)}
\\
& \quad + \sup_{\bm{e} \in \mathcal{H}_{\leq 1}}
\| T_{\langle \bm{e}, \bm{w}_2 \rangle_{\mathcal{H}} } \circ T_{\gamma_{23}} - T_{\gamma_{23}} \circ T_{\langle \bm{e}, \bm{w}_3 \rangle_{\mathcal{H}} } \|_{L^p(\I_3) \to L^q(\I_1)}
\\
& \ = \|\bm{w}_1 \gamma_{12} -\gamma_{12} \bm{w}_2\|_{p,q;\mathcal{H}} + \|\bm{w}_2 \gamma_{23} - \gamma_{23} \bm{w}_3\|_{p,q;\mathcal{H}}.
\end{aligned}
\end{equation*}
Here, in the second inequality, we use the observation that $T_{\gamma_{12}}$, $T_{\gamma_{23}}$ are doubly stochastic operators yet are contractions on $L^p$ spaces.
This argument can also be applied to the second term, which gives
\begin{equation*}
\begin{aligned}
& \|(\gamma_{12} \gamma_{23})^T \bm{w}_1 - \bm{w}_3 (\gamma_{12} \gamma_{23})^T \|_{p,q;\mathcal{H}}
\\
& \ \leq \|(\gamma_{12})^T \bm{w}_1 - \bm{w}_2 (\gamma_{12})^T \|_{p,q;\mathcal{H}} + \|(\gamma_{23})^T \bm{w}_2 - \bm{w}_3 (\gamma_{23})^T \|_{p,q;\mathcal{H}}.
\end{aligned}
\end{equation*}
Take the infinum over all $\gamma_{12}, \gamma_{23}$ we obtain the triangle inequality
\begin{equation*}
\begin{aligned}
\gamma_{p,q;\mathcal{H}}(\bm{w}_1,\bm{w}_3) \leq \gamma_{p,q;\mathcal{H}}(\bm{w}_1,\bm{w}_2) + \gamma_{p,q;\mathcal{H}}(\bm{w}_2,\bm{w}_3).
\end{aligned}
\end{equation*}

\end{proof}

\begin{proof} [Proof of Proposition~\ref{prop:unlabeled_and_coupling}]

The first part of the proof is to use the following argument repetitiously:
``If $f: \R_+ \to \R_+$ is continuous and monotonic and $x_\alpha \leq f(y_\alpha)$ for all $\alpha$, take the infimum over all $\alpha$ yields $\inf x \leq \inf f(y_\alpha)$, this leads to $\inf x \leq f(\inf y)$, allowing us to conclude $\inf x \leq \inf y$.''
The $\alpha$ in the above argument would be either the bijective measure preserving map $\Phi$ or the coupling $\gamma$.

For the last part of the proof, we use the following standard arguments from optimal transport.
Starting with a bijective measure preserving map $\Phi$, we can construct a specific coupling measure $\gamma$ concentrated on the graph of $\Phi$. 
This means $\gamma$ is defined via:
For measurable sets $A_1 \subset \I_1$ and $A_2 \subset \I_2$.
\begin{equation*}
\begin{aligned}
\gamma(A_1 \times A_2) = \big| \{ x \in A_1 \, |\, \Phi(x) \in A_2 \} \big|.
\end{aligned}
\end{equation*}
This is the pushforward $(\text{Id} \times \Phi)_{\#}$ of the probability measure on $\I_1$. It is straightforward to verify that $\gamma \in \Pi(\I_1, \I_2)$ and that
\begin{equation*}
\begin{aligned}
T_\gamma = \Phi^{-1}, \quad T_{\gamma^T} = \Phi.
\end{aligned}
\end{equation*}
Taking
\begin{equation*}
\begin{aligned}
\gamma_{p,q;\mathcal{H}}(\bm{w}_1,\bm{w}_2) & \leq \|\bm{w}_1 \gamma -\gamma \bm{w}_2\|_{p,q;\mathcal{H}} + \|\gamma^T \bm{w}_1 - \bm{w}_2 \gamma^T \|_{p,q;\mathcal{H}}.
\end{aligned}
\end{equation*}
It is straightforward to verify that
\begin{equation*}
\begin{aligned}
& \|\bm{w}_1 \gamma -\gamma \bm{w}_2\|_{p,q;\mathcal{H}}
\\
& \ = \sup_{\bm{e} \in \mathcal{H}_{\leq 1}}
\| T_{\langle \bm{e}, \bm{w}_1 \rangle_{\mathcal{H}} } \circ \Phi^{-1} - \Phi^{-1} \circ T_{\langle \bm{e}, \bm{w}_2 \rangle_{\mathcal{H}} } \|_{L^p(\I_2) \to L^q(\I_1)}
\\
& \ = \sup_{\bm{e} \in \mathcal{H}_{\leq 1}}
\| T_{\langle \bm{e}, \bm{w}_1 \rangle_{\mathcal{H}} } - \Phi^{-1} \circ T_{\langle \bm{e}, \bm{w}_2 \rangle_{\mathcal{H}} } \circ \Phi \|_{L^p(\I_1) \to L^q(\I_1)}
\\
& \ = \|\bm{w}_1 - \bm{w}_2^{\Phi}\|_{p,q;\mathcal{H}}.
\end{aligned}
\end{equation*}
Apply a similar argument to the second term we have
\begin{equation*}
\begin{aligned}
& \|\gamma^T \bm{w}_1 - \bm{w}_2 \gamma^T \|_{p,q;\mathcal{H}} = \|\bm{w}_1 - \bm{w}_2^{\Phi}\|_{p,q;\mathcal{H}}.
\end{aligned}
\end{equation*}
Take the summation and take the infimum over $\Phi$, we conclude that 
\begin{equation*}
\begin{aligned}
\gamma_{p,q;\mathcal{H}}(\bm{w}_1,\bm{w}_2) & \leq 2 \delta_{p,q;\mathcal{H}}(\bm{w}_1,\bm{w}_2).
\end{aligned}
\end{equation*}

\end{proof}

\subsection{Bi-coupling distances and functional-valued coupling distances} \label{subsec:bi_coupling_and_functional_proof}

\begin{proof}[Proof of Lemma~\ref{lem:negative_Sobolev_torus}]
Parts 1, 2, 3, and 4 of the lemma are straightforward exercises in Fourier analysis. As a hint for Part 4, note that on the real line $\R$, when $s=1$, the inverse Fourier transform of $1/(1 + 4\pi^2 \xi^2)$ is $\exp(-|x|)/2$. 

To compute the norm of the Dirac delta, we substitute it directly into the definition of the $H^{-1}(\T)$ norm, resulting in $\| \delta_0 \|_{H^{-1}(\T)} = \sqrt{\Lambda_{\max}}$. 
For the norm of other probability measures, observe that $\Lambda_{\max}$ is reached only at $x = 0$. Therefore, for any probability measure $f$ that is not concentrated as a Dirac delta, a part of the integral defining the $H^{-1}(\T)$ norm will be strictly less than $\Lambda_{\max}$. This results in the strict inequality $\| f \|_{H^{-1}(\T)} < \sqrt{\Lambda_{\max}}$.

Regarding the Lipschitz estimate, we first note that $|\Lambda'(x)| \leq \Lambda(x)$, which is more easily observed when considering $\Lambda(x)$ extended to the whole real line $\R$. For the lower bound, it is straightforward to verify that the function $\Lambda(x)$ is convex when viewed on the interval $[0,1]$ and reaches its minimum at $x = \tfrac{1}{2}$. The remainder of the estimate follows from the standard arguments in convex analysis and the properties of $\Lambda(x)$.
\end{proof}

\begin{proof}[Proof of Lemma~\ref{lem:Hilbert_measure_weak_star}]

To proof the compactness of the embedding $\mathcal{M}(\T) \to H^{-s}(\T)$ for $s > \frac{1}{2}$, we first observe that the embedding $\mathcal{M}(\T) \to H^{-\left(\frac{s}{2} + \frac{1}{4}\right)}(\T)$ is bounded. Additionally, the embedding $H^{-\left(\frac{s}{2} + \frac{1}{4}\right)}(\T) \to H^{-s}(\T)$ is compact according to the Rellich-Kondrachov compactness theorem. Consequently, the composition of these embeddings ensures that $\mathcal{M}(\T) \to H^{-s}(\T)$ is compact.

Next, we establish that the $H^{-s}(\T)$ norm induces the weak-* topology on $\mathcal{M}_{\leq 1}(\T)$. Consider a sequence $\{f_n\} \subset \mathcal{M}_{\leq 1}(\T)$ such that $\|f_n - f\|_{H^{-s}(\T)} \to 0$. For any test function $\varphi \in H^{s}(\T)$, the dual pairing satisfies
\begin{equation*}
\begin{aligned}
\int_{\T} \varphi (f_n - f) \, dx \to 0.
\end{aligned}
\end{equation*}
Since $H^{s}(\T)$ is dense in $C(\T)$, this convergence extends to all continuous test functions, implying that $f_n$ converges to $f$ in the weak-* topology of measures.

Conversely, suppose that $\{f_n\}$ is a sequence in $\mathcal{M}_{\leq 1}(\T)$ that converges to $f$ in the weak-* topology. Then, for each $x \in \T$, the convolution $\Lambda_s \ast (f_n - f)(x)$ satisfies
\begin{equation*}
\begin{aligned}
\Lambda_s \ast (f_n - f)(x) \to 0.
\end{aligned}
\end{equation*}
Since $\Lambda_s \in C^{\alpha}(\T)$ for any $0 < \alpha < s - \frac{1}{2}$ (which can be verified either directly or via Morrey's inequality), this convergence is uniform over $x \in \T$. Therefore, we have
\begin{equation*}
\begin{aligned}
\|\Lambda_s \ast (f_n - f)\|_{L^\infty(\T)} \to 0.
\end{aligned}
\end{equation*}
By the alternative formulation of the $H^{-s}(\T)$ norm, this implies that
\begin{equation*}
\begin{aligned}
\| f_n - f \|_{H^{-s}(\T)} = \left( \int_{\T \times \T} (f_n - f)(x) \Lambda_s(x - y) (f_n - f)(y) \, dx \, dy \right)^{\frac{1}{2}} \to 0.
\end{aligned}
\end{equation*}
Thus, convergence in the $H^{-s}(\T)$ norm coincides with weak-* convergence of measures. This completes the proof that the $H^{-s}(\T)$ norm induces the weak-* topology on $\mathcal{M}_{\leq 1}(\T)$.

\end{proof}

\begin{proof} [Proof of Lemma~\ref{lem:coupling_equivalence}]

To be concise, we prove only the equivalence of (1) and (3), as including (2) and (4) follows straightforwardly from Proposition~\ref{prop:unlabeled_and_coupling}.

In the proof, we establish uniform bounds across the two distances directly.
We begin by that the bi-coupling distance is controlled by the coupling distance.
Specifically, for $n = 1,2$, any kernels $w^{(n)} \in L^\infty(\I_n \times \I_n)$, and any states $X^{(n)}: \I_n \to \T$, we would like to show that for any coupling $\gamma \in \Pi(\I_1,\I_2)$,
\begin{equation*}
\begin{aligned}
& d_{\square,W_1} \big[ (w^{(1)}, X^{(1)}), (w^{(2)}, X^{(2)})\big]
\\
& \ \leq \int_{\I_1 \times \I_2} |X^{(1)}(\xi_1) - X^{(2)}(\xi_2)| \, \gamma(\rd \xi_1, \rd \xi_2)
\\
& \hspace{1cm} + \| w^{(1)} \gamma - \gamma w^{(2)} \|_{\square} + \|\gamma^T w^{(1)} - w^{(2)} \gamma^T \|_{\square}
\\
& \ \leq C^{\downarrow}_{w_{\max}} (\|\bm{w}_1 \gamma -\gamma \bm{w}_2\|_{\square;\mathcal{H}} + \|\gamma^T \bm{w}_1 - \bm{w}_2 \gamma^T \|_{\square;\mathcal{H}}),
\end{aligned}
\end{equation*}
where we take $\bm{w}_n = \bm{w}_{w^{(n_1)},X^{(n_1)}}$ to make the notation simpler. 
Once this inequality is completed, taking the infimum over all $\gamma \in \Pi(\I_1,\I_2)$ yields a uniform bound on one side.

For any $\gamma \in \Pi(\I_1,\I_2)$,
\begin{equation*}
\begin{aligned}
& (\bm{w}_1 \gamma -\gamma \bm{w}_2)(\xi_1,\xi_2) =
\\
&\begin{pmatrix}
\displaystyle \int_{\I_1} \delta_{X^{(1)}(\zeta)} \gamma(\zeta, \xi_2) \ \rd \zeta - \delta_{X^{(2)}(\xi_2)}
\\
\\
\displaystyle \int_{\I_1} w^{(1)}(\xi_1,\zeta) \delta_{X^{(1)}(\zeta)} \gamma(\zeta, \xi_2) \ \rd \zeta - \int_{\I_2} \gamma(\xi_1, \zeta) w^{(2)}(\zeta, \xi_2) \delta_{X^{(2)}(\xi_2)} \ \rd \zeta
\end{pmatrix}.
\end{aligned}
\end{equation*}
It is relatively easy to bound the cut distance terms:
Notice that $\mathbbm{1}(x) = 1, \forall x \in \T$ is a unit vector in $H^{-1}(\T)$. Hence by taking
\begin{equation*}
\begin{aligned}
\bm{e} =
\begin{pmatrix}
0
\\
\mathbbm{1}
\end{pmatrix},
\text{ which is a unit vector in } \mathcal{H} = H^{-1}(\T) \oplus H^{-1}(\T),
\end{aligned}
\end{equation*}
it is straightforward that
\begin{equation*}
\begin{aligned}
& \|w_1 \gamma -\gamma w_2\|_{\square} = \|\langle \bm{e}, \bm{w}_1 \gamma -\gamma \bm{w}_2 \rangle_{\mathcal{H}} \|_{\square} \leq \|\bm{w}_1 \gamma -\gamma \bm{w}_2\|_{\square;\mathcal{H}},
\\
& \|\gamma^T w_1 - w_2 \gamma^T \|_{\square} \leq \|\langle \bm{e}, \gamma^T \bm{w}_1 - \bm{w}_2 \gamma^T \rangle_{\mathcal{H}} \|_{\square} \leq \|\gamma^T \bm{w}_1 - \bm{w}_2 \gamma^T \|_{\square;\mathcal{H}}.
\end{aligned}
\end{equation*}

Bounding the Wasserstein-1 term is the most technical part of the proof.
Take
\begin{equation*}
\begin{aligned}
\epsilon = \int_{\I_1 \times \I_2} |X^{(1)}(\xi_1) - X^{(2)}(\xi_2)| \, \gamma(\rd \xi_1, \rd \xi_2)
\end{aligned}
\end{equation*}
Now, choose the subset $A_2 \subset \I_2$ as
\begin{equation*}
\begin{aligned}
A_2 = \bigg\{\xi_2 \in \I_2 : \int_{\I_1} |X^{(1)}(\xi_1) - X^{(2)}(\xi_2)| \ \gamma(\rd \xi_1, \xi_2) \geq \epsilon/2 \bigg\}.
\end{aligned}
\end{equation*}
By the definition of $A_2$,
\begin{equation*}
\begin{aligned}
& \int_{\I_1 \times A_2} |X^{(1)}(\xi_1) - X^{(2)}(\xi_2)| \, \gamma(\rd \xi_1, \rd \xi_2)
\\
& \ = \bigg( \int_{\I_1 \times \I_2} - \int_{\I_1 \times (\I_2 \setminus A_2)} \bigg) |X^{(1)}(\xi_1) - X^{(2)}(\xi_2)| \, \gamma(\rd \xi_1, \rd \xi_2)
\\
& \ \geq (\epsilon - \epsilon/ 2) = \epsilon / 2.
\end{aligned}
\end{equation*}
For any $\xi_2 \in A_2$,
\begin{equation*}
\begin{aligned}
& \bigg\langle - \delta_{X^{(2)}(\xi_2)}, \ \int_{\I_1} \delta_{X^{(1)}(\zeta)} \gamma(\zeta, \xi_2) \ \rd \zeta - \delta_{X^{(2)}(\xi_2)} \bigg\rangle_{H^{-1}(\T)}
\\
& \ = \Lambda\big(X^{(2)}(\xi_2) - X^{(2)}(\xi_2) \big) - \int_{\I_1} \Lambda\big(X^{(2)}(\xi_2) - X^{(1)}(\zeta)\big) \gamma(\zeta, \xi_2) \ \rd \zeta
\\
& \ \geq 2(\Lambda_{\max} - \Lambda_{\min}) \int_{\I_1} |X^{(2)}(\xi_2) - X^{(1)}(\zeta)| \gamma(\zeta, \xi_2) \ \rd \zeta
\end{aligned}
\end{equation*}
For any $\xi_2 \in A_2$, any $r > 0$ and any $X_0 \in \T$, such that $|X_0 - X^{(2)}(\xi_2)| \leq r$,
\begin{equation*}
\begin{aligned}
& \bigg\langle - \delta_{X_0}, \ \int_{\I_1} \delta_{X^{(1)}(\zeta)} \gamma(\zeta, \xi_2) \ \rd \zeta - \delta_{X^{(2)}(\xi_2)} \bigg\rangle_{H^{-1}(\T)}
\\
& \ \geq 2(\Lambda_{\max} - \Lambda_{\min}) \int_{\I_1} |X^{(2)}(\xi_2) - X^{(1)}(\zeta)| \gamma(\zeta, \xi_2) \ \rd \zeta 
\\
& \quad - \|\delta_{X_0} - \delta_{X^{(2)}(\xi_2)}\|_{H^{-1}(\T)} \bigg\| \int_{\I_1} \delta_{X^{(1)}(\zeta)} \gamma(\zeta, \xi_2) \ \rd \zeta - \delta_{X^{(2)}(\xi_2)} \bigg\|_{H^{-1}(\T)}
\\
& \ \geq 2(\Lambda_{\max} - \Lambda_{\min}) \int_{\I_1} |X^{(2)}(\xi_2) - X^{(1)}(\zeta)| \gamma(\zeta, \xi_2) \ \rd \zeta
\\
& \quad - \sqrt{2\Lambda_{\max} r} \sqrt{2 \Lambda_{\max} \int_{\I_1} |X^{(2)}(\xi_2) - X^{(1)}(\zeta)| \gamma(\zeta, \xi_2) \ \rd \zeta}
\\
& \ \geq \Big( 2(\Lambda_{\max} - \Lambda_{\min}) - 2\sqrt{2}\Lambda_{\max} \sqrt{ r} \epsilon^{-1/2} \Big) \int_{\I_1} |X^{(2)}(\xi_2) - X^{(1)}(\zeta)| \gamma(\zeta, \xi_2) \ \rd \zeta,
\end{aligned}
\end{equation*}
where the last inequality is by the definition of $A_2$.
By choosing $r < \epsilon (\Lambda_{\max} - \Lambda_{\min})^2/8\Lambda_{\max}^2$, we ensure that the coefficient in the last line is greater than $(\Lambda_{\max} - \Lambda_{\min})$.
Next, select $m > 4\Lambda_{\max}^2/ \epsilon (\Lambda_{\max} - \Lambda_{\min})^2$, which implies $1/2m < r$, and take $X_i = \frac{i-1}{m}$ for all $i = 0, \dots, (m-1)$. Let the subset
$A_{2,i} \subset A_2 \subset \I_2$ be defined as
\begin{equation*}
\begin{aligned}
A_{2,i} = \bigg\{\xi_2 \in A_2 : X^{(2)}(\xi_2) \in \Big[ \frac{i-1}{m} - \frac{1}{2m}, \frac{i-1}{m} + \frac{1}{2m} \Big) \bigg\}.
\end{aligned}
\end{equation*}
Using the inclusion-exclusion principle, there exists at least one $i = 0, \dots, (m-1)$ such that
\begin{equation*}
\begin{aligned}
& \int_{\I_1 \times A_{2,i}} |X^{(1)}(\xi_1) - X^{(2)}(\xi_2)| \, \gamma(\rd \xi_1, \rd \xi_2) \geq \epsilon / 2m.
\end{aligned}
\end{equation*}
Now, notice that $(-\delta_{X_i}/\sqrt{\Lambda_{\max}})$ is a unit vector in $H^{-1}(\T)$ and $\I_1, A_{2,i}$ are measurable set on $\I_1, \I_2$ respectively.
By the definition of $\|\cdot\|_{\square;\mathcal{H}}$ we have
\begin{equation*}
\begin{aligned}
& \|\bm{w}_1 \gamma -\gamma \bm{w}_2\|_{\square;\mathcal{H}}
\\
& \ \geq \bigg\langle
\begin{pmatrix}
-\delta_{X_i} /\sqrt{\Lambda_{\max}}
\\
0
\end{pmatrix}, \int_{\I_1 \times A_{2,i}} (\bm{w}_1 \gamma -\gamma \bm{w}_2)(\xi_1,\xi_2) \ \rd \xi_1 \rd \xi_2 \bigg\rangle
\\
& \ = \bigg\langle - \delta_{X_i} /\sqrt{\Lambda_{\max}}, \ \int_{\I_1 \times A_{2,i}} \delta_{X^{(1)}(\zeta)} \gamma(\zeta, \xi_2) \ \rd \zeta \rd \xi_2 - \int_{A_{2,i}} \delta_{X^{(2)}(\xi_2)} \ \rd \xi_2\bigg\rangle_{H^{-1}(\T)}
\\
& \ \geq (\Lambda_{\max} - \Lambda_{\min}) /\sqrt{\Lambda_{\max}} \int_{\I_1 \times A_{2,i}} |X^{(2)}(\xi_2) - X^{(1)}(\zeta)| \gamma(\zeta, \xi_2) \ \rd \zeta \rd \xi
\\
& \ \geq (\Lambda_{\max} - \Lambda_{\min}) \epsilon /2\sqrt{\Lambda_{\max}} \Big\lceil 4\Lambda_{\max}^2/ \epsilon (\Lambda_{\max} - \Lambda_{\min})^2 \Big\rceil.
\end{aligned}
\end{equation*}
Recalling that we took $\epsilon = \int_{\I_1 \times \I_2} |X^{(1)}(\xi_1) - X^{(2)}(\xi_2)| \, \gamma(\rd \xi_1, \rd \xi_2)$, this essentially provides a bound that
\begin{equation*}
\begin{aligned}
\int_{\I_1 \times \I_2} |X^{(1)}(\xi_1) - X^{(2)}(\xi_2)| \, \gamma(\rd \xi_1, \rd \xi_2) \lesssim \sqrt{\|\bm{w}_1 \gamma -\gamma \bm{w}_2\|_{\square;\mathcal{H}}}
\end{aligned}
\end{equation*}
This completes one direction of the proof.

Now we address the other direction. We aim to show that for any coupling $\gamma \in \Pi(\I_1,\I_2)$,
\begin{equation*}
\begin{aligned}
& \gamma_{\square;\mathcal{H}}(\bm{w}_1,\bm{w}_2)
\\
& \ \leq \|\bm{w}_1 \gamma -\gamma \bm{w}_2\|_{\square;\mathcal{H}} + \|\gamma^T \bm{w}_1 - \bm{w}_2 \gamma^T \|_{\square;\mathcal{H}}
\\
& \ \leq C^{\downarrow}_{w_{\max}} \bigg( \int_{\I_1 \times \I_2} |X^{(1)}(\xi_1) - X^{(2)}(\xi_2)| \, \gamma(\rd \xi_1, \rd \xi_2)
\\
& \hspace{2cm} + \| w^{(1)} \gamma - \gamma w^{(2)} \|_{\square} + \|\gamma^T w^{(1)} - w^{(2)} \gamma^T \|_{\square} \bigg).
\end{aligned}
\end{equation*}
Again, once this inequality is completed, taking the infimum over all $\gamma \in \Pi(\I_1,\I_2)$ yields a uniform bound to the other direction.

Recall that for any $\gamma \in \Pi(\I_1,\I_2)$,
\begin{equation*}
\begin{aligned}
& (\bm{w}_1 \gamma -\gamma \bm{w}_2)(\xi_1,\xi_2) =
\\
&
\begin{pmatrix}
\displaystyle \int_{\I_1} \delta_{X^{(1)}(\zeta)} \gamma(\zeta, \xi_2) \ \rd \zeta - \delta_{X^{(2)}(\xi_2)}
\\
\\
\displaystyle \int_{\I_1} w^{(1)}(\xi_1,\zeta) \delta_{X^{(1)}(\zeta)} \gamma(\zeta, \xi_2) \ \rd \zeta - \int_{\I_2} \gamma(\xi_1, \zeta) w^{(2)}(\zeta, \xi_2) \delta_{X^{(2)}(\xi_2)} \ \rd \zeta
\end{pmatrix}.
\end{aligned}
\end{equation*}
Therefore for any $S \subset \I_1, \ T \subset \I_2$ and
\begin{equation*}
\begin{aligned}
\bm{e} =
\begin{pmatrix}
\bm{e}_0
\\
\bm{e}_1
\end{pmatrix}
\in \mathcal{H} = H^{-1}(\T) \oplus H^{-1}(\T).
\end{aligned}
\end{equation*}
that is a unit vector,
\begin{equation*}
\begin{aligned}
& \bigg\langle \bm{e}, \int_{S \times T} (\bm{w}_1 \gamma -\gamma \bm{w}_2)(\xi,\zeta) \rd \xi \rd \zeta \bigg\rangle_{\mathcal{H}}
\\
& \ = \bigg\langle \bm{e}_0, \int_{S \times T} \bigg[
\int_{\I_1} \delta_{X^{(1)}(\zeta)} \gamma(\zeta, \xi_2) \ \rd \zeta - \delta_{X^{(2)}(\xi_2)}
\bigg] \rd \xi_1 \rd \xi_2
 \bigg\rangle
\\
& \quad + \bigg\langle \bm{e}_1,
\int_{S \times T} \bigg[
\int_{\I_1} w^{(1)}(\xi_1,\zeta) \delta_{X^{(1)}(\zeta)} \gamma(\zeta, \xi_2) \ \rd \zeta
\\
& \hspace{2.5cm} - \int_{\I_2} \gamma(\xi_1, \zeta) w^{(2)}(\zeta, \xi_2) \delta_{X^{(2)}(\xi_2)} \ \rd \zeta
\bigg] \rd \xi_1 \rd \xi_2 \bigg\rangle
\\
& \ \eqdef L_0 + L_1.
\end{aligned}
\end{equation*}
The first line $L_0$ is bounded by
\begin{equation*}
\begin{aligned}
L_0
& \leq \int_{\I_2} \bigg\| \displaystyle \int_{\I_1} \delta_{X^{(1)}(\zeta)} \gamma(\zeta, \xi_2) \ \rd \zeta - \delta_{X^{(2)}(\xi_2)} \bigg\|_{H^{-1}(\T)}
\ \rd \xi_2
\\
& \leq \int_{\I_2} \sqrt{2 \Lambda_{\max} \int_{\I_1} |X^{(2)}(\xi_2) - X^{(1)}(\zeta)| \gamma(\zeta, \xi_2) \ \rd \zeta}
\ \rd \xi_2
\\
& \leq \bigg( 2 \Lambda_{\max} \int_{\I_1 \times \I_2} |X^{(2)}(\xi_2) - X^{(1)}(\zeta)| \gamma(\zeta, \xi_2) \ \rd \zeta \rd \xi_2 \bigg)^{\frac{1}{2}}.
\end{aligned}
\end{equation*}
The second line $L_1$ can be further decomposed as $L_1 = L_{11} + L_{12}$ where
\begin{equation*}
\begin{aligned}
L_{11} &= \bigg\langle \bm{e}_1,
\int_{S \times T} \bigg[
\int_{\I_1} w^{(1)}(\xi_1,\zeta) \gamma(\zeta, \xi_2) \delta_{X^{(2)}(\xi_2)} \ \rd \zeta
\\
& \hspace{2.5cm} - \int_{\I_2} \gamma(\xi_1, \zeta) w^{(2)}(\zeta, \xi_2) \delta_{X^{(2)}(\xi_2)} \ \rd \zeta
\bigg] \rd \xi_1 \rd \xi_2 \bigg\rangle
\\
&= \int_{\I_1 \times \I_2} \mathbbm{1}_{S}(\xi_1) (w^{(1)} \gamma - \gamma w^{(2)})(\xi_1,\xi_2) \langle \bm{e}_1, \delta_{X^{(2)}(\xi_2)} \rangle \rd \xi_1 \rd \xi_2
\\
&\leq \sqrt{\Lambda_{\max}} \| w^{(1)} \gamma - \gamma w^{(2)} \|_{L^\infty \to L^1}
\\
&\leq 4 \sqrt{\Lambda_{\max}} \| w^{(1)} \gamma - \gamma w^{(2)} \|_{\square},
\end{aligned}
\end{equation*}
and $L_2$ is the commutator term 
\begin{equation*}
\begin{aligned}
L_{12} &= \bigg\langle \bm{e}_1,
\int_{S \times T} \bigg[
\int_{\I_1} w^{(1)}(\xi_1,\zeta) \big[\delta_{X^{(1)}(\zeta)} \gamma(\zeta, \xi_2) - \gamma(\zeta, \xi_2) \delta_{X^{(2)}(\xi_2)} \big] \ \rd \zeta 
\bigg] \rd \xi_1 \rd \xi_2 \bigg\rangle
\\
&\leq \|w^{(1)}\|_{L^\infty} \int_{\I_1 \times \I_2} \|\delta_{X^{(1)}(\zeta)} - \delta_{X^{(2)}(\xi_2)}\|_{H^{-1}(\T)} \gamma(\zeta, \xi_2) \rd \zeta \rd \xi_2
\\
&\leq \|w^{(1)}\|_{L^\infty} \int_{\I_1 \times \I_2} \sqrt{2 \Lambda_{\max} |X^{(1)}(\zeta) - X^{(2)}(\xi_2)|} \gamma(\zeta, \xi_2) \rd \zeta \rd \xi_2
\\
&\leq \|w^{(1)}\|_{L^\infty} \sqrt{2 \Lambda_{\max}} \bigg( \int_{\I_1 \times \I_2} |X^{(1)}(\zeta) - X^{(2)}(\xi_2)| \gamma(\zeta, \xi_2) \rd \zeta \rd \xi_2 \bigg)^{\frac{1}{2}}.
\end{aligned}
\end{equation*}
Combining the estimates for $L_0$, $L_{11}$, and $L_{12}$ provides the desired bound for $\|\bm{w}_1 \gamma -\gamma \bm{w}_2\|_{\square;\mathcal{H}}$. A similar estimate can be applied to $\|\gamma^T \bm{w}_1 - \bm{w}_2 \gamma^T \|_{\square;\mathcal{H}}$. This completes the proof of the second direction.

\end{proof}

\subsection{Functional-valued homomorphism density: Observables} \label{subsec:observables_proof}

\begin{proof} [Proof of Proposition~\ref{prop:negative_Sobolev_tensorized}]

The proof essentially involves verifying the definitions of the multi-dimensional Fourier transform and tensorization space and reproducing the argument used in the proof of Lemma~\ref{lem:Hilbert_measure_weak_star}.

To outline the key ideas, note that $(H^{s}(\T))^{\otimes k}$ includes $C^\infty(\T^k)$, which is dense in $C(\T^k)$. Thus, the dual pairing of $(H^{-s}(\T))^{\otimes k}$ and $(H^{s}(\T))^{\otimes k}$ results in the weak-* topology of the measures. Additionally, $\Lambda_s^{\otimes k}$ is Hölder continuous on $\T^k$, allowing us to reproduce the uniform convergence argument.

\end{proof}

\begin{proof}[Proof of Proposition~\ref{prop:diagonal_concentration}]

In Definition~\ref{defi:functionalize_graphon_restate} we can rewrite that
\begin{equation*}
\begin{aligned}
\bm{w}_{w,X}(\xi,\zeta) \defeq
\begin{pmatrix}
\big( w(\xi,\zeta) \big)^0 \delta_{X(\zeta)}
\\
\big( w(\xi,\zeta) \big)^1 \delta_{X(\zeta)}
\end{pmatrix}
\in \mathcal{M}(\T)^{\oplus \{0,1\}}, \quad \forall (\xi,\zeta) \in \T \times \T.
\end{aligned}
\end{equation*}
Hence we have
\begin{equation*}
\begin{aligned}
\big( \bm{w}_{w,X}(\xi_i,\xi_j) \big)_{s_{i,j}} = \big( w(\xi_i,\xi_j) \big)^{s_{i,j}} \delta_{X(\xi_j)}
\end{aligned}
\end{equation*}
and 
\begin{equation*}
\begin{aligned}
\big( \bm{t}(F,\bm{w}_{w,X}) \big)_s & = \bigg( \int_{\I^{\mathsf{v}(F)}} \bigotimes_{(i,j) \in \mathsf{e}(F)} \bm{w}_{w,X}(\xi_i,\xi_j) \prod_{i \in \mathsf{v}(F)} \rd \xi_i \bigg)_s
\\
& = \int_{\I^{\mathsf{v}(F)}} \bigotimes_{(i,j) \in \mathsf{e}(F)} \big( \bm{w}_{w,X}(\xi_i,\xi_j) \big)_{s_{i,j}} \prod_{i \in \mathsf{v}(F)} \rd \xi_i
\\
& = \int_{\I^{\mathsf{v}(F)}} \prod_{(i,j) \in \mathsf{e}(F)} \big( w(\xi_i,\xi_j) \big)^{s_{i,j}} \bigotimes_{(i,j) \in \mathsf{e}(F)} \delta_{X(\xi_j)} \prod_{i \in \mathsf{v}(F)} \rd \xi_i.
\end{aligned}
\end{equation*}
To see that it is supported in the subspace $\T^{\mathsf{e}(F)}|_{\textnormal{diag}}$, notice that
\begin{equation*}
\begin{aligned}
\big(X(\xi_j)\big)_{(i,j) \in \mathsf{e}(F)} \in \T^{\mathsf{e}(F)}|_{\textnormal{diag}}, \quad \forall \xi \in \I^{\mathsf{v}(F)},
\end{aligned}
\end{equation*}
hence the above integral should be supported in $\T^{\mathsf{e}(F)}|_{\textnormal{diag}}$ as well.

\end{proof}

\begin{proof} [Proof of Proposition~\ref{prop:normal_plain_equivalence}]

The bijectivity of $\mathsf{i}_{\mathsf{v}'(F) \to \mathsf{e}(F)}$ is straightforward. The equivalence of $\bm{\tau}(F,w,X)$ and $\bm{t}(F,\bm{w}_{w,X})$ up to the pushforward map can be verified more easily by checking each component individually. For each $s \in \{0,1\}^{\mathsf{e}(F)}$,
\begin{equation*}
\begin{aligned}
& \mathsf{i}_{\mathsf{v}'(F) \to \mathsf{e}(F)}^\# \Big(
\big( \bm{\tau}(F,w,X) \big)_s \Big)
\\
& \ = \mathsf{i}_{\mathsf{v}'(F) \to \mathsf{e}(F)}^\# \bigg( \int_{\I^{\mathsf{v}(F)}} \prod_{(i,j) \in \mathsf{e}(F)} \big( w(\xi_i,\xi_j) \big)^{s_{i,j}} \bigotimes_{j \in \mathsf{v}'(F)} \delta_{X(\xi_j)} \prod_{i \in \mathsf{v}(F)} \rd \xi_i \bigg)
\\
& \ = \int_{\I^{\mathsf{v}(F)}} \prod_{(i,j) \in \mathsf{e}(F)} \big( w(\xi_i,\xi_j) \big)^{s_{i,j}} \mathsf{i}_{\mathsf{v}'(F) \to \mathsf{e}(F)}^\# \bigg( \bigotimes_{j \in \mathsf{v}'(F)} \delta_{X(\xi_j)} \bigg) \prod_{i \in \mathsf{v}(F)} \rd \xi_i
\\
& \ = \int_{\I^{\mathsf{v}(F)}} \prod_{(i,j) \in \mathsf{e}(F)} \big( w(\xi_i,\xi_j) \big)^{s_{i,j}} \bigotimes_{j \in \mathsf{e}(F)} \delta_{X(\xi_j)} \prod_{i \in \mathsf{v}(F)} \rd \xi_i
\\
& \ = \big( \bm{t}(F,\bm{w}_{w,X}) \big)_s.
\end{aligned}
\end{equation*}
Finally, when $(w,X)$ is lifted from $(w,f)$, a componentwise calculation yields
\begin{equation*}
\begin{aligned}
& \big( \bm{\tau}(F,w,X) \big)_s
\\
& \ = \int_{(\I \ltimes_f \T)^{\mathsf{v}(F)}} \prod_{(i,j) \in \mathsf{e}(F)} \big( w(\xi_i,\xi_j) \big)^{s_{i,j}} \bigotimes_{j \in \mathsf{v}'(F)} \delta_{x_j} \prod_{i \in \mathsf{v}(F)} (\rd \xi_i \rd x_i)
\\
& \ = \int_{(\I \times \T)^{\mathsf{v}(F)}} \prod_{(i,j) \in \mathsf{e}(F)} \big( w(\xi_i,\xi_j) \big)^{s_{i,j}} \bigotimes_{j \in \mathsf{v}'(F)} \delta_{x_j} \prod_{i \in \mathsf{v}(F)} (f(x_i,\xi_i) \rd \xi_i \rd x_i)
\\
& \ = \int_{\I^{\mathsf{v}(F)}} \prod_{(i,j) \in \mathsf{e}(F)} \big( w(\xi_i,\xi_j) \big)^{s_{i,j}} \bigotimes_{j \in \mathsf{v}'(F)} \bigg( \int_{\T} \delta_{x_j} f(x_j,\xi_j) \rd x_j \bigg)
\\
& \hspace{4.5cm} \prod_{j \in \mathsf{v}(F) \setminus \mathsf{v}'(F)} \bigg( \int_{\T} f(x_j,\xi_j) \rd x_j \bigg) \prod_{i \in \mathsf{v}(F)} \rd \xi_i
\\
& \ = \int_{\I^{\mathsf{v}(F)}} \prod_{(i,j) \in \mathsf{e}(F)} \big( w(\xi_i,\xi_j) \big)^{s_{i,j}} \bigotimes_{j \in \mathsf{v}'(F)} f(\xi_j) \prod_{i \in \mathsf{v}(F)} \rd \xi_i
\\
& \ = \big( \bm{\tau}(F,w,f) \big)_s.
\end{aligned}
\end{equation*}

\end{proof}

\begin{proof} [Proof of Lemma~\ref{prop:normal_plain_equivalence}]
The equivalence of (1) and (2) follows directly from the fact that $\mathsf{i}_{\mathsf{v}'(F) \to \mathsf{e}(F)}: \T^{\mathsf{v}'(F)} \to \T^{\mathsf{e}(F)}|_{\textnormal{diag}}$ is an invertible linear map.

The equivalences of (1) and (3), as well as (2) and (4), are direct consequences of Part~2 of Proposition~\ref{prop:negative_Sobolev_tensorized}. To incorporate (5) and (6), consider (5) as an example. Note that we have an a priori bound:
\begin{equation*}
\begin{aligned}
& \|\bm{t}(F,w^{(n)},X^{(n)}) - \bm{t}(F)_\infty\|_{(H^{-1}(\T)^{\otimes \mathsf{e}(F)})^{\oplus (\{0,1\}^{\mathsf{e}(F)})}}
\\
& \ \leq (\sqrt{\Lambda_{\max}})^{|\mathsf{e}(F)|} \|\bm{t}(F,w^{(n)},X^{(n)}) - \bm{t}(F)_\infty\|_{(\mathcal{M}(\T)^{\otimes \mathsf{e}(F)})^{\oplus (\{0,1\}^{\mathsf{e}(F)})}}
\\
& \ \leq (2(1+w_{\max})\sqrt{\Lambda_{\max}})^{|\mathsf{e}(F)|}
\end{aligned}
\end{equation*}
where the first inequality follows from Part 5 of Lemma~\ref{lem:negative_Sobolev_torus}. This implies that
\begin{equation*}
\begin{aligned}
& (4(1+w_{\max})\sqrt{\Lambda_{\max}})^{-|\mathsf{e}(F)|} \|\bm{t}(F,w^{(n)},X^{(n)}) - \bm{t}(F)_\infty\|_{(H^{-1}(\T)^{\otimes \mathsf{e}(F)})^{\oplus (\{0,1\}^{\mathsf{e}(F)})}}
\\
& \ \leq 2^{-|\mathsf{e}(F)|}.
\end{aligned}
\end{equation*}
Thus, by standard arguments about the metrization of the product topology on countably many metric spaces, we conclude that the metric induces the same topology.

\end{proof}

\section*{Acknowledgments}

The author would like to thank Francis Bach, José A. Carrillo, Nicolas Fournier, Pierre-Emmanuel Jabin, Anping Pan, Benoît Perthame, Zhengfu Wang, and Zhennan Zhou for discussions which inspired the resolve to complete this work.
The author would like to thank Pierre-Emmanuel Jabin for the valuable comments on the presentation after the first draft of this article was written.

\bibliography{CandT}{} 

\begin{thebibliography}{10}

\bibitem{ayi2024mean}
{\sc N.~Ayi, N.~P. Duteil, and D.~Poyato}, {\em {Mean-field limit of
  non-exchangeable multi-agent systems over hypergraphs with unbounded rank}},
  arXiv preprint arXiv:2406.04691,  (2024).

\bibitem{backhausz2022action}
{\sc {\'A}.~Backhausz and B.~Szegedy}, {\em {Action convergence of operators
  and graphs}}, Canadian Journal of Mathematics, 74 (2022), pp.~72--121.

\bibitem{bayraktar2023graphon}
{\sc E.~Bayraktar, S.~Chakraborty, and R.~Wu}, {\em {Graphon mean field
  systems}}, The Annals of Applied Probability, 33 (2023), pp.~3587--3619.

\bibitem{bet2024weakly}
{\sc G.~Bet, F.~Coppini, and F.~R. Nardi}, {\em {Weakly interacting oscillators
  on dense random graphs}}, Journal of Applied Probability, 61 (2024),
  pp.~255--278.

\bibitem{boker2021graph}
{\sc J.~B{\"o}ker}, {\em {Graph Similarity and Homomorphism Densities}}, in
  48th International Colloquium on Automata, Languages, and Programming (ICALP
  2021), Schloss-Dagstuhl-Leibniz Zentrum f{\"u}r Informatik, 2021.

\bibitem{borgs2019Lp}
{\sc C.~Borgs, J.~Chayes, H.~Cohn, and Y.~Zhao}, {\em {An $L^p$ theory of
  sparse graph convergence I: Limits, sparse random graph models, and power law
  distributions}}, Transactions of the American Mathematical Society, 372
  (2019), pp.~3019--3062.

\bibitem{borgs2018Lp}
{\sc C.~Borgs, J.~T. Chayes, H.~Cohn, and Y.~Zhao}, {\em {An $L^p$ theory of
  sparse graph convergence II: LD convergence, quotients and right
  convergence}}, The Annals of Probability, 46 (2018), pp.~337--396.

\bibitem{borgs2008convergent}
{\sc C.~Borgs, J.~T. Chayes, L.~Lov{\'a}sz, V.~T. S{\'o}s, and
  K.~Vesztergombi}, {\em {Convergent sequences of dense graphs I: Subgraph
  frequencies, metric properties and testing}}, Advances in Mathematics, 219
  (2008), pp.~1801--1851.

\bibitem{braun1977vlasov}
{\sc W.~Braun and K.~Hepp}, {\em {The Vlasov dynamics and its fluctuations in
  the $1/N$ limit of interacting classical particles}}, Communications in
  mathematical physics, 56 (1977), pp.~101--113.

\bibitem{chaintron2022propagation2}
{\sc L.-P. Chaintron and A.~Diez}, {\em {Propagation of chaos: a review of
  models, methods and applications. I. Models and methods}}, Kinetic and
  Related Models, 15 (2022), pp.~895--1015.

\bibitem{chaintron2022propagation1}
\leavevmode\vrule height 2pt depth -1.6pt width 23pt, {\em {Propagation of
  chaos: a review of models, methods and applications. II. Applications}},
  Kinetic and Related Models, 15 (2022), pp.~1017--1173.

\bibitem{chiba2018mean}
{\sc H.~Chiba and G.~S. Medvedev}, {\em {The mean field analysis of the
  Kuramoto model on graphs I. The mean field equation and transition point
  formulas}}, Discrete and Continuous Dynamical Systems, 39 (2018),
  pp.~131--155.

\bibitem{cuturi2013sinkhorn}
{\sc M.~Cuturi}, {\em {Sinkhorn distances: Lightspeed computation of optimal
  transport}}, Advances in neural information processing systems, 26 (2013).

\bibitem{dell2018lovasz}
{\sc H.~Dell, M.~Grohe, and G.~Rattan}, {\em {Lov{\'a}sz Meets Weisfeiler and
  Leman}}, in 45th International Colloquium on Automata, Languages, and
  Programming (ICALP 2018), Schloss Dagstuhl-Leibniz-Zentrum fuer Informatik,
  2018.

\bibitem{dobrushin1979vlasov}
{\sc R.~L. Dobrushin}, {\em {Vlasov equations}}, Functional Analysis and Its
  Applications, 13 (1979), pp.~115--123.

\bibitem{elek2012measure}
{\sc G.~Elek and B.~Szegedy}, {\em {A measure-theoretic approach to the theory
  of dense hypergraphs}}, Advances in Mathematics, 231 (2012), pp.~1731--1772.

\bibitem{gkogkas2022graphop}
{\sc M.~A. Gkogkas and C.~Kuehn}, {\em {Graphop mean-field limits for
  Kuramoto-type models}}, SIAM Journal on Applied Dynamical Systems, 21 (2022),
  pp.~248--283.

\bibitem{grebik2022fractional}
{\sc J.~Greb{\'\i}k and I.~Rocha}, {\em {Fractional isomorphism of graphons}},
  Combinatorica, 42 (2022), pp.~365--404.

\bibitem{jabin2025mean}
{\sc P.-E. Jabin, D.~Poyato, and J.~Soler}, {\em {Mean-field limit of
  non-exchangeable systems}}, Communications on Pure and Applied Mathematics,
  78 (2025), pp.~651--741.

\bibitem{jabin2024non}
{\sc P.-E. Jabin, V.~Schmutz, and D.~Zhou}, {\em {Non-exchangeable networks of
  integrate-and-fire neurons: spatially-extended mean-field limit of the
  empirical measure}}, arXiv preprint arXiv:2409.06325,  (2024).

\bibitem{jabin2023mean}
{\sc P.-E. Jabin and D.~Zhou}, {\em {The mean-field limit of sparse networks of
  integrate and fire neurons}}, arXiv preprint arXiv:2309.04046,  (2023).

\bibitem{kaliuzhnyi2018mean}
{\sc D.~Kaliuzhnyi-Verbovetskyi and G.~S. Medvedev}, {\em {The mean field
  equation for the Kuramoto model on graph sequences with non-Lipschitz
  limit}}, SIAM Journal on Mathematical Analysis, 50 (2018), pp.~2441--2465.

\bibitem{kechris2012classical}
{\sc A.~Kechris}, {\em {Classical descriptive set theory}}, vol.~156, Springer
  Science \& Business Media, 2012.

\bibitem{kuehn2022vlasov}
{\sc C.~Kuehn and C.~Xu}, {\em {Vlasov equations on digraph measures}}, Journal
  of Differential Equations, 339 (2022), pp.~261--349.

\bibitem{lacker2023local}
{\sc D.~Lacker, K.~Ramanan, and R.~Wu}, {\em {Local weak convergence for sparse
  networks of interacting processes}}, The Annals of Applied Probability, 33
  (2023), pp.~843--888.

\bibitem{lacker2023marginal}
\leavevmode\vrule height 2pt depth -1.6pt width 23pt, {\em {Marginal dynamics
  of interacting diffusions on unimodular Galton--Watson trees}}, Probability
  Theory and Related Fields, 187 (2023), pp.~817--884.

\bibitem{lacker2024quantitative}
{\sc D.~Lacker, L.~C. Yeung, and F.~Zhou}, {\em Quantitative propagation of
  chaos for non-exchangeable diffusions via first-passage percolation}, arXiv
  preprint arXiv:2409.08882,  (2024).

\bibitem{lovasz2012large}
{\sc L.~Lov{\'a}sz}, {\em {Large networks and graph limits}}, vol.~60, American
  Mathematical Soc., 2012.

\bibitem{lovasz2006limits}
{\sc L.~Lov{\'a}sz and B.~Szegedy}, {\em {Limits of dense graph sequences}},
  Journal of Combinatorial Theory, Series B, 96 (2006), pp.~933--957.

\bibitem{nesterov1994interior}
{\sc Y.~Nesterov and A.~Nemirovskii}, {\em {Interior-point polynomial
  algorithms in convex programming}}, SIAM, 1994.

\bibitem{neunzert2006introduction}
{\sc H.~Neunzert}, {\em {An introduction to the nonlinear Boltzmann-Vlasov
  equation}}, in Kinetic Theories and the Boltzmann Equation: Lectures given at
  the 1st 1981 Session of the Centro Internazionale Matematico Estivo (CIME)
  Held at Montecatini, Italy, June 10--18, 1981, Springer, 2006, pp.~60--110.

\bibitem{oliveira2020interacting}
{\sc I.~Oliveira, G.~H. Reis, L.~M. Stolerman, et~al.}, {\em {Interacting
  diffusions on sparse graphs: hydrodynamics from local weak limits}},
  Electronic Journal of Probability, 25 (2020).

\bibitem{ramana1994fractional}
{\sc M.~V. Ramana, E.~R. Scheinerman, and D.~Ullman}, {\em {Fractional
  isomorphism of graphs}}, Discrete Mathematics, 132 (1994), pp.~247--265.

\bibitem{roddenberry2023limits}
{\sc T.~M. Roddenberry and S.~Segarra}, {\em {Limits of dense simplicial
  complexes}}, Journal of Machine Learning Research, 24 (2023), pp.~1--42.

\bibitem{szemeredi1975sets}
{\sc E.~Szemer{\'e}di}, {\em {On sets of integers containing no $k$ elements in
  arithmetic progression}}, Acta Arithmetica, 27 (1975), pp.~199--245.

\bibitem{szemeredi1978regular}
\leavevmode\vrule height 2pt depth -1.6pt width 23pt, {\em {Regular partitions
  of graphs}}, in Problèmes combinatoires et théorie des graphes, vol.~260 of
  Colloq. Internat. CNRS, 1978, pp.~399--401.

\bibitem{sznitman1991topics}
{\sc A.-S. Sznitman}, {\em {Topics in propagation of chaos}}, Ecole
  d’{\'e}t{\'e} de probabilit{\'e}s de Saint-Flour XIX—1989, 1464 (1991),
  pp.~165--251.

\bibitem{tinhofer1986graph}
{\sc G.~Tinhofer}, {\em {Graph isomorphism and theorems of birkhoff-type}},
  Computing, 36 (1986), pp.~285--300.

\bibitem{tinhofer1991note}
\leavevmode\vrule height 2pt depth -1.6pt width 23pt, {\em {A note on compact
  graphs}}, Discrete Applied Mathematics, 30 (1991), pp.~253--264.

\end{thebibliography}
\bibliographystyle{siam}

\appendix

\section{BBGKY hierarchy for non-exchangeable systems} \label{sec:kinetic}

The objective of this section is to prove Lemma~\ref{lem:main_observable}, focusing specifically on the stability part of the lemma. Given that all coefficients in our system are Lipschitz continuous, the well-posedness of both the ODE/SDE system and the Vlasov PDE can be established using the standard Picard iteration method. For readers who are not familiar with these techniques, we refer them to \cite{jabin2025mean, jabin2023mean}, which provide comprehensive treatments.

\subsection{Hierarchy of Linear PDEs}

We begin by establishing a system of linear PDEs, where each observable depends on a finite number of other observables within the family.
This is related to the algebraic/combinatorial operation of adding a leaf to an oriented tree.
\begin{defi} [\textbf{Adding a leaf}]
Without loss of generality, assume the vertex set of the oriented simple graph $F = (\mathsf{v}(F), \mathsf{e}(F))$ is labeled as
\begin{equation*}
\begin{aligned}
\mathsf{v}(F) = \{1, 2, \dots, |\mathsf{v}(F)|\}.
\end{aligned}
\end{equation*}

For any $j \in \mathsf{v}(F)$, we define a new graph $(F + j) \in \mathcal{G}$ by:
\begin{equation*}
\begin{aligned}
\mathsf{v}(F + j) &\defeq \mathsf{v}(F) \cup \{ |\mathsf{v}(F)| + 1 \} = \{1, 2, \dots, |\mathsf{v}(F)|, |\mathsf{v}(F)| + 1\}, \\
\mathsf{e}(F + j) &\defeq \mathsf{e}(F) \cup \{ (j, |\mathsf{v}(F)| + 1) \}.
\end{aligned}
\end{equation*}

Let $s \in \{0,1\}^{\mathsf{e}(F)}$ be an assignment of values to the edges of $F$. We define $(s + j) \in \{0,1\}^{\mathsf{e}(F + j)}$ as an assignment of values to the edges of $(F + j)$ given by:
\begin{equation*}
\begin{aligned}
&(s + j)_{j_1, j_2} = s_{j_1, j_2}, \quad \text{for all } (j_1, j_2) \in \mathsf{e}(F), \\
&(s + j)_{j, |\mathsf{v}(F)| + 1} = 1.
\end{aligned}
\end{equation*}
\end{defi}
\noindent
In other words, we add a new vertex labeled $|\mathsf{v}(F)| + 1$ to $F$ and connect it with a directed edge from vertex $j$ to the new vertex $|\mathsf{v}(F)| + 1$. This effectively adds a leaf to $F$ at the vertex $j$.
For $s$ an assignment of values to the edges, $s + j$ retains the same edge values as $s$ for all edges in $F$ and assigns the value $1$ to the new edge $(j, |\mathsf{v}(F)| + 1)$.

Let us first consider the case of the PDE. Recall the extended Vlasov PDE from \eqref{eqn:multi_agent_Vlasov_first}, which we restate here with $\D = \T$:
\begin{equation*}
\begin{aligned}
& \frac{\partial}{\partial t}f(t,x,\xi) + \partial_x \left( \left[ \mu(x) + \int_{\I} w(\xi,\zeta) \int_{\T} \sigma(x,y) f(t, y, \zeta) \, \rd y \rd \zeta \right] f(t, x, \xi) \right)
\\
& \hspace{8.5cm} - \frac{\nu^2}{2} \partial_{x}^2 f(t,x,\xi) = 0.
\end{aligned}
\end{equation*}
Also, recall the definition of the plain observables from Definition~\ref{defi:plain_observables}. We restate it here in a more concrete form. For each component $s \in \{0,1\}^{\mathsf{e}(F)}$, the definition is as follows for $z \in \T^{\mathsf{v}'(F)}$ in the distributional sense:
\begin{equation*}
\begin{aligned}
\big( \bm{\tau}(F,w,f) \big)_s(z) = \int_{\I^{\mathsf{v}(F)}} \prod_{(j_1,j_2) \in \mathsf{e}(F)} \big( w(\xi_{j_1},\xi_{j_2}) \big)^{s_{j_1,j_2}} \prod_{j \in \mathsf{v}'(F)} f(z_j, \xi_j) \prod_{j \in \mathsf{v}(F)} \rd \xi_j.
\end{aligned}
\end{equation*}

With the above definitions, the following computation is presented formally, but it is straightforward to justify rigorously in the distributional sense. The time derivative of $\big( \bm{\tau}(F,w,f) \big)_s(t,z)$, $z \in \T^{\mathsf{v}'(F)}$ involves taking the derivative of each factor $f(z_j, \xi_j)$ in the product. For any oriented simple graph $F \in \mathcal{G}$,
\begin{equation*}
\begin{aligned}
&\partial_t \big( \bm{\tau}(F,w,f) \big)_s(t,z)
\\
& \ = \sum_{j' \in \mathsf{v}'(F)} \int_{\I^{\mathsf{v}(F)}} \prod_{(j_1,j_2) \in \mathsf{e}(F)} \big( w(\xi_{j_1},\xi_{j_2}) \big)^{s_{j_1,j_2}} \prod_{j \in \mathsf{v}'(F)\setminus \{j'\}} f(z_j, \xi_j)
\\
& \quad \Bigg\{ - \partial_{z_{j'}} \bigg( \bigg[ \mu(z_{j'})
+ \int_{\I} w(\xi_{j'},\xi_{|\mathsf{v}(F)| + 1}) \int_{\D} \sigma(z_{j'}-z_{|\mathsf{v}(F)| + 1})
\\
& \hspace{3.5cm} f(t, z_{|\mathsf{v}(F)| + 1}, \xi_{|\mathsf{v}(F)| + 1}) \, \rd z_{|\mathsf{v}(F)| + 1} \rd \xi_{|\mathsf{v}(F)| + 1} \bigg] f(t, z_{j'}, \xi_{j'}) \bigg)
\\
& \hspace{7cm} + \frac{\nu^2}{2} \partial_{z_{j'}}^2 f(t, z_{j'}, \xi_{j'}) \Bigg\} \prod_{j \in \mathsf{v}(F)} \rd \xi_j
\\
& \ = \sum_{j' \in \mathsf{v}'(F)}\bigg[ - \partial_{z_{j'}} \Big(\mu(z_{j'})\big( \bm{\tau}(F,w,f) \big)_s(z) \Big) + \frac{\nu^2}{2} \partial_{z_{j'}}^2 \big( \bm{\tau}(F,w,f) \big)_s(z)
\\
& \hspace{1cm} - \partial_{z_{j'}} \bigg( \int_{\D} \sigma(z_{j'},z_{|\mathsf{v}(F)| + 1}) \big( \bm{\tau}(F+j',w,f) \big)_{s+j'}(z,z_{|\mathsf{v}(F)| + 1}) \ \rd z_{|\mathsf{v}(F)| + 1} \bigg) \bigg].
\end{aligned}
\end{equation*}
A critical observation is that for an oriented tree $T \in \mathcal{T}$, we have $(T + j) \in \mathcal{T}$. Thus, the hierarchy above is closed on the set of all oriented trees.

To extend this argument to the ODE/SDE system with finitely many agents, additional definitions are needed.
We can consider either the law of agents (always deterministic) or the empirical measure (a possibly random measure represented as a sum of Dirac measures). Furthermore, we can either allow repeated indices for agents or enforce unique indices. This orthogonal combination yields four possible definitions for our discussion, each with its own advantages.
\begin{defi} [\textbf{Several definition of observables}]\label{defi:empirical_observables}

Define $[N] = \{1,\dots,N\}$. Define $[N]^k = \{1,\dots,N\}^k$ product set, define $[N]^{\wedge k}$
\begin{equation*}
\begin{aligned}
[N]^{\wedge k} \defeq \{(l_1,\dots,l_k) \in [N]^k : l_1,\dots,l_k \text{ are all distinct } \}
\end{aligned}
\end{equation*}

When $\I$ in Definition~\ref{defi:plain_observables} is $[N]$ equipped with uniform atomic measure $1/N$, we define the following constructs for each component $s \in \{0,1\}^{\mathsf{e}(F)}$:
\begin{itemize}

\item
Re-define ``plain observables'' $\bm{\tau}$ (empirical, allowing repetition) as 
\begin{equation*}
\begin{aligned}
\big( \bm{\tau}(F,w,X) \big)_s = \frac{1}{N^{|\mathsf{v}(F)|}} \sum_{l \in [N]^{\mathsf{v}(F)}} \prod_{(j_1,j_2) \in \mathsf{e}(F)} \big( w_{l_{j_1},l_{j_2}} \big)^{s_{j_1,j_2}} \bigotimes_{j \in \mathsf{v}'(F)} \delta_{X_{l_j}}.
\end{aligned}
\end{equation*}

\item
Define ``probability distribution observables'' $\bm{\tau}_{p}$ (law, allowing repetition) as 
\begin{equation*}
\begin{aligned}
\big( \bm{\tau}_{p}(F,w,X) \big)_s &\defeq \E \bigg[ \big( \bm{\tau}(F,w,X) \big)_s \bigg]
\\
& = \frac{1}{N^{|\mathsf{v}(F)|}} \sum_{l \in [N]^{\mathsf{v}(F)}} \prod_{(j_1,j_2) \in \mathsf{e}(F)} \big( w_{l_{j_1},l_{j_2}} \big)^{s_{j_1,j_2}} \Law_{(X_{l_j})_{j \in \mathsf{v}'(F)}}.
\end{aligned}
\end{equation*}

\item
Define ``empirical observables'' $\bm{\tau}_e$ (empirical, no repetition) as 
\begin{equation*}
\begin{aligned}
\big( \bm{\tau}_{e}(F,w,X) \big)_s \defeq \frac{1}{N^{|\mathsf{v}(F)|}} \sum_{l \in [N]^{\wedge \mathsf{v}(F)}} \prod_{(j_1,j_2) \in \mathsf{e}(F)} \big( w_{l_{j_1},l_{j_2}} \big)^{s_{j_1,j_2}} \bigotimes_{j \in \mathsf{v}'(F)} \delta_{X_{l_j}}.
\end{aligned}
\end{equation*}

\item
Define ``marginal observables'' $\bm{\tau}_{m}$ (law, no repetition) as
\begin{equation*}
\begin{aligned}
\big( \bm{\tau}_{m}(F,w,X) \big)_s
& \defeq
\E \bigg[ \big( \bm{\tau}_{e}(F,w,X) \big)_s \bigg]
\\
& = \frac{1}{N^{|\mathsf{v}(F)|}} \sum_{l \in [N]^{\wedge \mathsf{v}(F)}} \prod_{(j_1,j_2) \in \mathsf{e}(F)} \big( w_{l_{j_1},l_{j_2}} \big)^{s_{j_1,j_2}} \Law_{(X_{l_j})_{j \in \mathsf{v}'(F)}}.
\end{aligned}
\end{equation*}
\end{itemize}

When $\I$ in Definition~\ref{defi:plain_observables} is an atomless standard probability space, and $(w,X)$ is a lift from some $(w,f)$, we formally define $\bm{\tau}(F,w,X) = \bm{\tau}_{p}(F,w,X) = \bm{\tau}_{e}(F,w,X) = \bm{\tau}_{m}(F,w,X)$.

\end{defi}

The following proposition states that excluding repetitive indices in the definition results in a negligible remainder as $N \to \infty$.
\begin{prop} \label{prop:empirical_observables_diff}
The following a priori bounds hold: For any $T \in \mathcal{G}$ and on each component $s \in \{0,1\}^{\mathsf{e}(F)}$,
\begin{equation*}
\begin{aligned}
\big\| \big( \bm{\tau}(F,w,X) \big)_s - \big( \bm{\tau}_{e}(F,w,X) \big)_s \big\|_{\mathcal{M}(\T^{\mathsf{v}'(F)})} & \leq (1 + w_{\max})^{|\mathsf{e}(F)|}
\\
& \hspace{1cm} \bigg(1 - \prod_{k=1}^{|\mathsf{v}(F)|} \frac{N - k + 1}{N} \bigg),
\\
\big\| \big( \bm{\tau}(F,w,X) \big)_s - \big( \bm{\tau}_{e}(F,w,X) \big)_s \big\|_{H^{-1}(\T)^{\otimes \mathsf{v}'(F)}} & \leq (1 + w_{\max})^{|\mathsf{e}(F)|} (\sqrt{\Lambda_{\max}})^{|\mathsf{v}'(F)|}
\\
& \hspace{1cm} \bigg(1 - \prod_{k=1}^{|\mathsf{v}(F)|} \frac{N - k + 1}{N} \bigg).
\end{aligned}
\end{equation*}
The same estimate holds if, in the inequalities above, we replace $\bm{\tau}$ and $\bm{\tau}_{e}$ with $\bm{\tau}_{p}$ and $\bm{\tau}_{m}$, respectively.

\end{prop}
\noindent
In particular, these uniform bounds vanish as $N \to \infty$. Since we are working with sequences satisfying $N_n \to \infty$, the convergence of $\bm{\tau}(F,w^{(n)},X^{(n)})$ and $\bm{\tau}_{e}(F,w^{(n)},X^{(n)})$ are equivalent, and similarly, the convergence of $\bm{\tau}_{p}(F,w^{(n)},X^{(n)})$ and $\bm{\tau}_{m}(F,w^{(n)},X^{(n)})$ are equivalent.

\begin{proof} [Proof of Proposition~\ref{prop:empirical_observables_diff}]

There is nothing to prove if $\I$ in Definition~\ref{defi:plain_observables} is an atomless standard probability space and $N = \infty$ is taken formally.
When $\I$ is $[N]$ equipped with uniform atomic measure $1/N$, for each component $s \in \{0,1\}^{\mathsf{e}(F)}$,
\begin{equation*}
\begin{aligned}
& \big( \bm{\tau}(F,w,X) \big)_s - \big( \bm{\tau}_{e}(F,w,X) \big)_s
\\
& \ = \frac{1}{N^{|\mathsf{v}(F)|}} \sum_{l \in [N]^{\mathsf{v}(F)} \setminus [N]^{\wedge \mathsf{v}(F)}} \prod_{(j_1,j_2) \in \mathsf{e}(F)} \big( w_{l_{j_1},l_{j_2}} \big)^{s_{j_1,j_2}} \bigotimes_{j \in \mathsf{v}'(F)} \delta_{X_{l_j}}.
\end{aligned}
\end{equation*}
Hence
\begin{equation*}
\begin{aligned}
& \big\| \big( \bm{\tau}(F,w,X) \big)_s - \big( \bm{\tau}_{e}(F,w,X) \big)_s \big\|_{\mathcal{M}(\T^{\mathsf{v}'(F)})}
\\
& \ \leq \frac{1}{N^{|\mathsf{v}(F)|}} \sum_{l \in [N]^{\mathsf{v}(F)} \setminus [N]^{\wedge \mathsf{v}(F)}} \prod_{(j_1,j_2) \in \mathsf{e}(F)} \big| w_{l_{j_1},l_{j_2}} \big|^{s_{j_1,j_2}}
\\
& \ \leq (1 + w_{\max})^{|\mathsf{e}(F)|} \frac{ \big| [N]^{\mathsf{v}(F)} \setminus [N]^{\wedge \mathsf{v}(F)} \big|}{N^{|\mathsf{v}(F)|}}
\\
& \ \leq (1 + w_{\max})^{|\mathsf{e}(F)|} \bigg(1 - \prod_{k=1}^{|\mathsf{v}(F)|} \frac{N - k + 1}{N} \bigg),
\end{aligned}
\end{equation*}
which is the first inequality.

The second inequality can be derived by combining the first inequality with the following estimate
\begin{equation*}
\begin{aligned}
& \big\| \big( \bm{\tau}(F,w,X) \big)_s - \big( \bm{\tau}_{e}(F,w,X) \big)_s \big\|_{H^{-1}(\T)^{\otimes \mathsf{v}'(F)}}
\\
& \ \leq (\sqrt{\Lambda_{\max}})^{|\mathsf{v}'(F)|} \big\| \big( \bm{\tau}(F,w,X) \big)_s - \big( \bm{\tau}_{e}(F,w,X) \big)_s \big\|_{\mathcal{M}(\T^{\mathsf{v}'(F)})},
\end{aligned}
\end{equation*}
which is obtained from Part~1 of Proposition~\ref{prop:negative_Sobolev_tensorized}.

\end{proof}

For simplicity in later calculations, we introduce the following notation for the law of agents:
For any multi-index $l \in [N]^k$, let the joint law of the agents be denoted by
\begin{equation*}
\begin{aligned}
f_{l}(z_1,\dots,z_k) \defeq \Law_{(X_{l_j})_{j =1}^k}(z_1,\dots,z_k) = \E \bigg[ \bigotimes_{j = 1}^k \delta_{X_{l_j}}(z_j) \bigg].
\end{aligned}
\end{equation*}
In particular, for the joint law of all agents, we use the simpler notation
\begin{equation*}
\begin{aligned}
f_{[N]}(x_1,\dots,x_N) \defeq \Law_{(X_i)_{i \in [N]}}(x_1,\dots,x_N) = \E \bigg[ \bigotimes_{i \in [N]} \delta_{X_{i}}(x_i) \bigg].
\end{aligned}
\end{equation*}
The Liouville equation governing the joint law of all agents is given by
\begin{equation} \label{eqn:Liouville}
\begin{aligned}
& \partial_t f_{[N]}(t,x) = \sum_{i \in [N]} \bigg[ - \partial_{x_i} \Big( \mu(x_i) f_{[N]}(t,x) \Big) + \frac{\nu^2}{2} \partial_{x_i}^2 f_{[N]}(t,x) 
\\
& \hspace{3.5cm} - \partial_{x_i} \Big( \frac{1}{N} \sum_{i'=1}^N w_{i,i'} \sigma(x_i,x_{i'}) f_{[N]}(t,x) \Big) \bigg]
\end{aligned}
\end{equation}
For any multi-index $l \in [N]^{\wedge k}$, we integrate out all dimensions except $l_1, \dots, l_k$ in the Liouville equation~\eqref{eqn:Liouville}. This yields
\begin{equation*}
\begin{aligned}
& \partial_t f_{l}(t,z)
\\
&\ = \sum_{j = 1}^k \bigg[ - \partial_{z_j} \Big( \mu(z_j) f_{l}(t,z) \Big) + \frac{\nu^2}{2} \partial_{z_j}^2 f_{l}(t,z)
\\
& \hspace{1cm} - \partial_{z_j} \bigg( \frac{1}{N} \sum_{l_{k+1} \in [N] \setminus \{l_1,\dots,l_k\}} w_{l_j,l_{k+1}} \int_{\T} \sigma(z_j,z_{k+1}) f_{l,l_{k+1}}(t,z,z_{k+1}) \rd z_{k+1} \bigg)
\\
& \hspace{1cm} - \partial_{z_j} \Big( \frac{1}{N} \sum_{j' = 1}^k w_{l_j,l_{j'}} \sigma(z_j,z_{j'}) f_{l}(t,z) \Big) \bigg]
\end{aligned}
\end{equation*}
where the last two terms correspond to binary interactions, one involving agents other than $l_1, \dots, l_k$ and the other involving one of $l_1, \dots, l_k$.

Next, we take the weighted summation in the definition of $\bm{\tau}_{m}$ (see Definition~\ref{defi:empirical_observables}), using multi-indices $l_{\mathsf{v}'} = (l_j)_{j \in \mathsf{v}'(F)}$ and $z = (z_j)_{j \in \mathsf{v}'(F)}$. This gives
\begin{equation*}
\begin{aligned}
& \partial_t \big( \bm{\tau}_{m}(F,w,X) \big)_s(t,z)
\\
&\ = \partial_t \bigg( \frac{1}{N^{|\mathsf{v}(F)|}} \sum_{l \in [N]^{\wedge \mathsf{v}(F)}} \prod_{(j_1,j_2) \in \mathsf{e}(F)} \big( w_{l_{j_1},l_{j_2}} \big)^{s_{j_1,j_2}} f_{l_{\mathsf{v}'}}(t,z) \bigg)
\\
&\ = \sum_{j \in \mathsf{v}'(F)} \bigg[ - \partial_{z_j} \Big( \mu(z_j) \big( \bm{\tau}_{m}(F,w,X) \big)_s(t,z) \Big) + \frac{\nu^2}{2} \partial_{z_j}^2 \big( \bm{\tau}_{m}(F,w,X) \big)_s(t,z)
\\
& \hspace{1.5cm} - \partial_{z_j} \bigg(\int_{\T} \sigma(z_j,z_{|\mathsf{v}(F)|+1}) \big( \bm{\tau}_{m}(F+j,w,X) \big)_{s+j}(t,z,z_{|\mathsf{v}(F)|+1}) \rd z_{|\mathsf{v}(F)|+1} \bigg)
\\
& \hspace{1.5cm} - \partial_{z_j} \mathscr{R}_{F,s,j}(t,z) \bigg].
\end{aligned}
\end{equation*}
where the remainder term is given by
\begin{equation*}
\begin{aligned}
& \mathscr{R}_{F,s,j}(t,z) \defeq
\\
& \ \frac{1}{N^{|\mathsf{v}(F)|}} \sum_{l \in [N]^{\wedge \mathsf{v}(F)}} \prod_{(j_1,j_2) \in \mathsf{e}(F)} \big( w_{l_{j_1},l_{j_2}} \big)^{s_{j_1,j_2}} \bigg( \frac{1}{N} \sum_{j' \in \mathsf{v}'(F)} w_{l_j,l_{j'}} \sigma(z_j,z_{j'}) f_{l_{\mathsf{v}'}}(t,z) \bigg).
\end{aligned}
\end{equation*}
Note that the components of $z$ are indexed by $\mathsf{v}'(F) \subset \mathsf{v}(F)$, but we still index the new variable as $z_{|\mathsf{v}(F)| + 1}$ for consistency in indexing calculations.

Following the proof of Proposition~\ref{prop:empirical_observables_diff}, it is straightforward to verify the a priori bound
\begin{equation*}
\begin{aligned}
\| \mathscr{R}_{F,s,j}(t,\cdot) \|_{\mathcal{M}(\T^{\mathsf{v}'(F)})} &\leq (1+w_{\max})^{|\mathsf{e}(F)|+1} \|\sigma\|_{L^\infty} |\mathsf{v}'(F)|/N,
\\
\| \mathscr{R}_{F,s,j}(t,\cdot) \|_{H^{-1}(\T)^{\otimes \mathsf{v}'(F)}} &\leq (1 + w_{\max})^{|\mathsf{e}(F)|+1} (\sqrt{\Lambda_{\max}})^{|\mathsf{v}'(F)|} \|\sigma\|_{L^\infty} |\mathsf{v}'(F)| / N.
\end{aligned}
\end{equation*}
Again, note that these uniform bounds vanish as $N \to \infty$.
For the Vlasov PDE, we formally take $N = \infty$, in which case the remainder term becomes zero.

We summarize the above arguments in the following proposition.
\begin{prop} \label{prop:BBGKY}
Let $\mu \in W^{1,\infty}(\T)$, $\sigma \in W^{2,\infty}(\T \times \T)$, $\nu \geq 0$.
Let $w \in L^\infty(\I \times \I)$ and $X_0: \I \to \T$ be either of the following:
\begin{itemize}
\item For some $N \geq 1$, $\I = \{1,\dots,N\}$ equipped with uniform atomic measure $1/N$ on each point, and $w \in \R^{N \times N}$ and $X_0 \in \T^{N_n}$, be connection weight matrices and initial data for the ODE/SDE system \eqref{eqn:multi_agent_SDE}.

\item For formally $N = \infty$, $\I'$ is an atomless probability space and $w \in L^\infty(\I' \times \I')$, $f_0 \in L^\infty(\I'; \mathcal{M}(\T))$ that for a.e. $\xi \in \I'$, 
$f_0(\cdot,\xi) \in \mathcal{P}(\T)$, be the weight kernel and initial data for the extended Vlasov PDE~\eqref{eqn:multi_agent_Vlasov_first}.

Let $\I = \I' \ltimes_{f_0} \T$, $w \in L^\infty(\I \times \I)$ and $X_0: \I \to \T$ be the lift.
\end{itemize}

Assume that $\|w\|_{L^\infty} \leq w_{\max} < \infty$. Then both the ODE/SDE system \eqref{eqn:multi_agent_SDE} and the extended Vlasov PDE~\eqref{eqn:multi_agent_Vlasov_first} have unique solutions. Moreover, the resulting pair $(w,X)$ has observables satisfying the following hierarchy of PDEs: For any $F \in \mathcal{G}$ and any $s \in \{0,1\}^{\mathsf{e}(F)}$,
\begin{equation*}
\begin{aligned}
& \partial_t \big( \bm{\tau}_{m}(F,w,X) \big)_s(t,z)
\\
&\ = \sum_{j \in \mathsf{v}'(F)} \bigg[ - \partial_{z_j} \Big( \mu(z_j) \big( \bm{\tau}_{m}(F,w,X) \big)_s(t,z) \Big) + \frac{\nu^2}{2} \partial_{z_j}^2 \big( \bm{\tau}_{m}(F,w,X) \big)_s(t,z)
\\
& \hspace{1.5cm} - \partial_{z_j} \bigg(\int_{\T} \sigma(z_j,z_{|\mathsf{v}(F)|+1}) \big( \bm{\tau}_{m}(F+j,w,X) \big)_{s+j}(t,z,z_{|\mathsf{v}(F)|+1}) \rd z_{|\mathsf{v}(F)|+1} \bigg)
\\
& \hspace{1.5cm} - \partial_{z_j} \mathscr{R}_{F,s,j}(t,z) \bigg]
\end{aligned}
\end{equation*}
with the remainder term bounded a priori by
\begin{equation*}
\begin{aligned}
\| \mathscr{R}_{F,s,j}(t,\cdot) \|_{H^{-1}(\T)^{\otimes \mathsf{v}'(F)}} &\leq (1 + w_{\max})^{|\mathsf{e}(F)|+1} (\sqrt{\Lambda_{\max}})^{|\mathsf{v}'(F)|} \|\sigma\|_{L^\infty} |\mathsf{v}'(F)| / N.
\end{aligned}
\end{equation*}

\end{prop}

\begin{proof} [Proof of Proposition~\ref{prop:BBGKY}]

Since all coefficients are at least Lipschitz continuous and bounded, the formal computations illustrated previously can be made rigorous in the distributional sense in a straightforward manner.

\end{proof}

\subsection{Energy estimate and the proof of Lemma~\ref{lem:main_observable}}

We now demonstrate how to perform an energy estimate on the BBGKY-like hierarchy, as summarized in Proposition~\ref{prop:BBGKY}. For preparation, we review the following lemmas, which are essentially drawn from Section~3 of \cite{jabin2023mean}. These provide tools for commutator estimates in negative Sobolev spaces.
\begin{lem} \label{lem:commutator_inequality}
Consider $\mu_m$ of form
\begin{equation*}
\begin{aligned}
\mu_m = 1 \otimes \dots \otimes \mu \otimes \dots \otimes 1,
\end{aligned}
\end{equation*}
where $\mu \in W^{1,\infty}(\R)$ appears in the $m$-th coordinate, i.e. $\mu_m(z) = \mu(z_m)$. Then for any $f \in H^{-1}(\T)^{\otimes k}$, the following inequality holds
\begin{equation*}
\begin{aligned}
\|\mu_m f\|_{H^{-1}(\T)^{\otimes k}} \leq 2 \|\mu\|_{W^{1,\infty}(\T)} \|f\|_{H^{-1}(\T)^{\otimes k}},
\end{aligned}
\end{equation*}
\end{lem}

\begin{proof}[Proof of Lemma~\ref{lem:commutator_inequality}]

We begin with the one-dimensional case. Start with the duality formula
\begin{equation*}
\begin{aligned}
\|f\|_{H^{-s}(\T)} = \sup_{\|g\|_{H^s(\T)} \leq 1} \bigg| \int_{\R} f(x)g(x) \;\rd x \bigg|,
\end{aligned}
\end{equation*}
and apply the inequality derived from Leibniz's rule for $s = 1$, we have
\begin{equation*}
\begin{aligned}
\|\mu f\|_{H^{-1}(\T)} & = \sup_{\|g\|_{H^1} \leq 1} \bigg| \int_{\R} g(x)\mu(x)f(x) \;\rd x \bigg|
\\
& \leq \sup_{\|g\|_{H^1} \leq 1} \|g\mu\|_{H^1} \|f\|_{H^{-1}} \leq 2 \|\mu\|_{W^{1,\infty}} \|f\|_{H^{-1}(\T)}.
\end{aligned}
\end{equation*}

Now, consider the higher-dimensional case. WLOG assume that $\mu_k$ is non-constant in the $k$-th dimension. We introduce the Fourier transform on the first $k-1$ dimensions,
\begin{equation*}
\begin{aligned}
\mathcal{F}^{\otimes k-1} \otimes I: \T^{k-1} \times \R \to \Z^{k-1} \times \R.
\end{aligned}
\end{equation*}
For clarity, we temporarily denote $m = (m_1, \dots, m_{k-1}) \in \mathbb{Z}^{k-1}$. This gives
\begin{equation*}
\begin{aligned}
\;& \big(\mathcal{F}^{\otimes k-1} \otimes I\big) \big( \Lambda_{1/2}^{\otimes k} \star (\mu_{k} f) \big) (m_1,\dots,m_{k-1}, z_{k})
\\
= \;& \bigg( \prod_{i=1}^{k-1} \frac{1}{\sqrt{1 + 4 \pi^2 m_i^2}} \bigg) \; \big( \Lambda_{1/2} \star_{k} (\mu_{k} \mathcal{F}^{\otimes k-1} f) \big) (m_1,\dots,m_{k-1}, z_{k}).
\end{aligned}
\end{equation*}
Recall that the definition of $\Lambda_{1/2}$ is given in Definition~\ref{defi:Lambda_s}.
By Plancherel’s identity,
\begin{equation*}
\begin{aligned}
& \|\mu_{k} f\|_{H^{-1}(\T)^{\otimes k}}^2
\\
& \ = \sum_{(m_1, \dots, m_{k-1}) \in \Z^{k-1}} \int
\\
& \hspace{1.5cm} \bigg| \bigg( \prod_{i=1}^{k-1} \frac{1}{\sqrt{1 + 4 \pi^2 m_i^2}} \bigg) \; \big( \Lambda_{1/2} \star_{k} (\mu_{k} \mathcal{F}^{\otimes k-1} f) \big) (m_1,\dots,m_{k-1}, z_{k}) \bigg|^2 \; \rd z_{k}
\\
& \ = \sum_{(m_1, \dots, m_{k-1}) \in \Z^{k-1}} \hspace{-0.2cm} \bigg( \prod_{i=1}^{k-1} \frac{1}{\sqrt{1 + 4 \pi^2 m_i^2}} \bigg) \Big\| \big( \mu_{k} \mathcal{F}^{\otimes k-1} f \big) (m_1,\dots,m_{k-1}, \cdot) \Big\|_{H^{-1}(\T)}^2.
\end{aligned}
\end{equation*}
Since $\mu \in W^{1,\infty}(\mathbb{T})$, apply the result in dimension $1$ gives
\begin{equation*}
\begin{aligned}
& \Big\| \big( \mu_{k} \mathcal{F}^{\otimes k-1} f \big) (m_1,\dots,m_{k-1}, \cdot) \Big\|_{H^{-1}(\T)}
\\
& \ \leq 2\|\mu\|_{W^{1,\infty}(\R)} \Big\| \mathcal{F}^{\otimes k-1} f (m_1,\dots,m_{k-1}, \cdot) \Big\|_{H^{-1}(\T)}.
\end{aligned}
\end{equation*}
Thus,
\begin{equation*}
\begin{aligned}
& \|\mu_{k} f\|_{H^{-1}(\T)^{\otimes k}}^2 
\\
& \ \leq 4\|\mu\|_{W^{1,\infty}(\T)}^2 \sum_{(m_1, \dots, m_{k-1}) \in \Z^{k-1}} \hspace{-0.2cm} \bigg( \prod_{i=1}^{k-1} \frac{1}{\sqrt{1 + 4 \pi^2 m_i^2}} \bigg) \Big\| \mathcal{F}^{\otimes k-1} f (m_1,\dots,m_{k-1}, \cdot) \Big\|_{H^{-1}(\T)}^2
\\
& \ = 4\|\mu\|_{W^{1,\infty}(\T)}^2 \|f\|_{H^{-1}(\T)^{\otimes k}}^2,
\end{aligned}
\end{equation*}
which completes the proof by taking the square root of both sides.
\end{proof}

The next lemma addresses the treatment of the binary interaction term.
\begin{lem} \label{lem:firing_rate_difference_tensorized}
For $k \geq 1$, let $f \in H^{-1}(\mathbb{T})^{\otimes (k+1)}$ and $\sigma \in W^{2,\infty}(\mathbb{T}^2)$. Define $z = (z_1, \dots, z_k)$. Then the following inequality holds:
\begin{equation*}
\begin{aligned}
& \int_{\T^{k}} \bigg[ \bigg( \Lambda_{1/2}^{\otimes k} \star \int_{\T} \sigma(\cdot_{k},z_{k+1}) f(\cdot, z_{k+1}) \; \rd z_{k+1} \bigg)(z) \bigg]^2 \prod_{j = 1}^{k} \;\rd z_j
\\
& \ = \int_{\T^{k}} \bigg[ \int_{\T^k} \prod_{j=1}^k \Lambda_{1/2}(z_j-u_j) 
\\
& \hspace{2cm} \int_{\T} \sigma(u_{k},z_{k+1}) f(u_1,\dots,u_k, z_{k+1}) \; \rd z_{k+1} \ \rd u_1,\dots,\rd u_k \bigg]^2 \prod_{j = 1}^{k} \;\rd z_j
\\
& \ \leq 16 \|\sigma\|_{W^{2,\infty}(\T^2)} \|f\|_{H^{-1}(\T)^{\otimes k+1}}^2.
\end{aligned}
\end{equation*}
The same estimate holds in an analogous manner when the two components in $\sigma$ are general indices rather than specifically $k$ and $k+1$.
\end{lem}

\begin{proof}[Proof of Lemma~\ref{lem:firing_rate_difference_tensorized}]

We begin with the lowest possible dimension. When $\sigma \in W^{2,\infty}(\mathbb{T}^2)$ and $f \in H^{-1}(\mathbb{T})^{\otimes 2}$,
\begin{equation*}
\begin{aligned}
& \int_{\T}\bigg[ \bigg( \Lambda_{1/2} \star \int_{\T} \sigma(\cdot,z_2) f(\cdot,z_2) \;\rd z_2\bigg)(z_1) \bigg]^2 \rd z_1
\\
& \ = \int_{\T}\bigg[ \int_{\T} \Lambda_{1/2}^{\otimes 2} \star (\sigma f)(z_1,z_2) \;\rd z_2 \bigg]^2 \rd z_1
\\
& \ \leq \int_{\T^2} \big[ \Lambda_{1/2}^{\otimes 2} \star (\sigma f)(z_1,z_2)\big]^2 \;\rd z_2 \rd z_1
\\
& \ \leq 16 \|\sigma\|_{W^{2,\infty}(\T^2)} \|f\|_{H^{-1}(\T)^{\otimes 2}}.
\end{aligned}
\end{equation*}

In the general case, by the result in dimension $2$,
\begin{equation*}
\begin{aligned}
& \int_{\T^{k}} \bigg[ \bigg( \Lambda_{1/2}^{\otimes k} \star \int_{\T} \sigma(\cdot_{k},z_{k+1}) f(\cdot, z_{k+1}) \; \rd z_{k+1} \bigg)(z) \bigg]^2 \prod_{j = 1}^{k} \;\rd z_j
\\
& \ = \int_{\T^{k}} \bigg[ \int_{\T} \Lambda_{1/2}^{\otimes k+1} \star (\sigma f)(z,z_{k+1}) \; \rd z_{k+1} \bigg]^2 \prod_{j = 1}^{k} \;\rd z_j
\\
& \ = \int_{\T^{k+1}} \big[ \Lambda_{1/2}^{\otimes k+1} \star (\sigma f)(z,z_{k+1}) \big]^2 \prod_{j = 1}^{k+1} \;\rd z_j
\\
& \ \leq 16 \|\sigma\|_{W^{2,\infty}(\T^2)} \|f\|_{H^{-1}(\T)^{\otimes k+1}}^2.
\end{aligned}
\end{equation*}

\end{proof}

With the two lemmas, we now proceed to prove Lemma~\ref{lem:main_observable} with the energy estimates.

\begin{proof} [Proof of Lemma~\ref{lem:main_observable}]

Let us begin with the case $\nu > 0$ and denote
\begin{equation} \label{eqn:defi_observable_difference}
\begin{aligned}
\Delta^{(n)}_{F,s} \defeq \big(\bm{\tau}_m(F,w^{(n)},X^{(n)}_0)\big)_s - \big(\bm{\tau}_m(F,w^{(\infty)},X^{(\infty)}_0)\big)_s,
\end{aligned}
\end{equation}
for any $F \in \mathcal{G}$ and any $s \in \{0,1\}^{\mathsf{e}(F)}$.
From Proposition~\ref{prop:BBGKY}, we obtain
\begin{equation*}
\begin{aligned}
& \partial_t \Delta^{(n)}_{F,s}(t,z)
\\
& \ = \sum_{j \in \mathsf{v}(F)} \bigg[ - \partial_{z_j} \Big( \mu(z_j) \Delta^{(n)}_{F,s}(t,z) \Big) + \frac{\nu^2}{2} \partial_{z_j}^2 \Delta^{(n)}_{F,s}(t,z)
\\
& \hspace{1.5cm} - \partial_{z_j} \bigg(\int_{\T} \sigma(z_j,z_{|\mathsf{v}(F)|+1}) \Delta^{(n)}_{F+j,s+j}(t,z,z_{|\mathsf{v}(F)|+1}) \rd z_{|\mathsf{v}(F)|+1} \bigg)
\\
& \hspace{8cm} - \partial_{z_j} \mathscr{R}_{F,s,j}^{(n)}(t,z) \bigg],
\end{aligned}
\end{equation*}
where the a priori $H^{-1}(\T)^{\otimes \mathsf{v}'(F)}$ bound of the remainder term $\mathscr{R}_{F,s,j}^{(n)}$ vanishes as $n \to \infty$.

The energy estimate in $H^{-1}(\mathbb{T})^{\otimes \mathsf{v}'(F)}$ from the above PDE reads
\begin{equation*}
\begin{aligned}
& \frac{\rd}{\rd t} \big\| \Delta^{(n)}_{F,s}(t,\cdot) \big\|_{H^{-1}(\T)^{\otimes \mathsf{v}'(F)}}^2
\\
& \ = \frac{\rd}{\rd t} \bigg( \int_{\T^{\mathsf{v}'(F)}} \Delta^{(n)}_{F,s}(t,z) \big( \Lambda^{\otimes \mathsf{v}'(F)} \star \Delta^{(n)}_{F,s} \big) (z) \rd z \bigg)
\\
& \ = \int_{\T^{\mathsf{v}'(F)}} 2 \partial_t \Delta^{(n)}_{F,s}(t,z) \big( \Lambda^{\otimes \mathsf{v}'(F)} \star \Delta^{(n)}_{F,s} \big) (z) \rd z
\\
& \ = \int_{\T^{\mathsf{v}'(F)}} 2 \big( \Lambda^{\otimes \mathsf{v}'(F)} \star \Delta^{(n)}_{F,s} \big) (z)
\\
& \hspace{1cm} \Bigg\{ \sum_{j \in \mathsf{v}(F)} \bigg[ - \partial_{z_j} \Big( \mu(z_j) \Delta^{(n)}_{F,s}(t,z) \Big) + \frac{\nu^2}{2} \partial_{z_j}^2 \Delta^{(n)}_{F,s}(t,z)
\\
& \hspace{1cm} - \partial_{z_j} \bigg(\int_{\T} \sigma(z_j,z_{|\mathsf{v}(F)|+1}) \Delta^{(n)}_{F+j,s+j}(t,z,z_{|\mathsf{v}(F)|+1}) \rd z_{|\mathsf{v}(F)|+1} \bigg)
\\
& \hspace{8cm} - \partial_{z_j} \mathscr{R}_{F,s,j}^{(n)}(t,z) \bigg] \Bigg\} \rd z.
\end{aligned}
\end{equation*}
Integration by parts yields
\begin{equation*}
\begin{aligned}
& \frac{\rd}{\rd t} \big\| \Delta^{(n)}_{F,s}(t,\cdot) \big\|_{H^{-1}(\T)^{\otimes \mathsf{v}'(F)}}^2
\\
& \ = \sum_{j \in \mathsf{v}(F)} \int_{\T^{\mathsf{v}'(F)}}\Bigg\{ - \nu^2 \big( \Lambda^{\otimes \mathsf{v}'(F)} \star \partial_{z_j} \Delta^{(n)}_{F,s} \big) (t,z) \partial_{z_j} \Delta^{(n)}_{F,s}(t,z)
\\
& \hspace{0.5cm} + 2 \big( \Lambda^{\otimes \mathsf{v}'(F)} \star \partial_{z_j} \Delta^{(n)}_{F,s} \big) (t,z) \bigg[ \Big( \mu(z_j) \Delta^{(n)}_{F,s}(t,z) \Big)
\\
& \hspace{0.5cm} + \bigg(\int_{\T} \sigma(z_j,z_{|\mathsf{v}(F)|+1}) \Delta^{(n)}_{F+j,s+j}(t,z,z_{|\mathsf{v}(F)|+1}) \rd z_{|\mathsf{v}(F)|+1} \bigg)
+ \mathscr{R}_{F,s,j}^{(n)}(t,z) \bigg] \Bigg\}\rd z.
\end{aligned}
\end{equation*}

As a standard argument, the positive diffusion term introduces a strictly negative contribution to the summation. Applying the Cauchy-Schwarz inequality, we get
\begin{equation*}
\begin{aligned}
& \frac{\rd}{\rd t} \big\| \Delta^{(n)}_{F,s}(t,\cdot) \big\|_{H^{-1}(\T)^{\otimes \mathsf{v}'(F)}}^2
\\
& \ \leq \sum_{j \in \mathsf{v}(F)} \bigg[ \bigg(- \nu^2 + \frac{\nu^2}{2} \bigg)\big\| \partial_{z_j} \Delta^{(n)}_{F,s}(t,\cdot) \big\|_{H^{-1}(\T)^{\otimes \mathsf{v}'(F)}}^2
\\
& \hspace{2cm} + \frac{3}{\nu^2}(L_{\mu;F,s,j}^{(n)} + L_{\sigma;F,s,j}^{(n)} + R_{F,s,j}^{(n)}) \bigg]
\end{aligned}
\end{equation*}
where, by Lemma~\ref{lem:commutator_inequality},
\begin{equation*}
\begin{aligned}
L_{\mu;F,s,j}^{(n)} = \big\| \mu(\cdot_j) \Delta^{(n)}_{F,s}(t,\cdot) \big\|_{H^{-1}(\T)^{\otimes \mathsf{v}'(F)}}^2 \leq 4 \|\mu\|_{W^{1,\infty}}^2 \big\| \Delta^{(n)}_{F,s}(t,\cdot) \big\|_{H^{-1}(\T)^{\otimes \mathsf{v}'(F)}}^2,
\end{aligned}
\end{equation*}
and by Lemma~\ref{lem:firing_rate_difference_tensorized},
\begin{equation*}
\begin{aligned}
& L_{\sigma;F,s,j}^{(n)}
\\
& \ = \int_{\T^{\mathsf{v}'(F)}} \bigg[ \Lambda_{1/2}^{\otimes \mathsf{v}'(F)} \star
\\
& \hspace{2cm} \bigg(\int_{\T} \sigma(\cdot_j,z_{|\mathsf{v}(F)|+1}) \Delta^{(n)}_{F+j,s+j}(t,\cdot,z_{|\mathsf{v}(F)|+1}) \rd z_{|\mathsf{v}(F)|+1} \bigg) (z) \bigg]^2 \rd z
\\
& \ \leq 16 \|\sigma\|_{W^{2,\infty}}^2 \big\| \Delta^{(n)}_{F+j,s+j}(t,\cdot) \big\|_{H^{-1}(\T)^{\otimes \mathsf{v}'(F+j)}}^2,
\end{aligned}
\end{equation*}
and by Proposition~\ref{prop:BBGKY},
\begin{equation*}
\begin{aligned}
R_{F,s,j}^{(n)} & = \| \mathscr{R}_{F,s,j}^{(n)}(t,\cdot) \|_{H^{-1}(\T)^{\otimes \mathsf{v}'(F)}}^2
\\
&\leq \big[ (1 + w_{\max})^{|\mathsf{e}(F)|+1} (\sqrt{\Lambda_{\max}})^{|\mathsf{v}'(F)|} \|\sigma\|_{L^\infty} |\mathsf{v}'(F)| / N_n \big]^2.
\end{aligned}
\end{equation*}

In conclusion,
\begin{equation} \label{eqn:negative_Sobolev_bound}
\begin{aligned}
& \frac{\rd}{\rd t} \big\| \Delta^{(n)}_{F,s}(t,\cdot) \big\|_{H^{-1}(\T)^{\otimes \mathsf{v}'(F)}}^2
\\
& \leq \sum_{j \in \mathsf{v}(F)} \bigg( \frac{12}{\nu^2} \|\mu\|_{W^{1,\infty}}^2 \big\| \Delta^{(n)}_{F,s}(t,\cdot) \big\|_{H^{-1}(\T)^{\otimes \mathsf{v}'(F)}}^2
\\
&\hspace{1.5cm} + \frac{48}{\nu^2} \|\sigma\|_{W^{2,\infty}}^2 \big\| \Delta^{(n)}_{F+j,s+j}(t,\cdot) \big\|_{H^{-1}(\T)^{\otimes \mathsf{v}'(F+j)}}^2
\\
&\hspace{1.5cm} + \frac{3}{\nu^2}\big[ (1 + w_{\max})^{|\mathsf{e}(F)|+1} (\sqrt{\Lambda_{\max}})^{|\mathsf{v}'(F)|} \|\sigma\|_{L^\infty} |\mathsf{v}'(F)| / N_n \big]^2 \bigg).
\end{aligned}
\end{equation}

We now consider $\mathcal{T}$, the set of all oriented trees. We take the maximum over all observables of oriented trees with an upper bound on the number of vertices:
\begin{equation*}
\begin{aligned}
M^{(n)}_k(t) \defeq \sup_{T \in \mathcal{T}, \ |\mathsf{v}(T)| \leq k, \ s \in \{0,1\}^{\mathsf{e}(T)} } \big\| \Delta^{(n)}_{T,s}(t,\cdot) \big\|_{H^{-1}(\T)^{\otimes \mathsf{v}'(T)}}^2.
\end{aligned}
\end{equation*}
Since we are considering trees, we have $|\mathsf{e}(T)| + 1 = |\mathsf{v}(T)| \leq k$. Hence by taking the maximum for \eqref{eqn:negative_Sobolev_bound} we have
\begin{equation*}
\begin{aligned}
& \frac{\rd}{\rd t} M^{(n)}_k(t) \leq k C_1 M^{(n)}_{k+1}(t) + C_2 C_3^{k} k^3/N_n^2,
\end{aligned}
\end{equation*}
where
\begin{equation*}
\begin{aligned}
& C_1 = \frac{12}{\nu^2} \|\mu\|_{W^{1,\infty}}^2 + \frac{48}{\nu^2} \|\sigma\|_{W^{2,\infty}}^2
\\
& C_2 = \frac{3}{\nu^2} \|\sigma\|_{L^\infty}^2, \quad C_3 = (1 + w_{\max})^2 \Lambda_{\max}.
\end{aligned}
\end{equation*}
In integral form, by iteration, this becomes:
\begin{equation*}
\begin{aligned}
M^{(n)}_k(t) &\leq M^{(n)}_k(0) + \int_0^t \bigg(k C_1 M^{(n)}_{k+1}(\tau_1) + C_2 C_3^{k} k^3/N_n^2\bigg) \rd \tau_1
\\
&\leq M^{(n)}_k(0) + \int_0^t \bigg(k C_1 M^{(n)}_{k+1}(0) + C_2 C_3^{k} k^3/N_n^2\bigg) \rd \tau_1
\\
&\quad + \int_0^t k C_1 \int_0^{\tau_1} \bigg((k+1) C_1 M^{(n)}_{k+2}(\tau_2) + C_2 C_3^{k+1} (k+1)^3/N_n^2\bigg) \rd \tau_2 \rd \tau_1
\\
&\leq \dots
\\
&\leq M^{(n)}_k(0)
\\
& \quad + \sum_{l = 1}^{m-1} \bigg[ \bigg((k-1+l) C_1 M^{(n)}_{k+l}(0) + C_2 C_3^{k-1+l} (k-1+l)^3/N_n^2\bigg)
\\
& \hspace{7cm} \bigg(\prod_{i = 1}^{l - 1} \frac{k-1 + i}{1 + i}\bigg) t^l \bigg]
\\
& \quad + \bigg((k-1+m) C_1 \sup_{0\leq \tau \leq t} M^{(n)}_{k+m}(\tau) + C_2 C_3^{k-1+m} (k-1+m)^3/N_n^2\bigg)
\\
& \hspace{7cm} \bigg(\prod_{i = 1}^{m - 1} \frac{k-1 + i}{1 + i}\bigg) t^m.
\end{aligned}
\end{equation*}

Now, taking the limit as $n \to \infty$, and using our assumptions of initial data convergence and that $N_n \to \infty$, all terms vanish except $\sup_{0 \leq \tau \leq t} M^{(n)}_{k+m}(\tau)$. For this term, we have an a priori bound
\begin{equation*}
\begin{aligned}
\sup_{0\leq \tau \leq t} M^{(n)}_{k+m}(\tau) \leq C_3^{k+m}, \quad C_3 = (1 + w_{\max})^2 \Lambda_{\max},
\end{aligned}
\end{equation*}
which is derived from the definition of $\Delta^{(n)}_{F,s}$ \eqref{eqn:defi_observable_difference} and Part~1 of Lemma~\ref{prop:negative_Sobolev_tensorized}.

Therefore,
\begin{equation*}
\begin{aligned}
& \limsup_{n \to \infty} M^{(n)}_k(t) 
\\
& \ \leq \bigg((k-1+m) C_1 \limsup_{n \to \infty} \sup_{0\leq \tau \leq t} M^{(n)}_{k+m}(\tau) \bigg) \bigg(\prod_{i = 1}^{m - 1} \frac{k-1 + i}{1 + i}\bigg) t^m
\\
& \ \leq \bigg((k-1+m) C_1 C_3^k\bigg) \bigg(\prod_{i = 1}^{m - 1} \frac{k-1 + i}{1 + i}\bigg) C_3^{m} t^m.
\end{aligned}
\end{equation*}
When $0 < t < \frac{1}{2C_3}$, we let $m \to 0$ and conclude that
\begin{equation*}
\begin{aligned}
\lim_{n \to \infty} M^{(n)}_k(t) = 0, \quad \forall k \in \N, \ \forall 0 < t < 1/2C_3.
\end{aligned}
\end{equation*}
We can take any $0 < t < 1/2C_3$ as our new starting time and repeat the argument, which gives $\lim_{n \to \infty} M^{(n)}_k(t) = 0$ for any $k$ in a longer interval. This essentially gives
\begin{equation*}
\begin{aligned}
\lim_{n \to \infty} M^{(n)}_k(t) = 0, \quad \forall k \in \N, \ \forall t \geq 0,
\end{aligned}
\end{equation*}
and completes the stability in terms of 
\begin{equation*}
\begin{aligned}
\Delta^{(n)}_{F,s} \defeq \big(\bm{\tau}_m(F,w^{(n)},X^{(n)}_0)\big)_s - \big(\bm{\tau}_m(F,w^{(\infty)},X^{(\infty)}_0)\big)_s,
\end{aligned}
\end{equation*}
for $\nu > 0$.

\hfill

Next, consider the case where $\nu = 0$. Let $(w, X^0)$ be the solution of \eqref{eqn:multi_agent_SDE} and $(w, X^\nu)$ be a solution with artificial viscosity $\nu$ added, both with the same initial data. It is straightforward to verify that
\begin{equation*}
\begin{aligned}
\frac{1}{N} \sum_{i \in [N]} |X_i^0(t) - X_i^\nu(t)| \leq C^{\downarrow}_{w_{\max}, \|\mu\|_{W^{1,\infty}}, \|\sigma\|_{W^{1,\infty}}, t}(\nu).
\end{aligned}
\end{equation*}
Similarly, let $(w, f^0)$ be the solution of \eqref{eqn:multi_agent_Vlasov_first} and $(w, f^\nu)$ be a solution with artificial viscosity $\nu$ added, also with the same initial data,
\begin{equation*}
\begin{aligned}
\int \Big[ d_{W_2}\big( f^0(t,\cdot,\xi), f^\nu(t,\cdot,\xi) \big) \Big]^2 \rd \xi \leq C^{\downarrow}_{w_{\max}, \|\mu\|_{W^{1,\infty}}, \|\sigma\|_{W^{1,\infty}}, t}(\nu).
\end{aligned}
\end{equation*}
where $d_{W_2}$ denotes the Wasserstein-2 distance.
These metrics induce the weak-* topology of measures. Hence the above bounds essentially imply the following: for all $T \in \mathcal{T}$ and $s \in \{0,1\}^{\mathsf{e}(T)}$,
\begin{equation*}
\begin{aligned}
& \big\|\big(\bm{\tau}_m(T,w,X^0(t))\big)_s - \big(\bm{\tau}_m(T,w,X^\nu(t)) \big)_s \big\|_{H^{-1}(\T)^{\otimes \mathsf{v}'(T)}}
\\
& \ \leq C^{\downarrow}_{w_{\max}, \|\mu\|_{W^{1,\infty}}, \|\sigma\|_{W^{1,\infty}}, T, t}(\nu).
\end{aligned}
\end{equation*}

Now, for a sequence of pairs $(w^{(n)}, X^{(n)})$ with converging initial data, we have convergence for any artificial viscosity $\nu > 0$ added at any later time. Using the a priori bound for the solutions (with and without artificial viscosity) and passing the limit $\nu \to 0$, we obtain the convergence of
\begin{equation*}
\begin{aligned}
\Delta^{(n)}_{F,s} \defeq \big(\bm{\tau}_m(F,w^{(n)},X^{(n)}_0)\big)_s - \big(\bm{\tau}_m(F,w^{(\infty)},X^{(\infty)}_0)\big)_s,
\end{aligned}
\end{equation*}
in the case $\nu = 0$.

\hfill

It remains to establish the stability estimate (in expectation) for $\bm{\tau}$, using the stability already obtained for $\bm{\tau}_m$.
We omit the time variable in the notation, as there is no dynamics in this part.

First, consider $\bm{\tau}_p$ (which allows repetitive agents in the law) instead of $\bm{\tau}_m$. Recall that
\begin{equation*}
\begin{aligned}
\big( \bm{\tau}_{p}(F,w,X) \big)_s &\defeq \E \bigg[ \big( \bm{\tau}(F,w,X) \big)_s \bigg]
\\
& = \frac{1}{N^{|\mathsf{v}(F)|}} \sum_{l \in [N]^{\mathsf{v}(F)}} \prod_{(j_1,j_2) \in \mathsf{e}(F)} \big( w_{l_{j_1},l_{j_2}} \big)^{s_{j_1,j_2}} \Law_{(X_{l_j})_{j \in \mathsf{v}'(F)}}.
\end{aligned}
\end{equation*}
For any oriented tree $T \in \mathcal{T}$ and $s \in \{0,1\}^{\mathsf{e}(T)}$, the convergence of $\bm{\tau}_m$ extends to
\begin{equation*}
\begin{aligned}
& \big( \bm{\tau}_{p}(T,w^{(n)},X^{(n)}) \big)_s \wsto \big( \bm{\tau}_{p}(T,w^{(\infty)},X^{(\infty)}) \big)_s,
\end{aligned}
\end{equation*}
as the difference between the two sequences $\bm{\tau}_p$, $\bm{\tau}_m$ in $H^{-1}(\T)^{\otimes \mathsf{v}'(F)}$ norm vanishes by Proposition~\ref{prop:empirical_observables_diff}.

For any oriented tree $T \in \mathcal{T}$, using the index $\mathsf{v}(T) = \{1, \dots, |\mathsf{v}(T)| \}$, define $T + T$ as follows:
\begin{equation*}
\begin{aligned}
\mathsf{v}(T+T) &= \{1,\dots, 2|\mathsf{v}(T) |\},
\\
\mathsf{e}(T+T) &= \{(1,1+|\mathsf{v}(T) |)\} \cup
\\
& \hspace{0.5cm} \{ (i,j) \in \mathsf{v}(T+T)^2 : (i,j) \in \mathsf{e}(T) \text{ or } (i-\mathsf{v}(T),j-\mathsf{v}(T)) \in \mathsf{e}(T)\}.
\end{aligned}
\end{equation*}
In other words, we define $(T + T)$ by drawing an additional edge between two copies of $T$.
For any $s \in \{0,1\}^{\mathsf{e}(F)}$, define $s + s$ by
\begin{equation*}
(s+s)_{i,j} = \left\{
\begin{aligned}
&0 &&\text{ if } (i,j) = (1,1+|\mathsf{v}(T) |),
\\
&s_{i,j} &&\text{ if } (i,j) \in \mathsf{e}(T),
\\
&s_{i-\mathsf{v}(T),j-\mathsf{v}(T)} &&\text{ if } (i-\mathsf{v}(T),j-\mathsf{v}(T)) \in \mathsf{e}(T).
\end{aligned} \right.
\end{equation*}

Then 
\begin{equation*}
\begin{aligned}
&\E \bigg[ \big\| \big( \bm{\tau}(F,w^{(n)},X^{(n)}) - \bm{\tau}(F,w^{(\infty)},X^{(\infty)}) \big)_s \big\|_{H^{-2}(\T)^{\otimes \mathsf{v}'(F)}}^2 \bigg]
\\
& \ = \E \bigg[ \int_{\T^{\mathsf{v}'(F)}} \int_{\T^{\mathsf{v}'(F)}} \Lambda_2^{\otimes \mathsf{v}'(F)}(z - u) \big( \bm{\tau}(F,w^{(n)},X^{(n)}) - \bm{\tau}(F,w^{(\infty)},X^{(\infty)}) \big)_s (z)
\\
& \hspace{4.5cm} \big( \bm{\tau}(F,w^{(n)},X^{(n)}) - \bm{\tau}(F,w^{(\infty)},X^{(\infty)}) \big)_s (u) \rd z \rd u \bigg]
\\
& \ = \int_{\T^{\mathsf{v}'(F)}} \int_{\T^{\mathsf{v}'(F)}} \Lambda_2^{\otimes \mathsf{v}'(F)}(z - u) \bar \Delta^{(n)}_{F+F,s+s}(z,u) \rd z \rd u,
\end{aligned}
\end{equation*}
where
\begin{equation*}
\begin{aligned}
\bar \Delta^{(n)}_{F+F,s+s}(z,u) & \defeq \big( \bm{\tau}_{p}(F+F,w^{(n)},X^{(n)}) \big)_{s+s} (z,u)
\\
& \quad - 2 \big( \bm{\tau}_{p}(F,w^{(n)},X^{(n)}) \big)_s(z) \big( \bm{\tau}(F,w^{(\infty)},X^{(\infty)}) \big)_s(u)
\\
& \quad + \big( \bm{\tau}(F,w^{(\infty)},X^{(\infty)}) \big)_s(z) \big( \bm{\tau}(F,w^{(\infty)},X^{(\infty)}) \big)_s(u).
\end{aligned}
\end{equation*}
From the previous estimate, we have
\begin{equation*}
\begin{aligned}
\lim_{n \to \infty} \big\| \bar \Delta^{(n)}_{F+F,s+s} \big\|_{H^{-1}(\mathbb{T})^{\otimes \mathsf{v}'(F+F)}} = 0.
\end{aligned}
\end{equation*}
By a duality argument,
\begin{equation*}
\begin{aligned}
&\E \bigg[ \big\| \big( \bm{\tau}(F,w^{(n)},X^{(n)}) - \bm{\tau}(F,w^{(\infty)},X^{(\infty)}) \big)_s \big\|_{H^{-2}(\T)^{\otimes \mathsf{v}'(F)}}^2 \bigg]
\\
& \ \leq \big\| \bar \Delta^{(n)}_{F+F,s+s} \big\|_{H^{-1}(\T)^{\otimes \mathsf{v}'(F+F)}}
\\
& \hspace{0.5cm} \bigg[ \int_{\T^{\mathsf{v}'(F)}} \int_{\T^{\mathsf{v}'(F)}} \Lambda_2^{\otimes \mathsf{v}'(F)}(z - u)
\\
& \hspace{2.5cm} \bigg(\prod_{i \in \mathsf{v}'(F)} (1-\partial_{z_i}^2 - \partial_{u_i}^2 + \partial_{z_i}^2 \partial_{u_i}^2)\bigg) \Lambda_2^{\otimes \mathsf{v}'(F)}(z - u) \rd z \rd u \bigg]^{\frac{1}{2}}.
\end{aligned}
\end{equation*}
Here, the second factor is bounded, and the first factor converges to zero as $n \to \infty$. This establishes the convergence of $\bm{\tau}(F,w^{(n)},X^{(n)})$ in expectation in the $H^{-2}(\T)^{\otimes \mathsf{v}'(F)}$ norm. By Proposition~\ref{prop:negative_Sobolev_tensorized}, the $H^{-2}(\T)^{\otimes \mathsf{v}'(F)}$ norm induces the weak-* topology for measures. By Lemma~\ref{lem:uniform_bound_1} and Lemma~\ref{lem:uniform_bound_2}, this convergence in expectation can be expressed in terms of any other metric that induces the same weak-* topology on measures, such as the $H^{-1}(\T)^{\otimes \mathsf{v}'(F)}$ norm.

Recall that by Proposition~\ref{prop:normal_plain_equivalence}, the convergence of all observables can be further expressed as the convergence in a single metric by taking the maximum. This completes the proof of Lemma~\ref{lem:main_observable}.

\end{proof}

\section{Graphon isomorphism, cut distance and homomorphism densities} \label{sec:graph_1}

The objective of this section is to prove Lemma~\ref{lem:graphon_compactness_first} and Lemma~\ref{lem:main_counting}. These results are extensions of those of \cite{lovasz2006limits} and \cite{borgs2008convergent}, but we follow the approach outlined in the textbook \cite{lovasz2012large}. 

In our treatment, we have omitted explicit calculations of convergence rates in Counting Lemma and Inverse Counting Lemma. This allows us to present a relatively simpler argument that highlights the conceptual key parts of the proof.

Since in this section we consider only atomless standard probability spaces and do not require any lifting operations, we will keep the discussion throughout on $\I = [0,1]$, in alignment with the classical literature.

Additionally, in this section, we need to consider unlabeled distances for finite graphs. We adopt the usual convention that the cut distance of a finite graph is defined via its piecewise constant extension to $[0,1]^2$ as follows.

\begin{defi}[\textbf{Piecewise constant extension of a finite graph}]
Let $\mathcal{H}$ be a Hilbert space.
Given $\bm{w} \in L^\infty([N] \times [N] ; \mathcal{H})$, we define the piecewise constant function $\bar {\bm{w}} \in L^\infty([0,1]^2 ; \mathcal{H})$ by
\begin{equation*}
\begin{aligned}
\bar {\bm{w}} (\xi, \zeta) \defeq \sum_{i,j \in [N]} \mathbbm{1}_{\left[\frac{i-1}{N}, \frac{i}{N}\right) \times \left[\frac{j-1}{N}, \frac{j}{N}\right)}(\xi, \zeta) \, \bm{w}(i,j), \quad \forall (\xi,\zeta) \in [0,1]^2.
\end{aligned}
\end{equation*}

If one or both of $\bm{w}_1$ and $\bm{w}_2$ are defined on a finite set, the unlabeled distance $\delta_{\square;\mathcal{H}}(\bm{w}_1, \bm{w}_2)$ should be understood as the cut distance between their piecewise constant extensions $\bar{\bm{w}}_1$ and $\bar{\bm{w}}_2$.
\end{defi}

\subsection{Regularity Lemma}

In this subsection, we reproduce Szemerédi's Regularity Lemma \cite{szemeredi1975sets, szemeredi1978regular} for $\bm{w} \in L^\infty([0,1]^2; \mathcal{H})$.

\begin{defi}
[\textbf{$k$-measurable partition and piecewise projection}]

A $k$-measurable partition $P$ of the interval $[0,1]$ is a collection of $k$ measurable subsets $V_1, V_2, \dots, V_k$ of $[0,1]$ such that
\begin{equation*}
\begin{aligned}
V_i \cap V_j = \varnothing \quad \text{for all } i \neq j.
\end{aligned}
\end{equation*}
and
\begin{equation*}
\begin{aligned}
\bigcup_{i=1}^k V_i = [0,1].
\end{aligned}
\end{equation*}
Each $V_i$ is called a class of partition $P$.

Given a function $\bm{w} \in L^\infty([0,1]^2; \mathcal{H})$, where $\mathcal{H}$ is a Hilbert space, and a $k$-measurable partition $P = \{ V_i \}_{i=1}^k$, the \emph{piecewise projection} $\bm{w}_P$ is defined as the function in $L^\infty([0,1]^2; \mathcal{H})$ that is constant on each rectangle $V_i \times V_j$ and equals the average of $\bm{w}$ on that rectangle. Formally, for $\xi, \zeta \in [0,1]$:
\begin{equation*}
\begin{aligned}
\bm{w}_P(\xi, \zeta) = \sum_{i=1}^k \sum_{j=1}^k \mathbbm{1}_{V_i}(\xi) \, \mathbbm{1}_{V_j}(\zeta) \, \bm{w}_{ij},
\end{aligned}
\end{equation*}
where
\begin{equation*}
\begin{aligned}
\bm{w}_{ij} = \frac{1}{|V_i| |V_j|} \int_{V_i \times V_j} \bm{w}(\xi', \zeta') \, \rd\zeta' \, \rd\xi',
\end{aligned}
\end{equation*}
\end{defi}
\begin{rmk}
Alternatively, $\bm{w}_P$ can be viewed as the conditional expectation of $\bm{w}$ with respect to the $\sigma$-algebra generated by $P \times P$:
\begin{equation*}
\begin{aligned}
\bm{w}_P = \mathbb{E}\left[ \bm{w} | P \times P \right].
\end{aligned}
\end{equation*}

\end{rmk}

The following results are classical in the scalar-valued case; we simply extend them to the Hilbert space setting.
\begin{lem} \label{lem:energy_increase}
Let $\mathcal{H}$ be a Hilbert space, and $\bm{w} \in L^\infty([0,1]^2; \mathcal{H})$. Then
\begin{equation*}
\begin{aligned}
\|\bm{w}\|_{\square;\mathcal{H}} \leq \|\bm{w}\|_{L^2([0,1]^2; \mathcal{H})}.
\end{aligned}
\end{equation*}
In addition, let $P$ be a $k$-measurable partition of $[0,1]$. Then there exists a partition $Q$ refining $P$ with at most $4k$ classes such that the piecewise projection
\begin{equation*}
\begin{aligned}
\|\bm{w} - \bm{w}_P\|_{\square;\mathcal{H}} = \|\bm{w}_Q - \bm{w}_P\|_{\square;\mathcal{H}}.
\end{aligned}
\end{equation*}

\end{lem}

\begin{proof}[Proof of Lemma~\ref{lem:energy_increase}]

We begin by proving the first inequality. By definition,
\begin{equation*}
\begin{aligned}
\|\bm{w}\|_{\square;\mathcal{H}} & = \sup_{S,T \subset [0,1]} \bigg\| \int_{S \times T} \bm{w}(\xi,\zeta) \rd \xi \rd \zeta \bigg\|_{\mathcal{H}}
\\
& \leq \sup_{S,T \subset [0,1]} \int_{S \times T} \big\| \bm{w}(\xi,\zeta) \big\|_{\mathcal{H}} \rd \xi \rd \zeta
\\
& \leq \bigg( \int_{[0,1]^2} \big\| \bm{w}(\xi,\zeta) \big\|_{\mathcal{H}}^2 \rd \xi \rd \zeta \bigg)^{\frac{1}{2}}
\\
& = \|\bm{w}\|_{L^2([0,1]^2; \mathcal{H})}.
\end{aligned}
\end{equation*}

Next, we proceed to prove the second identity.
Notice that for any refinement $Q$ of $P$, we have $\|\bm{w} - \bm{w}_P\|_{\square;\mathcal{H}} \geq \|\bm{w}_Q - \bm{w}_P\|_{\square;\mathcal{H}}$ by contractivity. For the inverse direction, let $\bm{e} \in \mathcal{H}$ with $\|\bm{e}\|_{\mathcal{H}} = 1$, and let $S$ and $T$ be measurable subsets of $[0,1]$ that achieve the optimum, i.e.,
\begin{equation*}
\begin{aligned}
\|\bm{w} - \bm{w}_P\|_{\square;\mathcal{H}} = \bigg\langle \bm{e}, \int_{S \times T} (\bm{w} - \bm{w}_P)(\xi,\zeta) \rd \xi \rd \zeta \bigg\rangle_{\mathcal{H}}.
\end{aligned}
\end{equation*}
Let $Q$ denote the partition generated by $P$, $S$, and $T$, which has at most $4k$ classes. Then
\begin{equation*}
\begin{aligned}
& \int_{S \times T} (\bm{w} - \bm{w}_P)(\xi,\zeta) \rd \xi \rd \zeta = \int_{[0,1]^2} (\bm{w} - \bm{w}_P)(\xi,\zeta) \mathbbm{1}_{S \times T}(\xi,\zeta) \rd \xi \rd \zeta
\\
& \ = \int_{[0,1]^2} (\bm{w} - \bm{w}_P)(\xi,\zeta) ( \mathbbm{1}_{S \times T})_{Q} (\xi,\zeta) \rd \xi \rd \zeta
\\
& \ = \int_{[0,1]^2} (\bm{w} - \bm{w}_P)_Q(\xi,\zeta) \mathbbm{1}_{S \times T} (\xi,\zeta) \rd \xi \rd \zeta
\\
& \ = \int_{[0,1]^2} (\bm{w}_Q - \bm{w}_P)(\xi,\zeta) \mathbbm{1}_{S \times T} (\xi,\zeta) \rd \xi \rd \zeta.
\end{aligned}
\end{equation*}
Hence, $\|\bm{w} - \bm{w}_P\|_{\square;\mathcal{H}} \leq \|\bm{w}_Q - \bm{w}_P\|_{\square;\mathcal{H}}$, and we conclude.
\end{proof}

With the results above, we can now derive the weak regularity lemma for $\bm{w} \in L^\infty([0,1]^2; \mathcal{H})$.

\begin{lem}[Weak regularity lemma in Hilbert space] \label{lem:weak_regularity_lemma}
Let $\mathcal{H}$ be a Hilbert space.
For every kernel $\bm{w} \in L^\infty([0,1]^2; \mathcal{H})$ and $k \geq 1$, there exists a partition $P$ of $[0,1]$ into at most $k$ sets with positive measure for which
\begin{equation*}
\begin{aligned}
\|\bm{w} - \bm{w}_P\|_{\square;\mathcal{H}} \leq \frac{2}{\sqrt{\log k}} \|\bm{w}\|_{L^\infty([0,1]^2; \mathcal{H})}^2.
\end{aligned}
\end{equation*}
\end{lem}

In addition, for every $m$-partition $Q$ of $[0,1]$, there is an equipartition $P$ with $k$ classes that 
\begin{equation*}
\begin{aligned}
\|\bm{w} - \bm{w}_P\|_{\square;\mathcal{H}} \leq 2 \|\bm{w} - \bm{w}_Q\|_{\square;\mathcal{H}} + \frac{2m}{k} \|\bm{w}\|_{L^\infty([0,1]^2; \mathcal{H})}^2.
\end{aligned}
\end{equation*}

\begin{proof}

Given any partition $P$, we can choose a refinement $Q$ as in Lemma~\ref{lem:energy_increase}. It is easy to check that $(\bm{w}_Q - \bm{w}_P)$ is orthogonal to $\bm{w}_P$ in $L^2([0,1]^2; \mathcal{H})$. Hence,
\begin{equation*}
\begin{aligned}
\|\bm{w} - \bm{w}_P\|_{\square;\mathcal{H}}^2 &= \|\bm{w}_Q - \bm{w}_P\|_{\square;\mathcal{H}}^2 \leq \|\bm{w}_Q - \bm{w}_P\|_{L^2([0,1]^2; \mathcal{H})}^2 \\
&= \|\bm{w}_Q\|_{L^2([0,1]^2; \mathcal{H})}^2 - \|\bm{w}_P\|_{L^2([0,1]^2; \mathcal{H})}^2.
\end{aligned}
\end{equation*}
Starting with the initial partition $P_0 = \{\varnothing, [0,1]\}$, let each $P_m$ be a refinement of $P_{m-1}$ as in Lemma~\ref{lem:energy_increase}. Then we have
\begin{equation*}
\begin{aligned}
\|\bm{w}_{P_m}\|_{L^2([0,1]^2; \mathcal{H})}^2 &\geq \|\bm{w}_{P_{m-1}}\|_{L^2([0,1]^2; \mathcal{H})}^2 + \|\bm{w} - \bm{w}_{P_{m-1}}\|_{\square;\mathcal{H}}^2 \\
&\geq \dots \geq \|\bm{w}_{P_0}\|_{L^2([0,1]^2; \mathcal{H})}^2 + \sum_{i=0}^{m-1} \|\bm{w} - \bm{w}_{P_i}\|_{\square;\mathcal{H}}^2.
\end{aligned}
\end{equation*}
Since $\|\bm{w}_{P_k}\|_{L^2([0,1]^2; \mathcal{H})}^2 \leq \|\bm{w}\|_{L^\infty([0,1]^2; \mathcal{H})}^2$, it follows that
\begin{equation*}
\begin{aligned}
\|\bm{w} - \bm{w}_{P_m}\|_{\square;\mathcal{H}} \leq \frac{1}{\sqrt{m}}.
\end{aligned}
\end{equation*}
By our construction, each $P_m$ has at most $4^m$ sets. Optimizing $m$ for any given $k$, we conclude the first inequality in the lemma.

To prove the last inequality, we partition each class of $Q$ into classes of measure $1/k$, with at most one exceptional class of size less than $1/k$. Keeping all classes of size $1/k$, let us take the union of exceptional classes and repartition it into classes of size $1/k$ to get a partition $P$.

\end{proof}

\subsection{Compactness in unlabeled distance}
As a standard approach in graph limit theory, we use the regularity lemma to further prove Lemma~\ref{lem:graphon_compactness_first} on the compactness of kernels in the unlabeled distance, which we restate here with $\I = [0,1]$:
\begin{lem} [\textbf{Compactness lemma, restated}] \label{lem:graphon_compactness_restate}
Let $\mathcal{B}$ be a Banach space compactly embedded into a separable Hilbert space $\mathcal{H}$.
For any sequence $\{\bm{w}_n\}_{n=1}^\infty \subset L^\infty([0,1]^2; \mathcal{B})$ satisfying uniform bound that
\begin{equation*}
\begin{aligned}
\sup_n \|\bm{w}_n\|_{L^\infty([0,1]^2; \mathcal{B})} < w_{\max} < \infty.
\end{aligned}
\end{equation*}
up to an extraction of subsequence (which we still index by $n$), there exists $\bm{w} \in L^\infty([0,1]^2; \mathcal{B})$ that 
\begin{equation*}
\begin{aligned}
\lim_{n \to \infty}
\delta_{\square;\mathcal{H}}(\bm{w}_n, \bm{w}) = 0.
\end{aligned}
\end{equation*}
Moreover, if the sequence $\{\bm{w}_n\}_{n=1}^\infty$ is defined from $w^{(n)} \in L^\infty([0,1]^2)$ and $X^{(n)} : [0,1] \to \T$, and $\bm{w}_n = \bm{w}_{w^{(n)},X^{(n)}} \in L^\infty([0,1]^2; H^{-1}(\T)^{\oplus \{0,1\}})$ defined as in Definition~\ref{defi:functionalize_graphon_restate},
there exist limiting $w \in L^\infty([0,1]^2)$ and $X: [0,1] \to \T$ that $\bm{w} = \bm{w}_{w,X}$.

\end{lem}

\begin{proof} [Proof of Lemma~\ref{lem:graphon_compactness_restate}]

For every $n \geq 1$, we can construct partitions $P_{n,m}$ of $[0,1]$ using the weak regularity lemma such that the corresponding piecewise constant projections $(\bm{w}_{n})_{P_{n,m}} \eqdef \bm{w}_{n,m}$ satisfy the following conditions:
\begin{enumerate}
  \item
  $
  \|\bm{w}_{n} - \bm{w}_{n,m}\|_{\square;\mathcal{H}} \leq \frac{1}{m},
  $
  \item
  The partition $P_{n,m+1}$ refines $P_{n,m}$,
  \item
  $|P_{n,m}| \leq k_m$, where $k_m$ depends only on $m$.
\end{enumerate}
Once we have such partitions, we can rearrange the points of $[0,1]$ for each fixed $n$ by a measure-preserving bijection so that every partition class in every $P_{n,k}$ is an interval.

Since $\bm{w}$ takes values in $\mathcal{B}_{\leq w_{\max}}$, which is compact in $\mathcal{H}$, we can, by extracting a subsequence indexed by $n$, assume that as $n \to \infty$, the length of the $i$-th interval of $\bm{w}_{n,m}$ converges for each $1 \leq i \leq k_m$ and that the values on the product of the $i$-th and $j$-th intervals converge in $\mathcal{H}$ for each $1 \leq i, j \leq k_m$. It follows that as $n \to \infty$, $\bm{w}_{n,m}$ converges almost everywhere to a step function $\bm{u}_m \in L^\infty([0,1]^2; \mathcal{B})$ with $k_m$ steps.

Let $P_m$ denote the partition of $[0,1]$ into the steps of $\bm{u}_m$. For every $m < l$, the partition $P_{n,l}$ is a refinement of $P_{n,m}$, and therefore $\bm{w}_{n,m} = (\bm{w}_{n,l})_{P_{n,m}}$. It is straightforward that this relation is inherited by the limiting step functions, i.e.,
\begin{equation*}
\begin{aligned}
\bm{u}_{m} = (\bm{u}_{l})_{P_m}.
\end{aligned}
\end{equation*}
Assuming $\xi, \zeta \in [0,1]$ are uniform random variables, the sequence
\begin{equation*}
\begin{aligned}
\big(\langle \bm{e}, \bm{u}_m(\xi,\zeta) \rangle_{\mathcal{H}}\big)_{m=1}^\infty
\end{aligned}
\end{equation*}
is a martingale for any $\bm{e} \in \mathcal{H}$. Since these random variables remain bounded as $m \to \infty$, the martingale convergence theorem implies that this sequence converges with probability $1$ (almost everywhere on $(\xi,\zeta) \in [0,1]^2$).

Then, by the separability of $\mathcal{H}$ and the compact embedding $\mathcal{B} \subset \mathcal{H}$, we conclude that $\bm{u}_m(\xi,\zeta)$ converges almost everywhere on $(\xi,\zeta) \in [0,1]^2$. Let $\bm{u}$ denote this limit. It satisfies
\begin{equation*}
\begin{aligned}
\delta_{\square;\mathcal{H}}(\bm{u},\bm{w}_n) \leq \delta_{\square;\mathcal{H}}(\bm{u},\bm{u}_m) + \delta_{\square;\mathcal{H}}(\bm{u}_m,\bm{w}_{n,m}) + \delta_{\square;\mathcal{H}}(\bm{w}_{n,m},\bm{w}_n).
\end{aligned}
\end{equation*}
Taking the limit as $n \to \infty$ and then $m \to \infty$, we obtain
\begin{equation*}
\begin{aligned}
\lim_{n \to \infty} \delta_{\square;\mathcal{H}}(\bm{u},\bm{w}_n) = 0,
\end{aligned}
\end{equation*}
which completes the proof of the limit.

Finally, suppose that the sequence $\{\bm{w}_n\}_{n=1}^\infty$ is defined from $w^{(n)} \in L^\infty([0,1]^2)$ and $X^{(n)} : [0,1] \to \T$. Using a similar argument using the martingale convergence theorem, we can identify the desired limits $w$ and $X$ such that $\bm{w} = \bm{w}_{w,X}$.

\end{proof}

\subsection{Counting Lemma}

From this point forward, our estimates will rely more heavily on the compact embedding $\mathcal{B} \subset \mathcal{H}$. We will begin by considering the finite-dimensional case, i.e. $\mathcal{H} = \mathbb{R}^d$, and then extend to the general case. 

This approach considerably simplifies the analysis. However, let us emphasize that in the previous proofs, we intentionally avoided this approach. Our goal is not only to establish the equivalence of convergence but also to identify the limit, so at least one compactness lemma must avoid techniques that rely on approximating compact operators in finite dimensions.

We begin with the extension of Counting Lemma, which is one direction of Lemma~\ref{lem:main_counting}.
\begin{lem}[\textbf{Counting lemma on compact subspace}] \label{lem:counting_lemma}
Let $\mathcal{B}$ be a Banach space compactly embedded into a separable Hilbert space $\mathcal{H}$.
If $\{\bm{w}_n\}_{n=1}^\infty \cup \{\bm{w}\} \subset L^\infty([0,1]^2; \mathcal{B})$ satisfy uniform bound that
\begin{equation*}
\begin{aligned}
\sup_n \|\bm{w}_n\|_{L^\infty([0,1]^2; \mathcal{B})} < w_{\max} < \infty
\end{aligned}
\end{equation*}
and convergence in unlabeled distance that
\begin{equation*}
\begin{aligned}
\lim_{n \to \infty} \delta_{\square;\mathcal{H}}( \bm{w}_n, \bm{w}) = 0,
\end{aligned}
\end{equation*}
then, for all oriented simple graph $F \in \mathcal{G}$,
\begin{equation*}
\begin{aligned}
\lim_{n \to \infty} \| \bm{t}(F,\bm{w}_n) - \bm{t}(F,\bm{w}) \|_{\mathcal{H}^{\otimes \mathsf{e}(F)}} = 0.
\end{aligned}
\end{equation*}
\end{lem}

\begin{proof} [Proof of Lemma~\ref{lem:counting_lemma}]

First, note that any measure-preserving rearrangement $\Phi: [0,1] \to [0,1]$ does not change the tensorized homomorphism densities. The following identity can be straightforwardly verified:
\begin{equation*}
\begin{aligned}
\bm{t}(F,\bm{w}) &= \int_{\I^{\mathsf{v}(F)}} \bigotimes_{(i,j) \in \mathsf{e}(F)} \bm{w}(\xi_i,\xi_j) \prod_{i \in \mathsf{v}(F)} \rd \xi_i
\\
&= \int_{\I^{\mathsf{v}(F)}} \bigotimes_{(i,j) \in \mathsf{e}(F)} \bm{w}(\Phi(\zeta_i),\Phi(\zeta_j)) \prod_{i \in \mathsf{v}(F)} \rd \Phi( \zeta_i)
\\
&= \int_{\I^{\mathsf{v}(F)}} \bigotimes_{(i,j) \in \mathsf{e}(F)} \bm{w}(\Phi(\zeta_i),\Phi(\zeta_j)) \prod_{i \in \mathsf{v}(F)} \rd \zeta_i
\\
&= \bm{t}(F,\bm{w}^{\Phi}).
\end{aligned}
\end{equation*}
Hence, by applying rearrangements, we may assume $\lim_{n \to \infty} \|\bm{w}_n - \bm{w}\|_{\square;\mathcal{H}} = 0$, and from this, we proceed to prove the convergence of all $\bm{t}(F, \bm{w}_n)$.

For any $F \in \mathcal{G}$, we assign an order to the edge set $\mathsf{e}(F)$ (any order will suffice for our argument). Then, by editing the edges one by one,
\begin{equation*}
\begin{aligned}
& \bm{t}(F,\bm{w}_n) - \bm{t}(F,\bm{w})
\\
& \ = \int_{\I^{\mathsf{v}(F)}} \bigg[ \bigotimes_{(i,j) \in \mathsf{e}(F)} \bm{w}_n(\xi_i,\xi_j) - \bigotimes_{(i,j) \in \mathsf{e}(F)} \bm{w}(\xi_i,\xi_j) \bigg] \prod_{i \in \mathsf{v}(F)} \rd \xi_i
\\
& \ = \sum_{(i',j') \in \mathsf{e}(F)} \int_{\I^{\mathsf{v}(F)}} \bigg( \bigotimes_{(i,j) \in \mathsf{e}(F), (i,j) < (i',j')} \bm{w}_n(\xi_i,\xi_j) \bigg) \otimes \big[ \bm{w}_n(\xi_{i'},\xi_{j'}) - \bm{w}(\xi_{i'},\xi_{j'}) \big]
\\
& \hspace{5.5cm} \otimes \bigg( \bigotimes_{(i,j) \in \mathsf{e}(F), (i',j') < (i,j)} \bm{w}(\xi_i,\xi_j) \bigg) \prod_{i \in \mathsf{v}(F)} \rd \xi_i.
\end{aligned}
\end{equation*}
Let $\bm{e}_1, \dots, \bm{e}_d$ be an orthonormal basis of $\mathcal{H} = \mathbb{R}^d$. Then
\begin{equation*}
\begin{aligned}
\bigotimes_{(i,j) \in \mathsf{e}(F)} \bm{e}_{l_{i,j}}, \quad l \in [d]^{\mathsf{e}(F)}
\end{aligned}
\end{equation*}
forms an orthonormal basis of $(\mathbb{R}^d)^{\otimes \mathsf{e}(F)}$. For any basis element,
\begin{equation*}
\begin{aligned}
& \bigg\langle \bigotimes_{(i,j) \in \mathsf{e}(F)} \bm{e}_{l_{i,j}}, \ \bm{t}(F,\bm{w}_n) - \bm{t}(F,\bm{w}) \bigg\rangle
\\
& \ = \sum_{(i',j') \in \mathsf{e}(F)} \int_{\I^{\mathsf{v}(F)}} \bigg( \prod_{(i,j) \in \mathsf{e}(F), (i,j) < (i',j')} \langle \bm{e}_{l_{i,j}}, \bm{w}_n(\xi_i,\xi_j) \rangle \bigg)
\\
& \hspace{4cm} \langle \bm{e}_{l_{i',j'}}, \bm{w}_n(\xi_{i'},\xi_{j'}) - \bm{w}(\xi_{i'},\xi_{j'}) \rangle
\\
& \hspace{4cm} \bigg( \prod_{(i,j) \in \mathsf{e}(F), (i',j') < (i,j)} \langle \bm{e}_{l_{i,j}}, \bm{w}(\xi_i,\xi_j) \rangle \bigg) \prod_{i \in \mathsf{v}(F)} \rd \xi_i.
\end{aligned}
\end{equation*}
For fixed $(i', j') \in \mathsf{e}(F)$, it is easy to verify that 
\begin{equation*}
\begin{aligned}
& \bigg| \int_{\I^{\mathsf{v}(F) \setminus \{i',j'\}}} \bigg( \prod_{(i,j) \in \mathsf{e}(F), (i,j) < (i',j')} \langle \bm{e}_{l_{i,j}}, \bm{w}_n(\xi_i,\xi_j) \rangle \bigg)
\\
& \hspace{2cm} \bigg( \prod_{(i,j) \in \mathsf{e}(F), (i',j') < (i,j)} \langle \bm{e}_{l_{i,j}}, \bm{w}(\xi_i,\xi_j) \rangle \bigg) \prod_{i \in \mathsf{v}(F) \setminus \{i',j'\}} \rd \xi_i \bigg|
\\
& \ \leq (w_{\max} L_{\mathcal{B} \to \mathcal{H}} )^{|\mathsf{e}(F)| - 1},
\end{aligned}
\end{equation*}
where $L_{\mathcal{B} \to \mathcal{H}}$ is the operator norm of the canonical embedding $\mathcal{B} \to \mathcal{H}$. Moreover, this function is of the form $f(\xi_{i'}) g(\xi_{j'})$, so we can use the $L^\infty \to L^1$ operator norm (or equivalently the cut norm, up to a constant $4$) to obtain
\begin{equation*}
\begin{aligned}
& \bigg| \int_{\I^{\mathsf{v}(F)}} \bigg( \prod_{(i,j) \in \mathsf{e}(F), (i,j) < (i',j')} \langle \bm{e}_{l_{i,j}}, \bm{w}_n(\xi_i,\xi_j) \rangle \bigg)
\\
& \hspace{1cm} \langle \bm{e}_{l_{i',j'}}, \bm{w}_n(\xi_{i'},\xi_{j'}) - \bm{w}(\xi_{i'},\xi_{j'}) \rangle
\\
& \hspace{1cm} \bigg( \prod_{(i,j) \in \mathsf{e}(F), (i',j') < (i,j)} \langle \bm{e}_{l_{i,j}}, \bm{w}(\xi_i,\xi_j) \rangle \bigg) \prod_{i \in \mathsf{v}(F)} \rd \xi_i \bigg|
\\
& \ \leq 4(w_{\max} L_{\mathcal{B} \to \mathcal{H}} )^{|\mathsf{e}(F)| - 1} \|\bm{w}_n- \bm{w}\|_{\square;\mathcal{H}}.
\end{aligned}
\end{equation*}
Apply the triangle inequality and let $n \to \infty$, we have
\begin{equation*}
\begin{aligned}
& \bigg\langle \bigotimes_{(i,j) \in \mathsf{e}(F)} \bm{e}_{l_{i,j}}, \ \bm{t}(F,\bm{w}_n) - \bm{t}(F,\bm{w}) \bigg\rangle
\\
& \ \leq 4 |\mathsf{e}(F)| (w_{\max} L_{\mathcal{B} \to \mathcal{H}} )^{|\mathsf{e}(F)| - 1} \|\bm{w}_n- \bm{w}\|_{\square;\mathcal{H}} \to 0,
\end{aligned}
\end{equation*}
for any basis element in $(\mathbb{R}^d)^{\otimes \mathsf{e}(F)}$.
Since it is a space of finite dimension,
\begin{equation*}
\begin{aligned}
\lim_{n \to \infty} \| \bm{t}(F,\bm{w}_n) - \bm{t}(F,\bm{w}) \|_{(\R^d)^{\otimes \mathsf{e}(F)}} = 0.
\end{aligned}
\end{equation*}

\hfill

Now, consider the general case of a compact embedding $\mathcal{B} \subset \mathcal{H}$.
By compactness, there exists a sequence of finite-dimensional subspaces $\mathcal{H}_d$, with $d \to \infty$, such that for any $\bm{u} \in \mathcal{B} \subset \mathcal{H}$, the projection $P_{\mathcal{H}_d}$ satisfies
\begin{equation*}
\begin{aligned}
\|P_{\mathcal{H}_d}\bm{u} - \bm{u}\|_{\mathcal{H}} \leq C^{\downarrow}_{\mathcal{B},\mathcal{H}}(d) \|\bm{u}\|_{\mathcal{B}}.
\end{aligned}
\end{equation*}
Consequently, when $\|\bm{w}\|_{L^\infty([0,1]^2; \mathcal{B})} \leq w_{\max}$,
\begin{equation*}
\begin{aligned}
\|P_{\mathcal{H}_d}\bm{w} - \bm{w}\|_{L^\infty([0,1]^2;\mathcal{H})} \leq C^{\downarrow}_{\mathcal{B},\mathcal{H}, w_{\max}}(d).
\end{aligned}
\end{equation*}
Using the previous triangle editing argument, for all $F \in \mathcal{G}$,
\begin{equation*}
\begin{aligned}
& \bm{t}(F,P_{\mathcal{H}_d} \bm{w}) - \bm{t}(F,\bm{w})
\\
& \ = \sum_{(i',j') \in \mathsf{e}(F)} \int_{\I^{\mathsf{v}(F)}} \bigg( \bigotimes_{(i,j) \in \mathsf{e}(F), (i,j) < (i',j')} P_{\mathcal{H}_d} \bm{w}(\xi_i,\xi_j) \bigg)
\\
& \hspace{3cm} \otimes \big[ P_{\mathcal{H}_d} \bm{w}(\xi_{i'},\xi_{j'}) - \bm{w}(\xi_{i'},\xi_{j'}) \big]
\\
& \hspace{3cm} \otimes \bigg( \bigotimes_{(i,j) \in \mathsf{e}(F), (i',j') < (i,j)} \bm{w}(\xi_i,\xi_j) \bigg) \prod_{i \in \mathsf{v}(F)} \rd \xi_i.
\end{aligned}
\end{equation*}
This time we have the stronger norm $L^\infty([0,1]^2; \mathcal{H})$ bound, so there is no need to take the dual pairing. We can directly conclude that
\begin{equation*}
\begin{aligned}
\| \bm{t}(F,P_{\mathcal{H}_d} \bm{w}) - \bm{t}(F,\bm{w}) \|_{\mathcal{H}^{\otimes \mathsf{e}(F)}} \leq C^{\downarrow}_{\mathcal{B},\mathcal{H}, w_{\max}, F}(d).
\end{aligned}
\end{equation*}
By contractivity, we have, as $n \to \infty$,
\begin{equation*}
\begin{aligned}
\delta_{\square;\mathcal{H}}(P_{\mathcal{H}_d} \bm{w}_n, P_{\mathcal{H}_d} \bm{w}) \leq \delta_{\square;\mathcal{H}}(\bm{w}_n,  \bm{w}) \to 0,
\end{aligned}
\end{equation*}
hence the convergence of the tensorized homomorphism densities projected onto finite-dimensional subspaces. By the a priori bound just established (up to a constant $2$ that we included in $C^{\downarrow}$),
\begin{equation*}
\begin{aligned}
\limsup_{n \to \infty} \| \bm{t}(F,\bm{w}_n) - \bm{t}(F,\bm{w}) \|_{\mathcal{H}^{\otimes \mathsf{e}(F)}} \leq C^{\downarrow}_{\mathcal{B},\mathcal{H}, w_{\max}, F}(d).
\end{aligned}
\end{equation*}
Taking the limit as $d \to \infty$, we conclude.
\end{proof}

\subsection{Inverse Counting Lemma}

Next, we extend the inverse counting lemma, which provides the other direction of Lemma~\ref{lem:main_counting}.

As before, we start with the case $\mathcal{H} = \mathbb{R}^d$. The convenience of this choice is obvious, as illustrated in the following lemma.
\begin{lem} \label{lem:norm_equivalence}
Let $\bm{w} \in L^\infty([0,1]^2; \R^d)$ and $\bm{e}_1,\dots,\bm{e}_d$ be an orthonormal basis of $\R^d$. Then
\begin{equation*}
\begin{aligned}
\|\bm{w}\|_{\square;\R^d} \leq \sum_{l \in [d]} \|\langle \bm{e}_l, \bm{w} \rangle\|_{\square} \leq d \|\bm{w}\|_{\square;\R^d}.
\end{aligned}
\end{equation*}
\end{lem}
\begin{proof} [Proof of Lemma~\ref{lem:norm_equivalence}]

The proof follows directly from the definitions and the bounds between the $\ell^1$ and $\ell^2$ norms in the finite-dimensional space $\mathbb{R}^d$.

\end{proof}

Following the approach of \cite{lovasz2012large}, we use several probabilistic techniques by interpreting $\xi \in [0,1]$ as a uniform random variable. The aim is not to reproduce rates comparable to those in the literature. To this end, we introduce the following notion.
\begin{defi}

If there exists a priori $\epsilon(k), r(k) \to 0$ as $k \to \infty$, such that for a sequence of random variables $Y_k$,
\begin{equation*}
\begin{aligned}
\Pb\big( |Y_k| \leq r(k) \big) \geq 1 - \epsilon(k),
\end{aligned}
\end{equation*}
we say that
\begin{equation*}
\begin{aligned}
Y_k \to 0
\end{aligned}
\end{equation*}
in probability with a priori rates $\epsilon(k)$ and $r(k)$.
\end{defi}
\noindent
This notion will be used to replace the explicit rate computations in \cite{lovasz2012large}, as the most important consideration (when we do not need to compute the exact rate) is the existence of some uniform rate of convergence.

The following definition of sampling is the starting point of the probabilistic argument.
\begin{defi}[\textbf{Sampling graphs from kernels}] \label{defi:sampling_kernel}

Let $\bm{w} \in L^\infty([0,1]^2; \mathcal{H})$, where $\mathcal{H}$ is a Hilbert space. For a tuple $\xi^{\oplus k} = (\xi_1, \dots, \xi_k) \in [0,1]^k$, define the $k \times k$ matrix $\bm{w}[\xi^{\oplus k}]$ with entries in $\mathcal{H}$ as
\begin{equation*}
\begin{aligned}
\bm{w}[\xi^{\oplus k}](i, j) \defeq \bm{w}(\xi_i, \xi_j), \quad \forall\, i, j = 1, \dots, k.
\end{aligned}
\end{equation*}
\end{defi}
\noindent
When considering $\xi^{\oplus k}$ as i.i.d. random variables, the sampling defined in Definition~\ref{defi:sampling_kernel} effectively approximates the kernel in terms of cut distance. For scalar-valued kernels $w \in L^\infty([0,1]^2;\R)$, the following result is commonly referred to as the First Sampling Lemma.
\begin{lem}[\textbf{First Sampling Lemma, scalar-valued}] \label{lem:first_sampling}
There exists a priori rates $\epsilon(k), r(k) \to 0$, such that for any kernel $w: [0,1]^2 \to [-1,1]$ and $\xi^{\oplus k}$ being uniformly random on $[0,1]^k$, as $k \to \infty$,
\begin{equation*}
\begin{aligned}
\|w[\xi^{\oplus k}]\|_{\square} - \|w\|_{\square} \to 0
\end{aligned}
\end{equation*}
in probability with such rates.
\end{lem}

\begin{proof} [Proof of Lemma~\ref{lem:first_sampling}]
We refer readers to Lemma~10.6 in \cite{lovasz2012large} for more details. Although the lemma there is stated specifically for symmetric kernels, the proof itself does not fundamentally rely on the assumption of symmetry and can be readily adapted to non-symmetric kernels.

\end{proof}

Next, we will use the technique of approximating kernels by random (oriented) simple graphs. Since the weights will be interpreted as probability densities of discrete random variables, it is convenient to assume non-negativity and that their summation is bounded by $1$. In what follows, we will complete the proof for $\bm{w}: [0,1]^2 \to [0,1/2d]^d$ and then extend the result to the general case.

\begin{defi} [\textbf{Random oriented simple graph}]

Let $\bm{H}$ be a $k \times k$ matrix with entries $\bm{H}(i,j) \in [0, 1/(2d)]^d$. Define a random graph where edge weights almost surely take values in the orthonormal basis $\{\bm{0}, \bm{e}_1, \dots, \bm{e}_d\} \subset \mathbb{R}^d$, as follows:
\begin{equation*}
\begin{aligned}
\mathbb{G}(\bm{H})(i,j) \defeq \sum_{l \in [d]} \bm{e}_l \mathbbm{1}_{[\frac{l-1}{2d},\frac{l-1}{2d} + \langle \bm{e}_{l}, \bm{H}(i,j) \rangle )}(z_{ij}), \quad \forall i,j = 1,\dots,k, \ i \neq j
\end{aligned}
\end{equation*}
where $(z_{ij})_{1 \leq i < j \leq k}$ are i.i.d. random variables on $[0,1]$, and $z_{ji} = 1 - z_{ij}$.

Define the random $k$-induced subgraph for a kernel $\bm{w}: [0,1]^2 \to [0,1/d]^d$ as
\begin{equation*}
\begin{aligned}
\mathbb{G}(k,\bm{w})(i,j) \defeq \mathbb{G}(\bm{w}[\xi^{\oplus k}])(i,j) = \sum_{l \in [d]} \bm{e}_l \mathbbm{1}_{[\frac{l-1}{d},\frac{l-1}{d} + \langle \bm{e}_{l}, \bm{w}(\xi_i,\xi_j) \rangle )}(z_{ij}),
\end{aligned}
\end{equation*}
where $(\xi_i)_{1 \leq i \leq k}$ and $(z_{ij})_{1 \leq i < j \leq k}$ are i.i.d. random variables on $[0,1]$, and $z_{ji} = 1 - z_{ij}$.
\end{defi}
\noindent
By our definition, at most one of $\mathbb{G}(\bm{H})(i,j)$ and $\mathbb{G}(\bm{H})(j,i)$ is nonzero, and there is at most one nonzero entry in $\R^d$.

\begin{lem} \label{lem:weight_simple_multi}
There exists a priori rates $\epsilon(k), r(k) \to 0$, such that for any $k \times k$ matrix $\bm{H}_{k}$ with entry values in $[0,1/2d]^d$, as $k \to \infty$,
\begin{equation*}
\begin{aligned}
\|\mathbb{G}(\bm{H}_k) - \bm{H}_k\|_{\square;\R^d} \to 0
\end{aligned}
\end{equation*}
in probability with such rates.
\end{lem}
\noindent
A quick consequence of the lemma is the following corollary:
\begin{cor} \label{cor:weight_simple_multi}
There exists a priori rates $\epsilon(k), r(k) \to 0$, such that for any kernel $\bm{w}: [0,1]^2 \to [0,1/2d]^d$, as $k \to \infty$,
\begin{equation*}
\begin{aligned}
\|\mathbb{G}(k,\bm{w}) - \bm{w}[\xi^{\oplus k}]\|_{\square;\R^d} \to 0
\end{aligned}
\end{equation*}
in probability with such rates.

\end{cor}
\begin{proof}[Proof of Lemma~\ref{lem:weight_simple_multi}]
It suffices to consider $\bm{H}_k$ with zero diagonal entries. Indeed, let $\bm{H}_k^{\setminus \Delta}$ denote the matrix with zero diagonals and the same off-diagonal entries as $\bm{H}_k$, then
\begin{equation*}
\begin{aligned}
\|\bm{H}_k^{\setminus \Delta} - \bm{H}_k\|_{\square;\R^d} \leq \frac{1}{k},
\end{aligned}
\end{equation*}
which converges to zero.

Now, let $S, T \subset \{1, \dots, k\}$. By definition,
\begin{equation*}
\begin{aligned}
\|\mathbb{G}(\bm{H}_k) - \bm{H}_k\|_{\square;\R^d} = \max_{S,T} \bigg\| \sum_{i \in S, j \in T} \frac{\mathbb{G}(\bm{H}_k)(i,j) - \bm{H}_k(i,j)}{k^2} \bigg\|_{\R^d}.
\end{aligned}
\end{equation*}
For fixed $S, T$ and $\bm{e}_l$ chosen from $\bm{e}_1,\dots,\bm{e}_d$ the orthonormal basis of $\R^d$, notice that
\begin{equation*}
\begin{aligned}
 \sum_{i \in S, j \in T} \frac{\langle \bm{e}_l, \mathbb{G}(\bm{H}_k)(i,j) - \bm{H}_k(i,j) \rangle}{k^2}
\end{aligned}
\end{equation*}
is the sum of at most $k(k-1)$ independent random variables (the pairs $(i,j)$ and $(j,i)$ are correlated, so we count each such pair as a single variable). The oscillation caused by each variable is $1/k^2$.

By applying Azuma-Bernstein-Chernoff-Hoeffding inequality,
\begin{equation*}
\begin{aligned}
\Pb\bigg(
\sum_{i \in S, j \in T} \frac{\langle \bm{e}_l, \mathbb{G}(\bm{H}_k)(i,j) - \bm{H}_k(i,j) \rangle}{k^2} \leq r \bigg) \geq 1 - e^{-k^2 r^2 /2}.
\end{aligned}
\end{equation*}
Note that there are at most $4^k$ choices of $S, T \subset \{1, \dots, k\}$ and $d$ directions in $\R^d$. Therefore,
\begin{equation*}
\begin{aligned}
\Pb\bigg(
\|\mathbb{G}(\bm{H}_k) - \bm{H}_k\|_{\square;\R^d} \leq d r \bigg) \geq 1 - 4^k d e^{-k^2 r^2 /2}.
\end{aligned}
\end{equation*}
By optimizing over $r$ for fixed $k$ and then passing the limit $k \to \infty$, we obtain convergence in probability with a priori rates.
\end{proof}

We next extend the Second Sampling Lemma to apply to our new definition.
\begin{lem}[\textbf{Second Sampling Lemma, vector-valued}] \label{lem:second_sampling}
There exists a priori rates $\epsilon(k), r(k)$, such that for any kernel $\bm{w}: [0,1]^2 \to [0,1/2d]^d$, as $k \to \infty$,
\begin{equation*}
\begin{aligned}
\delta_{\square;\R^d}(\mathbb{G}(k,\bm{w}),\bm{w}) \to 0
\end{aligned}
\end{equation*}
in probability with such rates.
\end{lem}
\begin{proof} [Proof of Lemma~\ref{lem:second_sampling}]
Firstly, notice that
\begin{equation*}
\begin{aligned}
\delta_{\square;\R^d}(\mathbb{G}(k,\bm{w}),\bm{w}) \leq \|\mathbb{G}(k,\bm{w}) - \bm{w}[\xi^{\oplus k}]\|_{\square;\R^d} + \delta_{\square;\R^d}(\bm{w}[\xi^{\oplus k}],\bm{w}).
\end{aligned}
\end{equation*}
By Corollary~\ref{cor:weight_simple_multi}, we have $\|\mathbb{G}(k, \bm{w})- \bm{w}[\xi^{\oplus k}]\|_{\square;\R^d} \to 0$ in probability with a priori rates. Therefore, it suffices to investigate $\delta_{\square;\mathbb{R}^d}(\bm{w}[\xi^{\oplus k}], \bm{w})$.

Let $m = \lceil k^{1/4} \rceil$. By Lemma~\ref{lem:weak_regularity_lemma}, there exists an equipartition $P = \{V_1,\dots,V_m\}$ of $[0,1]$ into $m$ sets such that
\begin{equation*}
\begin{aligned}
\|\bm{w}_P - \bm{w}\|_{\square;\R^d} \leq \frac{8}{\sqrt{\log k}}.
\end{aligned}
\end{equation*}
Then
\begin{equation} \label{eqn:unlabel_sampling_decompose}
\begin{aligned}
&\delta_{\square;\R^d}(\bm{w}[\xi^{\oplus k}],\bm{w})
\\
& \ \leq \delta_{\square;\R^d}(\bm{w}_P[\xi^{\oplus k}],\bm{w}_P) + \|\bm{w}_P - \bm{w}\|_{\square;\R^d} + \|\bm{w}_P[\xi^{\oplus k}] - \bm{w}[\xi^{\oplus k}]\|_{\square;\R^d},
\end{aligned}
\end{equation}
where the middle term $\|\bm{w}_P- \bm{w}\|_{\square;\R^d} \to 0$ by our construction of $P$. Moreover, the last term can be expanded as
\begin{equation*}
\begin{aligned}
& \big\|\bm{w}_P[\xi^{\oplus k}] - \bm{w}[\xi^{\oplus k}]\big\|_{\square;\R^d}
\\
& \ \leq \sum_{l=1}^d \big\| \langle \bm{e}_l , \bm{w}_P \rangle[\xi^{\oplus k}] - \langle \bm{e}_l , \bm{w} \rangle[\xi^{\oplus k}] \big\|_{\square;\R^d}
\\
& \ \leq \sum_{l=1}^d \bigg[ \big\| \langle \bm{e}_l , \bm{w}_P \rangle - \langle \bm{e}_l , \bm{w} \rangle\big\|_{\square;\R^d}
\\
& \hspace{1.5cm} + \Big( \big\| \langle \bm{e}_l , \bm{w}_P \rangle[\xi^{\oplus k}] - \langle \bm{e}_l , \bm{w} \rangle[\xi^{\oplus k}] \big\|_{\square;\R^d} - \big\| \langle \bm{e}_l , \bm{w}_P \rangle - \langle \bm{e}_l , \bm{w} \rangle \big\|_{\square;\R^d} \Big) \bigg].
\end{aligned}
\end{equation*}
By our construction of $P$, we have $\| \langle \bm{e}_l , \bm{w}_P \rangle - \langle \bm{e}_l , \bm{w} \rangle \|_{\square;\mathbb{R}^d} \to 0$ in probability with a priori rates. In addition,
\begin{equation*}
\begin{aligned}
\big\| \langle \bm{e}_l , \bm{w}_P \rangle[\xi^{\oplus k}] - \langle \bm{e}_l , \bm{w} \rangle[\xi^{\oplus k}] \big\|_{\square;\R^d} - \big\| \langle \bm{e}_l , \bm{w}_P \rangle - \langle \bm{e}_l , \bm{w} \rangle \big\|_{\square;\R^d} \to 0
\end{aligned}
\end{equation*}
in probability with a priori rates also holds, by Lemma~\ref{lem:first_sampling}. (Applying Lemma~\ref{lem:first_sampling} is the main reason for decomposing into components.)

It now remains to check the first term in \eqref{eqn:unlabel_sampling_decompose}, namely $\delta_{\square;\mathbb{R}^d}(\bm{w}_P[\xi^{\oplus k}], \bm{w}_P)$.
Taking the piecewise projection of $\bm{w}_P[\xi^{\oplus k}]$, we see that $\bm{w}_P[\xi^{\oplus k}]$ and $\bm{w}_P$, as kernels on $[0,1]$, are almost identical: both are step functions with $m = \lceil k^{1/4} \rceil$, having the same values on each step. The only difference is that the measure of the $i$-th step in $\bm{w}_P$ is $1/m$, whereas the measure of the $i$-th step in $\bm{w}_P[\xi^{\oplus k}]$ is
\begin{equation*}
\begin{aligned}
|V_i \cap \{\xi_1,\dots,\xi_k\}| / k,
\end{aligned}
\end{equation*}
which is expected to be close to $1/m$. Take
\begin{equation*}
\begin{aligned}
r_i \defeq 1/m - |V_i \cap \{\xi_1,\dots,\xi_k\}| / k.
\end{aligned}
\end{equation*}
It is straightforward to check that
\begin{equation*}
\begin{aligned}
\delta_{\square;\R^d}(\bm{w}_P[\xi^{\oplus k}],\bm{w}_P) \leq 4\sqrt{d} \sum_{i=1}^m |r_i|.
\end{aligned}
\end{equation*}
By applying Azuma-Bernstein-Chernoff-Hoeffding inequality,
\begin{equation*}
\begin{aligned}
\Pb\bigg(|r_i| \leq \frac{1}{m \sqrt{\log k}} \bigg) \geq 1 - 2e^{-2k/m^2\log k},
\end{aligned}
\end{equation*}
and thus
\begin{equation*}
\begin{aligned}
\Pb\bigg( \delta_{\square;\R^d}(\bm{w}_P[\xi^{\oplus k}],\bm{w}_P) \leq \frac{4\sqrt{d}}{\sqrt{\log k}} \bigg) \geq 1 - 2me^{-2k/m^2\log k}.
\end{aligned}
\end{equation*}
Recalling that $m = \lceil k^{1/4} \rceil$, we obtain convergence in probability with a priori rates.

Combining the analysis of all terms, we conclude convergence in probability with a priori rates (potentially at a slower rate, though this is not a concern here).
\end{proof}

The convergence in probability described above can be used to establish deterministic results through the inclusion-exclusion principle.
\begin{lem} \label{lem:inclusion_exclusion_deterministic}
Let $\epsilon(k), r(k)$ be the a priori rates in Lemma~\ref{lem:second_sampling}. If
\begin{equation*}
\begin{aligned}
\sum_{\bm{F}} \Big| \Pb(\mathbb{G}(k,\bm{u}) = \bm{F}) - \Pb(\mathbb{G}(k,\bm{w}) = \bm{F}) \Big| < 1 - 2\epsilon(k),
\end{aligned}
\end{equation*}
where $\bm{F}$ ranges over all graphs with $k$ nodes and edge values in $\{\bm{0}, \bm{e}_1, \dots, \bm{e}_d\} \subset \mathbb{R}^d$, such that for each pair $i \neq j$, at most one of $\bm{F}(i, j)$ and $\bm{F}(j, i)$ is nonzero, and there is at most one nonzero entry, then
\begin{equation*}
\begin{aligned}
\delta_{\square;\R^d}(\bm{u},\bm{w}) \leq 2r(k).
\end{aligned}
\end{equation*}

\end{lem}

\begin{proof}[Proof of Lemma~\ref{lem:inclusion_exclusion_deterministic}]
The assumption implies that we can couple $\mathbb{G}(k, \bm{u})$ and $\mathbb{G}(k, \bm{w})$ such that
\begin{equation*}
\begin{aligned}
\Pb\bigg( \mathbb{G}(k,\bm{u}) = \mathbb{G}(k,\bm{w}) \bigg) > 2\epsilon(k)
\end{aligned}
\end{equation*}
By Lemma~\ref{lem:second_sampling},
\begin{equation*}
\begin{aligned}
\Pb\bigg(
\delta_{\square;\R^d}(\mathbb{G}(k,\bm{u}),\bm{u}) + \delta_{\square;\R^d}(\mathbb{G}(k,\bm{w}),\bm{w}) \leq r(k) \bigg) \geq 1 - 2\epsilon(k).
\end{aligned}
\end{equation*}
Thus, there exists a positive probability that these events occur together, which implies
\begin{equation*}
\begin{aligned}
\delta_{\square;\R^d}(\bm{u},\bm{w}) \leq \delta_{\square;\R^d}(\mathbb{G}(k,\bm{u}),\bm{u}) + \delta_{\square;\R^d}(\mathbb{G}(k,\bm{w}),\bm{w}) \leq 2r_k.
\end{aligned}
\end{equation*}
\end{proof}

We are now ready to prove the Inverse Counting Lemma for kernels $\bm{w}: [0,1]^2 \to [0, 1/(2d)]^d$.
\begin{lem}[\textbf{Inverse Counting Lemma, positive finite}] \label{lem:inverse_counting_lemma_finite}
If $\{\bm{w}_n\}_{n=1}^\infty \cup \{\bm{w}\}$ are $[0,1]^2 \to [0, 1/(2d)]^d$ kernels and have all tensorized homomorphism densities converging, i.e. for all $F \in \mathcal{G}$,
\begin{equation*}
\begin{aligned}
\lim_{n \to \infty} \| \bm{t}(F,\bm{w}_n) - \bm{t}(F,\bm{w}) \|_{(\R^d)^{\otimes \mathsf{e}(F)}} = 0,
\end{aligned}
\end{equation*}
then
\begin{equation*}
\begin{aligned}
\lim_{n \to \infty} \delta_{\square;\R^d}( \bm{w}_n, \bm{w}) = 0.
\end{aligned}
\end{equation*}
\end{lem}
\begin{proof}[Proof of Lemma~\ref{lem:inverse_counting_lemma_finite}]

Let $k$ be an arbitrary integer and let $\bm{F}$ be a graph with $k$ nodes and edge values in $\{\bm{0}, \bm{e}_1, \dots, \bm{e}_d\} \subset \mathbb{R}^d$, such that for each pair $i \neq j$, at most one of $\bm{F}(i, j)$ and $\bm{F}(j, i)$ is nonzero, and there is at most one nonzero entry. Then
\begin{equation*}
\begin{aligned}
& \Pb(\mathbb{G}(k,\bm{w}) = \bm{F})
\\
& \ = \int_{[0,1]^{\mathsf{v}(F)}} \prod_{(i,j) \in \mathsf{v}(F), \bm{F}(i,j) = \bm{F}(j,i) = \bm{0}} \Big(1 - \langle \bm{e}_1+ \dots+ \bm{e}_d, \bm{w}(\xi_i,\xi_j) + \bm{w}(\xi_j,\xi_i) \rangle_{\R^d} \Big)
\\
& \hspace{5cm} \prod_{(i,j) \in \mathsf{v}(F), \bm{F}(i,j) \neq \bm{0}} \langle \bm{F}(i,j), \bm{w}(\xi_i,\xi_j) \rangle_{\R^d} \prod_{i \in \mathsf{v}(F)} \rd \xi_i.
\end{aligned}
\end{equation*}
Expanding the products by the distributive property, it is straightforward to verify that each term should be a linear functional applied to some $\bm{t}(F', \bm{w})$, $F' \in \mathcal{G}$ with $|\mathsf{v}(F')| = |\mathsf{v}(F)| = k$. Hence, for any fixed $k \geq 1$, we have that
\begin{equation*}
\begin{aligned}
\lim_{n \to \infty} \Big| \Pb(\mathbb{G}(k,\bm{w}_n) = \bm{F}) - \Pb(\mathbb{G}(k,\bm{w}) = \bm{F}) \Big| = 0
\end{aligned}
\end{equation*}
for all possible choices of $\bm{F}$. Therefore, by definition,
\begin{equation*}
\begin{aligned}
\lim_{n \to \infty} \sum_{\bm{F}} \Big| \Pb(\mathbb{G}(k,\bm{w}_n) = \bm{F}) - \Pb(\mathbb{G}(k,\bm{w}) = \bm{F}) \Big|
\end{aligned}
\end{equation*}
By Lemma~\ref{lem:inclusion_exclusion_deterministic}, we obtain
\begin{equation*}
\begin{aligned}
\limsup_{n \to \infty} \delta_{\square;\R^d}(\bm{w}_n,\bm{w}) \leq 2r(k),
\end{aligned}
\end{equation*}
where $r(k)$ is the a priori rates in Lemma~\ref{lem:second_sampling}.
Taking the limit as $k \to \infty$, we conclude the lemma.

\end{proof}

Finally, we proceed with the Inverse Counting Lemma on $L^\infty([0,1]^2; \mathcal{B})$.

\begin{lem}[\textbf{Inverse Counting lemma on compact subspace}] \label{lem:inverse_counting_lemma_subspace}
Let $\mathcal{B}$ be a Banach space compactly embedded into a separable Hilbert space $\mathcal{H}$.
If $\{\bm{w}_n\}_{n=1}^\infty \cup \{\bm{w}\} \subset L^\infty([0,1]^2; \mathcal{B})$ satisfy uniform bound that
\begin{equation*}
\begin{aligned}
\sup_n \|\bm{w}_n\|_{L^\infty([0,1]^2; \mathcal{B})} < w_{\max} < \infty
\end{aligned}
\end{equation*}
and have all tensorized homomorphism densities converging, i.e. for all $F \in \mathcal{G}$,
\begin{equation*}
\begin{aligned}
\lim_{n \to \infty} \| \bm{t}(F,\bm{w}_n) - \bm{t}(F,\bm{w}) \|_{\mathcal{H}^{\otimes \mathsf{e}(F)}} = 0,
\end{aligned}
\end{equation*}
then
\begin{equation*}
\begin{aligned}
\lim_{n \to \infty} \delta_{\square;\mathcal{H}}( \bm{w}_n, \bm{w}) = 0.
\end{aligned}
\end{equation*}

\end{lem}

\begin{proof} [Proof of Lemma~\ref{lem:inverse_counting_lemma_subspace}]

We begin by assuming $\mathcal{B} = \mathcal{H} = \mathbb{R}^d$. The strategy is to make Lemma~\ref{lem:inverse_counting_lemma_finite} applicable. Define
\begin{equation*}
\begin{aligned}
\bm{u}_n(\xi,\zeta) &= \frac{\bm{w}_n(\xi,\zeta)}{4d w_{\max}} + \frac{1}{4d} \sum_{l \in [d]} \bm{e}_l, && \forall (\xi,\zeta) \in [0,1]^2, \ \forall n \in \N
\\
\bm{u}(\xi,\zeta) &= \frac{\bm{w}(\xi,\zeta)}{4d w_{\max}} + \frac{1}{4d} \sum_{l \in [d]} \bm{e}_l, && \forall (\xi,\zeta) \in [0,1]^2.
\end{aligned}
\end{equation*}
It is straightforward to see that $\bm{u}_n$ and $\bm{u}$ take values in $[0, 1/(2d)]^d$, and by checking the definition,
\begin{equation*}
\begin{aligned}
\delta_{\square;\R^d}( \bm{w}_n, \bm{w}) = \frac{1}{4d w_{\max}} \delta_{\square;\R^d}( \bm{u}_n, \bm{u}).
\end{aligned}
\end{equation*}
To verify the convergence $\bm{t}(F, \bm{u}_n) \to \bm{t}(F, \bm{u})$, note that for any $F \in \mathcal{G}$ and any basis element $\bigotimes_{(i,j) \in \mathsf{e}(F)} \bm{e}_{l_{i,j}}$,
\begin{equation*}
\begin{aligned}
& \bigg\langle \bigotimes_{(i,j) \in \mathsf{e}(F)} \bm{e}_{l_{i,j}}, \ \bm{t}(F,\bm{u}_n) \bigg\rangle
\\
& \ = \int_{\I^{\mathsf{v}(F)}} \prod_{(i,j) \in \mathsf{e}(F)} \big\langle \bm{e}_{l_{i,j}}, \ \bm{u}_n(\xi_i,\xi_j) \big\rangle \prod_{i \in \mathsf{v}(F)} \rd \xi_i
\\
& \ = \int_{\I^{\mathsf{v}(F)}} \prod_{(i,j) \in \mathsf{e}(F)} \bigg( \frac{\big\langle \bm{e}_{l_{i,j}}, \ \bm{w}_n(\xi_i,\xi_j) \big\rangle}{4d w_{\max}} + \frac{1}{4d} \bigg) \prod_{i \in \mathsf{v}(F)} \rd \xi_i.
\end{aligned}
\end{equation*}
Expanding the products by the distributive property, it is straightforward to verify that each term is a linear functional applied to some $\bm{t}(F', \bm{w}_n)$, where $F' \in \mathcal{G}$ with $|\mathsf{v}(F')| = |\mathsf{v}(F)|$. Hence, $\bm{t}(F, \bm{w}_n) - \bm{t}(F, \bm{w})$ implies $\bm{t}(F, \bm{u}_n) - \bm{t}(F, \bm{u})$.

Applying Lemma~\ref{lem:inverse_counting_lemma_finite}, we obtain $\lim_{n \to \infty} \delta_{\square; \R^d}(\bm{u}_n, \bm{u}) = 0$, which in turn implies $\lim_{n \to \infty} \delta_{\square; \R^d}(\bm{w}_n, \bm{w}) = 0$.

\hfill

Now, consider the general case of a compact embedding $\mathcal{B} \subset \mathcal{H}$.
By compactness, there exists a sequence of finite-dimensional subspaces $\mathcal{H}_d$, with $d \to \infty$, such that for any $\bm{u} \in \mathcal{B} \subset \mathcal{H}$, the projection $P_{\mathcal{H}_d}$ satisfies
\begin{equation*}
\begin{aligned}
\|P_{\mathcal{H}_d}\bm{u} - \bm{u}\|_{\mathcal{H}} \leq C^{\downarrow}_{\mathcal{B},\mathcal{H}}(d) \|\bm{u}\|_{\mathcal{B}}.
\end{aligned}
\end{equation*}
Consequently, when $\|\bm{w}\|_{L^\infty([0,1]^2; \mathcal{B})} \leq w_{\max}$,
\begin{equation*}
\begin{aligned}
\delta_{\square;\mathcal{H}}( P_{\mathcal{H}_d} \bm{w}, \bm{w}) \leq C^{\downarrow}_{\mathcal{B},\mathcal{H}, w_{\max}}(d).
\end{aligned}
\end{equation*}
By contractivity
\begin{equation*}
\begin{aligned}
\| \bm{t}(F,P_{\mathcal{H}_d}\bm{w}_n) - \bm{t}(F,P_{\mathcal{H}_d}\bm{w}) \|_{\mathcal{H}^{\otimes \mathsf{e}(F)}} & = \| P_{\mathcal{H}_d}^{\otimes \mathsf{e}(F)}(\bm{t}(F,\bm{w}_n) - \bm{t}(F,\bm{w})) \|_{\mathcal{H}^{\otimes \mathsf{e}(F)}}
\\
& \leq \| \bm{t}(F,\bm{w}_n) - \bm{t}(F,\bm{w}) \|_{\mathcal{H}^{\otimes \mathsf{e}(F)}}.
\end{aligned}
\end{equation*}
This suggests the convergence of the tensorized homomorphism densities, hence the convergence of the unlabeled distance, when projected onto finite-dimensional subspaces.
Notice that 
\begin{equation*}
\begin{aligned}
\delta_{\square;\mathcal{H}}( \bm{w}_n, \bm{w}) &\leq \delta_{\square;\mathcal{H}_d}( P_{\mathcal{H}_d} \bm{w}_n, P_{\mathcal{H}_d} \bm{w}) + \delta_{\square;\mathcal{H}}( P_{\mathcal{H}_d} \bm{w}_n, \bm{w}_n) + \delta_{\square;\mathcal{H}}( P_{\mathcal{H}_d} \bm{w}, \bm{w})
\\
& \leq \delta_{\square;\mathcal{H}_d}( P_{\mathcal{H}_d} \bm{w}_n, P_{\mathcal{H}_d} \bm{w}) + C^{\downarrow}_{\mathcal{B},\mathcal{H}, w_{\max}}(d).
\end{aligned}
\end{equation*}
We conclude by passing the limit first to $n$ and then to $d$.

\end{proof}

We conclude this section by proving Lemma~\ref{lem:main_counting}.

\begin{proof}[Proof of Lemma~\ref{lem:main_counting}]
The result follows directly by combining Lemma~\ref{lem:counting_lemma} and Lemma~\ref{lem:inverse_counting_lemma_subspace}.
\end{proof}

\section{Fractional isomorphism and tree homomorphism densities} \label{sec:graph_2}

The objective of this section is to prove Lemma~\ref{lem:main_tree_counting}. We begin with the finite-dimensional case $\mathbb{R}^d$, spanned by the orthonormal basis $\bm{e}_1, \dots, \bm{e}_d$, and follow the roadmap outlined in \cite{grebik2022fractional}. We then generalize the result to $\mathcal{B} \subset \mathcal{H}$.

As in Appendix~\ref{sec:graph_1}, probabilistic techniques are used. However, while the primary focus in Appendix~\ref{sec:graph_1} was on martingales and their concentration properties, the main tools in this section are Borel algebras and conditional expectations. In this section, we will explicitly specify the Borel algebra and the measure when discussing standard probability spaces. Accordingly, we denote the measures by $\mu$ and $\nu$. Note that this usage of $\mu$ and $\nu$ differs from their usage elsewhere in the article.

\subsection{Invariant subspace and sub-$\sigma$-algebra}
We begin with a critical lemma in the $L^2$ space, which we will later use to construct sub-$\sigma$-algebras on standard probability space $\I$. It is worth noting that this lemma exists in \cite{grebik2022fractional}, but only in the context of symmetric kernels. This extension to the general case is essential for our argument, although it is neither surprising nor particularly difficult.

\begin{lem} \label{lem:L2_orthogonal}
Let $(\I, \mathscr{B}, \mu)$ be a standard probability space. Let $\bm{w} \in L^\infty(\I \times \I; \R^d)$ and let $V \subseteq L^2(\I)$ be a closed linear subspace. Then the following are equivalent:

\begin{itemize}
  \item For any $\bm{e} \in \R^d$, $V$ is $T_{\langle \bm{e}, \bm{w} \rangle}$-invariant and $T_{\langle \bm{e}, \bm{w}^T \rangle}$-invariant.
  \item For any $\bm{e} \in \R^d$, $T_{\langle \bm{e}, \bm{w} \rangle}$ and $T_{\langle \bm{e}, \bm{w}^T \rangle}$ commute with the projection $P_V$.
\end{itemize}

\end{lem}
\begin{proof} [Proof of Lemma~\ref{lem:L2_orthogonal}]
It suffices to prove the case of scalar-valued $w \in L^\infty(\I \times \I)$, as the general case can be handled by applying the result for any $\bm{e} \in \R^d$.

Let us begin by considering $w$ symmetric. In this case, $T_w$ is a self-adjoint operator, and the result follows as a standard application of the compact self-adjoint operator theory.

For the non-symmetric case, define the direct sum space $V \oplus V \subseteq L^2(\I) \oplus L^2(\I)$. It is straightforward to verify that the projection $P_{V \oplus V}$ satisfies $P_{V \oplus V} = P_V \oplus P_V$.
If $P_V$ commutes with both $T_w$ and $T_{w^T}$, then the following identity is true: 
\begin{equation} \label{eqn:symmetrize_commute}
\begin{aligned}
&
\begin{pmatrix}
0 & T_{w^T}
\\
T_w & 0
\end{pmatrix}
\begin{pmatrix}
P_V & 0
\\
0 & P_V
\end{pmatrix}
=
\begin{pmatrix}
0 & T_{w^T} \circ P_V
\\
T_w \circ P_V & 0
\end{pmatrix}
\\ \\
& \ =
\begin{pmatrix}
0 & P_V \circ T_{w^T}
\\
P_V \circ T_w & 0
\end{pmatrix}
=
\begin{pmatrix}
P_V & 0
\\
0 & P_V
\end{pmatrix}
\begin{pmatrix}
0 & T_{w^T}
\\
T_w & 0
\end{pmatrix}.
\end{aligned}
\end{equation}
By the symmetric case, this implies the invariance of $V \oplus V$ under the off-diagonal block matrix representation of $T_w$ and $T_{w^T}$, and hence the invariance of $V$ under each of $T_w$ and $T_{w^T}$.

Conversely, if $V$ is invariant under both $T_w$ and $T_{w^T}$, then $V \oplus V$ is invariant under the off-diagonal block matrix representation of $T_w$ and $T_{w^T}$. This gives \eqref{eqn:symmetrize_commute} and ensures that $P_V$ commutes with both $T_w$ and $T_{w^T}$.

\end{proof}

Next, we define the $\bm{w}$-invariant sub-$\sigma$-algebra. The motivation, heuristically, is to group points that cannot be distinguished through operations involving $\bm{w}$ or $\bm{w}^T$. For further discussion, see \cite{grebik2022fractional}.

\begin{defi} \label{defi:w_invariant_algebra}
Let $(\I, \mathscr{B}, \mu)$ be a standard probability space, and let $\bm{w} \in L^\infty(\I \times \I; \mathbb{R}^d)$. A sub-$\sigma$-algebra $\mathscr{C} \subset \mathscr{B}$ is said to be $\bm{w}$-invariant if, for any $\mathscr{C}$-measurable function $f \in L^2(\I, \mathscr{C})$ and any $\bm{e} \in \mathbb{R}^d$, both $T_{\langle \bm{e}, \bm{w} \rangle}(f)$ and $T_{\langle \bm{e}, \bm{w}^T \rangle}(f)$ belong to $L^2(\I, \mathscr{C})$. That is, these transforms are not only $\mathscr{B}$-measurable but also $\mathscr{C}$-measurable.
\end{defi}

\noindent
The following supplementary definition ensures the consistency of the concept of null sets:
\begin{defi} \label{defi:relative_complete}
We say that $\mathscr{C} \subseteq \mathscr{B}$ is a relatively complete sub-$\sigma$-algebra of $\mathscr{B}$ if it is a sub-$\sigma$-algebra and if $Z \in \mathscr{C}$ whenever there exists $Z_0 \in \mathscr{C}$ such that $\mu(Z \setminus Z_0) = \mu(Z_0 \setminus Z) = 0$.
\end{defi}

Next, we define canonical projections and injections in the presence of a relatively complete sub-$\sigma$-algebra, which are closely related to conditional expectation.

\begin{defi}
Let $\mathscr{C}$ be a relatively complete sub-$\sigma$-algebra.
Define the canonical injection and projection
\begin{equation*}
\begin{aligned}
I_{\mathscr{C}} &: L^2(\I/\mathscr{C}) \to L^2(\I), \\
P_{\mathscr{C}} &: L^2(\I) \to L^2(\I/\mathscr{C}).
\end{aligned}
\end{equation*}
\end{defi}

\noindent
It follows that $(P_{\mathscr{C}} \circ I_{\mathscr{C}})$ is the identity map on $L^2(\I / \mathscr{C})$, and $(I_{\mathscr{C}} \circ P_{\mathscr{C}})$ is the conditional expectation operator $\E(\cdot | \mathscr{C})$.

These operations can be applied to $\bm{w}$ when it is viewed as a Lebesgue function on $\I \times \I$.

\begin{defi}
Let $\mathscr{C}$ be a $\bm{w}$-invariant and relatively complete sub-$\sigma$-algebra.
Define
\begin{equation*}
\begin{aligned}
\bm{w}_{\mathscr{C}} &\defeq \E(\bm{w} | \mathscr{C} \times \mathscr{C}) \in L^\infty(\I \times \I; \mathbb{R}^d), \\
\bm{w}/\mathscr{C} &\defeq (\bm{w} \circ I_{\mathscr{C} \times \mathscr{C}}) \in L^\infty(\I/\mathscr{C} \times \I/\mathscr{C}; \mathbb{R}^d).
\end{aligned}
\end{equation*}
\end{defi}
\noindent
However, we do not need to take the conditional expectation on both dimensions. This is stated in the following proposition.
\begin{prop} \label{prop:half_averaging}
Let $\mathscr{C}$ be a $\bm{w}$-invariant and relatively complete sub-$\sigma$-algebra.
The following holds:
\begin{equation*}
\begin{aligned}
\bm{w}_{\mathscr{C}} = \E(\bm{w} | \mathscr{C} \times \mathscr{C}) = \E(\bm{w} | \mathscr{B} \times \mathscr{C}).
\end{aligned}
\end{equation*}

\end{prop}
\begin{proof}[Proof of Proposition~\ref{prop:half_averaging}]

The proof is a standard application of conditional expectation in probability theory. We refer to the discussion around Claim~5.7 in \cite{grebik2022fractional} for details.

\end{proof}

The next proposition states that taking the conditional expectation with respect to a $\bm{w}$-invariant and relatively complete sub-$\sigma$-algebra preserves certain properties of $\bm{w}$.

\begin{prop} \label{prop:L2_orthogonal}
Let $(\I, \mathscr{B}, \mu)$ be a standard probability space, and let $\bm{w} \in L^\infty(\I \times \I; \mathbb{R}^d)$. Let $\mathscr{C}$ be a $\bm{w}$-invariant and relatively complete sub-$\sigma$-algebra of $\mathscr{B}$. Then:
\begin{enumerate}
  \item $\delta_{\square;\R^d}(\bm{w}_{\mathscr{C}}, \bm{w}/\mathscr{C}) = 0$.
  \item $T_{\langle \bm{e}, \bm{w}/\mathscr{C} \rangle} \circ P_{\mathscr{C}} = P_{\mathscr{C}} \circ T_{\langle \bm{e}, \bm{w} \rangle}$ and $T_{\langle \bm{e}, \bm{w}^T/\mathscr{C} \rangle} \circ P_{\mathscr{C}} = P_{\mathscr{C}} \circ T_{\langle \bm{e}, \bm{w}^T \rangle}$ for any $\bm{e} \in \mathbb{R}^d$.
\end{enumerate}
\end{prop}
\noindent
Aside from adjustments for the multi-dimension, our proof of the proposition essentially replicates the proof of Proposition~5.9 in \cite{grebik2022fractional}.
\begin{proof}[Proof of Proposition~\ref{prop:L2_orthogonal}]
To verify (1), we proceed directly by definition. 

For (2), it is straightforward to check that, for example, $T_{\langle \bm{e}, \bm{w}/\mathscr{C} \rangle} \circ P_{\mathscr{C}} = P_{\mathscr{C}} \circ T_{\langle \bm{e}, \bm{w}_{\mathscr{C}} \rangle}$. By Lemma~\ref{lem:L2_orthogonal}, the $\bm{w}$-invariance of $\mathscr{C}$ implies that the projection operator $\E(\cdot|\mathscr{C}): L^2(\I) \to L^2(\I)$ commutes with $T_{\langle \bm{e}, \bm{w} \rangle}$, so
\begin{equation*}
\begin{aligned}
T_{\langle \bm{e}, \bm{w}/\mathscr{C} \rangle} \circ P_{\mathscr{C}} &= P_{\mathscr{C}} \circ T_{\langle \bm{e}, \bm{w}_{\mathscr{C}} \rangle}
\\
&= P_{\mathscr{C}} \circ T_{\langle \bm{e}, \bm{w} \rangle} \circ \E(\cdot|\mathscr{C})
\\
&= P_{\mathscr{C}} \circ \E(\cdot|\mathscr{C}) \circ T_{\langle \bm{e}, \bm{w} \rangle}
\\
&= P_{\mathscr{C}} \circ T_{\langle \bm{e}, \bm{w} \rangle},
\end{aligned}
\end{equation*}
where the second equality follows from Proposition~\ref{prop:half_averaging}.
\end{proof}

We now consider the construction of $\bm{w}$-invariant and relatively complete sub-$\sigma$-algebras. To achieve this, we first introduce the following auxiliary definition.
\begin{defi}
Let $\mathscr{D}$ and $\mathscr{E}$ be relatively complete sub-$\sigma$-algebras of $\mathscr{B}$. We say that $(\mathscr{D}, \mathscr{E})$ is a $\bm{w}$-invariant pair if, for all $\bm{e} \in \R^d$,
\begin{equation*}
\begin{aligned}
T_{\langle \bm{e}, \bm{w} \rangle}(L^2(\I, \mathscr{D})) &\subseteq L^2(\I, \mathscr{E}),
\\
T_{\langle \bm{e}, \bm{w}^T \rangle}(L^2(\I, \mathscr{D})) &\subseteq L^2(\I, \mathscr{E}).
\end{aligned}
\end{equation*}
\end{defi}

Note that $\mathscr{C}$ is $\bm{w}$-invariant if and only if $(\mathscr{C}, \mathscr{C})$ is a $\bm{w}$-invariant pair. Given $\mathscr{D}$, define $\Phi$ as the collection of all $\mathscr{E}$ such that $(\mathscr{D}, \mathscr{E})$ is a $\bm{w}$-invariant pair. Then $\Phi$ is non-empty because $\mathscr{B} \in \Phi$. We define
\begin{equation*}
m(\mathscr{D}) \defeq \{ Z \in \mathscr{B} : \forall \mathscr{E} \in \Phi, Z \in \mathscr{E} \}.
\end{equation*}
It is straightforward to verify that for any relatively complete $\mathscr{D}$, the $\sigma$-algebra $m(\mathscr{D})$ is relatively complete and $(\mathscr{D}, m(\mathscr{D}))$ forms a $\bm{w}$-invariant pair.

\begin{defi} [\textbf{Canonical sequence $\{ \mathscr{C}_n^{\bm{w}} \}_{n \in \mathbb{N}}$}] \label{defi:canonical_sequence_Borel}
Let $(\I, \mathscr{B}, \mu)$ be a standard probability space, and let $\bm{w} \in L^\infty(\I \times \I; \mathbb{R}^d)$.
Define $\mathscr{C}_0^{\bm{w}} \defeq \langle \{\varnothing, \I\} \rangle$ and, inductively, $\mathscr{C}_{n+1}^{\bm{w}} \defeq m(\mathscr{C}_n^{\bm{w}})$. Furthermore, define
\begin{equation*}
\mathscr{C}(\bm{w}) \defeq \left\langle \bigcup_{n \in \mathbb{N}} \mathscr{C}_n^{\bm{w}} \right\rangle.
\end{equation*}
\end{defi}
\noindent
The canonical sequence generates the minimum $\bm{w}$-invariant relatively complete sub-$\sigma$-algebra.
\begin{prop} \label{prop:minimum_algebra}
Let $(\I, \mathscr{B}, \mu)$ be a standard probability space, and let $\bm{w} \in L^\infty(\I \times \I; \mathbb{R}^d)$. Then $\mathscr{C}(\bm{w})$ is the minimum $\bm{w}$-invariant relatively complete sub-$\sigma$-algebra of $\mathscr{B}$.
\end{prop}
\begin{proof}[Proof of Proposition~\ref{prop:minimum_algebra}]

We refer to the proof of Proposition~5.13 in \cite{grebik2022fractional}.

\end{proof}

\subsection{Iterated degree of measures}

We next adapt the definition of the iterated degree of measures from \cite{grebik2022fractional} to the multi-dimensional, non-symmetric case.

\begin{defi} \label{defi:iterated_measure_space}
Let $P^0 = \{\star\}$ denote the one-point space, and define inductively
\begin{equation*}
\begin{aligned}
\mathbb{M}_n &= \prod_{i \leq n} P^i, \quad \text{and} \quad P^{n+1} = \big( \mathcal{M}_{\leq 1}(\mathbb{M}_n) \big)^{2d}
\end{aligned}
\end{equation*}
for every $n \in \mathbb{N}$. We set $\mathbb{M} = \mathbb{M}_\infty = \prod_{n \in \mathbb{N}} P^n$ and denote by $p_{n,k}: \mathbb{M}_k \to \mathbb{M}_n$ the canonical projection, where $n \leq k \leq \infty$.

Define
\begin{equation*}
\begin{aligned}
\mathbb{P} = \left\{ \alpha \in \mathbb{M} : \forall n \in \mathbb{N}, \forall 1 \leq l \leq 2d, \, \alpha(n + 1)(l) = (p_{n, n+1})_{\#} \big[ \alpha(n + 2)(l) \big] \right\},
\end{aligned}
\end{equation*}
where $(p_{n, n+1})_{\#} \alpha(n + 2)(l) \in (\mathcal{M}_{\leq 1}(\mathbb{M}_{n}))^{2d}$ denotes the push-forward of $\alpha(n + 2)(l)$ via $p_{n, n+1}$.

\end{defi}
\noindent
It follows from Kolmogorov’s Existence Theorem that for every $\alpha \in \mathbb{P}$, there exists a unique measure $\mu_\alpha \in (\mathcal{M}_{\leq 1}(\mathbb{M}))^{2d}$ such that
\begin{equation*}
(p_{n, \infty})_\ast \mu_\alpha = \alpha(n + 1)
\end{equation*}
for every $n \in \mathbb{N}$. In fact, we have the following uniform version.

\begin{prop} \label{prop:DIDM_close}
The set $\mathbb{P}$ is closed in $\mathbb{M}$, and the map $\alpha \mapsto \mu_\alpha$, which satisfies
\begin{equation*}
(p_{n, \infty})_\ast \mu_\alpha = \alpha(n + 1)
\end{equation*}
for every $n \in \mathbb{N}$, is a continuous map from $\mathbb{P}$ to $(\mathcal{M}_{\leq 1}(\mathbb{M}))^{2d}$.
\end{prop}

Hence, we have the following definition, adapted from the notion of distribution on iterated degree measures in \cite{grebik2022fractional}.

\begin{defi}
We say that $\nu \in \mathcal{P}(\mathbb{M})$ is a $2d$-dimensional distribution on iterated degree measures, or a $2d$-DIDM, if:
\begin{enumerate}
  \item $\nu(\mathbb{P}) = 1$,
  \item $\mu_\alpha$ is absolutely continuous with respect to $\nu$, with the corresponding Radon-Nikodym derivative satisfying $\frac{d\mu_\alpha}{d\nu} \in [-1,1]^{2d}$ for $\nu$-almost every $\alpha \in \mathbb{M}$.
\end{enumerate}
\end{defi}

We now construct $\nu_{\bm{w}} \in \mathcal{P}(\mathbb{M})$ from our kernel $\bm{w}$.

\begin{defi} \label{defi:iterated_measure}
Let $(\I, \mathscr{B}, \mu)$ be a standard probability space, and let $\bm{w} \in L^\infty_{\leq 1}(\I \times \I; \mathbb{R}^d)$. We define $i_{\bm{w},0} : \I \to \mathbb{M}_0 = \{\star\}$ as the constant map. Inductively, define
\begin{equation*}
\begin{aligned}
i_{\bm{w},n+1} : \I \to \mathbb{M}_{n+1} = \textstyle \prod_{j \leq n+1} P^j
\end{aligned}
\end{equation*}
such that:
\begin{enumerate}
  \item For every $j \leq n$, $i_{\bm{w},n+1}(\xi)(j) = i_{\bm{w},n}(\xi)(j)$.
  \item When $j = n+1$, for $1 \leq l \leq 2d$ and any Borel set $A \subset \mathbb{M}_{n}$,
  \begin{equation*}
  i_{\bm{w},n+1}(\xi)(n+1)(l)(A) = \left\{ \begin{aligned}
    &\int_{i_{\bm{w},n}^{-1}(A)} \langle \bm{e}_{l}, \bm{w}(\xi, \zeta) \rangle \, \rd \mu(\zeta) && \text{if } 1 \leq l \leq d, \\
    &\int_{i_{\bm{w},n}^{-1}(A)} \langle \bm{e}_{l - d}, \bm{w}^T(\xi, \zeta) \rangle \, \rd \mu(\zeta) && \text{if } d+1 \leq l \leq 2d.
  \end{aligned} \right.
  \end{equation*}
\end{enumerate}

Define
\begin{equation*}
\begin{aligned}
i_{\bm{w}} : \I \to \mathbb{M} = \textstyle \prod_{n \in \mathbb{N}} P^n
\end{aligned}
\end{equation*}
as the unique map such that $\forall \xi \in \I, n \in \N$,
\begin{equation*}
\begin{aligned}
i_{\bm{w}}(\xi)(n) = i_{\bm{w},n}(\xi)(n).
\end{aligned}
\end{equation*}
Finally, define $\nu_{\bm{w}} \in \mathcal{P}(\mathbb{M})$ as the pushforward measure of $\mu$ by $i_{\bm{w}}$, i.e.,
\begin{equation*}
\begin{aligned}
\nu_{\bm{w}}(A) = \int_{i_{\bm{w}}^{-1}(A)} \rd \mu(\xi).
\end{aligned}
\end{equation*}

\end{defi}
\noindent
When constructing $\nu_{\bm{w}}$ iteratively through $i_{\bm{w},n}$, the canonical sequence $\{ \mathscr{C}_n^{\bm{w}} \}_{n \in \mathbb{N}}$ from Definition~\ref{defi:canonical_sequence_Borel} is reproduced. We state this in detail as the following proposition.
\begin{prop} \label{prop:minimum_algebra_from_iteration}
Let $(\I, \mathscr{B}, \mu)$ be a standard probability space, and let $\bm{w} \in L^\infty_{\leq 1}(\I \times \I; \mathbb{R}^d)$. The maps $i_{\bm{w},n}$, $n \in \mathbb{N}$, are measurable, and
\begin{equation*}
\Big\langle \Big\{ i_{\bm{w},n}^{-1}(A) : A \in \mathscr{B}(\mathbb{M}_n) \Big\} \Big\rangle = \mathscr{C}_n^{\bm{w}},
\end{equation*}
i.e., the minimum relatively complete sub-$\sigma$-algebra of $\mathscr{B}$ that makes the map $i_{\bm{w},n}$ measurable is $\mathscr{C}_n^{\bm{w}}$.
\end{prop}
\noindent
Aside from adjustments for the multi-dimension, our proof of the proposition essentially replicates the proof of Proposition~6.6 in \cite{grebik2022fractional}.
\begin{proof} [Proof of Proposition~\ref{prop:minimum_algebra_from_iteration}]
It is clear that the claim holds for $n = 0$ because $\mathscr{C}_0^{\bm{w}} = \langle \{\varnothing, \I\} \rangle = \langle \{i_{\bm{w},0}^{-1}(\varnothing), i_{\bm{w},0}^{-1}(\{\star\})\} \rangle$.
Suppose that the claim holds for $n \in \mathbb{N}$. It follows from Theorem~17.24 of \cite{kechris2012classical} and the definition of $\mathbb{M}_{n+1}$ that the Borel $\sigma$-algebra $\mathscr{B}(\mathbb{M}_{n+1})$ is generated by the sets $\{p_{n,n+1}^{-1}(A) : A \in \mathscr{B}(\mathbb{M}_{n})\}$ and the maps
\begin{equation*}
\mathbb{M}_{n+1} \ni \kappa \mapsto \sum_{l=1}^{2d} \int_{\mathbb{M}_{n}} f_l \, \rd \kappa(n+1)(l) \in \mathbb{R},
\end{equation*}
where for each $1 \leq l \leq 2d$, $f_l : \mathbb{M}_{n} \to \mathbb{R}$ is a bounded Borel function.

Let $A \in \mathscr{B}(\mathbb{M}_{n})$. Then, by the inductive hypothesis, we have
\begin{equation*}
i_{\bm{w},n+1}^{-1}(p_{n,n+1}^{-1}(A)) = i_{\bm{w},n}^{-1}(A) \in \mathscr{C}_n^{\bm{w}} \subset \mathscr{C}_{n+1}^{\bm{w}}.
\end{equation*}
For each $1 \leq l \leq 2d$, let $f_l : \mathbb{M}_{n} \to \mathbb{R}$ be a bounded Borel function. Then the map
\begin{equation*}
\begin{aligned}
\I \ni \xi &\mapsto \sum_{l=1}^{2d} \int_{\mathbb{M}_{n}} f_l \ \big( i_{\bm{w},n+1}(\xi)(n+1)(l) \big)_{\#}(\rd \mu)
\\
& = \sum_{l=1}^{d} \int_{\I} \langle \bm{e}_{l}, \bm{w}(\xi,\zeta) \rangle (f_l \circ i_{\bm{w},n}) (\zeta) \rd \mu(\zeta)
\\
& \hspace{0.5cm} + \sum_{l=d+1}^{2d} \int_{\I} \langle \bm{e}_{l-d}, \bm{w}^T(\xi,\zeta) \rangle (f_l \circ i_{\bm{w},n}) (\zeta) \rd \mu(\zeta)
\end{aligned}
\end{equation*}
is $\mathscr{C}_{n+1}^{\bm{w}}$-measurable by the definition and the inductive hypothesis (that $\mathscr{C}_{n}^{\bm{w}}$ is the minimum sub-$\sigma$-algebra making $i_{\bm{w},n}$ measurable). This shows that $i_{\bm{w},n+1}$ is measurable.

Let $\mathscr{D}_{n+1}^{\bm{w}}$ denote the minimum relatively complete sub-$\sigma$-algebra that makes $\mathscr{C}_{n+1}^{\bm{w}}$ measurable; then by definition, $\mathscr{D}_{n+1}^{\bm{w}} \subset \mathscr{C}_{n+1}^{\bm{w}}$. It remains to show that $\mathscr{D}_{n+1}^{\bm{w}} = \mathscr{C}_{n+1}^{\bm{w}}$. 

For $A \in \mathscr{C}_{n}^{\bm{w}}$, we find $B \in \mathscr{B}(\mathbb{M}_{n})$ such that $\mu(A \setminus i_{\bm{w},n}^{-1}(B)) = \mu(i_{\bm{w},n}^{-1}(B) \setminus A) = 0$ by the inductive hypothesis. Then we see that the function
\begin{equation*}
\begin{aligned}
\I &\ni \xi \mapsto \int_{B} \ \big( i_{\bm{w},n+1}(\xi)(n+1)(l) \big)_{\#}(\rd \mu)
\\
&
= \left\{
\begin{aligned}
&\int_{\I} \langle \bm{e}_{l}, \bm{w}(\xi,\zeta) \rangle \mathbbm{1}_{A} (\zeta) \rd \mu(\zeta) = T_{\langle \bm{e}_{l}, \bm{w} \rangle}(\mathbbm{1}_A)(\xi) && \text{ if } 1 \leq l \leq d
\\
&\int_{\I} \langle \bm{e}_{l-d}, \bm{w}^T(\xi,\zeta) \rangle \mathbbm{1}_{A} (\zeta) \rd \mu(\zeta) = T_{\langle \bm{e}_{l-d}, \bm{w}^T \rangle}(\mathbbm{1}_A)(\xi) && \text{ if } d+1 \leq l \leq 2d
\end{aligned} \right.
\end{aligned}
\end{equation*}
is $\mathscr{D}_{n+1}^{\bm{w}}$-measurable. It is straightforward to check that $\mathscr{C}_{n}^{\bm{w}}$ is the minimum relatively complete sub-$\sigma$-algebra that makes $T_{\langle \bm{e}_{l}, \bm{w} \rangle}(\mathbbm{1}_A)$ and $T_{\langle \bm{e}_{l-d}, \bm{w}^T \rangle}(\mathbbm{1}_A)$ measurable, where $1 \leq l \leq 2d$ and $A \in \mathscr{C}_{n}^{\bm{w}}$.

Consequently, $\mathscr{D}_{n+1}^{\bm{w}} = \mathscr{C}_{n+1}^{\bm{w}}$, and the proof is complete.
\end{proof}

The following corollary is then straightforward:
\begin{cor} \label{cor:minimum_algebra_from_iteration}
Let $(\I, \mathscr{B}, \mu)$ be a standard probability space, and let $\bm{w} \in L^\infty_{\leq 1}(\I \times \I; \mathbb{R}^d)$. The map $i_{\bm{w}}$ is measurable, and
\begin{equation*}
\Big\langle \Big\{ i_{\bm{w}}^{-1}(A) : A \in \mathscr{B}(\mathbb{M}) \Big\} \Big\rangle = \mathscr{C}^{\bm{w}},
\end{equation*}
i.e., the minimum relatively complete sub-$\sigma$-algebra of $\mathscr{B}$ that makes the map $i_{\bm{w}}$ measurable is $\mathscr{C}^{\bm{w}}$.
\end{cor}

We now show that $\nu_{\bm{w}}$ is a $2d$-DIDM.

\begin{prop} \label{prop:kernel_DIDM}
Let $(\I, \mathscr{B}, \mu)$ be a standard probability space, and let $\bm{w} \in L^\infty_{\leq 1}(\I \times \I; \mathbb{R}^d)$. Then the measure $\nu_{\bm{w}}$ is a $2d$-DIDM, and $i_{\bm{w}}(\xi) \in \mathbb{P}$ for every $\xi \in \I$.
\end{prop}
\noindent
Aside from adjustments for the multi-dimension, our proof of the proposition essentially replicates the proof of Proposition~6.8 in \cite{grebik2022fractional}.
\begin{proof} [Proof of Proposition~\ref{prop:kernel_DIDM}]
First, we show that $i_{\bm{w}}(\xi) \in \mathbb{P}$ for every $\xi \in \I$. This immediately implies that $\nu_{\bm{w}}(\mathbb{P}) = 1$. Let $A \in \mathscr{B}(\mathbb{M}_n)$ and $1 \leq l \leq d$. We have, by the definition of $i_{\bm{w}}$,
\begin{equation*}
\begin{aligned}
& i_{\bm{w}}(\xi)(n+1)(l)(A)
= i_{\bm{w},n+1}(\xi)(n+1)(l)(A)
\\
& \ = \int_{i_{\bm{w},n}^{-1}(A)} \langle \bm{e}_{l} , \bm{w}(\xi, \zeta) \rangle \, \rd \mu(\zeta)
= \int_{i_{\bm{w},n+1}^{-1}(p_{n, n+1}^{-1}(A))} \langle \bm{e}_{l} , \bm{w}(\xi, \zeta) \rangle \, \rd \mu(\zeta)
\\
& \ = i_{\bm{w},n+2}(\xi)(n+2)(l)(p_{n, n+1}^{-1}(A))
= i_{\bm{w}}(\xi)(n+2)(l)(p_{n, n+1}^{-1}(A))
\\
& \ = (p_{n, n+1})_{\#} \big[ i_{\bm{w}}(\xi)(n+2) (l) \big](A).
\end{aligned}
\end{equation*}
This shows that $i_{\bm{w}}(\xi) \in \mathbb{P}$ for any $\xi \in \I$. The argument is essentially the same for $d+1 \leq l \leq 2d$, replacing $\langle \bm{e}_{l}, \bm{w}(\xi, \zeta) \rangle$ with $\langle \bm{e}_{l-d}, \bm{w}^T(\xi, \zeta) \rangle$.

Let $\xi \in \I$, and write $\mu_{\xi} = \mu_{i_{\bm{w}}(\xi)} \in (\mathcal{M}_{\leq 1}(\mathbb{M}))^{2d}$. It follows from Corollary~\ref{cor:minimum_algebra_from_iteration} and Corollary E.2 in \cite{grebik2022fractional}, that there are functions $g_{\xi}$ and $g_{\xi}'$ on $\mathbb{M}$ such that
\begin{equation*}
\begin{aligned}
\E(\bm{w}(\xi,\zeta) | \mathscr{C}(\bm{w})) &= g_{\xi} (i_{\bm{w}}(\zeta)), \\
\E(\bm{w}^T(\xi,\zeta) | \mathscr{C}(\bm{w})) &= g_{\xi}' (i_{\bm{w}}(\zeta)),
\end{aligned}
\end{equation*}
holds for $\mu$-a.e. $\zeta \in \I$.

We show that $(g_{\xi}, g_{\xi}')$ is the Radon–Nikodym derivative $\frac{\rd \mu_{\xi}}{\rd \nu_{\bm{w}}}$ (for $\nu_{\bm{w}}$-almost every $i_{\bm{w}}(\xi) \in \mathbb{M}$). To do this, let $A \in \bigcup_{n \in \mathbb{N}} \mathscr{B}(\mathbb{M}_n)$. Then we have
\begin{equation*}
\begin{aligned}
& \mu_{\xi}(l)(p_{n,\infty}^{-1}(A)) = i_{\bm{w}}(\xi)(n+1)(l)(A)
\\
& \ = \int_{i_{\bm{w},n}^{-1}(A)} \langle \bm{e}_{l} , \bm{w}(\xi, \zeta) \rangle \, \rd \mu(\zeta)
= \int_{i_{\bm{w},n}^{-1}(A)} \Big\langle \bm{e}_{l} , \E\big( \bm{w}(\xi,\cdot)| \mathscr{C}_{n}^{\bm{w}})(\zeta) \Big\rangle \, \rd \mu(\zeta)
\\
& \ = \int_{i_{\bm{w},n}^{-1}(A)} \Big\langle \bm{e}_{l} , \E\big( \bm{w}(\xi,\cdot)| \mathscr{C}(\bm{w}))(\zeta) \Big\rangle \, \rd \mu(\zeta)
= \int_{i_{\bm{w},n}^{-1}(A)} \Big\langle \bm{e}_{l} , g_{\xi} (i_{\bm{w}}(\zeta)) \Big\rangle \, \rd \mu(\zeta)
\\
& \ = \int_{i_{\bm{w}}^{-1}(p_{n,\infty}^{-1}(A))} \Big\langle \bm{e}_{l} , g_{\xi} (i_{\bm{w}}(\zeta)) \Big\rangle \, \rd \mu(\zeta)
\\
& \ = \int_{p_{n,\infty}^{-1}(A)} \Big\langle \bm{e}_{l} , g_{\xi} (\cdot) \Big\rangle \, \rd \nu_{\bm{w}}.
\end{aligned}
\end{equation*}

The rest follows from the fact that $\mu_{\xi}$ and $\nu_{\bm{w}}$ are well-defined and that
\begin{equation*}
\bigcup_{n \in \mathbb{N}} \big\{ p_{n,\infty}^{-1}(A): A \in \mathscr{B}(\mathbb{M}_n) \big\}
\end{equation*}
generates $\mathscr{B}(\mathbb{M})$.
\end{proof}

The next proposition illustrates how to define a kernel from a $2d$-DIDM using the Radon-Nikodym derivative.

\begin{prop} \label{prop:DIDM_to_kernel}
Let $\nu$ be a $2d$-DIDM. Then there exists $\bm{U}[\nu] \in L^\infty\big((\mathbb{M} \times \mathbb{M}, \nu \times \nu); \mathbb{R}^{d} \big)$ such that for $1 \leq l \leq d$,
\begin{equation*}
\langle \bm{e}_l , \bm{U}[\nu](\alpha,\cdot) \rangle = \bigg( \frac{\rd \mu_\alpha}{\rd \nu} \bigg)(l)
\end{equation*}
for $\nu$-almost every $\alpha \in \mathbb{M}$.
\end{prop}
\noindent
Aside from adjustments for the multi-dimension, our proof of the proposition essentially replicates the proof of Claim~6.9 in \cite{grebik2022fractional}.
\begin{proof}[Proof of Proposition~\ref{prop:DIDM_to_kernel}]
Let $A \in \mathcal{B}(\mathbb{M} \times \mathbb{M})$ and define $A_{\alpha} = \{\beta \in \mathbb{M} : (\alpha, \beta) \in A\}$. Then the assignment
\begin{equation*}
\mathbb{M} \ni \alpha \mapsto \mu_\alpha(l)(A_\alpha) \in [0,1]
\end{equation*}
is defined for all $1 \leq l \leq 2d$ and $\nu$-almost everywhere, and it follows from Proposition~\ref{prop:DIDM_close} that this assignment is measurable. This allows us to compute
\begin{equation*}
\Phi(A) = \int_{\mathbb{M}} \mu_\alpha (A_\alpha) \, \rd \nu.
\end{equation*}
It is straightforward to check that $\Phi$ is a Borel probability measure on $\mathbb{M} \times \mathbb{M}$ that is absolutely continuous with respect to $\nu \times \nu$. Let $\bm{U}[\nu]$ be the corresponding Radon-Nikodym derivative. It is then straightforward to verify that $\bm{U}[\nu]$ coincides with $\frac{\rd \mu_\alpha}{\rd \nu}$ for $1 \leq l \leq d$ and for $\nu$-almost every $\alpha \in \mathbb{M}$.
\end{proof}

The generated kernel $\bm{U}[\nu]$ is a pushforward of $\bm{w}_{\mathscr{C}(\bm{w})}$, as established in the following theorem.

\begin{thm} \label{thm:iterated_measure_kernel}
Let $\bm{w} \in L^\infty(\I \times \I; \mathbb{R}^d)$. Then
\begin{equation*}
\bm{w}_{\mathscr{C}(\bm{w})} (\xi,\zeta) = \bm{U}[\nu_{\bm{w}}] (i_{\bm{w}}(\xi), i_{\bm{w}}(\zeta))
\end{equation*}
for $\mu \times \mu$-almost every $(\xi,\zeta) \in \I \times \I$.
\end{thm}
\noindent
Aside from adjustments for the multi-dimension, our proof of the theorem essentially replicates the proof of Theorem~6.10 in \cite{grebik2022fractional}.
\begin{proof}[Proof of Theorem~\ref{thm:iterated_measure_kernel}]
Recall that by Proposition~\ref{prop:kernel_DIDM}, $\bm{U}[\nu_{\bm{w}}]$ is well-defined because $\nu_{\bm{w}}$ is a $2d$-DIDM and $i_{\bm{w}}(\xi) \in \mathbb{P}$ for every $\xi \in \I$. Consequently, $\bm{U}[\nu_{\bm{w}}](i_{\bm{w}}(\xi), \cdot) = \frac{\rd i_{\bm{w}}(\xi)}{\rd \nu_{\bm{w}}}$ for $\mu$-almost every $\xi \in \I$ by Proposition~\ref{prop:DIDM_to_kernel}.

Define an integral kernel $\bm{U}$ on $\I$ as 
\begin{equation*}
\bm{U}(\xi,\zeta) = \bm{U}[\nu_{\bm{w}}] (i_{\bm{w}}(\xi), i_{\bm{w}}(\zeta)).
\end{equation*}
It is sufficient to show that $T_{\langle \bm{e}_l, \bm{w}_{\mathscr{C}(\bm{w})} \rangle} = T_{\langle \bm{e}_l, \bm{U} \rangle}$ for $1 \leq l \leq d$.
By the definition of $\bm{w}_{\mathscr{C}(\bm{w})}$ and Corollary~\ref{cor:minimum_algebra_from_iteration}, both $\bm{w}_{\mathscr{C}(\bm{w})}$ and $\bm{U}$ are $(\mathscr{C}(\bm{w}) \times \mathscr{C}(\bm{w}))$-measurable. This implies $T_{\langle \bm{e}_l, \bm{w}_{\mathscr{C}(\bm{w})} \rangle}(f) = T_{\langle \bm{e}_l, \bm{U} \rangle}(f)$ for $1 \leq l \leq d$ and $f \in L^2(\I, \mathscr{C}(\bm{w}), \mu)^\perp$. Therefore, it suffices to show that
\begin{equation*}
T_{\langle \bm{e}_l, \bm{w}_{\mathscr{C}(\bm{w})} \rangle}(\mathbbm{1}_{A}) = T_{\langle \bm{e}_l, \bm{U} \rangle}(\mathbbm{1}_{A})
\end{equation*}
for every $A \in \bigcup_{n \in \mathbb{N}} \mathscr{C}_{n}^{\bm{w}}$.

To this end, choose such an $A \in \mathscr{C}_{n}^{\bm{w}}$ for some $n \in \mathbb{N}$. By Proposition~\ref{prop:minimum_algebra_from_iteration}, we may assume (up to a $\mu$-null set) that there is $B \in \mathscr{B}(\mathbb{M}_n)$ such that $A = i_{\bm{w},n}^{-1}(B)$. Since $i_{\bm{w}}^{-1}(p_{n,\infty}^{-1}(B)) = A$, we have
\begin{equation*}
\begin{aligned} 
&T_{\langle \bm{e}_l, \bm{w}_{\mathscr{C}(\bm{w})} \rangle}(\mathbbm{1}_{A})(\xi) 
= \int_{A} \langle \bm{e}_l, \bm{w}_{\mathscr{C}(\bm{w})}(\xi,\zeta) \rangle \, \rd \mu(\zeta) = \int_{i_{\bm{w},n}^{-1}(B)} \langle \bm{e}_l, \bm{w}_{\mathscr{C}(\bm{w})}(\xi,\zeta) \rangle \, \rd \mu(\zeta)
\\
& \ = i_{\bm{w}}(\xi)(n+1)(l)(B) = \mu_{i_{\bm{w}}(\xi)}(l)(p_{n,\infty}^{-1}(B))
\\
& \ = \int_{p_{n,\infty}^{-1}(B)} \bigg(\frac{\rd i_{\bm{w}}(\xi)}{\rd \nu_{\bm{w}}}\bigg)(l) \, \rd \nu_{\bm{w}} = \int_{p_{n,\infty}^{-1}(B)} \Big\langle \bm{e}_l, \bm{U}[\nu_{\bm{w}}](i_{\bm{w}}(\xi), \cdot) \Big\rangle \, \rd \nu_{\bm{w}}
\\
& \ = \int_{A} \Big\langle \bm{e}_l, \bm{U}(\xi, \zeta) \Big\rangle \, \rd \mu(\zeta)
\\
& \ = T_{\langle \bm{e}_l, \bm{U} \rangle}(\mathbbm{1}_{A})(\xi),
\end{aligned}
\end{equation*}
by definition of $i_{\bm{w}}$, $\mu_{\alpha}$, and $\bm{U}[\nu_{\bm{w}}]$ for $\mu$-almost every $\xi \in \I$.
\end{proof}

\begin{cor} \label{cor:iterated_measure_kernel}
Let $\bm{w} \in L^\infty(\I \times \I; \mathbb{R}^d)$. Then $\bm{w} / \mathscr{C}(\bm{w})$ is isomorphic to $\bm{U}[\nu_{\bm{w}}]$.
\end{cor}

\subsection{Tree functions}

In this section, we introduce the tensorized graph homomorphism of rooted trees as an alternative to $\nu_{\bm{w}}$ to group the points in $\I$.

Recall Definition~\ref{defi:tensorized_homomorphism}, which we restate here in the context of $\mathbb{R}^d$: For any $\bm{w} \in L^\infty(\I \times \I; \mathbb{R}^d)$ and an oriented simple graph $F$, define the tensorized graph homomorphism $\bm{t}(F, \bm{w}) \in (\mathbb{R}^d)^{\otimes \mathsf{e}(F)}$ as
\begin{equation*}
\bm{t}(F, \bm{w}) \defeq \int_{\I^{\mathsf{v}(F)}} \bigotimes_{(i,j) \in \mathsf{e}(F)} \bm{w}(\xi_i, \xi_j) \prod_{i \in \mathsf{v}(F)} \rd \xi_i.
\end{equation*}

We define a rooted oriented tree as a pair $\mathfrak{T} = (T, v)$, where $T \in \mathcal{T}$ is an oriented tree and $v \in \mathsf{v}(T)$ is a distinguished vertex of $T$. The height $h(\mathfrak{T})$ is the maximum number of edges in a path (considering $T$ as undirected) that starts at $v$. We denote by $c(\mathfrak{T})$ the neighbors of $v$ in $T$ (also considering $T$ as undirected).

Every rooted oriented tree $\mathfrak{T}$ can be decomposed into subtrees rooted in the neighbors of $v$. Specifically, there exists a sequence of $(\mathfrak{T}_i)_{i \in c(\mathfrak{T})}$, where each $\mathfrak{T}_i$ is a rooted, oriented tree with root $i$, and a sequence $(e_i)_{i \in c(\mathfrak{T})}$, where each $e_i = (e_{i;1}, e_{i;2})$ is either $(v, i)$ or $(i, v)$, such that
\begin{equation*}
\begin{aligned}
\mathsf{v}(T) &= \{v\} \cup \bigcup_{i \in c(\mathfrak{T})} \mathsf{v}(T_i), \\
\mathsf{e}(T) &= \bigcup_{i \in c(\mathfrak{T})} \big( \mathsf{e}(T_i) \cup \{e_i\} \big).
\end{aligned}
\end{equation*}
We call $(\mathfrak{T}_i)_{i \in c(\mathfrak{T})}$ and $(e_i)_{i \in c(\mathfrak{T})}$ the corresponding decomposition of $\mathfrak{T}$.

Note that if $h(\mathfrak{T}) > 0$, then $h(\mathfrak{T}_i) < h(\mathfrak{T})$ for every $i \in c(\mathfrak{T})$, and there exists $i \in c(\mathfrak{T})$ such that $h(\mathfrak{T}_i) + 1 = h(\mathfrak{T})$.

\begin{defi} \label{defi:tensorized_homomorphism_root}
Let $\bm{w} \in L^\infty(\I \times \I; \mathbb{R}^d)$ and $\mathfrak{T} = (T, v)$ be a rooted oriented tree. We define $f^{\bm{w}}_{\mathfrak{T}}: \I \to (\mathbb{R}^d)^{\otimes \mathsf{e}(F)}$ as
\begin{equation*}
\begin{aligned}
f^{\bm{w}}_{\mathfrak{T}}(\xi_v) = \int_{\I^{c(\mathfrak{T})}} \bigotimes_{i \in c(\mathfrak{T})} \Big[ \bm{w}(\xi_{e_{i;1}}, \xi_{e_{i;2}}) \otimes f^{\bm{w}}_{\mathfrak{T}_i}(\xi_i) \Big] \prod_{i \in c(\mathfrak{T})} \rd \xi_i,
\end{aligned}
\end{equation*}
where $(\mathfrak{T}_i)_{i \in c(\mathfrak{T})}$ and $(e_i)_{i \in c(\mathfrak{T})}$ are the corresponding decomposition of $\mathfrak{T}$.
\end{defi}
\noindent
This definition is compatible with $\bm{w}$-invariant sub-$\sigma$-algebras, as stated by the next proposition.

\begin{prop} \label{prop:invariant_on_tree}
Let $\bm{w} \in L^\infty(\I \times \I; \mathbb{R}^d)$, let $\mathfrak{T} = (T, v)$ be a rooted oriented tree and let $\mathscr{C}$ be a $\bm{w}$-invariant and relatively complete sub-$\sigma$-algebra. Then $f^{\bm{w}}_{\mathfrak{T}}$ is $\mathscr{C}_{h(\mathfrak{T})}^{\bm{w}}$-measurable, and
\begin{equation*}
\begin{aligned}
f^{\bm{w}_{\mathscr{C}}}_{\mathfrak{T}}(\xi_v) = f^{\bm{w}}_{\mathfrak{T}}(\xi_v)
\end{aligned}
\end{equation*}
for $\mu$-almost every $\xi \in \I$.
\end{prop}
\noindent
Aside from adjustments for the multi-dimension, our proof of the proposition essentially replicates the proof of Proposition~7.2 in \cite{grebik2022fractional}.
\begin{proof}[Proof of Proposition~\ref{prop:invariant_on_tree}]
We prove both statements simultaneously by induction. If $h(\mathfrak{T}) = 0$, then the claim clearly holds. Suppose that $h(\mathfrak{T}) = n+1$ and that the claim holds for all rooted oriented trees of height at most $n$. Let $(\mathfrak{T}_i)_{i \in c(\mathfrak{T})}$ and $(e_i)_{i \in c(\mathfrak{T})}$ be the corresponding decomposition of $\mathfrak{T}$.
We have
\begin{equation*}
\begin{aligned}
f^{\bm{w}}_{\mathfrak{T}}(\xi_v) &= \int_{\I^{c(\mathfrak{T})}} \bigotimes_{i \in c(\mathfrak{T})} \Big[ \bm{w}(\xi_{e_{i;1}}, \xi_{e_{i;2}}) \otimes f^{\bm{w}}_{\mathfrak{T}_i}(\xi_i) \Big] \prod_{i \in c(\mathfrak{T})} \rd \xi_i \\
&= \bigotimes_{i \in c(\mathfrak{T})} \bigg( \int_{\I} \bm{w}(\xi_{e_{i;1}}, \xi_{e_{i;2}}) \otimes f^{\bm{w}}_{\mathfrak{T}_i}(\xi_i) \, \rd \xi_i \bigg) \\
&= \bigotimes_{i \in c(\mathfrak{T})} \bigg( \int_{\I} \bm{w}(\xi_{e_{i;1}}, \xi_{e_{i;2}}) \otimes f^{\bm{w}_{\mathscr{C}}}_{\mathfrak{T}_i}(\xi_i) \, \rd \xi_i \bigg) \\
&= \bigotimes_{i \in c(\mathfrak{T})} \bigg( \int_{\I} \bm{w}_{\mathscr{C}}(\xi_{e_{i;1}}, \xi_{e_{i;2}}) \otimes f^{\bm{w}_{\mathscr{C}}}_{\mathfrak{T}_i}(\xi_i) \, \rd \xi_i \bigg) \\
&= f^{\bm{w}_{\mathscr{C}}}_{\mathfrak{T}}(\xi_v),
\end{aligned}
\end{equation*}
for $\mu$-almost every $\xi \in \I$. By definition, we see that $f^{\bm{w}}_{\mathfrak{T}}$ is $\mathscr{C}_{n+1}^{\bm{w}}$-measurable and that ﬁnishes the proof.
\end{proof}

When integrating over $\I$, the function $f^{\bm{w}}_{\mathfrak{T}}$ reduces to $\bm{t}(T, \bm{w})$, as stated in the following proposition.

\begin{prop} \label{prop:two_def_of_tree}
Let $\bm{w} \in L^\infty(\I \times \I; \mathbb{R}^d)$, let $\mathfrak{T} = (T, v)$ be a rooted directed tree, and let $\mathscr{C}$ be a $\bm{w}$-invariant and relatively complete sub-$\sigma$-algebra. Then
\begin{equation*}
\bm{t}(T, \bm{w}) = \int_{\I} f^{\bm{w}}_{\mathfrak{T}}(\xi_v) \, \rd \xi_v.
\end{equation*}
In particular, $\bm{t}(T, \bm{w}) = \bm{t}(T, \bm{w}_{\mathscr{C}}) = \bm{t}(T, \bm{U}[\nu_{\bm{w}}])$ for every $T \in \mathcal{T}$.
\end{prop}
\noindent
Aside from adjustments for the multi-dimension, our proof of the proposition essentially replicates the proof of Proposition~7.3 in \cite{grebik2022fractional}.
\begin{proof}[Proof of Proposition~\ref{prop:two_def_of_tree}]
We proceed by induction on the height $h(\mathfrak{T})$. If $h(\mathfrak{T}) = 0$, then the claim clearly holds. Suppose that $h(\mathfrak{T}) = n+1$ and that the claim holds for all rooted oriented trees of height at most $n$. Let $(\mathfrak{T}_i)_{i \in c(\mathfrak{T})}$ and $(e_i)_{i \in c(\mathfrak{T})}$ be the corresponding decomposition of $\mathfrak{T}$.
Then for fixed $\xi_v \in \I$, we have
\begin{equation*}
\int_{\I} \bm{w}(\xi_{e_{i;1}}, \xi_{e_{i;2}}) \otimes f^{\bm{w}}_{\mathfrak{T}_i}(\xi_i) \, \rd \xi_i = \int_{\I^{\mathsf{v}(T_i)}} \bm{w}(\xi_{e_{i;1}}, \xi_{e_{i;2}}) \otimes \bigotimes_{(k, j) \in \mathsf{e}(T_i)} \bm{w}(\xi_k, \xi_j) \prod_{i \in \mathsf{v}(T_i)} \rd \xi_i.
\end{equation*}
This gives
\begin{equation*}
\begin{aligned}
\bm{t}(T, \bm{w}) &= \int_{\I^{\mathsf{v}(T)}} \bigotimes_{(i, j) \in \mathsf{e}(T)} \bm{w}(\xi_i, \xi_j) \prod_{i \in \mathsf{v}(T)} \rd \xi_i \\
&= \int_{\I} \bigotimes_{i \in c(\mathfrak{T})} \bigg( \int_{\I^{\mathsf{v}(T_i)}} \bm{w}(\xi_{e_{i;1}}, \xi_{e_{i;2}}) \otimes \bigotimes_{(k, j) \in \mathsf{e}(T_i)} \bm{w}(\xi_k, \xi_j) \prod_{i \in \mathsf{v}(T_i)} \rd \xi_i \bigg) \, \rd \xi_v \\
&= \int_{\I} \bigotimes_{i \in c(\mathfrak{T})} \bigg( \int_{\I} \bm{w}(\xi_{e_{i;1}}, \xi_{e_{i;2}}) \otimes f^{\bm{w}}_{\mathfrak{T}_i}(\xi_i) \, \rd \xi_i \bigg) \, \rd \xi_v \\
&= \int_{\I} f^{\bm{w}}_{\mathfrak{T}}(\xi_v) \, \rd \xi_v.
\end{aligned}
\end{equation*}
The rest follows from Proposition~\ref{prop:invariant_on_tree}.
\end{proof}

\subsection{Stone-Weierstrass Theorem}
In this section, we show that $\nu_{\bm{w}}$ is determined by the family $f^{\bm{w}}_{\mathfrak{T}}$, where $\mathfrak{T}$ ranges over all rooted oriented trees.

We define a collection $\mathscr{T} \subset C(\mathbb{M}, \mathbb{R})$ that is closed under multiplication and contains the constant function $\mathbbm{1}_{\mathbb{M}}$.
The construction proceeds recursively on $n \in \mathbb{N}$: in each step $n \in \mathbb{N}$, we construct a subset $\mathscr{T}_n \subset C(\mathbb{M}, \mathbb{R})$ that factors through $\mathbb{M}_n$. Specifically, for every $f \in \mathscr{T}_n$, there exists an $f' \in C(\mathbb{M}_n, \mathbb{R})$ such that $f = f' \circ p_{n, \infty}$, and $\mathscr{T}_n$ is uniformly dense in $C(\mathbb{M}_n, \mathbb{R}) \circ p_{n, \infty}$.

The set $\mathscr{T}_{n+1}$ is constructed from $\mathscr{T}_n$ using two operations. Informally, these operations correspond to the following constructions on finite trees, with the correspondence made precise in the proof.

\textbf{(I)}: Given a rooted tree, add an extra vertex to serve as the new root, with its only neighbor being the original root.

\textbf{(II)}: Given a collection of finitely many rooted trees $\{\mathfrak{T}^j\}$, define a new rooted tree $\mathfrak{T}$ as the disjoint union of $\{\mathfrak{T}^j\}$ whose roots are glued together to form a single new root for $\mathfrak{T}$.

\begin{defi} \label{defi:Stone_Weierstrass_algebra}
Let $n, k \in \mathbb{N}$, $1 \leq l \leq 2d$, and $f, f_1, \dots, f_k \in C(\mathbb{M}, \mathbb{R})$ be such that $f$ factors through $\mathbb{M}_n$. Then, for every $\alpha \in \mathbb{M}$, define
\begin{equation*}
\begin{aligned}
F(f, n, l)(\alpha) &= \int_{\mathbb{M}_n} f' \, \rd \alpha(n+1)(l), \quad \text{where } f' \in C(\mathbb{M}_n, \mathbb{R}) \text{ and } f = f' \circ p_{n, \infty},
\end{aligned}
\end{equation*}
and
\begin{equation*}
\begin{aligned}
G(f_1, \dots, f_k)(\alpha) = \prod_{j=1}^k f_j(\alpha).
\end{aligned}
\end{equation*}
\end{defi}
It is straightforward to see, by the definition of $\mathbb{M}$, that $F(f, n, l)$ and $G(f_1, \dots, f_k)$ are elements of $C(\mathbb{M}, \mathbb{R})$, where $F(f, n, l)$ factors through $\mathbb{M}_{n+1}$ and $G(f_1, \dots, f_k)$ factors through $\mathbb{M}_n$.

We begin by setting $\mathscr{T}_0 = \{\mathbbm{1}_{\mathbb{M}}\}$. Suppose $\mathscr{T}_n$ has been defined. Then we define
\begin{equation*}
\begin{aligned} 
&\mathscr{T}_{n+1} = \big\{ G(f_1, \dots, f_k) : \forall 1 \leq i \leq k, 
\\
& \hspace{2cm} \exists g_i \in \mathscr{T}_n \, (g_i = f_i \text{ or } f_i = F(g_i, n, l) \text{ for some } 1 \leq l \leq 2d) \big\}.
\end{aligned}
\end{equation*}
That is, we first apply \textbf{(I)} to add new elements, then apply \textbf{(II)} to all existing and newly added elements.

\begin{prop} \label{prop:Stone_Weierstrass_algebra}
The collection $\mathscr{T}$ is closed under multiplication, contains $\mathbbm{1}_{\mathbb{M}}$, and separates the points of $\mathbb{M}$.
\end{prop}
\noindent
Aside from adjustments for the multi-dimension, our proof of the proposition essentially replicates the proof of Proposition~7.5 in \cite{grebik2022fractional}.
\begin{proof} [Proof of Proposition~\ref{prop:Stone_Weierstrass_algebra}]
We need only to show that $\mathscr{T}$ separates points. We will demonstrate by induction on $n \in \mathbb{N}$ that $\mathscr{T}_n$ separates any pair of points $\alpha, \beta \in \mathbb{M}$ whenever there exist $1 \leq i \leq n$ and $1 \leq l \leq 2d$ such that $\alpha(i)(l) \neq \beta(i)(l)$. This is sufficient to prove the claim, since each $\mathscr{T}_n$ is closed under multiplication and contains $\mathbbm{1}_{\mathbb{M}}$ by \textbf{(II)}.

If $n = 0$, there is nothing to prove. Suppose that the claim holds for some $n \in \mathbb{N}$. Now, consider $\alpha, \beta \in \mathbb{M}$ with $\alpha(i)(l) \neq \beta(i)(l)$ for some $1 \leq i \leq n+1$ and $1 \leq l \leq 2d$. By the inductive assumption, either there already exists $f \in \mathscr{T}_n$ such that $f(\alpha) \neq f(\beta)$, or $i = n+1$.

Define $\mathscr{T}_n' = \{f' \in C(\mathbb{M}_n, \mathbb{R}) : \exists f \in \mathscr{T}_n, \, f = f' \circ p_{n, \infty}\}$. By the inductive assumption, $\mathscr{T}_n'$ is closed under multiplication, contains $\mathbbm{1}_{\mathbb{M}_n}$, and separates points of $\mathbb{M}_n$. By the Stone-Weierstrass theorem, the algebra generated by $\mathscr{T}_n'$ is dense in $C(\mathbb{M}_n, \mathbb{R})$.

Since $\alpha(n+1)(l), \beta(n+1)(l) \in \mathcal{M}(\mathbb{M}_n)$ and $\alpha(n+1)(l) \neq \beta(n+1)(l)$, there exists $f' \in \mathscr{T}_n'$ such that
\begin{equation*}
\begin{aligned}
\int_{\mathbb{M}_n} f' \, \rd \alpha(n+1)(l) \neq \int_{\mathbb{M}_n} f' \, \rd \beta(n+1)(l).
\end{aligned}
\end{equation*}
Thus, we have $F(f, n, l)(\alpha) \neq F(f, n, l)(\beta)$, where $f \in \mathscr{T}_n$ is such that $f = f' \circ p_{n, \infty}$.
Since $F(f, n, l) \in \mathscr{T}_{n+1}$ by \textbf{(I)}, the proof is complete.

\end{proof}

\begin{prop} \label{prop:Stone_Weierstrass_tree_algebra}
Let $f \in \mathscr{T}$. Then there exists a finite rooted tree $\mathfrak{T} = (T, v)$ and a sequence $(l_{i,j})_{(i,j) \in \mathsf{e}(T)} \in \{1, \dots, d\}^{\mathsf{e}(T)}$ such that for every $2d$-DIDM $\nu$, we have 
\begin{equation*}
\begin{aligned}
f(\alpha) = \bigg\langle \bigotimes_{(i,j) \in \mathsf{e}(T)} \bm{e}_{l_{i,j}}, \, f^{\bm{U}[\nu]}_{\mathfrak{T}} (\alpha) \bigg\rangle
\end{aligned}
\end{equation*}
for $\nu$-almost every $\alpha \in \mathbb{M}$.
\end{prop}
\noindent
Aside from adjustments for the multi-dimension, our proof of the proposition essentially replicates the proof of Proposition~7.6 in \cite{grebik2022fractional}.

\begin{proof} [Proof of Proposition~\ref{prop:Stone_Weierstrass_tree_algebra}]

We prove the claim by induction on $n \in \mathbb{N}$. It is easy to see that if $f = 1_{\mathbb{M}_1}$, then a tree $\mathfrak{T}$ satisfying $h(\mathfrak{T}) = 0$ works; in other words, the claim holds for $\mathscr{T}_0$.

Suppose that the claim holds for $\mathscr{T}_n$ where $n \in \mathbb{N}$. Let $f = F(g, n, l)$ for some $g \in \mathscr{T}_n$ and $1 \leq l \leq 2d$. Fix a rooted directed tree $\mathfrak{S} = (S, w)$ and $(l_{i,j})_{(i,j) \in \mathsf{e}(G)} \in \{1,\dots,d\}^{\mathsf{e}(G)}$, corresponding to $g$, and let $g' \in C(\mathbb{M}_n, \mathbb{R})$ such that $g = g' \circ p_{n, \infty}$. Define a rooted, directed tree $\mathfrak{T}$ such that $c(\mathfrak{T}) = 1$ and $(\mathfrak{S})$ and $(e_w)$ are the corresponding decomposition of $\mathfrak{T}$, where $e_w = (v, w)$ if $1 \leq l \leq d$ and $e_w = (w, v)$ if $d+1 \leq l \leq 2d$.
We also define $l_{v, w} = l$ if $1 \leq l \leq d$ and $l_{w, v} = l - d$ if $d+1 \leq l \leq 2d$.
In other words, we add an extra vertex as the new root. The only neighbor of the new root is the old root and the direction of the new edge depends on $l$.

Given a DIDM $\nu$, we have
\begin{equation*}
\begin{aligned}
& \bigg\langle \bigotimes_{(i,j)\in \mathsf{e}(T)} \bm{e}_{l_{i,j}}
, \, f^{\bm{U}[\nu]}_{\mathfrak{T}} (\alpha_v) \bigg\rangle
\\
& \ = \int_{\mathbb{M}} \bigg\langle \bm{e}_{l_{e_{w;1},e_{w;2}}}, \bm{U}[\nu](\alpha_{e_{w;1}}, \alpha_{e_{w;2}}) \bigg\rangle \bigg\langle \bigotimes_{(i,j)\in \mathsf{e}(S)} \bm{e}_{l_{i,j}}, f_{\mathfrak{S}}^{\bm{U}[\nu]}(\alpha_w) \bigg\rangle \, \rd\nu(\alpha_w)
\\
& \ = \int_{\mathbb{M}} \bigg\langle \bm{e}_{l_{e_{w;1},e_{w;2}}}, \bm{U}[\nu](\alpha_{e_{w;1}}, \alpha_{e_{w;2}}) \bigg\rangle g(\alpha_w) \, \rd\nu(\alpha_w)
\\
& \ = \int_{\mathbb{M}} g \ \rd \mu_{\alpha}(l) = \int_{\mathbb{M}} g' \circ p_{n,\infty} \ \rd \mu_{\alpha}(l) = \int_{\mathbb{M}_n} g' \rd (p_{n,\infty})_\# \mu_{\alpha}(l)
\\
& \ = \int_{\mathbb{M}_n} g' \, d\alpha(n + 1)(l) = F(g, n, l)(\alpha) = f(\alpha)
\end{aligned}
\end{equation*}
for $\nu$-almost every $\alpha \in \mathbb{M}$.

Let $f \in \mathscr{T}_{n+1}$. By definition, we have $f = G(f_1, \dots, f_k)$ for some $f_i$ such that either $f_i \in \mathscr{T}_n$ or $f_i = F(g_i, n, l_i)$ for some $g_i \in \mathscr{T}_n$, $1 \leq l_i \leq 2d$. In both cases, either by the inductive assumption or by the previous paragraph, we find a rooted, directed tree $\mathfrak{T}^i = (T^i, v^i)$ and a labeling $(l_{i,j})_{(i,j) \in \mathsf{e}(T^i)} \in \{1, \dots, d\}^{\mathsf{e}(T^i)}$ that satisfies the claim for $f_i$ for every $1 \leq i \leq k$. Let $\{\mathfrak{T}_j^i\}_{j = 1}^{|c(\mathfrak{T}^i)|}$ and $(e_j^i)_{j = 1}^{|c(\mathfrak{T}^i)|}$ be the corresponding decomposition of $\mathfrak{T}^i$, where $\mathfrak{T}_j^i = (T_j^i, v_j^i)$ for every $1 \leq i \leq k$. 

Define the index set $I = \{(i, j) : 1 \leq i \leq k, 1 \leq j \leq |c(\mathfrak{T}^i)|\}$ and construct $\mathfrak{T} = (T, v)$ as
\begin{equation*}
\begin{aligned}
\mathsf{v}(T) = \{v\} \cup \bigcup_{(i, j) \in I} \mathsf{v}(T_j^i) \quad \text{and} \quad \mathsf{e}(T) = \bigcup_{(i, j) \in I} \{\bar e_{j;1}^i, \bar e_{j;2}^i\} \cup \mathsf{e}(T_j^i).
\end{aligned}
\end{equation*}
where
\begin{equation*}
\bar e_{j}^i = (\bar e_{j;1}^i, \bar e_{j;2}^i) = \left\{
\begin{aligned}
&(v, v_j^i) && \text{ if } e_{j}^i = (v^i, v_j^i),
\\
&(v_j^i, v) && \text{ if } e_{j}^i = (v_j^i, v^i).
\end{aligned} \right.
\end{equation*}
Moreover, set $l_{\bar{e}_{j;1}^i, \bar{e}_{j;2}^i} = l_{e_{j;1}^i, e_{j;2}^i}$ for all $(i, j) \in I$.

Note that $\{\mathfrak{T}_j^i\}_{(i,j) \in I}$ and $(\bar{e}_j^i)_{(i,j) \in I}$ are the corresponding decomposition of $\mathfrak{T}$.
Given a $2d$-DIDM $\nu$, we have
\begin{equation*}
\begin{aligned}
& \bigg\langle \bigotimes_{(p,q)\in \mathsf{e}(T)} \bm{e}_{l_{p,q}}
, \, f^{\bm{U}[\nu]}_{\mathfrak{T}} (\alpha_v) \bigg\rangle
\\
& \ = \int_{\mathbb{M}^I} \prod_{(i,j) \in I} \bigg(
\bigg\langle \bm{e}_{l_{\bar e_{j;1}^i, \bar e_{j;2}^i}}, \bm{U}[\nu](\alpha_{\bar e_{j;1}^i}, \alpha_{\bar e_{j;2}^i}) \bigg\rangle
\\
& \hspace{4cm} \bigg\langle \bigotimes_{(p,q)\in \mathsf{e}(T_j^i)} \bm{e}_{l_{p,q}}, f_{\mathfrak{S}}^{\bm{U}[\nu]}(\alpha_{v_j^i}) \bigg\rangle \, \rd\nu(\alpha_{v_j^i}) \bigg)
\\
& \ = \prod_{i = 1}^{k} \bigg[ \int_{\mathbb{M}^{c(\mathfrak{T}^i)}} \prod_{j=1}^{|c(\mathfrak{T}^i)|} \bigg( \bigg\langle \bm{e}_{l_{\bar e_{j;1}^i, \bar e_{j;2}^i}}, \bm{U}[\nu](\alpha_{\bar e_{j;1}^i}, \alpha_{\bar e_{j;2}^i}) \bigg\rangle
\\
& \hspace{4cm} \bigg\langle \bigotimes_{(p,q)\in \mathsf{e}(T_j^i)} \bm{e}_{l_{p,q}}, f_{\mathfrak{S}}^{\bm{U}[\nu]}(\alpha_{v_j^i}) \bigg\rangle \, \rd\nu(\alpha_{v_j^i}) \bigg) \bigg]
\\
& \ = \prod_{i = 1}^{k} \bigg\langle \bigotimes_{(p,q)\in \mathsf{e}(T^i)} \bm{e}_{l_{p,q}}
, \, f^{\bm{U}[\nu]}_{\mathfrak{T}^i} (\alpha_v) \bigg\rangle = \prod_{i = 1}^{k} f_i(\alpha_v) = f(\alpha_v)
\end{aligned}
\end{equation*}
for $\nu$-almost every $\alpha \in \mathbb{M}$, which concludes the proof.

\end{proof}

\begin{cor} \label{cor:Stone_Weierstrass_tree_algebra}

The map $\bm{w} \mapsto \nu_{\bm{w}}$ is continuous when $\bm{w} \in L^\infty_{\leq 1}(\I \times \I; \mathbb{R}^d)$ is endowed with the cut distance $\delta_{\square;\mathbb{R}^d}$, and $\mathcal{P}(\mathbb{M})$ is endowed with the weak-* topology.

Moreover, if $\bm{w} \in L^\infty_{\leq 1}(\I_1 \times \I_1; \R^d)$, and $\bm{u} \in L^\infty_{\leq 1}(\I_2 \times \I_2; \R^d)$ such that $\nu_{\bm{u}} \neq \nu_{\bm{w}}$, then there exists an oriented tree $T \in \mathcal{T}$ such that
\begin{equation*}
\begin{aligned}
\bm{t}(T, \bm{w}) \neq \bm{t}(T, \bm{u}).
\end{aligned}
\end{equation*}

\end{cor}
\noindent
Aside from adjustments for the multi-dimension, our proof of the corollary essentially replicates the proof of Corollary~7.7 in \cite{grebik2022fractional}.

\begin{proof} [Proof of Corollary~\ref{cor:Stone_Weierstrass_tree_algebra}]
It follows from the Stone-Weierstrass theorem together with Proposition~\ref{prop:Stone_Weierstrass_algebra} that the algebra generated by $\mathscr{T}$ is uniformly dense in $C(\mathbb{M}, \mathbb{R})$. Consequently, the weak-* topology on $\mathcal{P}(\mathbb{M})$ is generated by functionals corresponding to elements of $\mathscr{T}$.

Let $\bm{w}_n$ converge to $\bm{w}$ in $\delta_{\square;\mathbb{R}^d}$ and let $f \in \mathscr{T}$. Fix a (rooted) oriented tree $\mathfrak{T} = (T,v)$ and $(l_{i,j})_{(i,j) \in \mathsf{e}(T)}$ that corresponds to $f$ as in Proposition~\ref{prop:Stone_Weierstrass_tree_algebra}. By Proposition~\ref{prop:two_def_of_tree} and Proposition~\ref{prop:Stone_Weierstrass_tree_algebra}, we have
\begin{equation*}
\begin{aligned}
\int_{\mathbb{M}} f \rd \nu_{\bm{w}_n} &= \bigg\langle \bigotimes_{(p,q)\in \mathsf{e}(T)} \bm{e}_{l_{p,q}}, \, \bm{t}(T,\bm{U}[\nu_{\bm{w}_n}]) \bigg\rangle
\\
&= \bigg\langle \bigotimes_{(p,q)\in \mathsf{e}(T)} \bm{e}_{l_{p,q}}, \, \bm{t}(T,\bm{w}_n) \bigg\rangle,
\end{aligned}
\end{equation*}
and
\begin{equation*}
\begin{aligned}
\lim_{n \to \infty} \int_{\mathbb{M}} f \rd \nu_{\bm{w}_n} &= \bigg\langle \bigotimes_{(p,q)\in \mathsf{e}(T)} \bm{e}_{l_{p,q}}, \, \bm{t}(T,\bm{w}) \bigg\rangle
\\
&= \bigg\langle \bigotimes_{(p,q)\in \mathsf{e}(T)} \bm{e}_{l_{p,q}}, \, \bm{t}(T,\bm{U}[\nu_{\bm{w}}]) \bigg\rangle
\\
&= \int_{\mathbb{M}} f \rd \nu_{\bm{w}}.
\end{aligned}
\end{equation*}
This proves the continuity.

Suppose that $\nu_{\bm{w}} \neq \nu_{\bm{u}}$. By the Stone-Weierstrass theorem together with Proposition~\ref{prop:Stone_Weierstrass_algebra}, there exists an $f \in \mathscr{T}$ such that
\begin{equation*}
\begin{aligned}
\int_{\mathbb{M}} f \rd \nu_{\bm{w}} \neq \int_{\mathbb{M}} f \rd \nu_{\bm{u}}.
\end{aligned}
\end{equation*}
The rooted oriented tree $\mathfrak{T} = (T,v)$ and the assignment $(l_{i,j})_{(i,j) \in \mathsf{e}(T)}$ corresponding to $f$ then satisfies
\begin{equation*}
\begin{aligned}
\bigg\langle \bigotimes_{(p,q)\in \mathsf{e}(T)} \bm{e}_{l_{p,q}}, \, \bm{t}(T,\bm{w}) \bigg\rangle \neq \bigg\langle \bigotimes_{(p,q)\in \mathsf{e}(T)} \bm{e}_{l_{p,q}}, \, \bm{t}(T,\bm{u}) \bigg\rangle
\end{aligned}
\end{equation*}
which completes the proof.
\end{proof}

\subsection{Summary of results}

We are now at the point of summarizing all the results of the section:

\begin{thm} \label{thm:pointwise_equivalent}

Let $\bm{w} \in L^\infty_{\leq 1}(\I_1 \times \I_1; \R^d)$, and $\bm{u} \in L^\infty_{\leq 1}(\I_2 \times \I_2; \R^d)$. Then the following are equivalent:
\begin{itemize}

\item[(1)] $\bm{t}(T,\bm{w}) = \bm{t}(T,\bm{u})$, $\forall T \in \mathcal{T}$.

\item[(2)] $\nu_{\bm{w}} = \nu_{\bm{u}}$.

\item[(3)] $\bm{w}/\mathscr{C}(\bm{w})$ and $\bm{u}/\mathscr{C}(\bm{u})$ are isomorphic.

\item[(4)] There is a coupling $\gamma \in \Pi(\I_1,\I_2)$ such that $T_{\langle \bm{e}, \bm{w} \rangle} \circ T_{\gamma} = T_{\gamma} \circ T_{\langle \bm{e}, \bm{u} \rangle}$ for all $\bm{e} \in \R^d$.

\item[(5)] There is a $\bm{w}$-invariant sub-$\sigma$-algebra $\mathscr{C}$ and a $\bm{u}$-invariant sub-$\sigma$-algebra $\mathscr{D}$ such that $\delta_{\square;\mathbb{R}^d}(\bm{w}_{\mathscr{C}},\bm{u}_{\mathscr{D}}) = 0$.

\end{itemize}

\end{thm}
\noindent
What we need is the equivalence of (1) and (4), which we restate as the following corollary:
\begin{cor} \label{cor:pointwise_equivalent}

Let $\bm{w} \in L^\infty_{\leq 1}(\I_1 \times \I_1; \R^d)$, and $\bm{u} \in L^\infty_{\leq 1}(\I_2 \times \I_2; \R^d)$. Then the following are equivalent:
\begin{itemize}
\item[(1')] $\| \bm{t}(T,\bm{w}) - \bm{t}(T,\bm{u}) \|_{(\R^d)^{\otimes \mathsf{e}(T)}}$, $\forall T \in \mathcal{T}$

\item[(4')] $\gamma_{\square;\R^d}(\bm{w}, \bm{u}) = 0$.
\end{itemize}

\end{cor}

\begin{proof} [Proof of Theorem~\ref{thm:pointwise_equivalent}]
The implications (5) $\implies$ (1) $\implies$ (2) $\implies$ (3) follow directly from Proposition~\ref{prop:two_def_of_tree}, Corollary~\ref{cor:Stone_Weierstrass_tree_algebra}, and Corollary~\ref{cor:iterated_measure_kernel}.

For parts (3) $\implies$ (4) $\implies$ (5), the proofs are identical to those of Theorem~8.1 in \cite{grebik2022fractional}, so we refer the readers to the original literature and provide only a brief outline.

For (3) $\implies$ (4), the adjacency operator $T_\gamma: L^2(\I_2) \to L^2(\I_1)$ can be indentified from the canonical maps
\begin{equation*}
\begin{aligned}
L^2(\I_2) \to L^2(\I_2/\mathscr{C}(\bm{u})) \cong L^2(\I_1/\mathscr{C}(\bm{w})) \to L^2(\I_1).
\end{aligned}
\end{equation*}
It is straightforward to verify that $\gamma$ is a coupling and that the identity $T_{\langle \bm{e}, \bm{w} \rangle} \circ T_{\gamma} = T_{\gamma} \circ T_{\langle \bm{e}, \bm{u} \rangle}$ holds.

For (4) $\implies$ (5), this follows by a careful yet straightforward application of the Mean Ergodic Theorem to $(T_{\gamma} \circ T_{\gamma^T}): L^2(\I_1) \to L^2(\I_1)$ and $T_{\gamma^T} \circ T_{\gamma}: L^2(\I_2) \to L^2(\I_2)$, both of which are Markov operators.

\end{proof}

This pointwise equivalence can be extended to the topological equivalence in Lemma~\ref{lem:main_tree_counting}, following the argument suggested in \cite{boker2021graph}.

\begin{proof}[Proof of Lemma~\ref{lem:main_tree_counting}]
We begin by assuming $\mathcal{B} = \mathcal{H} = \mathbb{R}^d$ and prove the lemma by contradiction.

Suppose there exists a sequence $\bm{w}_n$ such that $\lim_{n \to \infty} \gamma_{\square; \mathbb{R}^d}(\bm{w}_n, \bm{w}) = 0$, but there exists a tree $T \in \mathcal{T}$ and an $\epsilon > 0$ such that 
\begin{equation*}
\begin{aligned}
\| \bm{t}(T,\bm{w}_n) - \bm{t}(T,\bm{w}) \|_{(\R^d)^{\otimes \mathsf{e}(F)}} > \epsilon.
\end{aligned}
\end{equation*}
By compactness of $\delta_{\square; \mathbb{R}^d}$, up to extracting a subsequence (which we still denote by $\bm{w}_n$ for simplicity), we can assume that $\bm{w}_n$ converges in $\delta_{\square; \mathbb{R}^d}$ (and in $\gamma_{\square; \mathbb{R}^d}$) to some limit $\bm{w}'$.
Since the limit in $\gamma_{\square; \mathbb{R}^d}$ is unique, we have $\gamma_{\square; \mathbb{R}^d}(\bm{w}', \bm{w}) = 0$. Thus, by Corollary~\ref{cor:pointwise_equivalent}, we conclude that
\begin{equation*}
\begin{aligned}
\| \bm{t}(T,\bm{w}') - \bm{t}(T,\bm{w}) \|_{(\R^d)^{\otimes \mathsf{e}(F)}} = 0.
\end{aligned}
\end{equation*}
Hence, on this extracted subsequence $\bm{w}_n$, as $n \to \infty$,
\begin{equation*}
\begin{aligned}
& \| \bm{t}(T,\bm{w}_n) - \bm{t}(T,\bm{w}) \|_{(\R^d)^{\otimes \mathsf{e}(F)}} = \| \bm{t}(T,\bm{w}_n) - \bm{t}(T,\bm{w}') \|_{(\R^d)^{\otimes \mathsf{e}(F)}} \to 0,
\end{aligned}
\end{equation*}
where the convergence follows from Lemma~\ref{lem:main_counting} (Counting Lemma associated with $\delta_{\square; \mathbb{R}^d}$) and the fact that $\delta_{\square; \mathbb{R}^d}(\bm{w}_n, \bm{w}') \to 0$. This leads to a contradiction.

For the converse direction, assume there exists a sequence such that for all $T \in \mathcal{T}$,
\begin{equation*}
\begin{aligned}
\lim_{n \to \infty} \| \bm{t}(T,\bm{w}_n) - \bm{t}(T,\bm{w}) \|_{(\R^d)^{\otimes \mathsf{e}(F)}} = 0,
\end{aligned}
\end{equation*}
but there exists an $\epsilon > 0$ such that for all $n \in \N$,
\begin{equation*}
\begin{aligned}
\gamma_{\square;\R^d}( \bm{w}_n, \bm{w}) > \epsilon.
\end{aligned}
\end{equation*}
Again, by the compactness of $\delta_{\square; \mathbb{R}^d}$, we can extract a subsequence (which we still denote by $\bm{w}_n$ for simplicity) that converges in $\delta_{\square; \mathbb{R}^d}$ to some $\bm{w}'$.
It is straightforward that for all $T \in \mathcal{T}$,
\begin{equation*}
\begin{aligned}
\| \bm{t}(T,\bm{w}') - \bm{t}(T,\bm{w}) \|_{(\R^d)^{\otimes \mathsf{e}(F)}} = 0,
\end{aligned}
\end{equation*}
hence, by Corollary~\ref{cor:pointwise_equivalent}, $\gamma_{\square; \mathbb{R}^d}(\bm{w}', \bm{w}) = 0$.
Hence, on this extracted subsequence $\bm{w}_n$, as $n \to \infty$,
\begin{equation*}
\begin{aligned}
\gamma_{\square;\R^d}( \bm{w}_n, \bm{w}) &= \gamma_{\square;\R^d}( \bm{w}_n, \bm{w}') \leq \delta_{\square;\R^d}( \bm{w}_n, \bm{w}') \to 0,
\end{aligned}
\end{equation*}
leading again to a contradiction.

\hfill

Finally, consider a general compact embedding $\mathcal{B} \subset \mathcal{H}$.
By compactness, there exists a sequence of finite-dimensional subspaces $\mathcal{H}_d$, with $d \to \infty$, such that for any $\bm{u} \in \mathcal{B} \subset \mathcal{H}$, the projection $P_{\mathcal{H}_d}$ satisfies
\begin{equation*}
\begin{aligned}
\|P_{\mathcal{H}_d}\bm{u} - \bm{u}\|_{\mathcal{H}} \leq C^{\downarrow}_{\mathcal{B},\mathcal{H}}(d) \|\bm{u}\|_{\mathcal{B}}.
\end{aligned}
\end{equation*}

Recall the techniques from the proof of Lemma~\ref{lem:counting_lemma} and Lemma~\ref{lem:inverse_counting_lemma_subspace}, we have, for any $\bm{w} \in L^\infty([0,1]^2; \mathcal{B})$ with $\|\bm{w}\|_{L^\infty([0,1]^2; \mathcal{B})} \leq w_{\max}$,
that 
\begin{equation*}
\begin{aligned}
\gamma_{\square;\mathcal{H}}( P_{\mathcal{H}_d} \bm{w}, \bm{w}) &\leq C^{\downarrow}_{\mathcal{B},\mathcal{H}, w_{\max}}(d),
\\
\gamma_{\square;\mathcal{H}}(P_{\mathcal{H}_d} \bm{w}_n, P_{\mathcal{H}_d} \bm{w}) &\leq \gamma_{\square;\mathcal{H}}(\bm{w}_n,  \bm{w}),
\end{aligned}
\end{equation*}
and, for all $T \in \mathcal{T}$,
\begin{equation*}
\begin{aligned}
\| \bm{t}(T,P_{\mathcal{H}_d} \bm{w}) - \bm{t}(T,\bm{w}) \|_{\mathcal{H}^{\otimes \mathsf{e}(T)}} &\leq C^{\downarrow}_{\mathcal{B},\mathcal{H}, w_{\max}, T}(d),
\\
\| \bm{t}(T,P_{\mathcal{H}_d}\bm{w}_n) - \bm{t}(T,P_{\mathcal{H}_d}\bm{w}) \|_{\mathcal{H}^{\otimes \mathsf{e}(T)}} &\leq \| \bm{t}(T,\bm{w}_n) - \bm{t}(T,\bm{w}) \|_{\mathcal{H}^{\otimes \mathsf{e}(T)}}.
\end{aligned}
\end{equation*}
This allows extending the finite-dimensional convergence to $\mathcal{B} \subset \mathcal{H}$ in both directions by passing to the limit $d \to \infty$.

\end{proof}

\end{document}